\documentclass[11pt]{article}

\pdfoutput=1

\usepackage[utf8]{inputenc}
\usepackage[T1]{fontenc}
\usepackage[british]{babel}
\usepackage{graphicx}
\usepackage[sectionbib,numbers]{natbib}
\usepackage{amstext,amsthm,amsmath,amssymb,mathtools,bbm,mathrsfs,mathtools}
\usepackage{url}
\usepackage{geometry}
\usepackage[colorlinks=false]{hyperref}
\usepackage{graphicx}
\usepackage{enumerate}
\usepackage{xcolor}

\usepackage{tikz}

\newtheorem{proposition}{Proposition}[section]
\newtheorem{theorem}[proposition]{Theorem}
\newtheorem{assumption}[proposition]{Assumption}%
\newtheorem{corollary}[proposition]{Corollary}
\newtheorem{lemma}[proposition]{Lemma}
\theoremstyle{definition}
\newtheorem{remark}[proposition]{Remark}
\newtheorem{constr}[proposition]{Construction}

\newcommand{\N}{\mathbb{N}}
\newcommand{\bZ}{\mathbb{Z}}
\newcommand{\Z}{\mathbb{Z}}
\newcommand{\R}{\mathbb{R}}
\renewcommand{\Pr}{\mathbb{P}}
\newcommand{\bE}{\mathbb{E}}

\newcommand{\abs}[1]{\lvert#1\rvert} 
\newcommand{\Abs}[1]{\left\lvert#1\right\rvert}
\newcommand{\norm}[1]{\lVert#1\rVert} 
\newcommand{\Norm}[1]{\big\lVert#1\big\rVert}

\newcommand{\ind}[1]{\mathbbm{1}_{\{#1\}}} 
\newcommand{\indset}[1]{\mathbbm{1}_{#1}} 

\newcommand\restr[2]{\ensuremath{\left.#1\right|_{#2}}} 

\newcommand{\cs}{{\mathsf{cs}}}
\newcommand{\dcs}{{\mathsf{dcs}}}
\newcommand{\cone}{{\mathsf{cone}}}

\newcommand{\base}{{\mathsf{bas}}}
\newcommand{\inn}{{\mathrm{inn}}}
\newcommand{\out}{{\mathrm{out}}}
\newcommand{\rf}{{\mathrm{ref}}}
\newcommand{\loc}{{\mathrm{loc}}}
\newcommand{\coupl}{{\mathrm{coupl}}}
\newcommand{\wt}{\widetilde}

\newcommand{\joint}{{\mathrm{joint}}}
\newcommand{\tube}{{\mathrm{tube}}}
\newcommand{\dtube}{{\mathrm{dtube}}}
\newcommand{\indi}{{\mathrm{ind}}}
\newcommand{\simi}{{\mathrm{sim}}}
\newcommand{\bP}{{\mathbb{P}}}
\newcommand{\compl}{\mathsf{C}}

\newcommand{\block}{{\mathsf{block}}}
\newcommand{\DC}{{\mathrm{DC}}}
\newcommand{\DTube}{{\mathrm{DTube}}}
\newcommand{\Tube}{{\mathrm{Tube}}}
\newcommand{\Lt}{L_{\mathrm{t}}}
\newcommand{\Ls}{L_{\mathrm{s}}}

\numberwithin{equation}{section}

\title{Quenched CLT for ancestral lineages of logistic branching
  random walks}
\author{M. Birkner, A. Depperschmidt, T. Schlüter}
\date{\today}

\begin{document}
\maketitle

\begin{abstract}
  We consider random walks in dynamic random environments which arise
  naturally as spatial embeddings of ancestral lineages in spatial
  locally regulated population models. In particular, as the main
  result, we prove the quenched central limit theorem for a random
  walk in dynamic random environment generated by time reversal of
  logistic branching random walks in a regime where the population
  density is sufficiently high.

  As an important tool we consider as auxiliary models random
  walks in dynamic random environments defined in terms of the
  time-reversal of oriented percolation. We show that the quenched
  central limit theorem holds if the influence of the random medium on
  the walks is suitably weak. The proofs of the quenched central limit
  theorems in these models rely on coarse-graining arguments and a
  construction of regeneration times for a pair of conditionally
  independent random walks in the same medium, combined with a
  coupling that relates them to a pair of independent random walks in
  two independent copies of the medium.
\end{abstract}


\section{Introduction}
\label{sec:introduction}

Ancestral lineages of individuals in many spatial population models
exhibiting some kind of non-trivial stable behaviour are known to
perform random walks whose increments of course depend on the
particular model. In the case of stepping stone models or (multi-type)
voter models on $\Z^d$ the local population size is constant and the
ancestral lineages are given by random walks which are homogeneous in
space and time. In spatial models with interactions and fluctuating
local population size the ancestral lineages are random walks in
rather complicated dynamic random environments. The investigation of
random walks in static and dynamic random environments under different
conditions on the environments is an active research area with a
lot of recent progress. In many cases limit theorems, such as laws of
large numbers, annealed and quenched (functional) central limit
theorems and others have been obtained.

In the present paper our main goal is to prove the quenched central
limit theorem for the \emph{directed} random walk on $\Z^d$ which
arises as an ancestral lineage in a particular spatial population
model, namely the \emph{logistic branching random walk} from
\cite{BirknerDepperschmidt2007}. For this random walk, considering
first as an auxiliary model a particular random walk on the backbone
of the cluster of supercritical oriented percolation on $\Z^d \times
\Z,$ the annealed central limit theorem was obtained in
\cite{BirknerCernyDepperschmidt2016}. Here we combine methods
from \cite{BirknerCernyDepperschmidtGantert2013} and
\cite{BirknerCernyDepperschmidt2016} to prove the corresponding
quenched central limit theorem.

As mentioned above, the random walk that we consider here is a
random walk in a dynamic random environment which itself is generated
by an individual based population model. In the literature there are
many examples of random walks in dynamic random environments generated
by various interacting particle systems; see for example
\cite{AvenaDosSantosVoellering:2013,AvenaJaraVoellering:2014,BethuelsenHeydenreich2015,HKS14,MV15,
  RedigVoellering:2013} and more recently \cite{BlondelHilarioetal2019,MiltonMenezes2020,BlondelHilarioetal2020,AvenaChinodaCostadenHollander2022}. In these models typically the random walk and
the particle system generating the dynamic environment have the same
`natural' forwards time direction. The main difference to the random walks we
consider here and which were also considered in
\cite{BirknerCernyDepperschmidt2016} and
\cite{BirknerCernyDepperschmidtGantert2013} is that in our case, due
to the interpretation of the walks as the ancestral lineages, the time
direction of the walk is the backwards direction of the population
model which generates the random environment. We refer to
\cite[Remark~1.7]{BirknerCernyDepperschmidtGantert2013} for a more detailed
discussion and further references.

\section{Models and main results}
\label{sec:main-results}

In this paper we are mainly interested in the behaviour of ancestral
lineages of a population model, namely the version of \emph{logistic
  branching random walk} introduced in
\cite{BirknerDepperschmidt2007}. In that model individuals live on the
integer lattice $\bZ^d$. In each generation (the Poissonian) branching
is subcritical in crowded regions and supercritical in not so densely
populated regions. After birth the offspring migrate in some
finite neighbourhood of their birth place. In a certain parameter
regime (we will be more precise in
Section~\ref{sec:logist-branch-rand} below) the logistic branching
random walk has a unique nontrivial ergodic invariant distribution and 
in particular there is a stationary version. For a realisation of this stationary
version of the process we can consider ancestral lineages which track
back the position of the ancestors of an individual sampled at random
from the present generation. In \cite{BirknerCernyDepperschmidt2016}
it was proven that such an ancestral lineage satisfies a law of large
numbers and an annealed central limit theorem. Here our goal is to
show that it also satisfies the quenched central limit theorem.

\medskip

In the following three subsections we briefly describe three models
and state the assumptions and the quenched central limit theorems in
the respective models.
\begin{itemize}
\item We start with the with the logistic branching random walk, our main
  model of interest.
\item The second model can be seen as an auxiliary toy model and is
  concerned with random walk in random environments induced by the
  (backbone of the) oriented percolation cluster.
\item The third model is a more abstract model which covers the
  logistic branching random walk. The assumptions for this model allow
  a comparison with oriented percolation which in turn allows us to
  use results that we obtain on the second model.
\end{itemize}

Our approach follows the one from
\cite{BirknerCernyDepperschmidt2016} where these three models were considered as well (in a
different order). We give a more detailed
outline of the proof strategy at the end of
Section~\ref{sec:model_in_oriented_percolation}.

In the following, we write $\norm{\cdot}$ for the $\sup$-norm on $\R^d$
or $\Z^d$, and $\norm{\cdot}_2$ denotes the Euclidean norm.

\subsection{Logistic branching random walk}
\label{sec:logist-branch-rand}

Here we give a formal description of the model \emph{logistic
  branching random walk}. Let
$p=(p_{xy})_{x,y \in \bZ^d} =(p_{y-x})_{x,y \in \bZ^d}$ be a symmetric
aperiodic stochastic kernel with a finite range $R_p \ge 1$.
Furthermore let $\lambda=(\lambda_{xy})_{x,y \in \bZ^d}$ be a
non-negative symmetric kernel satisfying
$0 \le \lambda_{xy} =\lambda_{0,y-x}$ and $\lambda_{00}>0$ and having
a finite range $R_\lambda$. We set $\lambda_0 \coloneqq \lambda_{00}$
and for a configuration $\zeta \in \R_+^{\bZ^d}$ and $x \in \bZ^d$ we
define
\begin{align}
  \label{def:f}
  f(x; \zeta) \coloneqq \zeta(x) \Big( m - \lambda_0 \zeta(x) -
  \sum_{z \ne x} \lambda_{xz} \zeta(z)  \Big)^+,
\end{align}
with $m>1$. \eqref{def:f} can be interpreted as the (conditional) mean
number of offspring produce at size $x$ given that the present
population configuration is $\zeta$. We consider a population process
$\eta\coloneqq (\eta_n)_{n \in \bZ}$ with values in $\bZ_+^{\bZ^d}$,
where $\eta_n(x)$ is the number of individuals at time $n \in \bZ$ at
site $x \in \bZ^d$. For a formal definition of $\eta$ let
\begin{align}
  \label{eq:Uxy}
  U \coloneqq \{U^{(y,x)}_n: n \in \Z,\; x, y \in \Z^d, \norm{x-y} \leq R_p\}
\end{align}
be a collection of independent Poisson processes on $[0,\infty)$ with
intensity measures of $U^{(y,x)}_n$ given by $p_{yx}\, dt$. For given
$\eta_n \in \Z_+^{\Z^d}$, we define $\eta_{n+1} \in \Z_+^{\Z^d}$ via
\begin{align}
  \label{eq:flowdef}
  \eta_{n+1}(x) \coloneqq \sum_{y \, : \, \norm{x-y} \leq R_p}
  U^{(y,x)}_n\big(f(y;\eta_n)\big),  \quad x\in \Z^d.
\end{align}
Considering the definition of $f$ in \eqref{def:f}, we see that in the
absence of competition an individual has on average $m$ offspring. For
obvious reasons we call $\lambda$ the \emph{competition} and $p$ the
\emph{migration kernel}. Note that for each $x$, the right-hand side
of \eqref{eq:flowdef} is a finite sum of (conditionally) Poisson
random variables $U^{(y,x)}_n\big(f(y;\eta_n)\big)$ with finite means
$p_{yx}f(y;\eta_n)$ which are bounded by $\norm{f}_\infty$. Using
standard superposition properties of Poisson random variables it
follows that
\begin{align}
  \label{eq:law of eta in log branching}
  \eta_{n+1} (x)  \sim  \mathrm{Pois}\Bigl( \sum_{y \in \Z^d}
  p_{yx} f(y;\eta_{n}) \Bigr), \quad x \in \Z^d
\end{align}
and that the $\eta_{n+1}(x)$ are conditionally independent given
$\eta_n$.

For $-\infty < k < n$ set
\begin{align}
  \label{eq:Gmn}
  \mathcal{G}_{k,n} \coloneqq \sigma(U_\ell^{(x,y)} : k \leq \ell < n,
  \, x, y \in \Z^d).
\end{align}
By iterating \eqref{eq:flowdef}, we can define a random family of
$\mathcal{G}_{k,n}$-measurable mappings
\begin{align}
  \label{eq:Phimn}
  \Phi_{k,n} : \Z_+^{\Z^d} \to \Z_+^{\Z^d}, \; -\infty < k<n \quad
  \text{such that} \quad \eta_n = \Phi_{k,n}(\eta_k).
\end{align}
To this end define $\Phi_{k,k+1}$ as in
\eqref{eq:flowdef} via
\begin{align}
  \label{eq:flowdef2}
  (\Phi_{k,k+1}(\zeta))(x)
  & \coloneqq \sum_{y \, : \, \norm{x-y} \leq R_p} \hspace{-0.8em}
    U^{(y,x)}_k\big(f(y;\zeta)\big) \quad \mbox{for $y \in \Z^d$ and
    $\zeta\in \Z_+^{\Z^d}$} \\ \intertext{and then put}
  \label{eq:flowdef2a}
  \Phi_{k,n}
  & \coloneqq \Phi_{n-1,n} \circ \cdots \circ \Phi_{k,k+1}.
\end{align}
Using these mappings we can define the dynamics of
$(\eta_n)_{n=k,k+1,\dots}$ simultaneously for all initial conditions
$\eta_k \in \Z_+^{\Z^d}$ for any $k \in \Z$. In
\cite{BirknerDepperschmidt2007} it is shown that for $m \in (1,3)$ and
suitable $\lambda$ there exists a unique non-trivial invariant
extremal distribution $\bar\nu$. Furthermore, conditioned on
non-extinction, the distribution of $\eta_n$ converges in the vague
topology to $\bar\nu$. The relevant conditions are the assumptions in
Theorem~3 from \cite{BirknerDepperschmidt2007} which we recall for
future reference.

\begin{assumption}[Sufficient conditions for survival and complete
  convergence of $\eta$]
  \label{ass:log_branch_walk}
  We assume
  \begin{enumerate}[(i)]
  \item $m\in(1,3)$,
  \item $\sum_x\lambda_{0x}$ is sufficiently small and
    $\sum_{x\neq 0} \lambda_{0x}\le a^*\lambda_0$ for $a^*=a^*(m,p)>0$
    as required in Theorem~3 from \cite{BirknerDepperschmidt2007}.
  \end{enumerate}
\end{assumption}
Under this assumption we can consider the stationary process
$\eta=(\eta_n)_{n\in\bZ}$ where $\eta_n$ is distributed according to
$\bar\nu$. One can enrich this model with genealogical information,
see Chapter~4 in \cite{Depperschmidt08} or Section~4 in
\cite{BirknerCernyDepperschmidt2016}. We set
\begin{align}
  \label{eq:anclinquedyn}
  p_{\eta}(k; x,y) \coloneqq
  \frac{p_{yx}f(y;\eta_{-k-1})}{\sum_z p_{zx}f(z;\eta_{-k-1})}, \quad x,y \in
  \Z^d, \, k \in \Z_+
\end{align}
with some arbitrary convention if the denominator is $0$. For a given
configuration $\eta$, conditioned on $\eta_0(0)>0$, let
$X\coloneqq (X_k)_{k=0,1,\dots}$ be a time-inhomogeneous Markov chain
with $X_0=0$ and
\begin{align}
  \label{eq:defX-real}
  \Pr(X_{k+1}=y \, | \, X_k=x,\eta) = p_{ \eta}(k; x,y).
\end{align}
With these dynamics the path $(X_k,-k)_{k =0,1,\dots}$ is the
space-time embedding of an ancestral lineage of an individual sampled
at the space-time origin, conditioned on the configuration $\eta$.
For fixed $\eta$, we write $P_\eta(\cdot)$ for the law of this process $X$
in the environment $\eta$ and $E_\eta[\cdot]$ for the corresponding expectations.

In \cite{BirknerCernyDepperschmidt2016} it was proven that under
Assumption~\ref{ass:log_branch_walk} the ancestral lineage satisfies a
law of large numbers and an annealed central limit theorem, see
Theorem~4.3 therein. We strengthen this result in the following
theorem.
\begin{theorem}[Quenched CLT]
  \label{thm:quenched_clt_log_branching}
  For $d\ge 1$, if the competition kernel $\lambda$ and the
  migration kernel $p$ satisfy Assumption~\ref{ass:log_branch_walk},
  the ancestral lineage $X$ satisfies a quenched central limit
  theorem, i.e.\ for any continuous and bounded function $f$ on
  $\R^d$, we have
  \begin{align}
    E_\eta[f(X_n/\sqrt{n})] \xrightarrow[]{n\to\infty} \Phi(f)
    \quad \text{for almost all }\eta,
  \end{align}
  where $\Phi(f)=\int_{\R^d}f(x)\Phi(dx)$ and $\Phi$ is a nontrivial
  centred $d$-dimensional normal distribution.
\end{theorem}

As mentioned above the proof of the theorem relies on a quenched CLT
for a more abstract class of models studied in
\cite{BirknerCernyDepperschmidt2016}. In particular,
Theorem~\ref{thm:quenched_clt_log_branching} is a corollary of the
quenched CLT for these abstract models; see
Remark~\ref{rem:log_branch_walk_part_of_abstract}.

\subsection{Random walks in oriented percolation: auxiliary model}
\label{sec:model_in_oriented_percolation}

Let us recall the `auxiliary model' from
\cite[Section~2]{BirknerCernyDepperschmidt2016}, which is built using
supercritical oriented percolation as a main ingredient. Let thus
$\omega \coloneqq \{\omega(x,n): (x,n) \in \bZ^d \times \bZ\}$ be a
family of i.i.d.\ Bernoulli random variables with parameter $p>0$. We
call a site $(x,n)$ \emph{open (with respect to $\omega$)} if
$\omega(x,n)=1$ and \emph{closed} if $\omega(x,n)=0$. If not stated
otherwise, in the following $\norm{\, \cdot\,}$ denotes the sup-norm.

For $m \le n$, we say that there is an \emph{open path} from $(x,m)$
to $(y,n)$ if there is a sequence $x_m,\dots, x_n\in \bZ^d$ such that
$x_m=x$, $x_n=y$, $\norm{x_k-x_{k-1}} \le 1$ for $k=m+1, \dots, n$ and
$\omega(x_k,k)=1$ for all $k=m,\dots,n$. In this case we write
$(x,m) \to^\omega (y,n)$, otherwise we write
$(x,m) \not\to^\omega (y,n)$. For sets $A ,B \subseteq \bZ^d$ and
$m \le n$ we write $A\times \{m\} \to^\omega B \times\{n\}$, if there
exist $x \in A$ and $y\in B$ so that $(x,m) \to^\omega (y,n)$. Here,
slightly abusing the notation, we use the convention that
$\omega(x,m)=\indset{A}(x)$ while for $k>m$ the $\omega(x,k)$ are
i.i.d.\ Bernoulli random variables as above. With this convention for
$A \subset \bZ^d$, $m \in \bZ$ we define the \emph{discrete time
  contact process} $\eta^A \coloneqq (\eta_n^A)_{n=m,m+1,\dots}$
driven by $\omega$ as
\begin{align}
  \label{eq:CPA}
  \eta_{m}^A = \indset{A} \quad \text{and} \quad  \eta_n^A(x)
  \coloneqq \ind{A \times \{m\} \to^\omega (x,n)},\; \; n > m.
\end{align}
Alternatively $\eta^A = (\eta_n^A)_{n=m,m+1,\dots}$ can be viewed as a
Markov chain on $\{0,1\}^{\bZ^d}$ with $\eta^A_m = \indset{A}$ and the
following local dynamics:
\begin{align}
  \label{eq:DCP-dyn}
  \eta_{n+1}^A(x) =
  \begin{cases}
    1 & \text{if $\omega(x,n+1)=1$ and $\eta_n^A(y)=1$ for some
      $y \in \bZ^d$ with $\norm{x-y} \le 1$}, \\
    0 & \text{otherwise}.
  \end{cases}
\end{align}

We say that the process \emph{dies out} if it ever hits the (unique)
absorbing state consisting only of zeros. Otherwise we say that the
process \emph{survives}. It is well known that there is a critical
percolation parameter $p_c\in (0,1)$ such that for $p>p_c$ and any
non-empty $A \subset \bZ^d$ the process $\eta^A$ survives with
positive probability. Furthermore, in this case there is a unique
non-trivial extremal invariant measure $\nu$, referred to as the
\emph{upper invariant measure}, such that, starting at any time
$m \in \bZ$ the distribution of $\eta^{\bZ^d}_n$ converges to $\nu$ as
$n\to \infty$. We assume $p>p_c$ throughout the following. Given a
configuration $\omega \in \{0,1\}^{\bZ^d \times \bZ}$, we define the
\emph{stationary discrete time contact process} driven by $\omega$ as
\begin{align}
  \label{eq:CP}
  \eta \coloneqq (\eta_n)_{n \in \bZ}
  \coloneqq \{\eta_n(x): x\in \bZ^d, n\in \bZ \}
  \quad \text{with} \quad \eta_n(x)
  \coloneqq \ind{\bZ^d \times \{-\infty\} \to^\omega (x,n)}.
\end{align}
The event on the right hand side should be understood as
$\cap_{m\le n} \big\{\bZ^d \times \{m\} \to^\omega (x,n)\big\}$.

To define a random walk in the random environment generated by
$\eta$, or more precisely by its time-reversal, let
\begin{align}
  \label{eq:Def.kappa}
  \kappa \coloneqq \bigl\{\kappa_n(x,y) : {n \in \bZ, \,x,y \in
  \bZ^d}\bigr\}
\end{align}
be a family of random transition kernels defined on the same
probability space as $\eta$, in particular, with probability one
$\kappa_n(x, \,\cdot\, ) \ge 0$ and
$\sum_{y \in \bZ^d} \kappa_n(x,y)=1$ holds for all $n\in \bZ$ and
$x \in \bZ^d$. Given $\kappa$, we consider a $\bZ^d$-valued random
walk, i.e.\ a time-inhomogeneous Markov chain
$X\coloneqq (X_n)_{n =0,1,\dots }$ with $X_0=0$ and transition
probabilities given by
\begin{align}
  \label{eq:defX}
  \Pr\big(X_{n+1}=y \, \big| \, X_n=x, \kappa \big) = \kappa_n(x,y),
\end{align}
that is, the random walk at time $n$ takes a step according to the
kernel $\kappa_n(x,\,\cdot\,)$ if $x$ is its position at time $n$. We
recall from \cite{BirknerCernyDepperschmidt2016} the following four
assumptions on the distribution of $\kappa$.

\begin{assumption}[Locality]
  The \label{ass:local} transition kernels in the family $\kappa $
  depend locally on the time-reversal of $\eta$, that is for some
  fixed $R_\loc \in \N$
  \begin{align}
    \label{eq:ass:local}
    \kappa_n(x,\cdot) \; \text{ depends only on } \; \big\{\omega(y,-n),
    \eta_{-n}(y) : \norm{x-y} \leq R_\loc\big\}.
  \end{align}
\end{assumption}
Note that \eqref{eq:ass:local} is in keeping with the interpretation of
$X_n$ as the spatial embedding of an ancestral lineage: Forwards time for the
walk corresponds to backwards time for the underlying population process,
i.e.\ backwards time in the space-time environment.

\begin{assumption}[Closeness to a symmetric reference measure on
  $\eta_{-n}(x)=1$]
  There \label{ass:appr-sym} is a deterministic symmetric probability
  measure $\kappa_\rf$ on $\bZ^d$ with finite range $R_\rf \in \N$,
  that is $\kappa_\rf(x)=0$ if $\norm{x} > R_\rf$, and a suitably
  small $\varepsilon_\rf >0$ such that
  \begin{align}
    \label{eq:ass:appr-sym}
    \norm{ \kappa_n(x,x+\,\cdot\,) -
    \kappa_\rf(\,\cdot\,) }_{\mathrm{TV}} <
    \varepsilon_\rf   \quad \text{whenever} \quad \eta_{-n}(x)=1.
  \end{align}
  Here $\norm{\, \cdot \, }_{\mathrm{TV}}$ denotes the total variation
  norm.
\end{assumption}

\begin{assumption}[Space-time shift invariance and spatial point
  reflection invariance]
  The \label{ass:distr-sym} kernels in the family $\kappa$ are
  shift-invariant on $\bZ^d \times \bZ$, that is, using the notation
  \begin{align*}
    \theta^{z,m} \omega (\,\cdot\,,\,\cdot\,)=\omega
    (z+\,\cdot\,,m+\,\cdot\,),
  \end{align*}
  we have
  \begin{align}
    \label{eq:sp-time-stat}
    \kappa_n(x,y)(\omega ) = \kappa_{n+m}
    (x+z,y+z)(\theta^{z,m}\omega).
  \end{align}
  Moreover, if $\varrho$ is the spatial point reflection operator
  acting on $\omega$, i.e., $\varrho \omega (x,n) = \omega(-x,n)$ for
  any $n \in \bZ$ and $x \in \bZ^d$, then
  \begin{align}
    \label{eq:point-refl-inv}
    \kappa_n(0,y)(\omega )=\kappa_n(0,-y)(\varrho \omega).
  \end{align}
\end{assumption}

\begin{assumption}[Finite range]
  There \label{ass:finite-range} is $R_\kappa \in \N$ such that a.s.\
  \begin{align}
    \label{eq:ass:kappa2}
    \kappa_n(x,y) = 0 \quad \text{whenever \;} \norm{y-x}>R_\kappa.
  \end{align}
\end{assumption}

\medskip

We write $P_\omega$ for the conditional law of $\Pr$, given $\omega$
and $E_\omega$ for the corresponding expectation. In particular, $X$
is a time-inhomogeneous Markov chain under $P_\omega$ with transition
probabilities $P_\omega(X_{n+1}=y| X_n =x) = \kappa_n(x,y)$. Note that
by \eqref{eq:ass:local}, $\kappa_n$ is a function of $\omega$, but is
in principle non-local (through its dependence on $\eta$). Our main
result for this class of models is the following theorem which (in a
part of the parameter region) strengthens the corresponding annealed
CLT from \cite{BirknerCernyDepperschmidt2016}.

\begin{theorem}[Quenched CLT for the auxiliary model]
  For any $d\ge 1$ one \label{thm:LLNuCLTmodel1} can choose
  $0 < \varepsilon_\rf$ sufficiently small and $p$ sufficiently close
  to $1$, so that if $\kappa$ satisfies
  Assumptions~\ref{ass:local}--\ref{ass:finite-range} then $X$
  satisfies a quenched central limit theorem with non-trivial
  covariance matrix, that is, for any continuous and bounded function
  $f$ on $\R^d$ we have
  \begin{align}
    \label{eq:32}
    E_\omega\Bigl[f(X_n/\sqrt{n})\Bigr] \xrightarrow{n \to\infty}
    \Phi(f) \quad \text{for almost all $\omega$},
  \end{align}
  where $\Phi(f) = \int_{\R^d} f(x) \, \Phi(dx)$ and $\Phi$ is a
  nontrivial centred $d$-dimensional normal distribution.
\end{theorem}

\noindent
The limit law $\Phi$ in \eqref{eq:32} necessarily agrees with the
limit law from Theorem~2.6 in \cite{BirknerCernyDepperschmidt2016}.
The proof of Theorem \ref{thm:LLNuCLTmodel1} is given in
Section~\ref{sect:qCLTd>=2-1} for the case $d\ge 2$ and in
Section~\ref{sec:dimension1} we complete the proof for the case $d=1$.

\subsection{General model}
\label{sec:mgenermod}

As in the case of logistic branching random walk, our general
motivation is to study the dynamics of `ancestral lineages' of spatial
populations evolving on $\bZ^d$ which we will also denote by
$\eta =(\eta_n)_{n\in \bZ}$ slightly abusing the notation. In the
abstract setting we need to make a few assumptions on the environment
and the random walk itself to get some control on the behaviour of the
`ancestral lineages'.

We start by defining the environment. Let
\begin{align}
  \label{eq:defn_U}
  U\coloneqq \{U(x,n): x\in\Z^d, n\in\Z\}
\end{align}
be an i.i.d.\ random field, $U(0,0)$ taking values in some Polish
space $E_U$ ($E_U$ could be $\{-1,+1\}$, $[0,1]$, a path space, etc.).
Furthermore for $R_{ \eta} \in \N$ let
$B_{R_{ \eta}} = B_{R_{ \eta}}(0) \subset \Z^d$ be the ball of radius
$R_{ \eta}$ around $0$ with respect to $\sup$-norm. Let
\begin{align*}
  \varphi : \Z_+^{B_{R_{  \eta}}} \times E_U^{B_{R_{\eta}}} \to \Z_+
\end{align*}
be a measurable function, which describes the random evolution of the
population process (i.e., the environment for the walk), as detailed below.
We make the following assumption on the environment.

\begin{assumption}[Markovian, local dynamics, flow construction]
  \label{ass:markovdyn}
  We assume that $\eta \coloneqq (\eta_n)_{n \in \Z}$ is a Markov
  chain with values in $\Z_+^{\Z^d}$ whose evolution is local in the
  sense that $\eta_{n+1}(x)$ depends only on $\eta_n(y)$ for $y$ in a
  finite ball around $x$. In particular we assume that $\eta$ can be
  realised using the `driving noise' $U$ as
  \begin{align}
    \label{eq:etadynabstr}
    \eta_{n+1}(x) = \varphi\big( \restr{\theta^x \eta_n}{B_{R_{\eta}}},
    \restr{\theta^x U(\,\cdot\,,n+1)}{B_{R_{\eta}}} \big), \quad
    x\in\Z^d, \; n \in \Z.
  \end{align}
  Here $\theta^x$ denotes the spatial shift by $x$, i.e.,
  $\theta^x \eta_n(\,\cdot\,) = \eta_n(\,\cdot\,+x)$ and
  $\theta^x U(\,\cdot\,, n+1) = U(\,\cdot\,+x,n+1)$. Furthermore
  $\restr{\theta^x \eta_n}{B_{R_{ \eta}}}$ and
  $\restr{\theta^x U(\,\cdot\,,n+1)}{B_{R_{\eta}}}$ are the
  corresponding restrictions to the ball $B_{R_{\eta}}$.
\end{assumption}

Note that \eqref{eq:etadynabstr} defines a flow, in the sense that
given a realisation of $U$ we can construct $\eta$ simultaneously for
all starting configurations. In most situations we have in mind the
constant zero configuration $\underline{0} \in \Z_+^{\Z^d}$ is an
equilibrium for $\eta$, that is,
\begin{align*}
  \varphi\big(\restr{\underline{0}}{B_{R_{\eta}}}, \, \cdot\, \big)
  \equiv 0,
\end{align*}
and there is another non-trivial equilibrium. It will be a consequence
of our assumptions that the latter is in fact the unique non-trivial
ergodic equilibrium.

Moreover we make an assumption that allows us to compare the
environment to oriented percolation in a suitable way. Let
$L_{\mathrm{s}} \in \N$ and $L_{\mathrm{t}}\in\N$ be a spatial and a temporal
scale, respectivly. We consider space-time boxes
whose `bottom parts' are centred at points in the coarse-grained grid
$L_{\mathrm{s}} \Z^d \times L_{\mathrm{t}}\Z$. We typically think of
$L_{\mathrm{t}} > L_{\mathrm{s}} \gg R_\eta$ and for
$(\tilde{x}, \tilde{n}) \in \Z^d \times \Z$ we set
\begin{align}
  \mathsf{block}_m(\tilde{x},\tilde{n}) \coloneqq \big\{ (y,k)
  \in \Z^d \times \Z \, : \,
  \norm{y-L_{\mathrm{s}} \tilde{x}} \le
  m L_{\mathrm{s}}, \tilde{n} L_{\mathrm{t}} < k \le
  (\tilde{n}+1) L_{\mathrm{t}}\big\}.
\end{align}
We abbreviate $\mathsf{block}(\tilde{x},\tilde{n}) \coloneqq
\mathsf{block}_1(\tilde{x},\tilde{n})$ and for a set $A \subset \Z^d
\times \Z$, slightly abusing the notation, we denote by $\restr{U}{A}$
the restriction of the random field $U$ to $A$.

In Sections~\ref{sect:qCLTd>=2-1} and \ref{sec:dimension1} we will
prove a quenched CLT for a random walk on oriented percolation and
transfer the results therein via the following assumption to the
abstract setting.

\begin{assumption}[`Good' noise configurations and propagation of
  coupling]
  There \label{ass:coupling} exist a finite set of `good' local
  configurations $G_{\eta} \subset \Z_+^{B_{2 L_{\mathrm{s}}}(0)}$ and
  a set of `good' local realisations of the driving noise
  $G_U \subset E_U^{B_{4 L_{\mathrm{s}}}(0) \times
    \{1,2,\dots,L_{\mathrm{t}}\}}$ with the following properties:
  \begin{itemize}
  \item For a suitably small $\varepsilon_U$,
    \begin{align}
      \label{eq:goodblockprob}
      \Pr\big( \restr{U}{\mathsf{block}_4(0,0)} \in G_U \big)
      \ge 1-\varepsilon_U.
    \end{align}
  \item For any $(\tilde{x}, \tilde{n}) \in \Z^d \times \Z$ and any
    configurations
    $\eta_{\tilde{n} L_{\mathrm{t}}}, \eta'_{\tilde{n} L_{\mathrm{t}}} \in
    \Z_+^{\Z^d}$ at time $\tilde{n} L_{\mathrm{t}}$,
    \begin{align}
      \label{eq:contraction}
      \begin{split}
        & \restr{\eta_{\tilde{n} L_{\mathrm{t}}}}{B_{2 L_{\mathrm{s}}}
          (L_{\mathrm{s}} \tilde{x})}, \, \restr{\eta'_{\tilde{n}
            L_{\mathrm{t}}}}{B_{2 L_{\mathrm{s}}} (L_{\mathrm{s}}
          \tilde{x})} \in G_{\eta} \quad \text{and} \quad
        \restr{U}{\mathsf{block}_4(\tilde{x},\tilde{n})} \in G_U  \\[+0.5ex]
        & \Rightarrow \;\; \eta_{(\tilde{n}+1) L_{\mathrm{t}}}(y) =
        \eta'_{(\tilde{n}+1) L_{\mathrm{t}}}(y) \quad \text{for all $y$
          with} \; \norm{y-L_{\mathrm{s}}
          \tilde{x}} \le 3 L_{\mathrm{s}}  \\
        & \qquad \text{and} \quad \restr{\eta_{(\tilde{n}+1)
            L_{\mathrm{t}}}}{B_{2 L_{\mathrm{s}}} (L_{\mathrm{s}}
          (\tilde{x}+\wt{e}))} \in G_{\eta} \; \text{for all $\wt{e}$
          with} \; \norm{\wt{e}} \le 1,
      \end{split}
      \\ \intertext{and}
      \label{eq:propagation.coupling}
        & \restr{\eta_{\tilde{n} L_{\mathrm{t}}}}{B_{2 L_{\mathrm{s}}}(L_{\mathrm{s}} \tilde{x})}
          =  \restr{\eta'_{\tilde{n} L_{\mathrm{t}}}}{B_{2 L_{\mathrm{s}}}(L_{\mathrm{s}} \tilde{x})}
          \quad \Rightarrow \quad
          \eta_k(y) = \eta'_k(y) \;\; \text{for all} \; (y,k) \in
          \mathsf{block}(\tilde{x},\tilde{n}),
    \end{align}
    where $\eta=(\eta_n)$ and $\eta'=(\eta'_n)$ are given by
    \eqref{eq:etadynabstr} with the same $U$ but possibly different
    initial conditions.
  \item There is a fixed (e.g., $L_{\mathrm{s}}$-periodic or even
    constant in space) reference configuration
    $\eta^{\mathrm{ref}} \in \Z_+^{\Z^d}$ such that
    $\restr{\eta^{\mathrm{ref}}}{B_{2 L_{\mathrm{s}}}(L_{\mathrm{s}}
      \tilde{x})} \in G_{\eta}$ for all $\tilde{x} \in \Z^d$.
        \end{itemize}
\end{assumption}

\begin{remark}
  \label{rem:eta-top_eta-bot_independent}
  One consequence of the above assumption is that the configurations
  of $\eta$ at the bottom of a space-time block
  $\block(\tilde{x},\tilde{n})$ is independent of the configuration of
  $\eta$ at the top of the same space-time block, conditioned on both
  being a good configuration, i.e.\ in $G_\eta$, and
  $U\vert_{\block(\tilde{x},\tilde{n})}\in G_U$. The proof can be
  found in Section~\ref{sec:abstract_proofs} below,
  see Equations~\eqref{eq:eta_rem_1} and \eqref{eq:eta_rem_3}.
\end{remark}

In \cite{BirknerCernyDepperschmidt2016}, Birkner et al.\ show the
connection between $\eta$ and oriented percolation which we recall
here; see Lemma~3.5 in \cite{BirknerCernyDepperschmidt2016}.

\begin{lemma}[Coupling with oriented percolation]
  \label{lem:abstr-OCcoupl}
  Put
  \begin{align}
    \wt{U}(\tilde{x}, \tilde{n}) \coloneqq
    \ind{\restr{U}{\textstyle
    \mathsf{block}_4(\tilde{x},\tilde{n})} \in G_U},
    \quad (\tilde{x}, \tilde{n}) \in \Z^d \times \Z.
  \end{align}
  If $\varepsilon_U$ is sufficiently small, we can couple
  $\wt U(\tilde{x}, \tilde{n})$ to an i.i.d.\ Bernoulli random field
  $\wt{\omega}(\tilde{x}, \tilde{n})$ with
  $\Pr(\wt{\omega}(\tilde{x}, \tilde{n})=1) \ge 1 -
  \varepsilon_{\wt{\omega}}$ such that $\wt{U} \ge \wt{\omega}$, and
  $\varepsilon_{\wt{\omega}}$ can be chosen small (how small depends
  on $\varepsilon_U$, of course).

  \noindent
  Moreover, the process $\eta$ then has a unique non-trivial ergodic
  equilibrium and one can couple a stationary process
  $\eta = (\eta_n)_{n\in\Z}$ with $\eta_0$ distributed according to
  that equilibrium with $\wt \omega $ so that
  \begin{align}
    \label{e:tildeG}
    \wt{G}(\tilde{x}, \tilde{n}) \coloneqq
    \wt U(\tilde{x},\tilde{n})
    \ind{\restr{\eta_{\tilde{n} L_{\mathrm{t}}}}{B_{2
    L_{\mathrm{s}}}(L_{\mathrm{s}} \tilde{x})}\in G_\eta}
    \ge \tilde{\xi}(\tilde{x}, \tilde{n}),
    \quad (\tilde{x}, \tilde{n}) \in \Z^d \times \Z
  \end{align}
  where
  $\tilde{\xi}\coloneqq \{\tilde{\xi}(\tilde{x}, \tilde{n}): \tilde{x}
  \in \Z^d, \tilde{n} \in \Z\}$ is the discrete time contact process
  driven by the random field $\wt \omega$, i.e.\ $\wt\xi$ is defined
  by
  \begin{align}
    \label{eq:CP-coarse-grained}
    \tilde{\xi}(\tilde{x}, \tilde{n}) \coloneqq  \ind{\Z^d
    \times \{-\infty\} \to^{\wt \omega} (x,n)}.
  \end{align}
\end{lemma}

Finally we need to make a technical assumption on the environment.

\begin{assumption}[Irreducibility on $G_{\eta}$]
  On \label{ass:irred} the event $\{ \wt{G}(\tilde{x}, \tilde{n}) = 1 \}$, conditioned on
  seeing a particular local configuration $\chi \in \Z_+^{B_{2
      L_{\mathrm{s}}}(L_{\mathrm{s}} \tilde{x})} \cap G_{\eta}$ at the
  bottom of the space-time box [time coordinate $\tilde{n}
  L_{\mathrm{t}}$], every configuration $\chi' \in G_{\eta}$ has a
  uniformly positive chance of appearing at the top of the space-time
  box [time coordinate $(\tilde{n}+1) L_{\mathrm{t}}$].
\end{assumption}

Next we describe our assumptions for the random walk itself. We start by
defining the random walk $X=(X_k)_{k\in\N_0}$ in the random
environment generated by $\eta$. Let
$\widehat U \coloneqq (\widehat{U}(x,k) : x \in \Z^d, k \in \Z_+)$ be
an independent space-time i.i.d.\ field of random variables uniformly
distributed on $(0,1)$. Furthermore let
\begin{align}
  \label{eq:phix}
  \varphi_{ X} : \Z_+^{B_{R_{ X}}} \times \Z_+^{B_{R_{ X}}}
  \times [0,1] \to B_{R_{ X}}
\end{align}
a measurable function, where $R_{ X} \in \N$ is an upper bound on
the jump size as well as on the dependence range. Given $\eta$, let
$X_0=0$ and put
\begin{align}
  \label{eq:abstr-walk}
  X_{k+1} \coloneqq X_k
  + \varphi_{ X}\big(\restr{\theta^{X_k} \eta_{-k}}{B_{R_{ X}}},
  \restr{\theta^{X_k} \eta_{-k-1}}{B_{R_{ X}}} , \widehat{U}(X_k,k) \big),
  \quad k =0,1,\dots.
\end{align}
Note that again forwards time direction for $X$ is backwards time
direction for $\eta$. This is consistent with the interpretation of
$X$ as an ancestral lineage.

The first assumption on the random walk is useful to control the
behaviour of $X$ along the coarse-grained space-time grid
$L_{\mathrm{s}}\bZ^d\times L_{\mathrm{t}}\bZ$. This assumption
corresponds to Assumption~3.9 in \cite{BirknerCernyDepperschmidt2016}
but we require here more precise control of the quenched law on
``good'' space time boxes. In this form the assumption is well aligned
with Assumption~\ref{ass:appr-sym} from the toy model above.

\begin{assumption}[Closeness to symmetric random walk on $\wt{G}=1$]
  \label{ass:abstract_closeness_to_symmetric_transition_kernel}
  Let $\kappa_{L_{\mathrm{t}}}$ be a deterministic, symmetric,
  non-degenerate transition kernel with finite range
  $R_{L_{\mathrm{t}}}\le RL_{\mathrm{t}}$ for some $R >0$. There exist
  $\varepsilon_{\mathrm{symm}}>0$ such that for all
  $(\tilde{x},\tilde{n})\in\bZ^d\times\bZ$ and all
  $z\in B_{L_{\mathrm{s}}}(\tilde{x}L_{\mathrm{s}})$
  \begin{align}
    \label{eq:assumptions_close_to_symmetric_kernel}
    \norm{\bP(X_{\tilde{n}L_{\mathrm{t}}}=\cdot + z
    \,\vert\, X_{(\tilde{n}-1)L_{\mathrm{t}}}=z,
    \wt{G}(\tilde{x},-\tilde{n})=1,\eta)
    - \kappa_{L_{\mathrm{t}}}(\cdot)}_{\mathrm{TV}}\le
    \varepsilon_{\mathrm{symm}}.
  \end{align}
\end{assumption}
Note that the condition on $R_{L_{\mathrm{t}}}$ easily holds whenever
$X$ itself has finite range.

Finally, we make an assumption that ensures that the random walk has
$0$ speed with respect to the annealed law.
\begin{assumption}[Symmetry of $\varphi_{ X}$ w.r.t.\ point
  reflection]
  \label{ass:abstract_symmetry_of_point_reflection}
  Let $\varrho$ be the (spatial) point reflection operator acting on
  $\eta$, i.e., $\varrho \eta_k(x) = \eta_k(-x)$ for any $k \in \Z$
  and $x \in \Z^d$. We assume
  \begin{align}
    \label{eq:abstr-symm}
    \varphi_{ X}\bigl(\restr{\varrho
    \eta_0}{B_{R_{ X}}},\restr{\varrho \eta_{-1}}{B_{R_{ X}}},
    \widehat{U}(0,0)\bigr) = -
    \varphi_{ X} \bigl(\restr{\eta_0}{B_{R_{X}}},\restr{\eta_{-1}}{B_{R_{ X}}},
    \widehat{U}(0,0)\bigr).
  \end{align}
\end{assumption}

\medskip

Under these assumptions we obtain a quenched CLT for any dimension
$d\ge 1$. Its proof is given in Section~\ref{sec:abstract_proofs}.

\begin{theorem}[Quenched CLT for a more general class of environments]
  \label{thm:abstract_quenched_clt}
  Let the random environment $\eta$ and random walks in it satisfy
  Assumption~\ref{ass:markovdyn} through
  \ref{ass:abstract_symmetry_of_point_reflection} with $\varepsilon_U$
  from \eqref{eq:goodblockprob} and $\varepsilon_{\mathrm{symm}}$ in
  \eqref{eq:assumptions_close_to_symmetric_kernel} sufficiently small.
  Then the random walk $X$ satisfies a quenched central limit theorem
  with non-trivial covariance matrix for any $d\ge 1$.
\end{theorem}
How small exactly $\varepsilon_U$ and $\varepsilon_{\mathrm{symm}}$
need to be chosen will depend on the other parameters of the model and
we do not strive here to work out explicit bounds.

\begin{remark}[Theorem~\ref{thm:quenched_clt_log_branching} as a corollary]
  In \label{rem:log_branch_walk_part_of_abstract}
  \cite{BirknerCernyDepperschmidt2016} 
  it is proved that under Assumption~\ref{ass:log_branch_walk} the
  model of the logistic branching random walks are part of the more
  abstract class of models; see Proposition~4.7 therein. In particular
  in \emph{Step 5} in the proof of that proposition it is in fact
  verified that the property
  \eqref{eq:assumptions_close_to_symmetric_kernel} from
  Assumption~\ref{ass:abstract_closeness_to_symmetric_transition_kernel}
  of the present paper (which is formally stronger than Assumption~3.9
  in \cite{BirknerCernyDepperschmidt2016}) is satisfied by the
  logistic branching random walk. Thus,
  Theorem~\ref{thm:quenched_clt_log_branching} follows once we prove
  Theorem~\ref{thm:abstract_quenched_clt}.
\end{remark}

\begin{remark}
  We use the techniques developed here in a companion paper to study a
  pair of coalescing ancestral lineages in the logistic branching
  random walk model for $d=2$. More specifically we study the
  asymptotic behaviour of the time of coalescence for two ancestral
  lineages and show that
  \begin{equation*}
    \lim\limits_{N\to \infty} \bP_{Nx}(\tau_{\mathrm{coal}}>N^{2\gamma}) =\frac{1}{\gamma},
  \end{equation*}
  for $\gamma>1$, where we interpret $\bP_{Nx}$ as the distribution
  where one lineage is started at the origin and the other at site
  $Nx$ for some $x\in \mathbb{R}^2\setminus \{0\}$.
\end{remark}

\paragraph{Outline and proof strategies}

In Section~\ref{sec:preparations} we provide a framework and the
necessary tools to prove the quenched CLT
Theorem~\ref{thm:LLNuCLTmodel1} for the auxiliary model from
Section~\ref{sec:model_in_oriented_percolation}.
The proof is long and quite technical. We give a detailed `roadmap'
of its steps at the beginning of Section~\ref{sec:preparations}.
In the following, we provide a brief overview of the main arguments.

We follow ideas used in \cite{BirknerCernyDepperschmidtGantert2013}
and extend the regeneration construction introduced in
\cite{BirknerCernyDepperschmidt2016} to two random walks. An essential
tool in the construction of regeneration times of \emph{one} random
walk in \cite{BirknerCernyDepperschmidt2016} was a cone based at the
current position of the random walk. We extent this construction to a
double cone to encompass a setting with two random walks (see
Figure~\ref{fig:dcones}), with the goal to isolate the interior of the
double cone from the exterior via the double cone shell.
Furthermore we show that the basic properties
of the regeneration construction from
\cite{BirknerCernyDepperschmidt2016} for a single walker, and thus a
single cone, translate to our two walker setting and the double cone.

In Construction~\ref{constr:reg-time} we detail how to find
regeneration times for a pair of random walks $(X,X')$ evolving in the
same environment and the conditions that have to be met at those
times. We call this pair the $\joint$-pair and want to compare its
behaviour with another pair $(X,X'')$ evolving in two independent
environments, called $\indi$-pair. A more detailed description and how
we think of those two pairs can be found above
Proposition~\ref{prop:TVdistance-joint-ind-1step}.
In Propositions~\ref{prop:JointRegTimesBound} and
\ref{prop:IndRegTimesBound} we obtain tail bounds for the regeneration
times of both pairs which allow us to couple the $\joint$-pair with
the $\indi$-pair with a probability depending on the starting
distance, see Proposition~\ref{prop:TVdistance-joint-ind-1step}.

\medskip

\noindent
In Sections~\ref{sect:qCLTd>=2-1}--\ref{sec:dimension1} we prove the
quenched CLT for the auxiliary model, i.e.\
Theorem~\ref{thm:LLNuCLTmodel1}. The main tool is
Proposition~\ref{prop:2nd-mom}, which, in combination with Markov's
inequality and Lemma~\ref{lemma:convergence for subsequence to full
  sequence}, yields a quenched CLT along the regeneration times; see
display \eqref{eq:52} and below. To prove
Proposition~\ref{prop:2nd-mom} in $d\ge 2$ we use the fact that, by
Proposition~\ref{prop:TVdistance-joint-ind-1step}, the $\joint$-pair
of random walks behaves similarly to the $\indi$-pair whenever
the two walks are far enough apart from each other.

This requires different approaches for dimensions $d\ge 2$ and $d=1$,
due to the fact that in the case $d=1$ the random walks spend more
time close to each other and we have to do more careful calculations.

Starting with $d\ge 2$ we prove in Lemma~\ref{lem:coupl} that the
number of steps where a coupling of $(X,X')$ and $(X,X'')$ fails is
suitably small with high probability. Since the success of the coupling
depends on the starting distance, we make use of the Separation
lemma~\ref{lem:separ} to see that a fast separation leads to a high
proportion of coupled steps.

To finish the proof of Theorem~\ref{thm:LLNuCLTmodel1} for $d\ge 2$ it
remains to show that the convergence along the regeneration times
transfers to the original random walk, i.e.\ we need to show that the
random walk is well behaved between two consecutive regeneration
times.

\medskip

\noindent
In Section~\ref{sec:dimension1} we provide the necessary results to
prove Proposition~\ref{prop:2nd-mom} for $d=1$. As mentioned above in
the one dimensional case the random walkers are close to each other
too often which requires another approach. We again study a pair of
random walks evolving in the same environment along their joint
regeneration times. By
Proposition~\ref{prop:TVdistance-joint-ind-1step} we can couple this
pair to another where the walkers evolve in independent copies of the
environment. Here the number of steps at which the random walkers are
too close to use the coupling argument is of order $\sqrt{n}$, that
means the error accrued during those steps doesn't vanish on the
diffusive scale without further calculations. This results in a
``weaker'' version of a coupling lemma, see Lemma~\ref{lem:coupl_d=1}.
Note that this weaker version of coupling seems to present a problem
in proving the properties of the processes we needed to prove
Proposition~\ref{prop:2nd-mom} and Theorem~\ref{thm:LLNuCLTmodel1} for
$d\ge 2$. However, since we can choose $\varepsilon_\rf$ small it turns
out that the proofs are robust enough and we can transfer all results
to $d=1$. The consequence is that we have to choose $\beta$ from
Proposition~\ref{prop:TVdistance-joint-ind-1step} very large. We show
that the random walks along the joint regeneration times are close to
martingales, that is the previsible part of these processes vanishes in
expectation on the diffusive scale, see
Lemma~\ref{lem:bound_for_A1_and_A2}. We use a martingale CLT to prove
Lemma~\ref{lem:key lemma for d=1}, which is the key lemma to prove
Proposition~\ref{prop:2nd-mom} for $d=1$.

\medskip

In Section~\ref{sec:abstract_proofs} we tackle the proofs of the
general model. This class of random environments, due to the
assumptions we make, see Assumption~\ref{ass:markovdyn} and
Assumption~\ref{ass:coupling}, behaves similarly to the oriented
percolation model when viewed over large scales.
More specifically the environment is Markovian,
evolves in the positive time direction and has local dynamics.
Furthermore, Assumption~\ref{ass:coupling} allows us to couple the
environment on a coarse-grained level of space-time boxes with side
lengths $L_s$ and $L_t$ to supercritical oriented percolation, see
Lemma~\ref{lem:abstr-OCcoupl}. We use a coarse-graining function to
define a coarse-grained random walk, see \eqref{eq:coarse-graining
  function} and \eqref{eq:defn coarse-grained rw}. In addition to that
we also want to carry the displacements inside the box to have a
direct relation between the coarse-grained and original versions of
the random walk, see \eqref{eq:defn coarse-grained rw} and the display
below.

\smallskip

The coupling to supercritical oriented percolation allows us to define
regeneration times in a similar way to the regeneration times obtained
in Section~\ref{sec:preparations-1}, see
Construction~\ref{constr:reg-time abstract}.

\smallskip

The central object in Section~\ref{sec:abstract_proofs} is the
coarse-grained random walk. As described above the coarse-graining
frames the abstract model in a similar context we studied in
Section~\ref{sec:preparations} via the comparison with oriented
percolation. The general strategy to prove
Theorem~\ref{thm:abstract_quenched_clt} is then to transfer the
results obtained in Section~\ref{sec:preparations} to the the
coarse-grained random walk. However, the ``coarse-grained random
walk'' viewed along its regeneration times is only Markovian if we
track the displacements inside the space-time boxes, the local
configuration of the environment and the current distance of the two
walkers, see \eqref{eq:43} and comments below. Furthermore, even in
the case where we observe two random walks in independent copies of
the environment, say $(X,X^{''})$ (in the notation introduced above),
the increments between regenerations are not independent. This is due
to the fact that, by definition, the random walk has to evaluate the
environment one step into its future, see \eqref{eq:abstr-walk}.
However this only results in finite range correlations of the
increments. We argue this by observing that the regeneration
construction and assumptions on the environment, specifically
Assumption~\ref{ass:coupling}, directly result in the sequence of
local configurations of the environment
$(\widehat{\eta},\widehat{\eta}')$, observed at the time of the
regenerations, is an independent sequence; see
Remark~\ref{rem:eta_hat_independent_sequence}. This fact allows us to
conclude that the correlations of the regeneration increments for the
random walks and regeneration times have finite range, see
Remarks~\ref{rem:ind_increments_almost_independent} and
\ref{rem:regen_times_almost_independent}.

\smallskip

We then follow the same steps as in Section~\ref{sec:preparations} and
prove a similar ``one-step'' coupling result; see
Lemma~\ref{lem:abstract_TVdistance-joint-ind}. Here, the only
difference in the proof is the scaling by $L_s$ and $L_t$ to account
for the coarse-graining. Next we provide a separation lemma; see
Lemma~\ref{lem:abstract separ}. Here a new difficulty arises from the
application of the invariance principle since the sequence of the
random walk increments along the regeneration times in the case of two
independent environments is not independent itself. Otherwise the
proof follows along the same lines as for its version in
Section~\ref{sec:preparations}. The remaining steps towards the
quenched CLT~\ref{thm:abstract_quenched_clt} then follow by the same
arguments we used in Section~\ref{sec:preparations} to prove
Theorem~\ref{thm:LLNuCLTmodel1}.

\section{Random walks in oriented percolation: Proof of
  Theorem~\ref{thm:LLNuCLTmodel1}}
\label{sec:preparations}
In this section we provide a framework and the necessary tools to
prove the quenched CLT for the auxiliary model introduced in
Section~\ref{sec:model_in_oriented_percolation}. We follow ideas from
\cite{BirknerCernyDepperschmidtGantert2013} and extend the
regeneration construction introduced in
\cite{BirknerCernyDepperschmidt2016} to two (conditionally
independent) random walks in the same environment.

Here is a roadmap for proof of Theorem~\ref{thm:LLNuCLTmodel1}:
\begin{enumerate}
\item By the annealed CLT from \cite{BirknerCernyDepperschmidt2016},
  the mean of the quenched expectation $E_\omega[f(X_n/\sqrt{n})]$,
  which is the annealed expectation, converges (for Lipschitz test
  functions $f$, say) to $\int f \, d\Phi$, where $\Phi$ is a normal
  law.  For the quenched CLT we want to show that the variance of
  $E_\omega[f(X_n/\sqrt{n})]$ tends to $0$.

  The general idea is to compute and estimate this variance by
  considering two conditionally independent copies $X$ and $X'$ of the
  walk in the same environment $\omega$, noting that
  \begin{equation}
    \label{eq:roadmap1-1}
    \bE\left[\left(E_\omega[f(X_n/\sqrt{n})]\right)^2 \right]
    =  \bE\left[ E_\omega[f(X_n/\sqrt{n})] E_\omega[f(X'_n/\sqrt{n})] \right]
    = \bE\left[ f(X_n/\sqrt{n}) f(X'_n/\sqrt{n}) \right].
  \end{equation}
\item In order to work with the latter term, we develop in
  Section~\ref{sec:preparations-1} a joint regeneration construction
  for the pair of walks $(X,X')$ in the same environment under the
  annealed law, see Construction~\ref{constr:reg-time}. This extends
  the construction for a single walk from
  \cite{BirknerCernyDepperschmidt2016}, where the central tool was a
  cone based at the current position of the random walk. This cone
  consisted of an inner and an outer cone; the region between these,
  the `cone shell', was used to shell separate the information
  collected by the random walk on the random environment inside the
  inner cone from any information about the environment outside the
  outer cone.

  We replace this by a double cone to encompass the setting with two
  random walks (see Figure~\ref{fig:dcones}), again with the goal to
  isolate the interior of the double cone from the environment in the
  exterior via the double cone shell. The construction guarantees that
  the inter-regeneration times have algebraic tails with decay
  exponent $-\beta$, where $\beta$ can be chosen very large when $p$
  is close to $1$ and $\varepsilon_\rf$ from
  Assumption~\ref{ass:appr-sym} is small, see
  Proposition~\ref{prop:JointRegTimesBound}. This in particular
  entails that the spatial increments between regeneration times have
  finite $\gamma$-th moments for any $0 \le \gamma < \beta$.

  Let $T_n^\joint$ be the $n$-th (joint) regeneration time and write
  $\widehat{X}^\joint_n \coloneqq X_{T^\joint_n},
  \widehat{X}'_n{}^{\!\joint} \coloneqq X'_{T_n^\joint}$ for the walks
  observed along these regeneration times. By the regeneration
  construction, $(\widehat{X}^\joint_n,
  \widehat{X}'_n{}^{\!\joint})_n$ is a Markov chain under the annealed
  law.

\item In Section~\ref{sect:coupling-1} we provide a coupling between
  the pair of walks $(\widehat{X}^\joint_n,
  \widehat{X}'_n{}^{\!\joint})_n$ and an auxiliary pair of walks
  $(\widehat{X}^\indi_n, \widehat{X}'_n{}^{\!\indi})_n$, where
  $\widehat{X}^\indi$ and $\widehat{X}'_n{}^{\!\indi}$ use the same
  regeneration construction as $(\widehat{X}^\joint,
  \widehat{X}'{}^{\joint})$ but $\widehat{X}'_n{}^{\!\indi}$ moves in
  an independent copy $\omega'$ of the environment. When
  $\widehat{X}^\joint_n$ and $\widehat{X}'_n{}^{\!\joint}$ are far
  apart, their transitions are based on well-separated parts of the
  environment; in view of the space-time mixing properties of
  supercritical oriented percolation it is thus reasonable to expect
  that one can then replace the medium for one of the walkers by an
  independent copy without changing the law of the next
  inter-regeneration increment very much.

  We quantify this idea in
  Proposition~\ref{prop:TVdistance-joint-ind-1step} where we show that
  \begin{align}
    \label{eq:roadmap1-2}
    \bP\Big( (\widehat{X}^\joint_{n+1}, \widehat{X}'_{n+1}{}^{\hspace{-1.1em}\joint}) \neq
    (\widehat{X}^\indi_{n+1}, \widehat{X}'_{n+1}{}^{\hspace{-1.1em}\indi})\,\Big|\,
    (\widehat{X}^\joint_n, \widehat{X}'_n{}^{\!\joint}) = (\widehat{X}^\indi_n, \widehat{X}'_n{}^{\!\indi})
    =(x,x') \Big) \le \frac{C}{\norm{x-x'}^\beta}.
  \end{align}
  Intuitively, this comes from the fact that the coupling can
  essentially only fail if the next inter-regeneration time is so
  large that the two cones based at $\widehat{X}^\joint_n$ and at
  $\widehat{X}'^{\,\joint}_n$ respectively begin to overlap, and this
  is unlikely by Proposition~\ref{prop:JointRegTimesBound}.

  Note that by construction, $(\widehat{X}^\indi_n,
  \widehat{X}'_n{}^{\!\indi})_n$ is not only a Markov chain but in
  fact a $\Z^d \times \Z^d$-valued random walk, i.e.\ its increments
  are i.i.d. Furthermore, its increments are centered and also have
  finite $\gamma$-th moments for $0 \le \gamma<\beta$, see
  Proposition~\ref{prop:IndRegTimesBound} and Remark~\ref{rem:indMC}.
  Thus, $(\widehat{X}^\indi, \widehat{X}'{}^{\indi})$ fulfills a
  (joint) CLT.

  Equipped with all this, the approach to compute the right-hand side
  of \eqref{eq:roadmap1-1}, first along joint regeneration times,
  follows to some extent the template provided by
  \cite{BirknerCernyDepperschmidtGantert2013}: The general strategy is
  to replace the $\joint$-walks in \eqref{eq:roadmap1-1} by
  $\indi$-walks with quantitatively controlled error. However, there
  are some additional technical difficulties because unlike the
  situation in \cite{BirknerCernyDepperschmidtGantert2013}, our
  regeneration times do not have exponential tails.

\item As in \cite{BirknerCernyDepperschmidtGantert2013}, the rest of
  the proof is split into the case $d \ge 2$ and the case $d=1$.

  We begin with the (easier) case $d \ge 2$ in
  Section~\ref{sect:aux-1}: We show in Lemma~\ref{lem:coupl} that
  under our coupling of $(\widehat{X}^\joint,
  \widehat{X}'{}^{\joint})$ and $(\widehat{X}^\indi,
  \widehat{X}'{}^{\indi})$ there are with high probability at most
  $n^{b_1}$ uncoupled steps until time $n$, for some small $b_1>0$ (in
  particular, $b_1<1/2$). The central ingredient here is the idea,
  formalised in Lemma~\ref{lem:separ}, that $\widehat{X}^\joint$ and
  $\widehat{X}'{}^{\joint}$ separate sufficiently quickly in $d \ge 2$
  so that the control on the coupling error provided by
  \eqref{eq:roadmap1-2} becomes effective.

  \begin{enumerate}
  \item Armed with this, we can show in
    Section~\ref{sect:qCLT-along-1} that
    $E_\omega[f(\widehat{X}^\joint_m/\sqrt{m})]$ converges to $\int f
    \, d\wt\Phi$ in $L^2$, where $\wt\Phi$ is a suitable normal law,
    see Proposition~\ref{prop:2nd-mom}. This immediately gives the
    same almost sure limit along an algebraically growing growing
    subsequence of the regeneration times. We strengthen this a a full
    quenched CLT along the regeneration times by controlling the
    fluctuations inside the increments of the subsequence, see
    Lemma~\ref{lemma:convergence for subsequence to full sequence} and
    Proposition~\ref{prop:qCLT.regen.1}.

    The technical tools for this are the following: In
    Lemma~\ref{lemma:joint-ind-comparison} we replace the
    $\joint$-pair by an $\indi$-pair along regeneration times for
    Lipschitz test functions with controlled error; this is the main
    ``workhorse'' here and it relies crucially on the fact that by
    time $n$ we have $o(\sqrt{n})$ uncoupled steps.
    Lemma~\ref{lem:ind-to-joint} then shows ``essentially'' a law of
    large numbers and CLT fluctuations for $T^\joint_n$, which allows
    to replace the random times $T^\joint_n$ by their means with
    controlled error; Lemma~\ref{lem:joint_fluctuations} controls that
    the fluctuations of $\widehat{X}^\joint$ cannot be much larger
    than those of $\widehat{X}^\indi$ (not more than
    $O(n^{1/2+\varepsilon})$ up to time $n$).
  \item In Section~\ref{sect:qCLTd>=2-1} we complete the proof of
    Theorem~\ref{thm:LLNuCLTmodel1} for the case $d \ge 2$: We
    transfer the quenched CLT along the regeneration times to the full
    quenched CLT by suitably controlling the fluctuations of $X$
    between regeneration times, based on the estimates from
    Section~\ref{sect:qCLT-along-1}.
  \end{enumerate}
\item In Section~\ref{sec:dimension1} we provide the proof of
  Theorem~\ref{thm:LLNuCLTmodel1} for the case $d=1$. The difficulty
  compared to the case $d \ge 2$ lies in the fact that we cannot hope
  to have at most $o(\sqrt{n})$ uncoupled steps where
  $\widehat{X}^\joint$ and $\widehat{X}'^\joint$ behave differently
  because generically, centred one-dimensional random walks with
  finite variance will have on the order of $\sqrt{n}$ visits to the
  origin. In particular, an analogue of Lemma~\ref{lem:coupl} cannot
  hold in $d=1$. We circumvent this problem, following the template
  from \cite[Section~3.4]{BirknerCernyDepperschmidtGantert2013} with
  suitable changes, by using a martingale decomposition of
  $(\widehat{X}^\joint,\widehat{X}'^\joint)$ and then a CLT for
  martingales with quantitatively controlled error.
\end{enumerate}

\subsection{Construction of simultaneous regeneration times for two
  walks}
\label{sec:preparations-1}

In the setting of Subsection~\ref{sec:model_in_oriented_percolation}
let $X\coloneqq (X_n)_n$ and $X' \coloneqq (X'_n)_n$ be two
(conditionally independent) random walks in the same environment. We
enrich the underlying probability space in a straightforward way so
that the space-time Bernoulli field $\omega(\cdot,\cdot)$ (and
$(\eta_n(\cdot))_n$ and $(\kappa_n(\cdot,\cdot))_n$, which are
functions of $\omega(\cdot,\cdot)$) as well as the two conditionally
independent walks $X$ and $X'$ are defined on the same space. In the
following, we will work on this enriched space and write
$\bP^\joint_{x_0,x_0'}$ to denote the probability measure on that
space when the initial positions of the two walks are given by
$X_0=x_0, X'_0=x'_0$ with $x_0, x'_0 \in \Z^d$.

Our first goal here is to introduce a sequence of regeneration times
$T_1< T_2 < \cdots$, at which both $X$ and $X'$ regenerate jointly.
The construction of the regeneration times extends the corresponding
construction from \cite{BirknerCernyDepperschmidt2016} for a single
random walk.

We need to isolate the part of the environment that the two random
walks explore (this is finite in finitely many steps due to our
assumptions) from the rest of the environment until they regenerate.
This isolation will be achieved by certain space-time double cones
(each one for one random walk) with shells in which the two random
walks will move and the shells of certain growing thickness separate
the information from outside the cone. In particular the randomness
driving each of the random walks is contained inside the corresponding
cone.

Let us recall the definitions of cones and cone shells from equations
(2.25) and (2.27) in \cite{BirknerCernyDepperschmidt2016}. For
positive $b,s,h$ and $x_\base \in \bZ^d$ we set
\begin{align}
  \label{def:cone}
  \cone(x_\base;b,s,h) \coloneqq \big\{ (x,n) \in \bZ^d
  \times \bZ_+ \, \colon \, \norm{x_\base - x}_2 \leq b + s
  n, 0 \le n \le h \big\}
\end{align}
for a (truncated upside-down) \emph{cone} with base radius $b$, slope
$s$, height $h$ and base point $(x_\base,0)$. Furthermore for
\begin{align}
  \label{eq:bs-inn-out}
  b_\inn < b_\out \quad \text{and} \quad s_\inn < s_\out,
\end{align}
we define the \emph{conical shell} with inner base radius $b_\inn$,
inner slope $s_\inn$, outer base radius $b_\out$, outer slope
$s_\out$, and height $h \in \N \cup \{\infty\}$ by
\begin{multline}
  \label{eq:coneshell}
  \cs(x_\base;b_\inn , b_\out , s_\inn , s_\out ,h)
  \\ \coloneqq \big\{ (x,n) \in \bZ^d \times \bZ  :
  b_\inn  + s_\inn n \le \norm{x_\base - x}_2 \le
  b_\out  + s_\out n, \, 0 < n \le h \big\}.
\end{multline}
The conical shell can be thought of as a difference of the \emph{outer
  cone} $\cone(x_\base; b_\out,s_\out,h)$ and the \emph{inner cone}
$\cone(x_\base;b_\inn,s_\inn,h)$ with all boundaries except the bottom
boundary of that difference included. To shorten notation for fixed
parameters for radii and slopes as in \eqref{eq:bs-inn-out}, we write
$\cs(x_\base;h)$ for the cone shell as defined in
\eqref{eq:coneshell}.

For the proof of Theorem~\ref{thm:LLNuCLTmodel1} we will follow ideas
from \cite{BirknerCernyDepperschmidtGantert2013}. As mentioned before
we extend the cone construction for the regeneration times from
\cite{BirknerCernyDepperschmidt2016} to two cones and we define the so
called \emph{double cone shell} to isolate the area of the environment
that the two random walks have explored from the rest.

Consider two random walks $X$ and $X'$ located at time $n$ at
positions $x_\base$ and $x'_\base$ respectively, i.e.\ $X_n=x_\base$
and $X'_n=x'_\base$. A first attempt would be to just take the union
of both cone shells $\cs(x_\base;h)$ and $\cs(x'_\base;h)$ with base
points $x_\base$ and $x'_\base$. The problem with this attempt is that
the cone shell $\cs(x_\base;h)$ would grow into the interior
$\cone(x_\base';b_\inn,s_\inn,h)$ of $\cs(x'_\base;h)$ and vice versa
and in particular into the region which we want to isolate. Instead we
take the union of the cone shells without the elements of the inner
cones and define the double cone shell
\begin{multline}
  \label{eq:defn_double_cone_shell}
  \dcs(x_\base,x'_\base;h) \\ \coloneqq \left(\cs(x_\base;h)\cup
  \cs(x'_\base;h)\right)\setminus\left(\cone(x_\base;b_\inn,s_\inn,h)\cup
  \cone(x_\base';b_\inn,s_\inn,h) \right).
\end{multline}
Note that we again omitted the base radii $b_\inn$, $b_\out$ and
slopes $s_\inn$, $s_\out$ in this notation. Of course the double cone shell
$\dcs(x_\base,x'_\base;h)=\dcs(x_\base,x'_\base;b_\inn,b_\out,s_\inn,s_\out,h)$
depends on these parameters as well. We also write
$\cone(x_\base) = \cone(x_\base;\infty)$ and
$\dcs(x_\base,x'_\base)=\dcs(x_\base,x'_\base;\infty)$ if we consider
the cone or cone shell with infinite height.

For the case $d=2$ the double cone with the double cone shell is
illustrated in Figure~\ref{fig:dcones}. For $d=1$ we will use a
slightly different double cone shell definition. It will make the
arguments in the proof of Lemma~\ref{th:lemma2.13Analog} below more
streamlined and the difference of the definitions is explained
in that proof. For the case $d=1$ the double cone with the double cone
shell is illustrated in Figure~\ref{fig:coneshells_in_d=1}.

\begin{figure}
  \centering
  \begin{tikzpicture}[yscale=0.9,xscale=0.9]
    \definecolor{drawColor}{RGB}{0,0,0}


    \draw[->] (-6,-2) -- (-2,2);
    \draw (-6,0) -- (-1,0);
    \draw[dashed] (-1,0) -- (1,0);
    \draw (1,0) -- (3,0);
    \draw[dashed] (3,0) -- (5,0);
    \draw[->] (5,0) -- (6,0);
    \draw[->] (-4,-2) -- (-4,8);


    \node[text=drawColor,anchor=base,inner sep=0pt, outer sep=0pt, scale=  1.00] at (-4.2,7.7) {$n$};
    \node[text=drawColor,anchor=base,inner sep=0pt, outer sep=0pt, scale=  1.00] at (-2.5,1.8) {$x_2$};
    \node[text=drawColor,anchor=base,inner sep=0pt, outer sep=0pt, scale=  1.00] at (5.75,-.3) {$x_1$};

    \draw[dashed] (1,0) arc (0:180: 1cm and .5cm);
    \draw (-1,0) arc (180:360: 1cm and .5cm);
    \draw (-3.5,6) arc (180:55: 3.5cm and 1.75cm);
    \draw (-3.5,6) arc (180:305: 3.5cm and 1.75cm);
    \path[draw=drawColor,line width= 0.5pt,line join=round] (-1,0) -- (-3.35882,5.50116);
    \path[draw=drawColor,line width= 0.5pt,line join=round] (1,0) -- (2,2.4);

    \draw[dashed] (0.5,0) arc (0:180: 0.5cm and .25cm);
    \draw (-0.5,0) arc (180:360: 0.5cm and .25cm);
    \draw (-2.5,6) arc (180:37: 2.5cm and 1.25cm);
    \draw (-2.5,6) arc (180:323: 2.5cm and 1.25cm);
    \path[draw=drawColor,line width= 0.5pt,line join=round] (-0.5,0) -- (-2.428649,5.6945946);
    \path[draw=drawColor,line width= 0.5pt,line join=round] (0.5,0) -- (2,4.5);

    \draw[dashed] (5,0) arc (0:180: 1cm and .5cm);
    \draw (3,0) arc (180:360: 1cm and .5cm);
    \draw (7.5,6) arc (0:125: 3.5cm and 1.75cm);
    \draw (7.5,6) arc (360:235: 3.5cm and 1.75cm);
    \path[draw=drawColor,line width= 0.5pt,line join=round] (3,0) -- (2,2.4);
    \path[draw=drawColor,line width= 0.5pt,line join=round] (5,0) -- (7.35882,5.50116);

    \draw[dashed] (4.5,0) arc (0:180: 0.5cm and .25cm);
    \draw (3.5,0) arc (180:360: 0.5cm and .25cm);
    \draw (6.5,6) arc (0:143: 2.5cm and 1.25cm);
    \draw (6.5,6) arc (360:217: 2.5cm and 1.25cm);
    \path[draw=drawColor,line width= 0.5pt,line join=round] (3.5,0) -- (2,4.5);
    \path[draw=drawColor,line width= 0.5pt,line join=round] (4.5,0) -- (6.428649,5.6945946);


    \draw[blue,line width= 1.5pt] (-1.6875,2.5) arc (180:360:1.6875cm and 0.84375cm);
    \draw[dashed,line width= 1.5pt,blue] (1.6875,2.5) arc (0:180:1.6875cm and 0.84375cm);

    \draw[red] (2.5,0) -- (2.3959730,0.25) -- (2.6913578,0.5) -- (2.3182013,0.75) -- (2.4424640,1) -- (2.1008534,1.25) -- (2.0371980,1.5) -- (2.1338517,1.75) -- (1.7472884,2) -- (1.6095703,2.25) -- (1.5,2.5) -- (1.2227309,2.75) -- (0.9823701,3);

    \filldraw[red] (2.5,0) circle (1pt);
    \filldraw[red] (0.9823701,3) circle (1pt);


    \draw[fill=green, opacity=0.2] (0.5,0) arc (360:180:0.5cm and .25cm) -- (-2.428649,5.6945946) arc (194:323:2.5cm and 1.25cm) arc (217:346:2.5cm and 1.25cm) -- (4.5,0) arc (360:180:0.5cm and .25cm) -- (2,4.5) --(0.5,0);


    \draw[fill=green, opacity=0.4] (2,6.75) arc (37:323:2.5cm and 1.25cm) -- (2,5.25)
    arc (217:503:2.5cm and 1.25cm);


    \draw[fill=drawColor, opacity=0.2] (-1,0) -- (-3.35882,5.50116) arc (196:305:3.5cm and 1.75cm) -- (2,4.563859) arc (235:344:3.5cm and 1.75cm) -- (5,0) arc (360:180:1cm and .5cm) -- (3,0) -- (2,2.4) -- (1,0) arc (360:180:1cm and .5cm);


    \draw[fill=drawColor, opacity=0.4] (2,6.75) arc (143:-12:2.5cm and 1.25cm) -- (6.44120,5.73049) -- (6.05228,4.58241735)  arc (-54:125:3.5cm and 1.75cm)  -- (2,7.436141) -- (2,7.436141)  arc (55:234:3.5cm and 1.75cm) -- (-2.05228,4.58241735) -- (-2.44120,5.73049) arc (192:37:2.5cm and 1.25cm);

  \end{tikzpicture}
  \caption{Double cone with a double cone shell (grey), a time slice
    of the middle tube (blue), and a path of a random walk crossing
    the double cone shell from outside to inside (red).}
  \label{fig:dcones}
\end{figure}
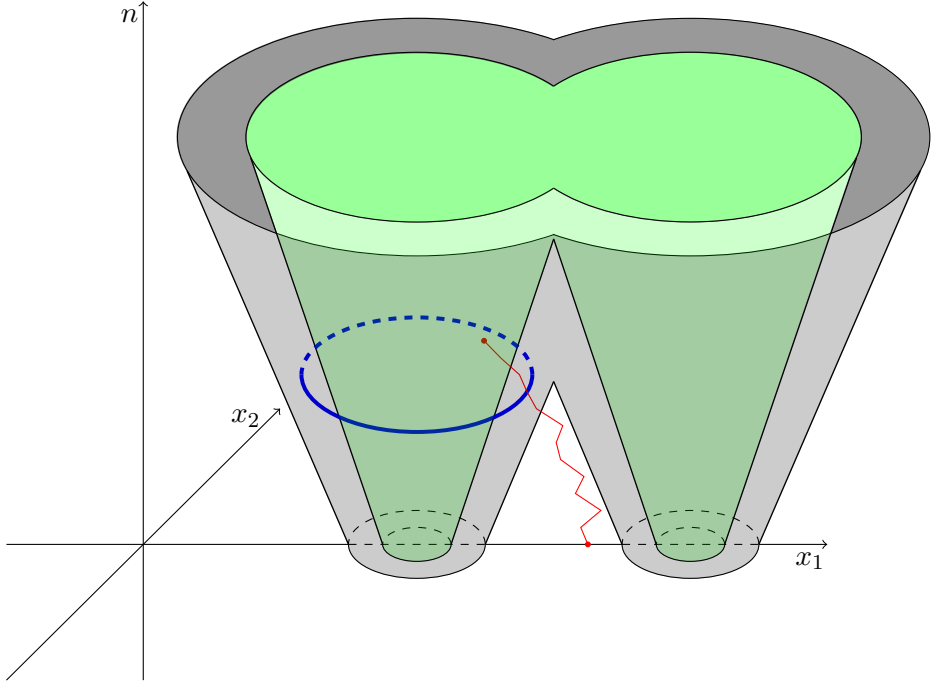

We follow the ideas of the proof of Lemma~2.13 from
\cite{BirknerCernyDepperschmidt2016}. For $d\ge 2$ we define a subset
$\mathcal{M}\subset \bZ^d\times\bZ$ of the double cone shell with the
following two properties:
\begin{enumerate}[1.]
\item Every path crossing from the outside to the inside has to hit a
  point in $\mathcal{M}$.
\item There exist small constants $\delta>0$ and $\tilde{\delta}>0$
  such that for every $(x,n)\in \mathcal{M}$ we have
  $B_{\tilde{\delta}n}(x)\times\{n-\delta n\}\subset
  \dcs(x_\base,x'_\base;\infty)$, where $B_r(y)$ is the ball of radius
  $r$ centred around $y$.
\end{enumerate}
Note that the number of elements in
$\mathcal{M}\cap\bZ^d\times[0,n] \subset \dcs(x_\base,x'_\base;n)$
grows polynomially in $n$. Such a set is given, for instance, by the
``middle tube'' of the double cone shell which we will now define more
precisely. Let
\begin{align*}
  d(n)\coloneqq \frac{1}{2}\big(n(s_\out+s_\inn)+b_\out+b_\inn \big)
\end{align*}
be the distance between the spatial origin and the middle of the cone
shell at time $n$ and define
\begin{align}
  \label{eq:defn_middle_tube}
  \begin{split}
    M_n & \coloneqq \{ x \in\bZ^d  :  \norm{x-x_\base}_2\in[d(n),d(n)+2d] \},\\
    M'_n & \coloneqq \{ x\in \bZ^d : \norm{x-x'_\base}_2\in[d(n),d(n)+2d] \}
  \end{split}
\end{align}
the middle tubes for the single cones at time $n$. We define the
middle tube in this a way, ``thickening'' it by $2d$, so that we can
ensure that a nearest neighbour path crossing the cone shell has to
hit a site in the middle tube and cannot jump over it. Note that we
define nearest neighbours for the jumps according to the sup-norm,
that is $y$ is a nearest neighbour of $x$ if $\norm{x-y} =
\norm{x-y}_\infty\le 1$, whereas the middle tubes in
\eqref{eq:defn_middle_tube} are based on the Euclidean norm.

Furthermore we define the middle tube for the double cone shell at
time $n$ by
\begin{align}
  \label{eq:defn_middle_tube_dcs}
  M^\dcs_n \coloneqq  \big(M_n\cup M'_n\big)\setminus \big\{
  x\in\bZ^d\colon \min\{\norm{x-x_\base}_2,\norm{x-x'_\base}_2 \} <
  d(n) \big\}
\end{align}
and set $\mathcal{M}= \bigcup_n (M^\dcs_n\times\{n\})$; see
Figure~\ref{fig:crossing path} for an illustration.
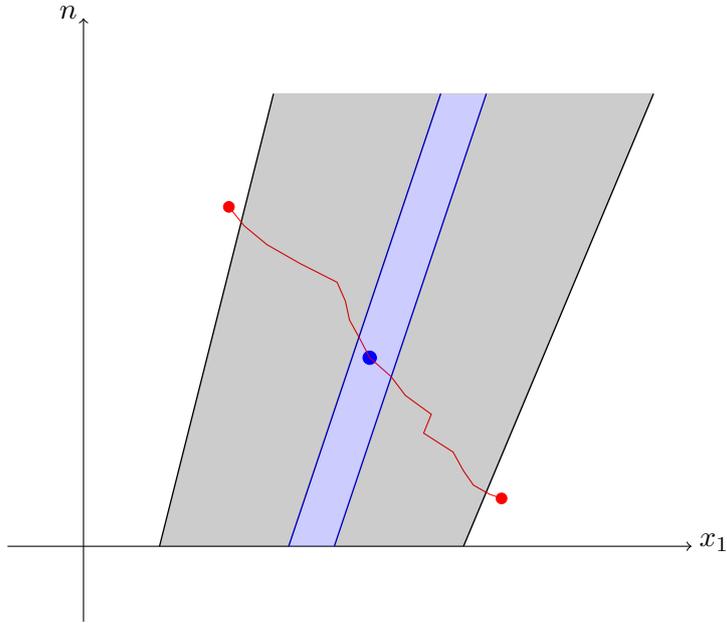
\begin{figure}
  \centering
  \begin{tikzpicture}
    \definecolor{drawColor}{RGB}{0,0,0}

    \draw[->] (-5,0) -- (4,0);
    \draw[->] (-4,-1) -- (-4,7);

    \node[text=drawColor,anchor=base,inner sep=0pt, outer sep=0pt, scale=  1.00] at (-4.2,7) {$n$};
    \node[text=drawColor,anchor=base,inner sep=0pt, outer sep=0pt, scale=  1.00] at (4.3,0) {$x_1$};

    \path[draw=drawColor,line width= 0.5pt,line join=round] (-3,0) -- (-1.5,6);
    \path[draw=drawColor,line width= 0.5pt,line join=round] (1,0) -- (3.5,6);


    \path[draw=drawColor,line width= 0.5pt,line join=round,blue] (-1.3,0) -- (0.7,6);
    \path[draw=drawColor,line width= 0.5pt,line join=round,blue] (-0.7,0) -- (1.3,6);


    \filldraw[blue] (-0.23422853,2.5) circle (2.5pt);


    \draw[red] (1.5, 0.6345) -- (1.341,0.6905) -- (1.13,0.8112) -- (1,1) -- (0.86228191,1.25) -- (0.47544254,1.5) -- (0.57600252,1.75) -- (0.23508174,2) -- (0.04721539,2.25) -- (-0.23422853,2.5) -- (-0.36288245,2.75) -- (-0.5,3) -- (-0.55415799,3.25) -- (-0.66275908,3.5) -- (-1.15305242,3.75) -- (-1.58709454,4) -- (-1.88709454,4.25) -- (-2.08709454,4.5);


    \filldraw[red] (1.5, 0.6345) circle (2pt);
    \filldraw[red] (-2.08709454,4.5) circle (2pt);


    \draw[fill=drawColor, opacity=0.2] (-3,0) -- (-1.3,0) -- (0.7,6) -- (-1.5,6);
    \draw[fill=drawColor, opacity=0.2] (-0.7,0) -- (1,0) -- (3.5,6) -- (1.3,6);


    \draw[fill=blue, opacity=0.2] (-1.3,0) -- (0.7,6) -- (1.3,6) -- (-0.7,0);

  \end{tikzpicture}
  \caption{A cross section of the cone shell including the middle tube
    $\mathcal M$ (blue) and a path $\gamma$ (red) crossing the cone
    shell from the outside to the inside of the double cone and
    hitting at least one point in $\mathcal M$ (blue dot).}
  \label{fig:crossing path}
\end{figure}

Let $\eta ^{\dcs}\coloneqq (\eta^{\dcs}_n)_{n=0,1,\dots}$ be the
contact process restricted to the infinite double cone shell
$\dcs(x_\base,x_\base';\infty)$ with initial condition
$\eta^{\dcs}_0(x) = \indset{\dcs(x_\base,x_\base';0)}((x,0))$ and
\begin{align*}
  \eta^{\dcs}_{n+1}(x)=
  \begin{cases}
    1 & \text{ if } (x,n+1) \in \dcs(x_\base,x_\base';\infty),\text{ }
    \omega(x,n+1)=1\\[-4pt]
    & \text{ and } \eta^{\dcs}_n(y) =1 \text{ for some } y\in
    \mathbb{Z}^d \text{ with } \norm{x-y}\leq 1,\\
    0 & \text{ otherwise.}
  \end{cases}
\end{align*}
We think of $\eta ^{\dcs}$ as a version of the contact process
where all $\omega$'s outside $\dcs(x_\base,x_\base';\infty)$ have been
set to $0$. For a directed nearest neighbour path
\begin{align}
  \label{eq:directedpath}
  \gamma = \left( ( x_m,m ), (x_{m+1},m+1),\dots,(x_n,n) \right),
  \quad m\leq n, \, x_i \in  \mathbb{Z}^d \text{ with }
  \norm{x_{i-1}-x_i} \leq 1
\end{align}
with starting position $x_m$ at time $m$ and final position $x_n$ at
time $n$ we say that $\gamma$ crosses the double cone shell
$\dcs(x_\base,x_\base';\infty)$ from the outside to the inside if the
following three conditions are fulfilled:
\begin{enumerate}[(i)]
\item the starting position lies outside the double cone shell, i.e.,
  $\norm{x_m - x_\base}_2 > b_\out + ms_\out$ and
  $\norm{x_m - x_\base'}_2 > b_\out + ms_\out$,
\item the terminal point lies inside one (or both) of the inner cones, i.e.,
  $\norm{x_n -x_\base}_2 < b_\inn + ns_\inn$ or
  $\norm{ x_n - x_\base'}_2 < b_\inn + ns_\inn$,
\item all the remaining points lie inside the shell, i.e.,
  $(x_i, i) \in \dcs(x_\base,x_\base';\infty)$ for $i=m+1,\dots,n-1$.
\end{enumerate}
We say that a path $\gamma$ (as in \eqref{eq:directedpath}) intersects
$\eta^{\dcs}$ if there exists $i \in \{ m+1,\dots,n-1 \}$ with
$\eta^{\dcs}_i(x_i)=1$. Finally we say that $\gamma$ is open in
$\dcs(x_\base,x_\base';\infty)$ if $\omega(x_i,i) = 1$ for all
$i=m+1,\dots,n-1$.

\begin{lemma}[2-cone analogue of Lemma~2.13 from
  \cite{BirknerCernyDepperschmidt2016}]
  Assume \label{th:lemma2.13Analog} that the relations in
  \eqref{eq:bs-inn-out} hold and consider the events
  \begin{align*}
    G_1 & \coloneqq \{ \eta^{\dcs} \text{ survives} \},\\
    G_2 & \coloneqq \{ \text{every open path } \gamma \text{ that crosses }
          \dcs(x_\base,x_\base';\infty) \text{ intersects }
          \eta^{\dcs} \}.
  \end{align*}
  For any $\varepsilon > 0$ and $0 \leq s_\inn < s_\out < 1$ one can
  choose $p$ sufficiently close to 1 and $b_\inn < b_\out$
  sufficiently large so that
  \begin{align*}
    \mathbb{P}(G_1 \cap G_2) \geq 1-\varepsilon.
  \end{align*}
\end{lemma}

\begin{remark}
  \label{rem:coneshifting}
  Note that in this preparation section we base the cones at time $0$.
  Later on they will be based at the current space-time positions of
  the random walks but all results here hold for the shifted
  constructions as well due to translation invariance. Furthermore,
  with the properties we impose on $\mathcal{M}$ on the event
  $G_1\cap G_2$ any site inside the inner cones which is connected to
  $\bZ^d\times \{0\}$ via an open path that crosses the cone shell also has
  a connection to $\bZ^d\times \{0\}$ inside the double cone shell.
  Thus, on the event $G_1\cap G_2$ we isolate all the sites in the
  inner cone from the information on the environment outside the outer
  cone in the sense that the value of $\eta$ inside the double cone
  can be determined using only the values of $\omega$'s inside the double
  cone.
\end{remark}

\begin{proof}[Proof of Lemma~\ref{th:lemma2.13Analog}]
  The central ideas, as in
  \cite[Lemma~2.13]{BirknerCernyDepperschmidt2016}, are as follows:
  The fact that $G_1$ is likely is inherited from the fact that
  supercritical oriented percolation survives in certain wedges, as
  proved by \cite{CoxMaricSchinazi2010} (for the contact process).
  Furthermore, on the complement of $G_2$, there would have to be a
  (suitably quantified) large `dry' cluster in supercritical oriented
  percolation, which is unlikely. We provide the details of the proof
  in Subsection~\ref{sect:proof.th:lemma2.13Analog}.
  \begin{figure}
    \centering
    \begin{tikzpicture}[x=1pt,y=1pt,yscale=0.8,xscale=0.8]]

      \definecolor{drawColor}{RGB}{0,0,0}

      \path[draw=drawColor,line width= 1.2pt,line join=round] (0,0) -- (100,0);
      \path[draw=drawColor,line width= 1.2pt,line join=round] (0,0) -- (-50,200);
      \path[draw=drawColor,line width= 1.2pt,line join=round] (100,0) -- (119.2308,76.923);


      \path[draw=drawColor,line width= 1.2pt,line join=round] (-20,0) -- (120,0);
      \path[draw=drawColor,line width= 1.2pt,line join=round] (-20,0) -- (-100,200);
      \path[draw=drawColor,line width= 1.2pt,line join=round] (120,0) -- (150.76923,76.923);


      \path[draw=drawColor,line width= 1.2pt,line join=round] (170,0) -- (270,0);
      \path[draw=drawColor,line width= 1.2pt,line join=round] (170,0) -- (150.76923,76.923);
      \path[draw=drawColor,line width= 1.2pt,line join=round] (270,0) -- (320,200);


      \path[draw=drawColor,line width= 1.2pt,line join=round] (150,0) -- (290,0);
      \path[draw=drawColor,line width= 1.2pt,line join=round] (150,0) -- (119.2308,76.923);
      \path[draw=drawColor,line width= 1.2pt,line join=round] (290,0) -- (370,200);


      \path[draw=drawColor,line width= 1.2pt,line join=round, dashed] (-100,200) -- (370,200);
      \path[draw=drawColor,line width= 1.2pt,line join=round] (119.2308,76.923) -- (150.76923,76.923);


      \path[draw=drawColor,line width= 1.2pt,line join=round,->] (-110,-55) -- (-110,250);
      \path[draw=drawColor,line width= 1.2pt,line join=round,->] (-115,-50) -- (380,-50);

      \node[text=drawColor,anchor=base,inner sep=0pt, outer sep=0pt, scale=  1.00] at (-124,240) {$n$};
      \node[text=drawColor,anchor=base,inner sep=0pt, outer sep=0pt, scale=  1.00] at (380,-70) {$x$};

      \path[draw=drawColor,line width= 1.2pt,line join=round, dashed] (50,0) -- (50,-5);
      \node[text=drawColor,anchor=base,inner sep=0pt, outer sep=0pt, scale=  1.00] at (50,-18) {$x_\base$};
      \path[draw=drawColor,line width= 1.2pt,line join=round, dashed] (220,0) -- (220,-5);
      \node[text=drawColor,anchor=base,inner sep=0pt, outer sep=0pt, scale=  1.00] at (220,-18) {$x'_\base$};

      \node[text=drawColor,anchor=base,inner sep=0pt, outer sep=0pt, scale=  1.00] at (-125,75.923) {$t^*$};
      \path[draw=drawColor,line width= 1.2pt,line join=round, dashed] (-110,76.923) -- (119.2308,76.923);

      \draw[fill=drawColor, opacity=0.4] (-20,0) -- (-100,200) -- (-50,200) -- (0,0);
      \draw[fill=drawColor, opacity=0.4] (120,0) -- (135,37.5) -- (150,0) -- (170,0) --  (150.76923,76.923) -- (119.2308,76.923) --(100,0);
      \draw[fill=drawColor, opacity=0.4] (270,0) -- (320,200) -- (370,200) -- (290,0);
      \draw[fill=green, opacity=0.4] (0,0) -- (-50,200) -- (320,200) -- (270,0) -- (170,0) -- (150.76923,76.923) -- (119.2308,76.923) -- (100,0) -- (0,0);
    \end{tikzpicture}
    \caption{Double cone and double cone shell in case $d=1$}
    \label{fig:coneshells_in_d=1}
  \end{figure}
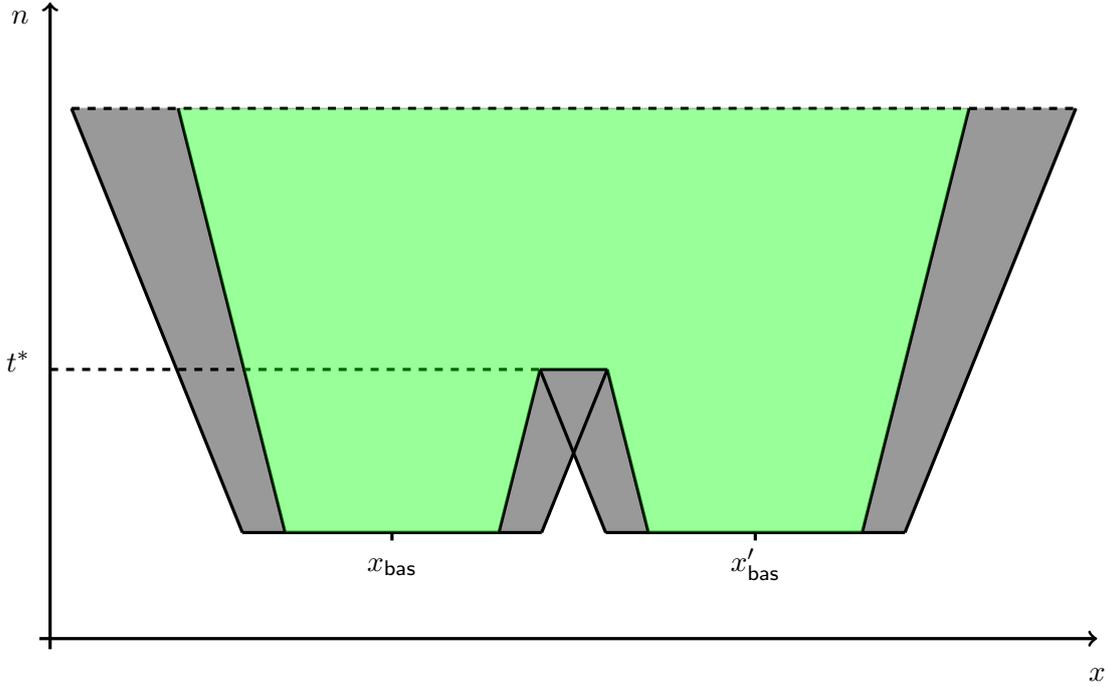
\end{proof}

Lemma~\ref{th:lemma2.13Analog} is the key tool to separate the
information inside the inner cone from the outside of the outer cone,
which provides a certain space-time isolation of the environment
around the random walk. Another important ingredient for the renewal
times is to \emph{locally} explore the reasons for $0$'s of $\eta$
along the paths of the random walks. This allows to stop in such a way
that the distribution of the environment viewed relative to the
stopped particle dominates the a priori law of the environment; see
equation \eqref{eq:stochDomSigmaJoint} in
Lemma~\ref{lemma:expTailsSigmaJoint} below.

We define stopping times at which the reasons for negative
information, i.e.\ reasons for $\eta=0$ at certain space-time points,
for both walkers $X$ and $X'$ are explored. For
$(x,n)\in\bZ^d\times\bZ$, let
\begin{displaymath}
  \ell(x,n) \coloneqq \max \left\{ m \ge 0 : \,
    \begin{array}{l}\exists \, x=x_0,x_1,\dots,x_m \text{ with } \norm{x_i-x_{i-1}} \le 1
      \text{ for } 1 \le i \le m \\
      \text{ and } \omega(x_i,n-i) = 1 \text{ for } 0 \le i \le m
    \end{array}\right\}
\end{displaymath}
be length of the longest directed open path (going in the negative
time direction) starting in $(x,n)$ with the convention $\ell(x,n)=-1$
if $\omega(x,n)=0$ and $\ell(x,n)=\infty$ if $\eta_n(x)=1$. As in
\cite{BirknerCernyDepperschmidt2016}, see equation~(2.45) there, we
define
\begin{align*}
  D_n \coloneqq n + \max \{\ell(y,-n) +2 : \norm{X_n -y} \leq R_\loc,\,
  \ell(y,-n) < \infty\}.
\end{align*}
Note that $D_n$ is the time, for the walk, at which the reasons for
$\eta_{-n}(y)=0$ for all $y$ from the $R_\loc$-neighbourhood of $X_n$
are explored by inspecting all the determining triangles
\begin{align}
  \label{eq:defn_determining_triangles}
  D(x,n) \coloneqq
  \begin{cases}
    \emptyset, & \text{if } \eta_n(x)=1,\\[0.5ex]
    \{ (y,m) : \norm{y-x}\leq n-m, \, n-\ell(x,n)-1 \leq m \leq n \}, &
    \text{if } \eta_n(x)=0,
  \end{cases}
\end{align}
with base points in $B_{R_\loc}(X_n)$. We aim to build the
regeneration times on exactly the stopping times at which we explored
all reasons for $0$'s of $\eta$ along the paths and define
\begin{align*}
  \sigma_0 \coloneqq 0, \quad \sigma_i \coloneqq \min \left\{ m >
  \sigma_{i-1} : \max_{\sigma_{i-1}\leq n \leq m } D_n \leq m \right\}, \; i \geq 1.
\end{align*}
Let $(\sigma'_i)_{i\geq 0}$ be defined analogously for $X'$.

The times at which we jointly explored the reasons for $0$'s of $\eta$
along the paths of $X$ and $X'$ are given by the sequence
$(\sigma^\simi_\ell)_{\ell \ge 0}$ defined by
\begin{align}
  \label{eq:sigmaJoint}
  \sigma^\simi_0 =0, \quad \sigma^{\simi}_\ell = \sigma_i = \sigma'_j,
  \; \ell \ge 1,
\end{align}
where $i$ and $j$ are the first times with respect to the
corresponding sequences so that
$|\{\sigma_0,\dots,\sigma_i\} \cap \{\sigma_0',\dots,\sigma_j'\}| =
\ell$.

Note that $\sigma^{\simi}$ are exactly the times when we have no
``negative'' influence on the environment in the future of the paths
of both random walks: The family $\{ \eta(\cdot, k), k >
\sigma^{\simi}_\ell\}$ is not independent of the information gathered
in the construction up to time $\sigma^{\simi}_\ell$ (for the walks).
However, the information concerning $\{ \eta(\cdot, k), k >
\sigma^{\simi}_\ell\}$ that we have at time $\sigma^{\simi}_\ell$ is
only of the form that we know that there must be $1$'s at certain
space-time positions (which lie in the future of the path). At time
$\sigma^{\simi}_\ell$, we do not have ``negative'' information in the
sense that we know that there must be $0$'s at certain space-time
positions. Thus, we can invoke the FKG inequality at such times, see
Lemma~\ref{lemma:expTailsSigmaJoint} and its proof below.

Again as in \cite{BirknerCernyDepperschmidt2016} we define
\begin{align}
  \label{eq:Rloc_tube}
  \tube_n \coloneqq \{ (y,-k): 0 \leq k \leq n, \norm{y-X_k} \leq R_\loc \}
\end{align}
and
\begin{align}
  \label{eq:det_triangle_Rloc_tube}
  \dtube_n = \bigcup_{(y,k) \in \tube_n} D(y,k)
\end{align}
(this is $\tube_n$ ``decorated'' with all the determining triangles
emanating from it). Similarly we define $\tube'_n$ and $\dtube'_n$ by
using $X'$ instead of $X$.
(These are random subsets of $\Z^d \times \Z_-$. Formally, we equip
$\{ A : A \subset \Z^d \times \Z_-\}$ with the $\sigma$-algebra generated by the
mappings $A \mapsto \ind{A}((y,m))$ for $(y,m) \in \Z^d \times \Z_-$. Thus,
$\tube_n$ and $\dtube_n$ are bona fide random variables.)

Let the filtration $\mathcal{F}^\simi =
(\mathcal{F}^\simi_n)_n$ be defined by
\begin{align*}
  \mathcal{F}^{\simi}_n
  & \coloneqq \sigma(X_j: 0 \leq j \leq n) \vee
    \sigma(\eta_j(y),\omega(y,j): (y,j) \in \tube_n)\vee
    \sigma(\omega(y,j):(y,j) \in \dtube_n)\\
  & \qquad \vee \sigma(X'_j: 0 \leq j \leq n) \vee
    \sigma(\eta_j(y),\omega(y,j): (y,j) \in \tube'_n)\vee
    \sigma(\omega(y,j):(y,j) \in \dtube'_n).
\end{align*}
In particular $\mathcal{F}^\simi_{\sigma^\simi_n}$ contains
all the information about the environment that the random walks gather
until time $\sigma^\simi_n$. This information includes the values of
$\omega$ and $\eta$ in their $R_\loc$-vicinity and the determining
triangles. Note that the $\sigma^{\simi}_n, n \in \N$ are stopping
times w.r.t.\ $\mathcal{F}^\simi$.

\begin{lemma}[Analogue of Lemma~2.17 from
  \cite{BirknerCernyDepperschmidt2016}]
  If \label{lemma:expTailsSigmaJoint} $p$ is sufficiently close to
  $1$, then there exist finite positive constants $c$ and $C$ so that
  uniformly in all pairs of starting points $x_0, x_0' \in \Z^d$,
  \begin{align}
    \label{eq:expTailsSigmaJoint}
    \mathbb{P}^\joint_{x_0,x'_0}(\sigma^{\simi}_{i+1} - \sigma^{\simi}_i > n \, | \,
    \mathcal{F}^{\simi}_{\sigma^{\simi}_i}) \leq Ce^{-cn} \quad
    \text{for all } n =1,2,\dots, i =0,1,\dots \, \mathbb{P}^\joint_{x_0,x'_0}\text{-a.s.},
  \end{align}
  in particular, all $\sigma^{\simi}_i$ are a.s.\ finite. Furthermore,
  we have
  \begin{align}
    \label{eq:stochDomSigmaJoint}
    \mathbb{P}^\joint_{x_0,x'_0}(\omega(\cdot,-j-\sigma^{\simi}_i)_{j=0,1,\dots} \in \cdot \, | \,
    \mathcal{F}^{\simi}_{\sigma^{\simi}_i} ) \succcurlyeq \mathscr{L}(
    \omega(\cdot,-j)_{j=0,1,\dots} ) \quad
    \text{for all } i=0,1,\dots \, \mathbb{P}^\joint_{x_0,x'_0}\text{-a.s.},
  \end{align}
  where `$\succcurlyeq$' denotes stochastic domination.
\end{lemma}

\begin{proof}
  The proof is basically the same as for Lemma~2.17 of
  \cite{BirknerCernyDepperschmidt2016} with only a few minor
  differences, using the fact that a large value of
  $\sigma^{\simi}_{i+1} - \sigma^{\simi}_i$ enforces the existence of
  a long open path in oriented percolation which is not connected to
  infinity. Furthermore, by construction at times $\sigma^{\simi}_i$
  the regeneration construction only retains information about certain
  $1$'s in $\eta$, i.e.\ about existence of infinite open paths in
  $\omega(\cdot,\cdot)$, so that at least intuitively
  \eqref{eq:stochDomSigmaJoint} follows from the FKG inequality.

  The details can be found in
  Section~\ref{proof:lemma:expTailsSigmaJoint}.
\end{proof}

\begin{corollary}[Analogue of Corollary~2.18 from
  \cite{BirknerCernyDepperschmidt2016}]
  For \label{corollary:dryPoints} $p$ large enough there exists
  $\varepsilon(p) \in (0,1]$ satisfying $\lim_{p\uparrow 1}
  \varepsilon(p) =0$ such that for all
  $V=\{ (x_1,t_1),\dots,(x_k,t_k) \}$ and
  $V'=\{ (x'_1,t'_1),\dots,(x'_\ell,t'_\ell) \}$, subsets of
  $\bZ^d\times \N$ with $t_1<\dots < t_k$, $t'_1<\dots < t'_\ell$, we
  have uniformly in all $x_0, x_0' \in \Z^d$
  \begin{align}
    \bP^\joint_{x_0,x_0'}\bigl(\eta_{-t-\sigma^{\simi}_i}(x+X_{\sigma^{\simi}_i})=0 \text{
    for all } (x,t) \in V \, \vert \,
    \mathcal{F}^{\simi}_{\sigma^{\simi}_i}\bigr) \leq \varepsilon(p)^k
  \end{align}
  and
  \begin{align}
    \bP^\joint_{x_0,x_0'}\bigl(\eta_{-t'-\sigma^{\simi}_i}(x'+X'_{\sigma^{\simi}_i})=0 \text{
    for all } (x',t') \in V' \, \vert \,
    \mathcal{F}^{\simi}_{\sigma^{\simi}_i}\bigr) \leq \varepsilon(p)^\ell.
  \end{align}
  a.s.\ under $\bP^\joint_{x_0,x_0'}$.
\end{corollary}
\begin{proof}
  With $\varepsilon(p)$ as in Lemma~2.11 of
  \cite{BirknerCernyDepperschmidt2016} the assertion follows from that
  lemma and \eqref{eq:stochDomSigmaJoint}; cf also Corollary~2.18 in
  \cite{BirknerCernyDepperschmidt2016}.
\end{proof}
For $t \in \N$ we define
$R_t \coloneqq \inf \{ i \in \bZ_+ : \sigma^{\simi}_i \geq t \}$ and
for $m=1,2,\dots$ we put
\begin{align}
  \wt{\tau}^{\simi}_m(t) \coloneqq
  \begin{cases}
    \sigma^{\simi}_{R_t-m+1} - \sigma^{\simi}_{R_t-m}  \quad &
    \text{if } m\leq R_t,\\
    0 & \text{else.}
  \end{cases}
\end{align}
In particular $\wt\tau^\simi_1(t)$ is the length of the interval
$(\sigma^\simi_{i-1},\sigma^\simi_{i}]$ which contains $t$.  (For
$m \ge 2$, $\wt\tau^\simi_m(t)$ is the length of the $(m-1)$-th
interval before it.)

\begin{lemma}[Analogue of Lemma~2.19 from \cite{BirknerCernyDepperschmidt2016}]
  When \label{lemma:expTailsTau} $p$ is sufficiently close to $1$ then
  there exist finite positive constants $c$ and $C$ so that for all
  starting points $x_0$, $x_0'$ and all $i,n =0,1,\dots$
  \begin{align}
    \label{eq:expTailsTau1}
    \bP^\joint_{x_0,x'_0}\bigl(\wt{\tau}^{\simi}_1(t) \geq n \,\vert \,
    \mathcal{F}^{\simi}_{\sigma^{\simi}_i}\bigr) \leq Ce^{-cn} \quad
    \text{a.s.\ on } \{ \sigma^{\simi}_i <t \},
  \end{align}
  and generally
  \begin{align}
    \label{eq:expTailsTaum}
    \bP^\joint_{x_0,x'_0}(R_t \geq i + m, \wt{\tau}^{\simi}_m(t) \geq n \, \vert \,
    \mathcal{F}^{\simi}_{\sigma^{\simi}_i}) \leq Cm^2e^{-cn} \quad
    \text{for } m=1,2,\dots \text{ a.s.\ on } \{
    \sigma^{\simi}_i < t \}.
  \end{align}
  under $\bP^\joint_{x_0,x'_0}$.
\end{lemma}
This lemma is direct consequence of some calculations used in the
proof of Lemma~2.19 in \cite{BirknerCernyDepperschmidt2016} combined
with Lemma~\ref{lemma:expTailsSigmaJoint}. The details can be found in
Section~\ref{sec:proof:lemma:expTailsTau}.

\begin{lemma}[Analogue of Lemma~2.20 in
  \cite{BirknerCernyDepperschmidt2016}]
  \label{lem:apriori speed conditioned on F}
  When $p$ is sufficiently close to $1$, then for all $\varepsilon>0$
  there exist finite positive constants $c=c(\varepsilon)$ and
  $C=C(\varepsilon)$ so that for any pair $x_0,x'_0\in\bZ^d$
  of starting positions and for all finite
  $\mathcal{F}^{\simi}$-stopping times $T$ with the property that
  almost surely $T \in \{ \sigma^{\simi}_i : i\in \N \}$ and for all
  $k,\ell \in \N$ and
  \begin{align}
    \label{eq:aprioriAlongStoppingTimesX}
    & \bP^\joint_{x_0,x'_0}(\norm{X_k - X_T} > s_{\mathrm{max}}(k-T) \, \vert \,
      \mathcal{F}^{\simi}_T) \leq Ce^{-c(k-T)} & & \text{a.s.\ on } \{T < k \},\\
    \label{eq:aprioriAlongStoppingTimesX'}
    & \bP^\joint_{x_0,x'_0}(\norm{X'_\ell - X'_T} > s_{\mathrm{max}}(\ell-T) \, \vert \,
      \mathcal{F}^{\simi}_T) \leq Ce^{-c(\ell-T)} & & \text{a.s.\ on } \{T < \ell \}
  \end{align}
  and for $j < k,\ell$
  \begin{align}
    \label{eq:aprioriAlongStoppingTimesX2}
    & \bP^\joint_{x_0,x'_0}(\norm{X_k -X_j} > (1+\varepsilon)s_{\mathrm{max}}(k-j) \, \vert \,
      \mathcal{F}^{\simi}_T) \leq Ce^{-c(k-j)} & & \text{a.s.\ on } \{T
                                                   \leq j\},\\
    \label{eq:aprioriAlongStoppingTimesX'2}
    & \bP^\joint_{x_0,x'_0}(\norm{X'_\ell -X'_j} > (1+\varepsilon)s_{\mathrm{max}}(\ell-j) \, \vert
      \, \mathcal{F}^{\simi}_T) \leq Ce^{-c(\ell-j)} & & \text{a.s.\ on }
                                                      \{T \leq j\}
  \end{align}
  under $\bP^\joint_{x_0,x'_0}$, where $s_{\mathrm{max}}$ can be
  chosen arbitrarily small by taking $\varepsilon_{\rf}\ll 1$ (see
  \eqref{eq:ass:appr-sym} from Assumption~\ref{ass:appr-sym}) and
  $1-p\ll 1$; see Lemma~2.16 in \cite{BirknerCernyDepperschmidt2016}.
\end{lemma}
The proof follows the same lines as in
\cite{BirknerCernyDepperschmidt2016} and can be found in
Section~\ref{sec:proof:lem:apriori speed conditioned on F}.

\medskip

\noindent
Next recall from \cite{BirknerCernyDepperschmidt2016} the definition
of a cone time point for the decorated path; see Equation~(2.56) and
also Figure~4 there. For $m < n$ we say that $n$ is a
$(b,s)$-\emph{cone time point for the decorated path of} $X$
\emph{beyond} $m$ if, recall definitions \eqref{eq:Rloc_tube} and
\eqref{eq:det_triangle_Rloc_tube},
\begin{align}
  \label{eq:TimePointDecoratedPathX}
  \begin{split}
    (\tube_n \, \cup \,
    & \dtube_n) \cap (\bZ^d \times \{ -n,-n+1,\dots,-m \})\\
    & \subset \{ (x,-j): m \leq j \leq n, \norm{x-X_n}\leq b +s(n-j) \}
  \end{split}
\end{align}
and for $X'$ if
\begin{align}
  \label{eq:TimePointDecoratedPathX'}
  \begin{split}
    (\tube'_n \,\cup \,
    & \dtube'_n) \cap (\bZ^d \times \{ -n,-n+1,\dots,-m \})\\
    & \subset \{ (x,-j): m \leq j \leq n, \norm{x-X'_n}\leq b +s(n-j) \}.
  \end{split}
\end{align}
Thus, $n$ is a cone time point for the decorated path of $X$ beyond
$m$ if the space-time path $(X_j,-j)_{j=m,\dots,n}$ together with its
$R_\loc$-tube and determining triangles is contained in
$\cone(b,s,n-m)$ shifted to the base point $(X_n,-n)$.

\begin{lemma}[Analogue to \cite{BirknerCernyDepperschmidt2016} Lemma~2.21]
  For \label{lem:conepoint} $\varepsilon>0$, when $p$ is sufficiently
  close to $1$, there exist $b>0$ and $s > s_{\mathrm{max}}$ (with $s_{\mathrm{max}}$
  from Lemma~\ref{lem:apriori speed conditioned on F}) such that
  for all $x_0, x_0' \in \Z^d $ and
  for all finite $\mathcal{F}^{\simi}$-stopping times $T$ with
  $T\in\{ \sigma^{\simi}_i : i \in \N \}$ a.s.\ and all $k \in \N$,
  with
  $T' \coloneqq \inf \{ \sigma^{\simi}_i : \sigma^{\simi}_i \geq k \}$
  \begin{multline}
    \label{eq:35}
    \bP^\joint_{x_0,x'_0}(T' \text{ is a } (b,s)\text{-cone time point for the decorated
      path of } X \text{ and } X'\text{ beyond } T \, \vert \,
    \mathcal{F}^{\simi}_T)\\
    \geq 1-\varepsilon,
  \end{multline}
  a.s.\ on $\{ T < k \}$ under $\bP^\joint_{x_0,x'_0}$. Furthermore,
  $0 < s - s_{\mathrm{max}} \ll 1$ can be chosen small.
\end{lemma}
\begin{proof}
  By Lemma~2.21 from \cite{BirknerCernyDepperschmidt2016} one can tune
  the parameters such that a.s.\ on $\{ T < k \}$ we have
  \begin{align*}
    \Pr(T' \text{ is a } (b,s)\text{-cone time point for the decorated
      path of } X \text{ beyond } T \, \vert \, \mathcal{F}^{\simi}_T)
    \geq 1-\varepsilon/2,
  \end{align*}
  and the analogous inequality with $X'$ instead of $X$ holds as well.
  From these estimates \eqref{eq:35} follows easily. The assertion
  that $0 < s - s_{\mathrm{max}} \ll 1$ can be chosen small is also a
  direct consequence of Lemma~2.21 from
  \cite{BirknerCernyDepperschmidt2016}.
\end{proof}

\begin{remark}
  Note that, since both inner cones are subsets of the double inner
  cone, $T'$ from Lemma~\ref{lem:conepoint} is a $(b,s)$-``double cone
  time point'' for the decorated path of $(X,X')$ beyond $T$ if it is
  a $(b,s)$-cone time point for $X$ and $X'$ beyond $T$ which is in
  accordance to our setting up the simultaneous regeneration times.
\end{remark}

As a part of the construction of the regeneration times we define the
set of `good' $\omega$-configurations in the double cone shell. Let
\begin{align}
  \label{eq:36}
  G_{x,x'}(b_\inn,b_\out,s_\inn,s_\out,h) \subset
  \{0,1\}^{\dcs(x,x',b_\inn,b_\out,s_\inn,s_\out,h)}
\end{align}
be the set of all $\omega$-configurations with the property
\begin{align}
  \label{eq:GoodOmegaSet}
  \begin{split}
    \forall \eta_0,\eta'_0 & \in \{0,1\}^{\bZ^d} \text{ with }
    \eta_0\vert_{B_{b_\out}(x) \cup B_{b_\out}(x')}=
    \eta'_0\vert_{B_{b_\out}(x) \cup B_{b_\out}(x')} \equiv 1 \quad
    \text{and} \\
    \omega & \in \{ 0,1 \}^{\bZ^d \times\{1,\dots,h \}} \text{ with }
    \omega\vert_{\dcs(x,x',b_\inn,b_\out,s_\inn,s_\out,h)} \in
    G_{x,x'}(b_\inn,b_\out,s_\inn,s_
    \out,h) :\\
    \eta_n(y) & = \eta'_n(y) \text{ for all } (y,n) \in
    \cone(x;b_\inn,s_\inn,h) \cup \cone(x';b_\inn,s_\inn,h),
  \end{split}
\end{align}
where $\eta$ and $\eta'$ are both constructed from \eqref{eq:DCP-dyn}
with the same $\omega$ but using possibly different initial
conditions. In words this means when there are $1$'s at the bottom of
the outer cones, i.e.\ $\cone(x;b_\out,s_\out,h)$ and
$\cone(x';b_\out,s_\out,h)$, a configuration from
$G_{x,x'}(b_\inn,b_\out,s_\inn,s_\out,h)$ guarantees successful
coupling inside the inner double cone irrespective of what happens
outside the outer cones.

\begin{lemma}
  For \label{lemma:GoodOmegaSet} parameters
  $p,b_\inn,b_\out,s_\inn,s_\out$ as in Lemma~\ref{th:lemma2.13Analog},
  \begin{align}
    \Pr(\omega\vert_{\dcs(x,x',b_\inn,b_\out,s_\inn,s_\out,h)} \in
    G_{x,x'}(b_\inn,b_\out,s_\inn,s_\out,h)) \geq 1 - \varepsilon,
  \end{align}
  uniformly in $h \in \N$ and $x,x' \in \bZ^d$.
\end{lemma}
\begin{proof}
  This lemma is a direct consequence of Lemma~\ref{th:lemma2.13Analog}
  because for $\omega \in G_1\cap G_2$ we have
  (after a suitable space-time shift)
  $\omega\vert_{\dcs(x,x',b_\inn,b_\out,s_\inn,s_\out,h)} \in
  G_{x,x'}(b_\inn,b_\out,s_\inn,s_\out,h)$.
\end{proof}
We denote a space-time shift of subsets of $\bZ^d \times \bZ$ by
$\Theta^{(x,n)}$, i.e.\
\begin{align}
  \label{eq:space-time-shift}
  \Theta^{(x,n)}(A) \coloneqq \{ (y+x,m+n) : (y,m) \in  A \} \quad
  \text{ for } A \subset \bZ^d \times \bZ.
\end{align}
From \cite{BirknerCernyDepperschmidt2016}, see there the discussion
around (2.62), we know that there exists a deterministic sequence
$t_\ell \nearrow \infty$ with the property that for $\ell \in \N$ and
$\norm{x_\base - y} \leq s_{\max}t_{\ell+1}$
\begin{align}
  \label{eq:space-time-shift.used}
  \Theta^{(0, -t_\ell)}\bigl( \cone (x_\base,t_\ell s_{\max} +
  b_\out, s_\out,t_\ell)\bigr) \subset \Theta^{(y,-t_{\ell+1})}
  \bigl(\cone(x_\base, b_\inn, s_\inn, t_{\ell+1}) \bigr).
\end{align}
We describe a possible choice in \eqref{eq:SequenceConditiontl}
below. The same sequence can be used for the double cone, since the
larger cones will overlap with each other before they can hit the
smaller cones. Thus, for $\norm{x_\base - y} \leq s_{\max}t_{\ell+1}$
and $\norm{x_\base' - y'} \leq s_{\max}t_{\ell+1}$ we have
\begin{align}
  \begin{split}
    & \Theta^{(0, -t_\ell)}\bigl( \cone (x_\base,t_\ell s_{\max} + b_\out,
      s_\out,t_\ell)\bigr) \cup \Theta^{(0, -t_\ell)}\bigl( \cone
      (x'_\base,t_\ell s_{\max} + b_\out, s_\out,t_\ell)\bigr)\\
    & \subset \Theta^{(y,-t_{\ell+1})}\bigl( \cone(x_\base, b_\inn,
      s_\inn, t_{\ell+1}) \bigr) \cup \Theta^{(y',-t_{\ell+1})}\bigl(
      \cone(x'_\base, b_\inn, s_\inn, t_{\ell+1}) \bigr)
  \end{split}
\end{align}
and the sequence satisfies
\begin{align}
  \label{eq:SequenceCondition}
  t_{\ell+1}s_\inn + b_\inn - t_{\ell+1}s_{\max} > t_\ell s_{\max}+b_\out + t_\ell
  s_\out \quad \text{for all } \ell=1,2,\dots.
\end{align}
Note that necessarily the sequence $(t_\ell)$ must grow exponentially
in $\ell$; cf.\ the discussion around \eqref{eq:RhoSoftCondition}.
For a concrete choice, put $t_1 =1$ and define
\begin{align}
  \label{eq:SequenceConditiontl}
  t_{\ell+1} \coloneqq \Big\lceil \frac{t_\ell s_{\max}+b_\out + t_\ell
  s_\out - b_\inn}{s_\inn - s_{\max}} \Big\rceil +1 \qquad \text{
  for } \ell=1,2,\dots.
\end{align}
This sequence $(t_\ell)_{\ell=1,2,\dots}$ satisfies
Condition \eqref{eq:SequenceCondition}.

\begin{constr}[Regeneration times]
  The \label{constr:reg-time} regeneration times will be constructed
  analogously to the construction in
  \cite{BirknerCernyDepperschmidt2016}. Let $X_0=x$ and $X'_0=x'$ be
  the starting positions of the two random walks. For the sequence
  $(\sigma_n^\simi)_{n=1,2,\dots}$ from \eqref{eq:sigmaJoint} and a
  sequence $(t_\ell)_{\ell=1,2,\dots}$ satisfying
  \eqref{eq:space-time-shift.used} we define the sequence $(\wt
  \sigma_n^\simi)_{n=1,2,\dots}$ by
  \begin{align}
    \label{eq:39}
      \wt{\sigma}^\simi_\ell \coloneqq \inf\{\sigma^\simi_i \in
  \{\sigma^\simi\} : \sigma^\simi_i \geq t_\ell \},
  \end{align}
  i.e.\ $\wt{\sigma}^\simi_\ell$ is the first element in the
  sequence $\sigma^\simi$ after time $t_\ell$. The sequence
  $(\wt \sigma_n^\simi)_{n=1,2,\dots}$ is the sequence along which
  we look for regeneration times.

  For the first regeneration time we check the following criteria:
  \begin{enumerate}[(i)]
  \item Go to $\wt\sigma^\simi_1$ and check if
    $\wt \sigma^\simi_1<t_2$, $\eta$ in the $b_\out$-neighbourhoods
    of $(X_{\wt\sigma^\simi_1},-\wt\sigma^\simi_1)$ and of
    $(X'_{\wt\sigma^\simi_1},-\wt\sigma^\simi_1)$ equals
    $\equiv 1$, the paths (together with their tubes and decorations)
    stayed inside the interior of the corresponding double cone based
    at the current space-time positions of the two random walks and
    $\omega$ in the respective conical shells is in the good set
    defined in \eqref{eq:36} and \eqref{eq:GoodOmegaSet}. If these
    events occur, we have found the first regeneration time
    $T_1 =\wt\sigma_1^\simi$.
  \item If the above attempt in (i) fails, we must try again. We
    successively check at times $t_2,t_3,$ etc.: If not previously
    successful up to time $t_{\ell -1}$, at the $\ell$-th step we
    check if $\wt{\sigma}^\simi_\ell <t_{\ell+1}$, if
    $\wt{\sigma}^\simi_\ell$ is a cone point for the decorated path
    beyond $t_{\ell-1}$ with
    \begin{align*}
      \max\{\norm{X_{\wt{\sigma}^\simi_\ell}-x},\norm{X'_{\wt{\sigma}^\simi_\ell}-x'}
      \} \leq s_{\mathrm{max}} \wt{\sigma}^\simi_\ell,
    \end{align*}
    if $\eta \equiv 1$ in the $b_\out$-neighbourhoods of
    $(X_{\wt{\sigma}^\simi_\ell},-\wt{\sigma}^\simi_\ell)$ and of
    $(X'_{\wt{\sigma}^\simi_\ell},-\wt{\sigma}^\simi_\ell)$, if
    $\omega$'s in the corresponding conical shells are in the good set
    defined in \eqref{eq:36} and \eqref{eq:GoodOmegaSet} and if the
    paths (with tubes and decorations) up to time $t_{\ell-1}$ are
    each contained in a box of diameter $s_\out t_{\ell-1}+b_\out$ and
    height $t_{\ell-1}$. If this all holds, we have found the first
    regeneration time $T_1$.

    Otherwise, we repeat Step~(ii).
  \end{enumerate}
\end{constr}

\begin{proposition}
  \label{prop:JointRegTimesBound}
  There are positive constants $C$ and $\beta$ so that
  \begin{align}
    \label{eq:37}
    \sup_{x_0,x_0'} \Pr_{x_0,x_0'}^{\joint} (T_1 >t) \le
    Ct^{-\beta}, \quad t \ge 1.
   \end{align}
  Furthermore, $\beta$ can be chosen arbitrarily large if $p$ is
  suitably close to $1$ and $\varepsilon_\rf$ is sufficiently small
  (see \eqref{eq:ass:appr-sym} in Assumption~\ref{ass:appr-sym}).
\end{proposition}

\begin{remark}
  In contrast to the setting in
  \cite{BirknerCernyDepperschmidtGantert2013} we do not have a literal
  connection (in the sense of being a subsequence) between the
  simultaneous regeneration times defined here and the regeneration
  times for a single random walk defined in
  \cite{BirknerCernyDepperschmidt2016}. For that reason we
  work here directly with the simultaneous regeneration times.
\end{remark}

\begin{proof}[Proof of Proposition~\ref{prop:JointRegTimesBound}]
  Fix $x_0,x'_0$ and write $\Pr$ for $\Pr^{\joint}_{x_0,x'_0}$ for the
  rest of the proof to shorten the notation. For $m \ge 2$, to be
  chosen appropriately later, we have
  \begin{align}
    \label{eq:40}
      \Pr(T_1 >t) = \sum_{\ell=1}^{\infty} \Pr(T_1 =
      \wt{\sigma}^\simi_\ell,
      \wt{\sigma}^\simi_\ell > t)
       \leq \sum_{\ell=1}^{m-1} \Pr(\wt{\sigma}^\simi_\ell > t) +
      \sum_{\ell=m}^{\infty} \Pr(T_1 = \wt{\sigma}^\simi_\ell).
  \end{align}

  Recall the definition of the sequence $(t_\ell)$ from
  \eqref{eq:SequenceConditiontl}.  For any $\rho$ with
  \begin{align}
    \label{eq:RhoSoftCondition}
    \rho > \frac{s_\out + s_{\max}}{s_\inn - s_{\max}}
  \end{align}
  we have $t_\ell \le \lceil \rho^\ell \rceil$ for all $\ell \ge
  \ell^*$, where $\ell^*=\ell^*(\rho) <\infty$ is the smallest index
  such that $t_\ell \ge \lceil\rho^\ell\rceil$ (see also
  Remark~\ref{rem:RhoHardCondition} below). Since we need to check the
  statement of the lemma only for large $t$, we may assume that $t \ge
  \lceil\rho^{\ell^*} \rceil^2$ (smaller $t$'s are covered by
  enlarging the constant $C$ in \eqref{eq:37} appropriately).

  For $t \ge \lceil\rho^{\ell^*} \rceil^2$ we have, choosing $m=
  \lceil \log(t)/(2\log(\rho))\rceil$ and using $\wt\sigma_0^\simi=0$
  \begin{align}
    \label{eq:UpperBoundFirstPartialSum}
    \begin{split}
      \sum_{i=1}^{m-1} \Pr(\tilde{\sigma}^\simi_i > t)
      & \leq \sum_{i=1}^{m-1} \Pr(\tilde{\sigma}^\simi_i -
      \tilde{\sigma}^\simi_{i-1}  > t - t_i)\\
      & \leq C(m-1)e^{-c(t-t_{m-1})} \leq C(m-1)e^{-c(t-\sqrt{t})}.
    \end{split}
  \end{align}
  (since $t_{m-1} \le \lceil \rho^{m-1} \rceil \le \sqrt{t}$).
  Obviously we have $m < \sqrt{t}$. Thus, the right hand side is
  bounded by $Ct^{-\beta}$ for any $\beta>0$.

  \medskip

  For the second sum on the right hand side of \eqref{eq:40} we first
  show that there is a uniform, positive lower bound for the
  probability of a successful regeneration (i.e.\ the conditions in
  Construction~\ref{constr:reg-time}~(ii) do hold) at the $\ell$-th
  attempt for all $\ell \geq 2$.

  By Lemma~\ref{lemma:expTailsTau} we know that for large $\ell$ the
  probability of $\wt{\sigma}_{\ell} > t_{\ell+1}$ is very small. The
  condition for the environment $\omega$, restricted to the respective
  cone shell, to be a good configuration is independent for different
  $\ell$'s since we check disjoint space-time regions of $\omega$ for
  each $\ell$.

  For the path containment we can argue as follows: by Lemma~2.16 in
  \cite{BirknerCernyDepperschmidt2016} we have
  \begin{align*}
    & \Pr\left(\exists \,n \leq t_\ell: \norm{X_n - x_0}>\frac{1}{2}s_{\max}t_\ell +
      n s_\out \: \text{ or } \: \exists \,n \leq t_\ell: \norm{X'_n -
      x'_0}>\frac{1}{2}s_{\max}t_\ell +n s_\out\right)\\
    & \quad \leq \Pr\left(\exists \,n \leq t_\ell: \norm{X_n -
      x_0}>\frac{1}{2}s_{\max}t_\ell + n s_\out\right)
      + \Pr\left(\exists \,n \leq
      t_\ell: \norm{X'_n - x'_0}>\frac{1}{2}s_{\max}t_\ell + n s_\out\right)\\
    & \quad \leq \sum_{n = \lceil \frac{t_\ell s_{\max}}{2R_\kappa} \rceil}^{t_\ell}
      \Pr(\norm{X_n - x_0}>s_{\max}n) + \sum_{n = \lceil \frac{t_\ell
      s_{\max}}{2R_\kappa} \rceil}^{t_\ell}  \Pr(\norm{X'_n -
      x'_0}>s_{\max}n)\\
    & \quad \leq Ce^{-ct_\ell}.
  \end{align*}
  Since $t_\ell$ grows exponentially in $\ell$, the right hand side is
  summable in $\ell$. Thus, from some time $\ell_0$ on, we have
  $\sup_{n\leq t_\ell}\norm{X_n - x_0} \leq s_{\max}t_\ell$ and
  $\sup_{n\leq t_\ell}\norm{X'_n - x'_0} \leq s_{\max}t_\ell$ for all
  $\ell \geq \ell_0$ a.s. In fact, we see from the argument above
  that
  \begin{align*}
    \bP\left(\exists \, \ell \ge \ell_0 : \, \max_{n\leq t_\ell}\norm{X_n - x_0}
    > s_{\max}t_\ell \text{ or } \max_{n\leq t_\ell}\norm{X'_n - x'_0} >
    s_{\max}t_\ell\right) \le C \sum_{\ell= \ell_0}^\infty e^{-ct_\ell}
    \le C' e^{-ct_\ell}.
  \end{align*}

  Next we bound the size of the decorations by showing that we only
  have finitely many large increments
  $\sigma^\simi_{i+1}-\sigma^\simi_i$ in the sequence $\sigma^\simi$.
  For $I_\ell \coloneqq \inf \{i : \sigma^\simi_i \geq t_\ell \}$ we
  have by Lemma~\ref{lemma:expTailsSigmaJoint}
  \begin{align*}
    \Pr(\exists \, i \leq I_\ell: \sigma^\simi_i - \sigma^\simi_{i-1} > k)
    & \leq \Pr(\exists \, i \leq t_\ell :\sigma^\simi_i - \sigma^\simi_{i-1} > k) \\
    & \leq \sum_{i=1}^{t_\ell} \Pr(\sigma^\simi_i - \sigma^\simi_{i-1} > k)
      \leq Ct_\ell e^{-ck}.
  \end{align*}
  Now if $\norm{X_n-x}\leq t_\ell s_{\max}$ for all $n\leq t_\ell$ and
  $\sigma^\simi_i - \sigma^\simi_{i-1} \leq t_\ell(s_\out-s_{\max})/2$
  for all $i\leq i_\ell$, then the path including the decorations is
  contained in the box $(x,0)+[-b_\out-t_\ell s_\out,b_\out+t_\ell
  s_\out]^d\times[0,t_\ell]$. Note that $ C t_\ell e^{-ck}$ is
  summable for the choice $k=t_\ell(s_\out-s_{\max})/2$. Hence there
  are only have finitely many times where one of the random walks
  moves ``too fast'' or the decorations are ``too large''. Combining
  the above it follows that a.s.\ there exists $\ell'$ so that the
  path containment property holds for all $\ell\geq \ell'$ and in fact
  \begin{displaymath}
    \bP(\text{the path containment
      property holds for all $\ell\geq \ell'$}) \ge 1 - C \exp(-c t_{\ell'}).
  \end{displaymath}
  for some $c, C \in (0,\infty)$.

  Assertion \eqref{eq:stochDomSigmaJoint} from
  Lemma~\ref{lemma:expTailsSigmaJoint} guarantees the existence of
  open paths in the future direction of the random walks (with
  uniformly positive probability) and thus yields a uniform lower
  bound on the event that $\eta=1$ in the $R_\loc$-neighbourhood of
  the random walks at $\wt{\sigma}^\simi_\ell$. Therefore there exists
  a uniform lower bound in $\ell$ on the probability of a successful
  regeneration at the $\ell$th attempt. Let $\delta_0$ be that uniform
  lower bound. Note that $\delta_0\to 1$ for $p\to 1$ since the
  probability for all three conditions in
  Construction~\ref{constr:reg-time}~(ii) to hold tends to $1$ for
  $p\to 1$.

  Plugging in the definition of $m$, the second partial sum on the
  right hand side of \eqref{eq:40} has the upper bound
  \begin{align}
    \label{eq:UpperBoundSecondPartialSum}
    \sum_{i=m}^{\infty} \Pr(T_1 = \wt{\sigma}^\simi_i) \leq
    \sum_{i=m}^\infty \delta_0(1-\delta_0)^i = (1-\delta_0)^{m+1} \le
    t^{\frac{\ln(1-\delta_0)}{2\ln(\rho)}+2}.
  \end{align}

  We see from Condition~\eqref{eq:RhoSoftCondition} that we can choose
  $\rho$ close to $1$ if $s_{\max}$ is close to $0$ and
  $s_\out-s_\inn$ is also close to $0$. The former requirement can be
  achieved for $p$ close to $1$, as observed in
  \cite[Lemma~2.16]{BirknerCernyDepperschmidt2016}, the latter
  requirement through Lemma~\ref{th:lemma2.13Analog} if $p$ is close
  enough to $1$. In this way we obtain
  $-\frac{\ln(1-\delta_0)}{2\ln(\rho)} > \beta+2$ and that concludes
  the proof. Note that in the proof we have a dependence between $t$
  and $\rho$, since we choose $t$ large enough such that
  $\rho^{\ell^*} \leq \sqrt{t}$, this results in the constant $C$
  being dependent on the choice of $\rho$, and thus $\beta$, as well.
  However this is not a problem since the condition $\rho^{\ell^*}
  \leq \sqrt{t}$ holds for smaller $t$ if we choose $\rho$ closer to
  $1$.
\end{proof}

\begin{remark}
  \label{rem:RhoHardCondition}
  One can be a bit more precise for which $\ell$ the relation $t_\ell
  \le \rho^\ell$ is true. This obviously depends on the choice of
  $\rho$ and yields that there exists $\ell^*=\ell^*(\rho)$ such that
  \begin{align}
    \label{eq:RhoHardCondition}
    \rho > \left( \frac{b_\out - b_\inn}{2s_{\max}+s_\out -s_\inn}
    \right)^{\frac{1}{\ell^*}} \frac{s_\out + s_{\max}}{s_\inn-s_{\max}},
  \end{align}
  and we have $t_\ell \leq \rho^\ell$ for all $\ell \geq \ell^*$.
\end{remark}

Write $\widehat{X}_1 = X_{T_1}, \widehat{X}'_1 = X'_{T_1}$. As a
direct consequence of Proposition~\ref{prop:JointRegTimesBound} we
obtain
\begin{corollary}
  \label{cor:tail bound random walks at sim regnerations}
  For $\beta$ from Proposition~\ref{prop:JointRegTimesBound} there
  exists a constant $C < \infty$ such that
  \begin{equation}
    \sup_{x_0,x_0'} \bP^\joint_{x_0,x_0'}\left(\norm{\widehat{X}_1 - x_0} + \norm{\hat
    X_1'-x_0'} >m \right) \le C m^{-\beta}
  \end{equation}
  and for any $0 \le \gamma < \beta$ we have
  \begin{equation}
    \label{eq:joint.gamma.moments}
    \sup_{x_0,x'_0} \bE^\joint_{x_0,x'_0}\left[ \norm{\widehat{X}_1 - x_0}^\gamma +
      \norm{\hat X_1'-x_0'}^\gamma + T_1^\gamma \right] < \infty.
  \end{equation}
 \end{corollary}

\begin{proof}
  By Assumption~\ref{ass:finite-range} the random walks $X$ and $X'$
  both have finite range for their transitions, hence
  \[
    \norm{\widehat{X}_1 - x_0} \le R_\kappa T_1, \qquad
    \norm{\hat X_1'-x_0'} \le R_\kappa T_1
  \]
  and thus the claim follows from
  Proposition~\ref{prop:JointRegTimesBound}. (For
  \eqref{eq:joint.gamma.moments} recall that $\bE[Y^\gamma] =
  \int_0^\infty \gamma y^{\gamma-1} \bP(Y > y) \,dy$ for any
  non-negative random variable $Y$.)
\end{proof}

A time $T_1$ for the walk, i.e.\ in time-slice $-T_1$ of the
environment $\eta = \eta(\omega)$, we have by construction that
$\eta_{-T_1}(\cdot)$ is $\equiv 1$ in the $b_\out$-neighbourhoods of
$X_{T_1}$ and of $X'_{T_1}$. Furthermore, the construction guarantees
that this is the only information we have about $\eta_u(\cdot), u \le
-T_1$. We can thus imagine that we re-draw all $\omega(x,u)$'s with $x
\in \Z^d, u \le -T_1$ from the original law conditioned on having all
$y$'s in the union of the $b_\out$-neighbourhoods of $X_{T_1}$ and of
$X'_{T_1}$ connected to $-\infty$ via open paths. Then we re-start at
time $T_1$ the construction of the regeneration time with the new
starting points $X_{T_1}$ and $X'_{T_1}$, etc. In this way we define
the sequence $0=T_0 < T_1 < T_2 < \cdots$ of regeneration times. We
abbreviate
\begin{equation}
  \label{def:Xhatn}
  \widehat{X}_n \coloneqq X_{T_n}, \quad \widehat{X}'_n \coloneqq X'_{T_n}, \qquad n \in \N.
\end{equation}
\begin{remark}
  \label{rem:jointMC}
  Note that the construction shows that under $\bP^\joint_{x_0,x_0'}$,
  the sequence $(\widehat{X}_n, \widehat{X}'_n)_n$ forms a Markov chain.
  (Strictly speaking, this holds only for $n \ge 2$ because we did not
  require that $\eta_0(\cdot)$ is $\equiv 1$ in the
  $b_\out$-neighbourhoods of $x_0$ and of $x_0'$, whereas this holds by
  construction at all later steps. Since the slightly special role of
  the first regeneration interval will be irrelevant for our limiting
  results, we silently drop this --mostly notational-- problem from now
  on.)  In fact, $(\widehat{X}_n, \widehat{X}'_n,T_n-T_{n-1})_n$ also
  forms a Markov chain and we could even enrich the state space to keep
  track of the path pieces in-between the regeneration times as well,
  analogous to \cite[Lemma~3.2]{BirknerCernyDepperschmidtGantert2013}.

  Furthermore, by shift-invariance of the whole construction,
  \[
    \bP^\joint_{x_0,x_0'}\big( (\widehat{X}_{n+1}-x, \widehat{X}'_{n+1}-x', T_{n+1}-T_n) \in \cdot
    \,\big|\, \widehat{X}_n = x, \widehat{X}'_n = x'\big)
  \]
  depends only on $x-x'$.
\end{remark}

\subsection{Coupling pairs of walks with copies in independent environments}
\label{sect:coupling-1}

By its very nature, the law of $(\widehat{X},\widehat{X}'{})$ under
$\bP^\joint_{x_0,x'_0}$ is quite complex and this makes explicit
computations or estimates for it a priori difficult. We address this
problem by comparing $\bP^\joint_{x_0,x'_0}$ with a new law
$\bP^\indi_{x_0,x'_0}$ under which the two walks evolve in independent
copies of the environment.

More precisely, we enrich the underlying probability space by an
independent copy $\omega'(\cdot,\cdot)$ of the space-time Bernoulli
field (and read off $(\eta'_n(\cdot))_n$ and
$(\kappa'_n(\cdot,\cdot))_n$ as functions of $\omega'(\cdot,\cdot)$,
analogous to $\eta$ and $\kappa$). Then we run the walk $X$ using the
kernels $\kappa_n$ and independently the walk $X'$ using the kernels
$\kappa'_n$. In this setting, we also implement a joint regeneration
construction for $X$ and for $X'$ analogous to
Construction~\ref{constr:reg-time} from
Section~\ref{sec:preparations-1}: Here, we we use the same sequences
$(t_\ell)$ as in Construction~\ref{constr:reg-time}. The deciding
difference is that we check the values of $\omega$ and $\eta$ in the
$R_\loc$-neighbourhood of $X$ along the $X$-path but the values of
$\omega'$ and $\eta'$ in the $R_\loc$-neighbourhood of $X'$ along the
$X'$-path. Since $\omega$ and $\omega'$ are independent we do not use
the double cone but instead two single cones (with the same geometry
as the constituent cones in the double cone construction) and each of
them has to satisfy the conditions of the construction. We denote the
resulting probability law on the enlarged space by
$\bP^\indi_{x_0,x'_0}$ when the initial positions of the two walks are
given by $X_0=x_0, X'_0=x'_0$ with $x_0, x'_0 \in \Z^d$. We write
$\bE^\indi_{x_0,x'_0}$ for corresponding expectations under this law.
\smallskip

The arguments from Section~\ref{sec:preparations-1} apply to this
construction, too. Thus we obtain the following analogue of
Proposition~\ref{prop:JointRegTimesBound} and Corollary~\ref{cor:tail
  bound random walks at sim regnerations} (using the notation
$(\widehat{X},\widehat{X}')$ as in \eqref{def:Xhatn} for the two walks
observed along joint regeneration times).

\begin{proposition}
  \label{prop:IndRegTimesBound}
    \begin{align}
    \label{eq:ind-tails}
    \sup_{x_0,x_0'} \Pr_{x_0,x_0'}^{\indi} (T_1 >t) \le
    Ct^{-\beta}, \quad t \ge 1.
   \end{align}
  (Recall that $\beta$ can be chosen arbitrarily large if $p$ is
  suitably close to $1$ and $\varepsilon_\rf$ from Assumption~\ref{ass:appr-sym}
  is sufficiently small.)

  Furthermore, there exists a constant $C < \infty$ such that
  \begin{equation}
    \sup_{x_0,x_0'} \bP^\indi_{x_0,x_0'}\left(\norm{\widehat{X}_1 - x_0} + \norm{\hat
    X_1'-x_0'} >m \right) \le C m^{-\beta}, \quad m \ge 0
  \end{equation}
  and for any $0 \le \gamma < \beta$ we have
  \begin{equation}
    \label{eq:ind.gamma.moments}
    \sup_{x_0,x'_0} \bE^\indi_{x_0,x'_0}\left[ \norm{\widehat{X}_1 - x_0}^\gamma +
      \norm{\hat X_1'-x_0'}^\gamma + T_1^\gamma \right] < \infty.
  \end{equation}
\end{proposition}

\begin{remark}
  \label{rem:indMC}
  Analogously to Remark~\ref{rem:jointMC}, the construction shows that
  under $\bP^\indi_{x_0,x_0'}$, the sequence $(\widehat{X}_n,
  \widehat{X}'_n)_n=(X_{T_n},X'_{T_n})_n$ also forms a Markov chain.
  (Again, strictly speaking this holds only for $n \ge 2$ because we
  did not require that $\eta_0(\cdot)$ is $\equiv 1$ in the
  $b_\out$-neighbourhoods of $x_0$ and of $x_0'$, whereas this holds
  by construction at all later steps. Since the slightly special role
  of the first regeneration interval will be irrelevant for our
  limiting results, we silently drop also this --mostly notational--
  problem from now on.) In fact, $(\widehat{X}_n,
  \widehat{X}'_n,T_n-T_{n-1})_n$ also forms a Markov chain under
  $\bP^\indi_{x_0,x_0'}$ and we could even enrich the state space to
  keep track of the path pieces in-between the regeneration times as
  well, analogous to
  \cite[Lemma~3.2]{BirknerCernyDepperschmidtGantert2013}.

  Furthermore, by shift-invariance of the whole construction and the
  fact that the two walks use independent copies of the environment,
  we in fact have that law of the increments
  \[
    \bP^\indi_{x_0,x_0'}\big( (\widehat{X}_{n+1}-x, \widehat{X}'_{n+1}-x', T_{n+1}-T_n) \in \cdot
    \,\big|\, \widehat{X}_n = x, \widehat{X}'_n = x'\big)
  \]
  does not depend on $x, x' \in \Z^d$ at all (unlike the situation in
  Remark~\ref{rem:jointMC}). In particular, the increments
  $(\widehat{X}_{n+1}-\widehat{X}_n,
  \widehat{X}'_{n+1}-\widehat{X}'_n)_{n \in \N}$ are i.i.d., i.e.\ the
  pair $(\widehat{X}, \widehat{X'})$ is a $\Z^d \times \Z^d$-valued
  random walk under $\bP^\indi_{x_0,x_0'}$. In addition, by symmetry
  of the roles of $\widehat{X}$ and $\widehat{X}'$, the difference
  random walk $\widehat{Z}_n \coloneqq \widehat{X}_n-\widehat{X}'_n$
  then has centred increments, i.e.\
  \begin{equation}
    \label{eq:ind.centred}
    \bE^\indi_{x_0,x_0'}\left[ \widehat{Z}_{n+1}-\widehat{Z}_n\right] = 0.
  \end{equation}
\end{remark}
\smallskip

Next we wish to introduce a coupling of $\bP^\joint_{x_0,x_0'}$ and
$\bP^\indi_{x_0,x_0'}$ in order to quantify the intuitive idea that
when $\norm{x_0-x'_0} \gg 1$, the two walks will (at least until their
first joint regeneration time) evolve under $\bP^\joint_{x_0,x_0'}$ in
`almost' independent parts of the environment. This can be achieved as
follows: In the setting of the product space for the independent
copies $\omega=\omega(\cdot,\cdot)$ and $\omega'=\omega'(\cdot,\cdot)$
we consider three random walks $X$, $X'$ and $X''$ with $X_0=x_0$,
$X'_0=X''_0=x'_0$. These processes exist on the product space
(enriched with sufficient extra randomness to run three random walks
in their respective environments). Furthermore, we then construct the
processes $X(\omega,\omega')=X(\omega)$ and
$X'(\omega,\omega')=X'(\omega)$ as evolving in the random environment
$\omega$ and $X''(\omega,\omega')=X''(\omega')$ as evolving in the
independent environment $\omega'$. This definition allows to compare
the pair $(X,X')$ with $(X,X'')$ realisation-wise.

Now we implement the first joint regeneration from
Construction~\ref{constr:reg-time} for $X$ and $X'$ in the environment
created by $\omega$ and at the same time its analogue discussed at the
beginning of this section for $X$ and $X''$, where $X''$ evolves in
the independent environment created by $\omega'$. We write
$T^\joint_1$ for the first joint regeneration time of $X$ and $X'$ and
$T^\indi_1$ for the first joint regeneration time of $X$ and $X''$
thus constructed.
\begin{proposition}
  \label{prop:TVdistance-joint-ind-1step}
  With $\beta$ from Proposition~\ref{prop:JointRegTimesBound} (and
  from Proposition~\ref{prop:IndRegTimesBound}) there exists a
  constant $C<\infty$
  \begin{align}
    & \sum_{x \in \Z^d,x'\in \Z^d,m\in \N} \left|
      \bP\left( (X_{T^\joint_1}, X'_{T^\joint_1},T^\joint_1)=(x,x',m)\right)
      - \bP\left( (X_{T^\indi_1}, X''_{T^\indi_1},T^\indi_1)=(x,x',m)\right)
      \right| \notag \\
      \label{eq:TVdistance-joint-ind-1step in prop}
    & \hspace{2em} \le C\norm{x_0-x'_0}^{-\beta},
  \end{align}
  where $x_0 \in \Z^d$ and $x_0' \in \Z^d$ are the initial positions
  of the two pairs (which are chosen identical for both pairs).
\end{proposition}
\begin{proof}[Proof idea]
  The general idea is a coupling argument similar to
  \cite[Lemma~3.4]{BirknerCernyDepperschmidtGantert2013}. Let $H
  \subset \R^d$ be the hyperplane perpendicular to $x'_0-x_0$ which
  passes through the mid-point $(x'_0-x_0)/2$ and define an auxiliary
  copy $\omega''$ of the medium which agrees with $\omega$ on one side
  of $H$ and with $\omega'$ on the other side. We then run copies
  $\overline{X}$ and $\overline{X}'$ of the walk in (the same)
  environment $\omega''$, starting from $\overline{X}_0=x_0$ and
  $\overline{X}'_0=x'_0$, respectively. By the space-time mixing
  properties of the environment, the dynamics of $\overline{X}$ will
  (at least initially) be with high probability identical to that of
  $X$ and the behaviour of $\overline{X}'$ identical to that of $X''$.
  Moreover, by construction the pairs $(X,X')$ and
  $(\overline{X},\overline{X}')$ are identical in law. Based on this,
  we implement yet another joint regeneration construction as in
  Construction~\ref{constr:reg-time}, this time for $\overline{X}$ and
  $\overline{X}'$ until the their first joint regeneration time
  $\overline{T}_1$.

  For large initial distance of the starting positions the cones from
  the regeneration construction will only overlap if the regeneration
  time itself is large, which has low probability by
  Propositions~\ref{prop:JointRegTimesBound} and
  \ref{prop:IndRegTimesBound}. Furthermore, as long the positions of
  the two walks are well separated, the double cone from
  Construction~\ref{constr:reg-time} is the same geometric object as
  the two single cones and thus the regeneration constructions from
  Section~\ref{sec:preparations-1} (for two walks in the same
  environment) and from the beginning of this section (for two walks
  in independent environments) involve identical steps. This yields
  that $T^\joint_1 = T^\indi_1 =\overline{T}_1$,
  $X_{T^\joint_1}=\overline{X}_{\overline{T}_1}$ and
  $X''_{T^\indi_1}=\overline{X}''_{\overline{T}_1}$ all occur with
  suitably high probability.

  The details can be found in
  Section~\ref{sec:proof:lem:TVdistance-joint-ind}.
\end{proof}

\begin{remark}
  \label{rem:TVdistance-joint-ind-1step}
  The total variation control provided by
  Proposition~\ref{prop:TVdistance-joint-ind-1step} together with the
  homogeneity properties of the walks under $\bP^\indi$ (see
  Remark~\ref{rem:indMC}) show that (with $\beta$ from
  Propositions~\ref{prop:JointRegTimesBound},
  \ref{prop:IndRegTimesBound}, \ref{prop:TVdistance-joint-ind-1step})
  there is a constant such that for all $x,x',\wt{x},\wt{x} \in \Z^d$
  \begin{align}
    \label{eq:TVdistance-joint-ind-1step}
    \Norm{\Pr^{\joint}_{x,x'}\big((\widehat{X}_1-x, \widehat{X}_1'-x',T_1)
    \in \cdot\big) - \Pr^{\indi}_{\wt{x},\wt{x}'}\big((\widehat{X}_1-\wt{x}, \widehat{X}_1'-\wt{x}',T_1) \in
    \cdot\big)}_{\mathrm{TV}} \le C \norm{x-x'}^{-\beta} .
  \end{align}
\end{remark}

Since $(\widehat{X}_n,\widehat{X}'_n,T_n)$ is a Markov chain under
$\bP^\joint_{x_0,x_0'}$ as well as under $\bP^\indi_{x_0,x_0'}$
(Remarks~\ref{rem:jointMC} and \ref{rem:indMC}), we can and shall in
the following iterate the coupling of their transition dynamics
provided by Remark~\ref{rem:TVdistance-joint-ind-1step} to produce a
Markov chain
\begin{equation}
  \label{eq:bigchain}
  \big( \widehat{X}^\joint_n, \widehat{X}'^{\,\joint}_n, T^\joint_n,
  \widehat{X}^\indi_n, \widehat{X}'^{\,\indi}_n, T^\indi_n \big)_{n \in \N_0}
\end{equation}
on $(\Z^d\times\Z^d\times\Z_+)^2$ such that $(\widehat{X}^\joint_n,
\widehat{X}'^{\,\joint}_n, T^\joint_n)$ evolves as under
$\bP^\joint_{x_0,x_0'}$ and $(\widehat{X}^\indi_n,
\widehat{X}'^{\,\indi}_n, T^\indi_n)$ evolves as under
$\bP^\indi_{x_0,x_0'}$ with the following property: We say that $m$ is
an \emph{uncoupled step} if
\begin{equation}
  \label{eq:uncoupledstep}
  \big( \widehat{X}^\joint_m-\widehat{X}^\joint_{m-1}, \widehat{X}'^{\,\joint}_m-\widehat{X}'^{\,\joint}_{m-1},
  T^\joint_m-T^\joint_{m-1}\big)
  \neq
  \big( \widehat{X}^\indi_m-\widehat{X}^\indi_{m-1}, \widehat{X}'^{\,\indi}_m-\widehat{X}'^{\,\indi}_{m-1},
  T^\indi_m-T^\indi_{m-1}\big).
\end{equation}
Then
\begin{align}
  \label{eq:uncoupledstep.prob}
  & \bP\Big( m\text{ is an uncoupled step} \,\Big|\,
  \big( \widehat{X}^\joint_{m-1}, \widehat{X}'^{\,\joint}_{m-1}, T^\joint_{m-1},
    \widehat{X}^\indi_{m-1}, \widehat{X}'^{\,\indi}_{m-1}, T^\indi_{m-1} \big)=(x,x',t,\wt{x},\wt{x}',\wt{t}) \Big) \notag \\
  & \qquad \le C \norm{x-x'}^{-\beta}
\end{align}
uniformly in $(x,x',t,\wt{x},\wt{x}',\wt{t})$ with $\beta$ as in
\eqref{eq:TVdistance-joint-ind-1step}.

We work in the following with the chain \eqref{eq:bigchain}. We will
see below that uncoupled steps are sufficiently rare to transfer many
properties from $\bP^\indi_{x_0,x_0'}$ to $\bP^\joint_{x_0,x_0'}$ (see
Lemma~\ref{lem:coupl} for $d \ge 2$, for $d=1$ see
Lemma~\ref{lem:bound_for_A1_and_A2} and the discussion around it).

\subsubsection{Auxiliary results}
\label{sect:aux-1}

  Write for $r>0$
  \begin{align}
    \label{eq:hittingtimes}
    \begin{split}
      h(r) & \coloneqq \inf\{ k \in \bZ_{+} : \norm{\widehat{X}_k -
        \widehat{X}_k'}_2 \leq r \}\\
      H(r) & \coloneqq \inf \{ k \in \bZ_{+} : \norm{\widehat{X}_k -
        \widehat{X}_k'}_2 \geq r \}
    \end{split}
  \end{align}
  where $\norm{ \cdot }_2$ denotes the Euclidean norm on $\bZ^d$.

\begin{lemma}[Hitting probabilities for spheres, analogous to
  \cite{BirknerCernyDepperschmidtGantert2013} Lemma~3.6]
  \label{lemma:exitAnnulus}
  Put for $r_1<r<r_2$
  \begin{align}
    f_d(r;r_1,r_2)= \begin{cases}
      \frac{r-r_1}{r_2-r_1}, & \text{when }d=1,\\
      \frac{\log r - \log r_1}{\log r_2 - \log r_1}, \quad & \text{when }d=2,\\
      \frac{r_1^{2-d}-r^{2-d}}{r_1^{2-d}-r_2^{2-d}}, \quad & \text{when }d\geq 3.
    \end{cases}
  \end{align}
  For every $\varepsilon >0$ there are (large) $R$ and $\wt{R}$ such
  that for all $r_2>r_1>R$ with $r_2-r_1>\wt{R}$ and
  $x,y \in \bZ^d$ satisfying $r_1 < r=\norm{x-y}_2<r_2$
  \begin{align}
    (1-\varepsilon)f_d(r;r_1,r_2) \leq \Pr^{\indi}_{x,y}(H(r_2)<h(r_1))
    \leq (1+\varepsilon)f_d(r;r_1,r_2).
  \end{align}
\end{lemma}
\begin{proof}[Proof sketch]
  It is sufficient to show that $(\widehat{X}_n-\widehat{X}_n')_n$
  satisfies an annealed invariance principle. Since
  $(\widehat{X}_n-\widehat{X}_n')_n$ has independent increments under
  $\Pr^{\indi}_{x,y}$, this follows from the observation that
  $\bE^\indi_{x,y}[(\widehat{X}_1-x)-(\widehat{X}_1'-y)] = 0$ and
  $\bE^\indi_{x,y}[\norm{\widehat{X}_1-x-\widehat{X}_1'+y}^2] <
  \infty$ by Remark~\ref{rem:indMC} and
  Proposition~\ref{prop:IndRegTimesBound} (assuming that $\beta>2$).
\end{proof}

The rest of this section is concerned with the case $d \ge 2$. We
first provide a separation lemma similar to Lemma~3.8 in
\cite{BirknerCernyDepperschmidtGantert2013} for the model from
\cite{BirknerCernyDepperschmidt2016}. Recall the definition of the
hitting times $h(\cdot)$ and $H(\cdot)$ from \eqref{eq:hittingtimes}.
\begin{lemma}[Separation lemma]\label{lem:separ}
  Let $d\geq 2$. For all small enough $\delta >0$ and
  $\varepsilon^*>0$ there exist $C,c>0$ so that
  \begin{align}
    \label{eq:2}
    \sup_{x_0, x_0' \in \Z^d} \Pr^{\joint}_{x_0, x_0'}\left( H(n^{\delta}) >
    n^{2\delta + \varepsilon^*} \right) \leq \exp(-C n^c).
  \end{align}
\end{lemma}
In fact, this should hold for any $\delta>0$, not just small
ones. However, the current statement suffices for our purposes.
\begin{proof}
  We will adapt the idea and the structure of the proof of Lemma~3.8
  in \cite{BirknerCernyDepperschmidtGantert2013}, dividing the
  argument into 5 steps. Most adaptations are relatively
  straightforward but Steps~4 and 5 will require additional work and
  some new ideas because unlike the situation in
  \cite{BirknerCernyDepperschmidtGantert2013}, the coupling error
  guaranteed by Proposition~\ref{prop:TVdistance-joint-ind-1step} does
  not decay exponentially in the initial separation of the two walks.
  We can hence, unlike in \cite{BirknerCernyDepperschmidtGantert2013},
  not immediately move from a distance which is a large multiple of
  $\log n$ to a distance $n^\delta$ without occasionally breaking the
  coupling of increments.

  \medskip
  \noindent
  \emph{Step 1.} We want to show that there exists a small
  $\varepsilon_1 >0$ and $b_4 \in (0,1/2)$, $b_5>0$ such that
  \begin{align}
    \label{eq:Step1}
    \Pr^{\joint}_{x,y}\left( H(\varepsilon_1 \log n) > n^{b_4}
    \right) \leq c n^{-b_5},
  \end{align}
  uniformly in $x,y \in \bZ^d$.

  To that end we construct suitable ``corridors'' to guide the random
  walks to a certain distance within the first steps.

  First have a closer look at a possible configuration of the
  environment that yields a suitable lower bound for the probability
  to reach a distance of $\varepsilon_1\log n$ in $\varepsilon_1 \log
  n$ steps. Recall the dynamics of the random walks whenever they are
  on a site with $\eta=1$, see Assumption~\ref{ass:appr-sym}. Since
  the reference transition kernel $\kappa_\rf$ is non degenerate there
  exist a possible steps $k\in\bZ^d$ whose first coordinate $k_1$ is
  strictly positive and we set $\delta_1 \coloneqq \kappa_\rf(\{k :
  k_1 \ge 1\})$. Then on any site $(y,-m)\in\bZ^d\times\bZ$ with
  $\eta_{-m}(y)=1$ the random walk can jump a distance of at least
  $+1$ in the first coordinate (and similarly the first coordinate of
  $X'$ can jump at least one position to the left). Thus, if we ensure
  that the random walks hit only sites where $\eta=1$, the distance
  increases by at least $2$ with probability at least $\delta_2
  \coloneqq \delta_1-\varepsilon_\rf>0$. (Here, we implicitly assume
  that $\varepsilon_\rf<\delta_1$, which is compatible with
  Assumption~\ref{ass:appr-sym}.)

  Recall Construction~\ref{constr:reg-time}. If the two random walks
  just regenerated they can do so at the next step if the values of
  $\eta$ in the $b_\out$-vicinity of their next positions are only
  $1$'s. This happens with a uniform lower bound (since a successful
  attempt at regenerating has a uniform lower bound as shown in the
  proof of Proposition~\ref{prop:JointRegTimesBound}) which we denote
  by $\delta_3>0$. Therefore, if $\eta=1$ in the $b_\out$-vicinity of
  both random walks along the path, they will regenerate at every such
  step.

  Consequently, for every step, the probability that the distance between
  the random walks increases by at least $2$ is
  \begin{align*}
        \bP^\joint_{x,y}(\norm{\widehat{X}_i-\widehat{X}'_i}\ge 2
    +\norm{\widehat{X}_{i-1}-\widehat{X}'_{i-1}}) > \delta_2^2\delta_3>0
  \end{align*}
  and iteratively
  \begin{align}
    \Pr^{\joint}_{x,y}\left(\norm{\widehat{X}_j -
    \widehat{X}'_j} \ge  2j \right) \geq \hat\delta^j,
  \end{align}
  for some $\hat\delta >0$. As in the proof in
  \cite[Lemma~3.8]{BirknerCernyDepperschmidtGantert2013} we conclude
  by a ``restart'' argument that
  \begin{align}
    \Pr^{\joint}_{x,y}\left( H(\varepsilon_1 \log n) > m
    \varepsilon_1 \log n \right) \leq \left( 1 - n^{-\varepsilon_1
    \log(1/\hat\delta)} \right)^m \leq \exp\left( -mn^{-\varepsilon_1
    \log(1/\hat\delta)} \right).
  \end{align}
  If we choose $\varepsilon_1$ so small that
  $-\varepsilon_1\log(\hat\delta) \in (0,1/2)$, and choose $b_4 \in
  (-\varepsilon_1 \log(\hat\delta),1/2)$, $b_5 >0$ and set $m=b_5
  n^{\varepsilon_1 \log(1/\hat\delta)}\log n$ we have shown
  \eqref{eq:Step1} since $m \varepsilon_1 \log n \le n^{b_4}$ if $n$
  is large. Furthermore note that $b_4$ can be chosen arbitrarily
  small by choosing $\varepsilon_1$ small. (Note that by choosing $m$
  somewhat larger we could arrive at superalgebraic decay in
  \eqref{eq:Step1}.) \medskip

  \noindent
  \emph{Step 2.} Next we show that for any $K_2>0$ there exists
  $\delta_2\in (0,1)$ such that for all $x,y \in \bZ^d$ with
  $\varepsilon_1 \log n \leq \norm{x-y} < K_2 \log n$ and $n$ large
  enough
  \begin{align}
    \label{eq:Step2}
    \Pr^{\joint}_{x,y}\Bigl( H(K_2 \log n) <
    h(\frac{1}{2}\varepsilon_1 \log n) \wedge (K_2 \log n)^{2+\alpha} \Bigr)
    \geq \delta_2,
  \end{align}
  where $h(m) \coloneqq \inf \{k : \norm{\widehat{X}_k -
    \widehat{X}'_k} \leq m\}$ and $\alpha>0$ is a small constant to be
  tuned later. To this end we couple $\Pr^{\indi}_{x,y}$ with
  $\Pr^{\joint}_{x,y}$. As mentioned above the probability that the
  coupling fails will decay algebraically in the distance of the
  starting positions of the random walks. Thus, the left-hand side of
  \eqref{eq:Step2} is bounded from below by
  \begin{align}
    \label{eq:AH2.1}
    \Pr^{\indi}_{x,y}\Bigl(H(K_2\log n) <
    h(\frac{1}{2}\varepsilon_1 \log n) \wedge (K_2 \log n)^{2+\alpha}\Bigr) -
    C(K_2 \log n)^{2+\alpha} (\frac{1}{2}\varepsilon_1 \log n)^{-\beta},
  \end{align}
  where $\beta > 0$ is from
  Remark~\ref{rem:TVdistance-joint-ind-1step} and can be chosen
  arbitrarily large so that the second term will go to 0 as $n$ tends
  to infinity. Under $\bP^\indi$ we can use
  Lemma~\ref{lemma:exitAnnulus}. Therefore we can obtain a lower bound
  for $d\ge 3$ for the left term in \eqref{eq:AH2.1} by combining
  \begin{align}
    \label{eq:34}
    \begin{split}
      \Pr^{\indi}_{x,y}\Bigl(H(K_2\log n) < h(\frac{1}{2}\varepsilon_1
      \log n) \Bigr)
      & \ge (1-\varepsilon) \frac{(\frac{\varepsilon_1}{2} \log n)^{2-d}-(\varepsilon_1
        \log n)^{2-d}}{(\frac{\varepsilon_1}{2} \log n)^{2-d} - (K_2 \log n)^{2-d}}\\
      &= (1-\varepsilon) \frac{(2^{d-2}-1)\varepsilon_1^{2-d}}{2^{d-2}\varepsilon_1^{2-d}
        - K_2^{2-d}}
    \end{split}
  \end{align}
  and
  \begin{align}
    \widetilde{\varepsilon}(n)
    & \coloneqq \bP^\indi_{x,y}\Bigl(H(K_2\log n) > (K_2\log n)^{2+\alpha} \Bigr) \notag \\
    & \, \le \bP^\indi_{x,y}\Bigl( \norm{ \widehat{X}_{(K_2\log n)^{2+\alpha}} - \widehat{X}'_{(K_2\log n)^{2+\alpha}}}
      \le K_2\log n\Bigr)
      \mathop{\longrightarrow}_{n\to\infty} 0
      \label{eq:HK2logn.too.big}
  \end{align}
  by the CLT for $\widehat{X}-\widehat{X}'$ under $\bP^\indi_{x,y}$.

  Thus, we have
  \begin{align*}
    \Pr^{\indi}_{x,y}\Bigl(H(K_2\log n)
    & < h(\frac{1}{2}\varepsilon_1 \log n) \wedge (K_2 \log n)^{2+\alpha}\Bigr)\\
    & \ge (1-\varepsilon)
      \frac{(2^{d-2}-1)\varepsilon_1^{2-d}}{2^{d-2}\varepsilon_1^{2-d} - K_2^{2-d}}
      - \widetilde{\varepsilon}(n)
  \end{align*}
  and the right hand side is bounded away from $0$.

  For $d=2$ we can use the same arguments with a slightly different
  lower bound when using Lemma~\ref{lemma:exitAnnulus} similar to
  \eqref{eq:34}. Here we have
  \begin{align*}
    \Pr^{\indi}_{x,y}\Bigl(H(K_2\log n) < h(\frac{1}{2}\varepsilon_1 \log n) \Bigr)
    & \ge (1-\varepsilon)\frac{\log (\varepsilon_1 \log n) - \log
      (\frac{1}{2}\varepsilon_1 \log n)}{\log (K_2 \log n)-\log
      (\frac{1}{2}\varepsilon_1 \log n)}\\
    & = (1-\varepsilon)\frac{\log 2}{\log K_2 - \log (\varepsilon_1/2)}.
  \end{align*}
  Combining the above estimates for $d=2$ and $d \ge 3$ with
  \eqref{eq:AH2.1} we obtain a uniform lower bound in \eqref{eq:Step2}
  for all $d\ge 2$.

  \medskip
  \noindent
  \emph{Step 3.} Combining the previous steps we see that we can
  choose a large $K_3$ and $b_6 \in (b_4, 1/2)$ such that uniformly
  in $x,y \in \bZ^d$ we have
  \begin{align}
    \Pr^{\joint}_{x,y}\left( H(K_3 \log n) \leq n^{b_6}
    \right) \geq \delta_3 > 0 \quad \text{ for } n \text{ large
    enough.}
  \end{align}

  At this point we will divert from the corresponding proof in
  \cite{BirknerCernyDepperschmidtGantert2013} since Step~4 there does
  not hold for our model. They were able, due to the exponential decay
  of the total variation distance of $\bP^\indi_{x,y}$ and
  $\bP^\joint_{x,y}$, to jump from a distance of $\log n$ directly to
  $n^{\delta}$, whereas we need to iterate through smaller distances.
  We first go to a distance of $\log^2 n$. For that we use the same
  arguments we used in Step~2 with starting positions $x,y\in\bZ^d$
  such that $K_3\log n \leq \norm{x-y} \leq \log^2 n$. Our next aim is
  to obtain a lower bound for
  \begin{align}
    \Pr^{\joint}_{x,y} \left( H(\log^2 n) < h(\varepsilon_1 \log
    n) \wedge \log^{4+\alpha} n \right).
  \end{align}
  We can bound that probability from below by using the arguments from
  Step~2 and get the lower bound
  \begin{align}
    \label{eq:53}
    \Pr^{\indi}_{x,y} \left( H(\log^2 n) < h(\varepsilon_1 \log n)
    \wedge \log^{4+\alpha} n \right) - \big(\log^{4+\alpha} n \big) (\varepsilon_1 \log n)^{-\beta}.
  \end{align}
  Since we can choose $\beta>8 $ the right-hand side will go to 0 as
  $n$ tends to infinity. Again we use Lemma~\ref{lemma:exitAnnulus}.
  For dimension $d\geq 3$
  \begin{align}
    \label{eq:30}
    \begin{split}
      \Pr^{\indi}_{x,y}\bigl(H(\log^2 n) < h(\varepsilon_1 \log n) \bigr)
      & \geq (1-\varepsilon)
        \frac{\varepsilon_1^{2-d}-K_3^{2-d}}{\varepsilon_1^{2-d}-(\log^2 n)^{2-d}}\\
      & = (1-\varepsilon) \Bigl(1 - \frac{K_3^{2-d}-(\log
          n)^{2-d}}{\varepsilon_1^{2-d}-(\log n)^{2-d}}\Bigr)
    \end{split}
  \end{align}
  and for $d=2$ note that $\log^{1/2}n < \varepsilon_1\log n$ holds for
  large $n$ and we get
  \begin{align}
    \Pr^{\indi}_{x,y}(H(\log^2 n) < h(\log^{1/2} n)) \geq
    (1-\varepsilon)\frac{\log(K_3\log n)-\log(\log^{1/2} n)}{\log(\log^2
    n)-\log(\log^{1/2} n)}\geq\frac{1-\varepsilon}{3}
  \end{align}
  which are both bounded from below for $n$ large enough. For the
  lower bound in $d=2$ we need to switch from $\varepsilon_1 \log n$
  to $\log^{1/2}n$ since the lower bound estimate with the
  $\varepsilon_1 \log n$ will tend to 0 for $n\to \infty$. The
  estimate for a failed coupling will still tend to 0 since we can
  choose $\beta>0$ large. If we do the same for the step from $\log^2
  n$ to $\log^4 n$ we get
  \begin{align}
    \Pr^{\joint}_{x,y}\left( H(\log^4 n) < h(K_3\log n) \wedge
    \log^{8+\alpha} n \right) \geq \delta_4 > 0.
  \end{align}
  with the coupling we can use the same arguments by choosing $\beta$
  accordingly. The invariance principle will yield a bound that goes
  to 0 if $n$ tends to infinity and we have for $d=3$
  \begin{align}
    \label{eq:31}
    \begin{split}
      \Pr^{\indi}_{x,y}(H(\log^4 n) < h(K_3\log n))
      & \geq (1-\varepsilon)
        \frac{(K_3 \log n)^{2-d}-(\log^2 n)^{2-d}}{(K_3 \log n)^{2-d} -
        (\log^4 n)^{2-d}} \\
      & = (1-\varepsilon)\frac{K_3^{2-d} - (\log n)^{2-d}}{K_3^{2-d} -
        (\log^3 n)^{2-d}},
    \end{split}
  \end{align}
  and for $d=2$
  \begin{align}
    \Pr^{\indi}_{x,y}(H(\log^4 n) < h(K_3\log n))\geq (1-\varepsilon)
    \frac{\log\log n - \log K_3}{3\log\log n - \log
    K_3}\geq\frac{(1-\varepsilon)}{4}
  \end{align}
  which as well are bounded from below by a positive constant. Thus we
  have for $d\geq2$ and $n$ large enough
  \begin{align}
    \Pr^{\joint}_{x,y} \left( H(\log^4 n) < h(\varepsilon_1 \log n)
    \wedge \log^{8+\alpha} n \right) \geq \delta_5 >0
  \end{align}
  Now we can combine this with Step~1 and Step~2 to get, in a similar
  way to Step~3,
  \begin{align}\label{eq:NewStep3}
    \Pr^{\joint}_{x,y}\left( H(\log^4n) \leq n^{b_6} \right)
    \geq \delta_6 >0 \qquad \text{ for } n \text{ large enough.}
  \end{align}

  Next we will go from $\log^4 \! n$ to $\log^8 \!n$ and then to
  $\log^{16} \!n$ and so on. Iterating, we go from $\log^{2^i} \!n$ to
  $\log^{2^{i+1}} \!n$ until we reach $n^{\delta}$. If we could do
  that in a number of times that does not depend on $n$ the proof
  would be completed. Instead we have to show that if $n$ tends to
  infinity, the product over the probabilities for all those steps is
  still bounded from below by a small constant away from $0$.

  From here on we can formulate the steps in more generally and handle
  the remaining part of the proof in a unified way. Say we are at a
  distance of $\log^{2^j} \! n$ and want to reach distance
  $\log^{2^{j+1}} \! n$. We will do that in the same way as in Step 2.
  For that we split the next step in cases $d\geq 3$ and $d=2$

  \medskip
  \noindent
  \emph{Step 4 $(d\geq 3)$.} We start with $d\geq 3$ and want to
  bound
  \begin{align}
    \label{eq:52}
    \Pr^{\joint}_{x,y} \left( H(\log^{2^{j+1}}\! n) <
    h(\log^{2^{j-1}}\! n) \wedge \log^{2^{j+2}+\alpha}\! n \right),
  \end{align}
  where $\log^{2^j} n \leq \norm{x-y} < \log^{2^{j+1}}\! n$. If we
  couple with $\Pr^{\indi}_{x,y}$ we get the lower bound
  \begin{align}
    \label{eq:100}
    \Pr^{\indi}_{x,y} \left(H(\log^{2^{j+1}}\! n) < h(\log^{2^{j-1}}\! n)
    \wedge \log^{2^{j+2}+\alpha}\! n \right)
    - \log^{2^{j+2}+\alpha}\! n \cdot (\log n)^{-2^{j-1}\beta}.
  \end{align}
  If $\beta > 8$ the second term will go to zero. For the first term
  we get a lower bound in a similar way to Step 2. We first observe
  \begin{multline}
    \Pr^{\indi}_{x,y} \left( H(\log^{2^{j+1}}\! n) <
      h(\log^{2^{j-1}}\! n) \wedge \log^{2^{j+2}+\alpha}\! n \right) \\
    \geq \Pr^{\indi}_{x,y} \left( H(\log^{2^{j+1}}\! n) <
      h(\log^{2^{j-1}}\! n) \right) - \Pr^{\indi}_{x,y} \left(
      H(\log^{2^{j+1}}\! n) \geq \log^{2^{j+2}+\alpha}\! n \right)
  \end{multline}
  Arguing as in \eqref{eq:HK2logn.too.big} we obtain
  \begin{align}
    \label{eq:101b}
    \widetilde{\varepsilon}'(n) := \Pr^{\indi}_{x,y} \left(
    H(\log^{2^{j+1}}\! n) \geq \log^{2^{j+2}+\alpha}\! n \right)
    \mathop{\longrightarrow}_{n\to\infty} 0
  \end{align}
  and with Lemma~\ref{lemma:exitAnnulus}
  \begin{align}
    \label{eq:102}
    \begin{split}
      \Pr^{\indi}_{x,y}\left( H(\log^{2^{j+1}}\! n) < h(\log^{2^{j-1}}\! n) \right)
      & \geq (1-\varepsilon)\frac{(\log^{2^{j-1}}\! n)^{2-d} -
        (\log^{2^{j}} n)^{2-d}}{(\log^{2^{j-1}}\! n)^{2-d} -
        (\log^{2^{j+1}}\! n)^{2-d}}\\
      & = (1-\varepsilon) \frac{(\log n)^{-2^{j-1}(2-d)} - 1}{(\log n)^{-2^{j-1}(2-d)}
        - (\log n)^{2^j (2-d)}}\\
      & = (1-\varepsilon) (1 - \frac{1 - (\log n)^{2^j (2-d)}}{(\log
        n)^{-2^{j-1}(2-d)} - (\log n)^{2^j (2-d)}}\\
      & \geq (1-\varepsilon)(1 - \frac{2}{(\log n)^{-2^{j-1}(2-d)}})
    \end{split}
  \end{align}
  for $n$ large enough, because for $d \ge 3$
  \begin{align*}
    (\log n)^{2^j (2-d)} \xrightarrow{n \to \infty} 0\quad
    \text{and} \quad (\log n)^{-2^{j-1}(2-d)} \xrightarrow{n \to \infty} \infty.
  \end{align*}
  Combining equations \eqref{eq:100}, \eqref{eq:101b}
  and \eqref{eq:102} we conclude
  \begin{multline}
    \label{eq:lowerBoundForIterationDim3}
    \Pr^{\joint}_{x,y}
    \left(H(\log^{2^{j+1}}\! n) < h(\log^{2^{j-1}}\! n) \wedge
      \log^{2^{j+2}+\alpha}\! n \right)\\
    \geq 1 - \varepsilon - (\log n)^{2^{j+2}+\alpha}
    \cdot (\log n)^{-2^{j-1}\beta} - 2(\log n)^{2^{j-1}(2-d)}
    - \widetilde{\varepsilon}'(n)
  \end{multline}

  \medskip
  \noindent
  \emph{Step 4 $(d=2)$.} Next we want to get a lower bound
  for the case $d=2$. Again assume
  $\log^{2^j}\!n\leq \norm{x-y} \leq \log^{2^{j+1}}\! n$ and we want to bound
    \begin{align}
      \Pr^{\joint}_{x,y} \left( H(\log^{2^{j+1}}\! n) <
      h(\log^{2^{j-1}}\! n) \wedge \log^{2^{j+2}+\alpha}\! n \right).
    \end{align}
    The difference to the case $d\geq 3$ lies in the application of
    Lemma~\ref{lemma:exitAnnulus} so we will concentrate on those
    differences.
    \begin{align}
      \label{eq:lowerBoundForIterationDim2}
      \begin{split}
        \Pr^{\indi}_{x,y}\left(H(\log^{2^{j+1}}\! n) < h(\log^{2^{j-1}}\! n)\right)
        & \geq (1-\varepsilon)\frac{\log(\log^{2^j}\!n) -
          \log(\log^{2^{j-1}}n)}{\log(\log^{2^{j+1}}\! n) -
          \log(\log^{2^{j-1}}n)}\\
        & =(1-\varepsilon) \frac{2^j-2^{j-1}}{2^{j+1}-2^{j-1}} =
        \frac{(1-\varepsilon)}{3}
      \end{split}
    \end{align}
    With this we get a lower bound for the left hand side of
    \eqref{eq:lowerBoundForIterationDim3} in $d=2$ that is bounded
    away from $0$.

    \medskip
    \noindent
    \emph{Step 5.} Now we bring the previous steps together. By
    \eqref{eq:NewStep3} we can get to a distance of $\log^4 n$ in at
    most $n^{b_6}$ steps, where $b_6>0$ is small. From there we can
    iterate Step~4. We want to get to a distance of $n^{\delta}$.
    Since
    \begin{align*}
       \log^{2^j}\!n = n^{\delta} &
      \; \iff \; 2^j \log\log n = \delta \log n\\
      & \; \iff \; j \log(2) + \log\log\log n = \log(\delta) + \log\log n ,
    \end{align*}
    we need $j_n^* \coloneqq \frac{1}{\log 2}(\log \delta + \log\log n
    -\log\log\log n)$ many iterations of Step~4. Following this
    construction the number of steps we need to get to $n^{\delta}$ is
    the number of steps to get to $\log^4 n$ and then the iterations
    of Step~2. Thus, the total number is
    \begin{align}
      \label{eq:1}
      \begin{split}
        n^{b_6} + \sum_{i=1}^{j_n^*-1} \log^{2^{i+2+\alpha}} n
        & = n^{b_6} + \sum_{i=1}^{j_n^*-1} \log^{\alpha}(n)(\log^{ 2^{i+1}} n)^{2}\\
        & \leq  n^{b_6} + c\log(\log n)\cdot\log^{\alpha}(n) n^{2 \delta}\\
        & \leq n^{b_6} + c n^{2\delta+\alpha}\\
        & \leq n^{2\delta+2\alpha}
      \end{split}
    \end{align}
    for $n$ large enough, where we used $j_n^*\leq c\log\log n$ for a
    positive constant $c>0$ and that $b_6$ can be chosen smaller than
    $\alpha$. Obviously since $\alpha$ can be chosen small we can
    assume that $2\alpha<\varepsilon^*$. We only need to show now that
    the probability to make it to a distance of $n^{\delta}$ in that
    time is positive. Then we can use the Markov property and have
    shown the claim. To show that we just multiply the above estimated
    probabilities. That gives us the following product for $d=3$
    \begin{align}
      \begin{split}
        \Pr^{\joint}_{x_0,x'_0}
        & \left(\norm{\widehat{X}_{n^{2\delta+2\alpha}}
          -\widehat{X}'_{n^{2\delta+2\alpha}}} \geq n^{\delta}
        \right)\\
        & \geq \delta_6 \cdot \prod_{i=2}^{c \log\log n} \left( 1 -
          \varepsilon - (\log n)^{2^{i+1}-2^{i-1}\beta} - 2(\log
          n)^{2^{j-1}(2-d)} -c\log^{-\alpha} n \right).
      \end{split}
    \end{align}
    Hence the probability that all this works in ``one go'' is at least
    \begin{align*}
    (1-\varepsilon')^{c\log\log n}\geq (\log
      n)^{-c\varepsilon'}
    \end{align*}
    for some small positive $\varepsilon'$.

    For $d=2$ it is easy to see that we also get a lower bound of the
    same sort by taking the product over the right hand side of
    \eqref{eq:lowerBoundForIterationDim2}. Thus, we will need
    approximately $(\log n)^{c\varepsilon'}$ many attempts. By
    \eqref{eq:1}, each of the attempts takes $n^{2\delta+2\alpha}$ steps for
    some constant $\alpha< \varepsilon^*/2$. We have $n^{\varepsilon^*-2\alpha}$
    attempts and thus
    \begin{align*}
      \Pr^\joint_{x_0,x'_0}(H(n^\delta)>n^{2\delta+\varepsilon^*}) \leq \Big(1-(\log
      n)^{c\log(1-\varepsilon')}\Big)^{n^{\varepsilon^*-2\alpha}} \leq \exp\Big(
      -c\frac{n^{\varepsilon^*-2\alpha}}{\log^{\tilde{c}}n} \Big)
    \end{align*}
    which shows \eqref{eq:2} and concludes the proof of the lemma.
\end{proof}

Recall
\begin{lemma}[Analogue to \cite{BirknerCernyDepperschmidtGantert2013}
  Lemma~3.9]
  \label{lem:coupl} Let $d\geq 2$. The Markov chain\
  \eqref{eq:bigchain} which describes the coupling of a $\joint$-pair
  an $\indi$-pair of walks has the following property: For any
  $0<b_1<1/2$ there is a constants $C>0$ and $c>0$ so that for
  arbitrary $x_0,x_0'$
  \begin{align}
    \label{eq:5}
    \Pr_{x_0, x_0'} (\text{at most $N^{b_1}$ uncoupled
    steps before time $N$}) \ge 1-CN^{-c},
  \end{align}
  where uncoupled steps are defined in \eqref{eq:uncoupledstep}.
\end{lemma}
Note that the constant $c$ in Lemma~\ref{lem:coupl} can be chosen
arbitrarily large if $\beta$ is large enough. This will be clear from
the proof.

\begin{proof}
  In order to make use of Remark~\ref{rem:TVdistance-joint-ind-1step}
  we need to ensure a minimal distance between the random walkers. To
  that end we decompose the trajectory of $(\widehat{X}^\joint_n,
  \widehat{X}'^\joint_n)_n$ into excursions. For a large constant
  $K'$, to be tuned later, and $\delta>0$ small enough to be
  admissible in Lemma~\ref{lem:separ} we define stopping times
  $\mathcal{R}_i,\mathcal{D}_i$ and $\mathcal{U}$ (for the canonical
  filtration of the coupled process \eqref{eq:bigchain}) by
  $\mathcal{R}_0=0$ and, for $i\ge 1$ and small $0<\alpha<\delta$,
  \begin{align*}
    \mathcal{D}_i
    & = \inf\{ k\ge \mathcal{R}_{i-1} \colon
      \norm{\widehat{X}^\joint_k -\widehat{X}'^\joint_k} \ge N^{\delta} \},\\
    \mathcal{R}_i
    & = \inf\{k\ge \mathcal{D}_i \colon
      \norm{\widehat{X}^\joint_k -\widehat{X}'^\joint_k} \le N^{\alpha} \},\\
    \mathcal{U}
    &= \inf\{k\ge 0 \colon
      \norm{\widehat{X}^\joint_k -\widehat{X}'^\joint_k} \ge K' N \}.
  \end{align*}
  We want to ensure that the time spent in the intervals
  $\mathcal{D}_i-\mathcal{R}_{i-1}$ is suitably short with high
  probability. We find an upper bound for this time if
  $\mathcal{D}_1-\mathcal{R}_{0}$, i.e.\ the initial time spent to
  separate to a distance of $N^\delta$, is as large as possible. This
  is the case when both random walks start on the same position. Let
  $x_0=x'_0=0$ and, to avoid too much notation, we write
  $\bP^\joint_{0,0} = \bP^\joint$. Furthermore let $J$ be the unique
  integer such that $\mathcal{D}_J \le \mathcal{U} \le \mathcal{R}_J$
  and set $Y^\joint_i=\norm{\widehat{X}^\joint_i
    -\widehat{X}'^\joint_i}$, the distance of the two walkers at the
  $i$-th regeneration time. In the following, for the sake of
  readability, we pretend that the transitions only depend on the
  distance and we write $\bP^\joint_r$ for the distribution of
  $Y^\joint$ starting from $Y^\joint_0=r$. (This is of course not
  literally true since the transitions depend on the displacement
  $\widehat{X}^\joint_i -\widehat{X}'^\joint_i$ itself, see
  Remark~\ref{rem:jointMC}. However, in view of
  Remark~\ref{rem:TVdistance-joint-ind-1step} and the invariance
  principle under $\bP^\indi$, it is `almost' true. The computation
  could be written completely rigorously but we think it would be less
  readable.)

  Adapting arguments from Step~2 in the proof of Lemma~\ref{lem:separ}
  we obtain, using the invariance principle for $d=2$,
  \begin{align*}
    \bP^\joint_{N^\delta}(H(K'N)< h(N^\alpha))
    & \ge \bP^\joint_{N^\delta}(H(K'N)< h(N^\alpha)\wedge (K'N)^3)\\
    & \ge \bP^\indi_{N^\delta}(H(K'N)< h(N^\alpha)\wedge (K'N)^3) - (K'N)^3N^{-\alpha\beta}\\
    & \ge\bP^\indi_{N^\delta}(H(K'N)< h(N^\alpha))
      - \bP^\indi_{N^\delta}(H(K'N) \ge (K'N)^3) \\
    & \hspace{9cm}- (K'N)^3N^{-\alpha\beta}\\
    & \ge \bP^\indi_{N^\delta}(H(K'N)< h(N^\alpha))
      - \widetilde{\varepsilon}(N) -(K'N)^3N^{-\alpha\beta}\\
    & \ge (1-\varepsilon)\frac{\log N^\delta - \log N^\alpha}{\log(K'N) -
      \log N^\alpha}
      - \widetilde{\varepsilon}(N) - CN^{3-\alpha\beta}\\
    & \ge (1-\varepsilon)\frac{\delta - \alpha}{1-\alpha - \frac{\log
      K'}{\log N}}
      - \widetilde{\varepsilon}(N)- CN^{3-\alpha\beta}\\
    & \ge \frac{1}{2}\frac{\delta - \alpha}{1-\alpha}
  \end{align*}
  for $N$ large enough and $3/\alpha < \beta$, with
  $\widetilde{\varepsilon}(N) \to 0 $ as $N\to\infty$.  For $d\ge 3$
  the only change is for the lower bound
  \begin{equation*}
    \bP^\indi_{N^\delta}\big( H(K'N) < h(N^\alpha) \big) \ge
    (1-\varepsilon) \frac{N^{\alpha(2-d)}-N^{\delta(2-d)}}{N^{\alpha(2-d)}-(K'N)^{2-d}}
  \end{equation*}
  which also is larger than
  $\frac{1}{2}\frac{\delta-\alpha}{1-\alpha}$. Thus $J$ is dominated by a geometric
  distribution with success parameter $\frac{1}{2}\frac{\delta -
    \alpha}{1-\alpha}$ as $N\to \infty$, for any $d\ge 2$. Therefore
  there is a constant $K(\alpha,\delta)>0$,
  \begin{equation*}
    \bP^\joint(J\ge K(\alpha,\delta)\log N) \le
    N^{K(\alpha,\delta)\log(1-(\delta-\alpha)/(2-2\alpha))}.
  \end{equation*}
  Note that we can make the exponent above arbitrarily small by
  choosing $K(\alpha,\delta)$ large. Applying the separation lemma,
  Lemma~\ref{lem:separ}, we get
  \begin{align}
    \bP^\joint(\mathcal{D}_i - \mathcal{R}_{i-1} \ge
    N^{2\delta+\varepsilon^*}) \le \exp(-b_3N^{b_4})
  \end{align}
  for constants $\varepsilon^*>0$ and $b_3,b_4>0$, since
  $N^\delta-N^\alpha\le N^{\delta+\varepsilon^*/2}$. Combining these
  we obtain
  \begin{align}
    \label{eq:9}
    \bP^\joint\Big( \sum_{i=1}^J \mathcal{D}_i - \mathcal{R}_{i-1}\ge
    K(\alpha,\delta)N^{2\delta+\varepsilon^*}\log N \Big) \le
    CN^{K(\alpha,\delta)\log(1-\frac{\delta-\alpha}{2-2\alpha})}.
  \end{align}
  Using
  Remark~\ref{rem:TVdistance-joint-ind-1step}
  to compare $\bP^\joint$ and $\bP^\indi$ and large deviation
  estimates for sums of independent, heavy tailed random variables
  (using the tail bounds from Proposition~\ref{prop:IndRegTimesBound}
  and, for example, \cite{Nagaev82}, Thm.~2)
  \begin{align*}
    & \bP^\indi\Big( \exists
      k\in \bigcup_{i=1}^J\{\mathcal{D}_i,\dots,\mathcal{R}_i
      \}\cap\{1,\dots,N \} \colon \norm{Y_k}\ge K'N\Big)\\
    & \le \bP^\indi\Big(\exists k\le N \colon T_k \cdot 2R_\kappa \ge K'N \Big)\\
    & \le \bP^\indi\Big( T_N \ge \frac{K'N}{2R_\kappa} \Big)\\
    & \le CN^{1-\beta}
  \end{align*}
  for $K'$ large enough, e.g. $K'>8R_\kappa \bE[\tau^\indi_2]$. Using
  that
  \begin{align*}
    \bP^\joint(\mathcal{U} \le N)
    & \le \bP^\joint\Big(\exists k\le N \colon \norm{Y_k}\ge K'N\Big) \\
    & \le \bP^\joint\Big(\exists k\in
      \bigcup_{i=1}^J\{\mathcal{D}_i,\dots,\mathcal{R}_i
      \}\cap\{1,\dots,N \} \colon \norm{Y_k}\ge K'N\Big) \\
    & \le \bP^\indi\Big( \exists k\in
      \bigcup_{i=1}^J\{\mathcal{D}_i,\dots,\mathcal{R}_i
      \}\cap\{1,\dots,N \} \colon \norm{Y_k}\ge K'N\Big)\\
    & \hspace{3cm} + N\cdot N^{-\alpha \beta}\\
    & \le CN^{1-\beta} + N^{1-\alpha\beta}
  \end{align*}
  and thus $\bP^\joint(\mathcal{U}\le N) \le CN^{1-\alpha\beta}$,
  since $\alpha$ is small. Combining this and \eqref{eq:9} gives
  \begin{align*}
    \bP^\joint
    & \Big( \# \{ k\le N \colon \norm{\widehat{X}^\joint_k -\widehat{X}'^\joint_k}
      \le N^\alpha \}\ge
      K(\alpha,\delta)N^{2\delta+\varepsilon^*}\log N \Big)\\
    & \le \bP^\joint\Big( \sum_{i\colon \mathcal{R}_{i-1}\le N}
      \mathcal{D}_i - \mathcal{R}_{i-1}\ge
      K(\alpha,\delta)N^{2\delta+\varepsilon^*}\log N \Big)\\
    &\le CN^{1-\alpha\beta} + CN^{K(\alpha,\delta)
      \log(1-\frac{\delta-\alpha}{2-2\alpha})}.
  \end{align*}
  If the event
  \[
    \# \{ k\le N \colon \norm{\widehat{X}^\joint_k -\widehat{X}'^\joint_k} \le N^\alpha \}\ge K(\alpha,\delta)
    N^{2\delta+\varepsilon^*}\log N
  \]
  does not occur, we can with
  probability at least $1-N^{1-\alpha\beta}$ couple $\bP^\indi$ and
  $\bP^\joint$ for all $k$ satisfying $\mathcal{D}_i\le k \le
  \mathcal{R}_{i-1}$ for some $i$. Taking $2\delta+\varepsilon^*<b_1$
  we proved \eqref{eq:5} for $d\ge2$, since
  $K(\alpha,\delta)N^{2\delta+\varepsilon^*}\log N \le N^{b_1}$ for
  large $N$ and we can tune $\beta$ and $K(\alpha,\delta)$ large
  enough such that all exponents, i.e.\
  $K(\alpha,\delta)\log(1-\frac{\delta-\alpha}{2-2\alpha})$ and
  $1-\alpha\beta$, are arbitrarily small. Note the tuning of the exponents
  even allows us to get $c>1$ which is relevant in the proof of
  Lemma~\ref{lem:ind-to-joint} below.
  \end{proof}

\subsection{Quenched CLT along regeneration times}
\label{sect:qCLT-along-1}

The main result of this section, Proposition~\ref{prop:qCLT.regen.1} below,
is the quenched CLT for a walk along the (joint) regeneration times.
The following proposition gives the key technical estimate for its proof.

\begin{proposition}
  There \label{prop:2nd-mom} exists $c>0$ and a non-trivial centered
  $d$-dimensional normal law $\wt{\Phi}$ such that for
  $f: \R^d \to \R$ bounded and Lipschitz we have
  \begin{align}
    \label{eq:2nd-mom-estimate}
    \bE\left[ \left( E_\omega[f(\widehat{X}^\joint_m/\sqrt{m})]-\wt{\Phi}(f)
    \right)^2 \right] \leq C_f m^{-c}, \quad m \in \N
  \end{align}
  where the prefactor $C_f$ depends on $f$ through its Lipschitz and
  supremum norms.
\end{proposition}
For the proof we will need some auxiliary results and it will be given
on page~\pageref{proof-prop:2nd-mom}.

\begin{lemma}[Analogue to Lemma~3.10 in
  \cite{BirknerCernyDepperschmidtGantert2013}]
  \label{lemma:joint-ind-comparison}
  Let $d\geq 2$. Then, there exist
  constants $0<a<1/2, C < \infty$ such that for every pair of bounded Lipschitz functions
  $f,g:\R^d\to \R$
  \begin{align}
    \begin{split}
      \abs{& \bE^{\joint}_{0,0}[f(\widehat{X}_n/\sqrt{n})g(\widehat{X}^{\prime}_n/\sqrt{n})]
             - \bE^{\indi}_{0,0}[f(\widehat{X}_n/\sqrt{n})g(\widehat{X}^{\prime}_n/\sqrt{n})]}\\
           & \leq C(1 + \norm{f}_\infty L_f)(1+\norm{g}_\infty L_g) n^{-a},
    \end{split}
  \end{align}
  where $L_f \coloneqq\sup_{x\neq y}\abs{f(y) - f(x)}/\norm{x-y}$ and
  analogously $L_g$ are the Lipschitz constants of $f$ and $g$.
\end{lemma}
Inspection of the proof shows that by increasing $\beta$, one can choose $1/2-a>0$ arbitrarily small.
\begin{proof}
  We work with the coupling of $\bP^\indi_{0,0}$ and $\bP^\joint_{0,0}$ from Lemma~\ref{lem:coupl}.
  Let us define the (random) set of all steps before $n$ at which the coupling
  between $\bP^\indi_{0,0}$ and $\bP^\joint_{0,0}$ is successful
  \begin{align}
    \label{eq:59}
    \mathcal{I}_n \coloneqq \{ k \leq n:\text{the coupling is successful
    for the } k\text{-th step}\}
  \end{align}
  (recall \eqref{eq:uncoupledstep} and the discussion around it),
  its complement
  $\mathcal{I}_n^\compl \coloneqq \{1,\dots,n \}\setminus\mathcal{I}_n$
  and the events
  \begin{align*}
    B_\joint
    & \coloneqq \Bigl\{ \sum_{k\in\mathcal{I}_n^\compl} \bigl(T^\joint_k -
      T^\joint_{k-1} \bigr) \le n^{b_1+2/\beta} \Bigr\},\\
    B_\indi
    & \coloneqq \Bigl\{ \sum_{k\in\mathcal{I}_n^\compl} \bigl (T^\indi_k -
      T^\indi_{k-1} \bigr) \le n^{b_1+2/\beta} \Bigr\},
  \end{align*}
  where $b_1$ is from Lemma~\ref{lem:coupl}. By Lemma~\ref{lem:coupl},
  $\Pr(\abs{\mathcal{I}_n^\compl} > n^{b_1}) \leq n^{-c}$.

  Set $M^\joint_n \coloneqq \max_{k \le n} \{ T^\joint_k-T^\joint_{k-1} \}$,
  $M^\indi_n \coloneqq \max_{k \le n} \{ T^\indi_k-T^\indi_{k-1} \}$ and
  \begin{align*}
    \widetilde{B}_\joint \coloneqq \{ M^\joint_n \le n^{2/\beta} \},
    \quad \widetilde{B}_\indi \coloneqq \{ M^\indi_n \le n^{2/\beta} \}.
  \end{align*}
  By Proposition~\ref{prop:JointRegTimesBound}, we have
  \begin{align*}
    \bP(\widetilde{B}_\joint^\compl)
    & = \bP\left(\exists  \, k \le n \, : T^\joint_k-T^\joint_{k-1} > n^{2/\beta} \right) \\
    & \le \sum_{k=1}^n \bP\left( T^\joint_k-T^\joint_{k-1} > n^{2/\beta} \right) \le C n \cdot (n^{2/\beta})^{-\beta} = \frac{C}{n}
  \end{align*}
  and analogously $\bP(\widetilde{B}_\indi^\compl) \le C/n$ by Proposition~\ref{prop:IndRegTimesBound}.
  Note that
  \begin{align*}
    \{ \abs{\mathcal{I}_n^\compl} \le n^{b_1} \} \cap \widetilde{B}_\joint \subset B_\joint
    \quad \text{and} \quad
    \{ \abs{\mathcal{I}_n^\compl} \le n^{b_1} \} \cap \widetilde{B}_\indi \subset B_\indi .
  \end{align*}

  On the event
  \begin{equation}
    \label{def:eventA}
    A=\{\abs{\mathcal{I}_n^\compl} \le n^{b_1}\} \cap B_\joint\cap B_\indi
  \end{equation}
  we have, due to Assumption \ref{ass:finite-range},
  \begin{align}
    \label{eq:proof lemma 3.10 analogue 1}
    \begin{split}
      & \norm{X_{T^\joint_n} - X_{T^\indi_n}} \le C n^{b_1+2/\beta},\\
      & \norm{X'_{T^\joint_n} - X'_{T^\indi_n}} \le C n^{b_1+2/\beta}.
    \end{split}
  \end{align}
  Thus
  \begin{align*}
    & \Abs{\bE[f(X_{T^\joint_n}/\sqrt{n})g(X'_{T^\joint_n}/\sqrt{n})] -
      \bE[f(X_{T^\indi_n}/\sqrt{n})g(X'_{T^\indi_n}/\sqrt{n})]}\\
    & \le\Big\vert\bE[f(X_{T^\joint_n}/\sqrt{n})g(X'_{T^\joint_n}/\sqrt{n})\mathbbm{1}_A]
      - \bE[f(X_{T^\indi_n}/\sqrt{n})g(X'_{T^\indi_n}/\sqrt{n})\mathbbm{1}_A]\Big\vert\\
    & \qquad+2\norm{f}_\infty\norm{g}_\infty \Bigl(\bP(\abs{\mathcal{I}_n^\compl}
    > n^{b_1}) + \bP\Big(\widetilde{B}_\joint^\compl \Big) + \bP\Big(\widetilde{B}_\indi^\compl \Big) \Bigr).
  \end{align*}

  Lastly, observe that
  \begin{align*}
    \abs{f(x)g(y)-f(x')g(y')} \le \norm{g}_\infty L_f\norm{x-x'} +
    \norm{f}_\infty L_g\norm{y-y'}
  \end{align*}
  and therefore, combining this with \eqref{eq:proof lemma 3.10
    analogue 1} and the last equation yields
  \begin{align*}
    & \Abs{\bE[f(X_{T^\joint_n}/\sqrt{n})g(X'_{T^\joint_n}/\sqrt{n})] -
      \bE[f(X_{T^\indi_n}/\sqrt{n})g(X'_{T^\indi_n}/\sqrt{n})]}\\
    & \le C\norm{g}_\infty L_f
      n^{b_1+2/\beta-1/2} + C\norm{f}_\infty L_g
      n^{b_1+2/\beta-1/2}+C \norm{f}_\infty\norm{g}_\infty (n^{-c}+n^{-1}).
  \end{align*}
  Note that can choose $\beta > 4$ and $b_1$ from
  Lemma~\ref{lem:coupl} so small that $b_1+2/\beta < 1/2$. This
  concludes the proof.
\end{proof}

\begin{remark}
  In \label{rem:cij} the proof of
  Lemma~\ref{lemma:joint-ind-comparison} we have used the coupling
  between a pair of independent random walks on the same environment
  and a pair of random walks on independent environments to carry over
  some properties that are known for the independent random walks on
  independent environments.

  Thus, if we consider some property $W_n$ for independent random
  walks on independent environments up to time $n$ then using the
  above arguments and notation we obtain
  \begin{align}
    \label{eq:57}
    \Pr^\joint(W_n)
    \le \Pr(W_n \cap A ) + \Pr(A^\compl ) \le  \Pr^\indi (W_n) + \Pr(A^\compl )
  \end{align}
  with $A$ from \eqref{def:eventA}.
  By the bounds from the proof of
  Lemma~\ref{lemma:joint-ind-comparison} we have
  \begin{align}
    \label{eq:58}
    \Pr(A^\compl) \le \bP(\abs{\mathcal{I}^\compl}
    > n^{b_1}) + \bP\Big(B_\joint^\compl , \abs{\mathcal{I}^\compl}
    \le n^{b_1}\Big) + \bP\Big(B_\indi^\compl ,
    \abs{\mathcal{I}^\compl} \le n^{b_1}\Big)\le n^{-c} + Cn^{b_1 -\varepsilon\beta}
  \end{align}
  and $c$ and $\beta$ can be made large by tuning the other parameters
  of the model. In particular they can be chosen such that the right
  hand side of \eqref{eq:58} is summable so that Borel-Cantelli
  arguments can be applied when needed.
\end{remark}

For $m \in \N$ we set $\tau^\indi_m \coloneqq T^\indi_m(0,0)-
T^\indi_{m-1}(0,0)$. Furthermore we denote by $V^\indi_n \coloneqq
\max\{m \in \bZ_+ \colon T^\indi_m(0,0) \leq n \}$ and $V^\joint_n
\coloneqq \max\{m \in \bZ_+ \colon T^\joint_m(0,0) \leq n \}$ the
indices of the last regeneration times prior to $n$.
\begin{lemma}
  We have  \label{lem:ind-to-joint} for all $x_0, x_0' \in \Z^d$
  \begin{align}
    \label{eq:60}
        \limsup_{n\to \infty}  \frac{\vert V^\joint_n -n/\bE[\tau^\indi_2]
    \vert}{\sqrt{n\log\log n}} <\infty \quad \bP^\joint_{x_0,x_0'}\text{-a.s.}
  \end{align}
  Furthermore there exist $C < \infty, \gamma >0$ so that
  \begin{align}
    \label{eq:Matthias4}
    \sup_{x_0,x_0' \in \Z^d} \Pr^{\joint}_{x_0,x_0'}\left( \exists \, m \le n : \, | T_m^\joint - m
    \bE[\tau_2^{\indi}] | > n^{1/2 + \varepsilon} \right)
    \le C n^{-\gamma} \quad \text{for } n \in \N.
  \end{align}
\end{lemma}
\begin{proof}
  One can deduce the assertion \eqref{eq:60} from an iterated
  logarithm law for $V^\joint_n$. Using the triangle inequality we
  have
  \begin{align}
    \label{eq:3}
    \limsup_{n\to \infty}  \frac{\vert V^\joint_n -n/\bE[\tau^\indi_2]
    \vert}{\sqrt{n\log\log n}}
    \le  \limsup_{n\to \infty} \frac{\vert V^\joint_n - V^\indi_n \vert}{\sqrt{n\log\log n}}
    +  \limsup_{n\to \infty}  \frac{\vert V^\indi_n -n/\bE[\tau^\indi_2]
    \vert}{\sqrt{n\log\log n}}.
  \end{align}
  By classical renewal theory the second term on the right hand side
  is a.s.\ finite, see e.g.\ \cite[Thm.~III.11.1]{Gut1988}.

  As discussed in Remark~\ref{rem:TVdistance-joint-ind-1step} we can
  couple the regeneration times as well as the spatial positions,
  since the proof relies on a construction that ensures that the two
  pairs of random walks see the same environment and therefore both
  can regenerate at the same time. Thus, any difference between
  $V^\joint_n$ and $V^\indi_n$ comes only from the uncoupled steps.
  Hence $\abs{V^\joint_n -V^\indi_n} \ge K$ implies that there are
  uncoupled steps $i$ in which the increments $T^\joint_i -
  T^\joint_{i-1}$ and $T^\indi_i - T^\indi_{i-1}$ differ and the sum
  of such differences up to time $n$ is larger than $K$.

  Obviously we have $n\le \min\{T^\joint_n,T^\indi_n\}$. Let
  $\mathcal I_n$ be the (random) set of coupled steps up to $n$ from
  \eqref{eq:59}. For $c>1$ we obtain
  \begin{align*}
    \bP&\big(\abs{V^\joint_n - V^\indi_n}\ge n^{1/2-\varepsilon}\big)\\
       &\le \bP\big( \abs{V^\joint_n - V^\indi_n}\ge
         n^{1/2-\varepsilon}, \abs{\mathcal{I}^\compl_n} \le n^{b_1} \big) +
         \bP(\abs{\mathcal{I}^\compl_n} > n^{b_1})\\
       &\le \bP\big( \sum_{i \in \mathcal{I}^\compl_n} \abs{T^\joint_i
         - T^\joint_{i-1} - T^\indi_i + T^\indi_{i-1}}\ge
         n^{1/2-\varepsilon},  \abs{\mathcal{I}^\compl_n} \le n^{b_1} \big) + n^{-c}\\
       &\le \bP\big(\exists i \le n \colon \abs{T^\joint_i -
         T^\joint_{i-1} - T^\indi_i + T^\indi_{i-1}}\ge
         n^{1/2-\varepsilon-b_1}  \big)+ n^{-c}\\
       &\le n\Big( \bP\big( \abs{T^\joint_2-T^\joint_1} \ge \frac{1}{2}
         n^{1/2-\varepsilon-b_1} \big)+\bP\big(
         \abs{T^\indi_2-T^\indi_1} \ge \frac{1}{2}
         n^{1/2-\varepsilon-b_1} \big) \Big) + n^{-c}\\
       &\le C'n^{1-\beta(1/2-\varepsilon-b_1)} + n^{-c},
  \end{align*}
  where $1-\beta(1/2-\varepsilon-b_1)<-1$ for $\beta$ large enough;
  recall that $b_1<1/2$ from Lemma~\ref{lem:coupl} and note that we
  can choose $\varepsilon$ suitably. Therefore, using the
  Borel-Cantelli lemma, we have
  $\abs{V^\joint_n - V^\indi_n }/\sqrt{n} \to 0$ a.s. This implies
  that the first summand on the right hand side in \eqref{eq:3} is
  equal to $0$ and therefore the right hand side is almost surely
  finite proving \eqref{eq:60}.

  For \eqref{eq:Matthias4} by similar conditioning and arguments as
  above we have
  \begin{align*}
    \Pr
    & \left( \exists \, m \le n : \, | T_m^\joint- m \bE[\tau_2^{\indi}] | >
      n^{1/2 + \varepsilon} \right)\\
     & \le\bP\left( \exists \, m \le n : \, | T^\joint_m- m
       \bE[\tau_2^{\indi}] | > n^{1/2 + \varepsilon}, A_n\right) +
       Cn^{-\gamma}\\
     & \le \bP\left( \exists \, m \le n : \, | T^\indi_m- m
       \bE[\tau_2^{\indi}] |>n^{1/2 + \varepsilon}, A_n \right)\\
     & \hspace{3cm}+\bP\left( \exists \, m \le n : \,
       \abs{T^\joint_m-T^\indi_m}>n^{1/2 + \varepsilon}, A_n \right) +
       Cn^{-\gamma}
  \end{align*}
  where
  \begin{equation}
    \label{event:fewuncoupledsteps}
    A_n \coloneqq \{ \abs{\mathcal{I}_n^\compl} \le n^{b_1} \}
  \end{equation}
  is the event that there are at most $n^{b_1}$ uncoupled steps before
  time $n$. By Lemma~\ref{lem:coupl} we obtain $\bP(A^\compl_n) \le
  Cn^{-\gamma}$. For the first term we make use of a Burkholder
  inequality \cite{Burkholder1973} (a version which works for
  discrete-time martingales and is particularly convenient for our
  purposes here appears in \cite[Ch.~II.4, (PSF) on
  p.~152]{Perkins1999}). Note that
  \begin{equation*}
        M_m \coloneqq T^\indi_m- m\bE[\tau_2^{\indi}] = \sum_{i=1}^m T^\indi_i - T^\indi_{i-1} -\bE[\tau_2^{\indi}]
  \end{equation*}
  is a sum of independent random variables with zero mean and thus a martingale. Therefore, using the Markov- and
  Burkholder-inequalities with the notation
  \begin{equation*}
        \langle M \rangle_n = \sum_{m=1}^{n}\bE[ (M_m-M_{m-1})^2 \, \vert\, \mathcal{F}_{m-1} ] + \bE[M_0^2]
  \end{equation*}
  we obtain for $1 \le q < \beta$
  \begin{align}
    \bP%
    &\left( \exists \, m \le n : \, | T^\indi_m- m
      \bE[\tau_2^{\indi}] |>n^{1/2 + \varepsilon}\right) \notag \\
    &\le \frac{1}{n^{q(1/2+\varepsilon)}} \bE\Big[ (\max_{m\le n} \norm{M_m}\big)^q \Big] \notag \\
    &\le \frac{c}{n^{q(1/2+\varepsilon)}} \Big( \bE\big[ \langle M \rangle_n^{q/2} \big] + \bE[(\max_{m\le n} \norm{M_m-M_{m-1}})^q] \Big) \notag  \\
    &\le \frac{c}{n^{q(1/2+\varepsilon)}} \Big( \bE\big[ \langle M \rangle_n^{q/2} \big] + \bE[\max_{m\le n} \norm{M_m-M_{m-1}}^q] \Big) \notag \\
    &\le \frac{c}{n^{q(1/2+\varepsilon)}} \Big( \bE\big[ \langle M \rangle_n^{q/2} \big] + C n \Big) \notag \\
    &\le \frac{c}{n^{q(1/2+\varepsilon)}} \Big( \bE\big[ \big(\sum_{m=1}^{n}\bE[ (M_m-M_{m-1})^2 \, \vert\, \mathcal{F}_{m-1} ]\big)^{q/2} \big] + C n \Big) \notag \\
    &\le \frac{c}{n^{q(1/2+\varepsilon)}} \Big( \big( C n \big)^{q/2} + Cn \Big) \notag \\
    &\le \frac{C}{n^{q\varepsilon}},
      \label{eq:useBurkholder}
  \end{align}
  where we need to tune $\beta$ large enough to get a finite $q$-th
  moment, using Proposition~\ref{prop:IndRegTimesBound}. A consequence
  is that we can make the exponent $q\varepsilon$ in the above display
  arbitrarily large, and therefore get $q\varepsilon > \gamma$, by
  tuning $\beta$ large enough. Combining the above estimates we obtain
  \begin{align*}
        \Pr
        & \left( \exists \, m \le n : \, | T_m^\joint- m \bE[\tau_2^{\indi}] | >
        n^{1/2 + \varepsilon} \right)\\
    & \le n\Big( \bP\big( \abs{T^\joint_2-T^\joint_1} \ge \frac{1}{2}
      n^{1/2+\varepsilon-b_1} \big)+\bP\big( \abs{T^\indi_2-T^\indi_1}
      \ge \frac{1}{2} n^{1/2+\varepsilon-b_1} \big) \Big) +
      Cn^{-\gamma}\\
    & \le Cn^{1-\beta(1/2+\varepsilon-b_1)} + C n^{-\gamma},
  \end{align*}
  where we have used \eqref{eq:37} in the last step.

  By choosing
  $\beta$ large enough we arrive at \eqref{eq:Matthias4}.
\end{proof}

\begin{lemma}
  \label{lem:joint_fluctuations}
  For $\delta\in(1/2,1)$ and $\varepsilon>0$ small enough such that
  $\delta/2+\varepsilon < 1/2$ there exists $\gamma>0$ and $C <
  \infty$, if the parameters are tuned correctly (in particular,
  $\beta$ must be large), such that
  \begin{align}
    \label{eq:Matthias5}
    \sup_{0 \le \theta \le 1} \Pr^{\joint}\left( \sup_{\abs{k-[\theta n]}  \leq n^\delta}
    \norm{\widehat{X}_k - \widehat{X}_{[\theta n]}} >  n^{\delta/2+\varepsilon}\right)
    \le C n^{-\gamma}.
  \end{align}
\end{lemma}
We will see in the proof that $\gamma$ can be chosen large when $\beta$ is large.
\begin{proof}
  Indeed, if we denote by $A_n$ as in \eqref{event:fewuncoupledsteps}
  in the proof of Lemma~\ref{lem:ind-to-joint} the event that there
  are at most $n^{b_1}$ uncoupled steps before time $n$ and follow
  these steps
  \begin{align*}
    \bP
    & \Big( \sup_{\abs{[\theta n]-k}  \leq n^\delta}
      \Norm{\widehat{X}^\joint_k-\widehat{X}^\joint_{[\theta n]}} >
      n^{\delta/2+\varepsilon} \Big)\\
    & \le \bP\Big( \sup_{\abs{[\theta n]-k}  \leq n^\delta}
      \Norm{\widehat{X}^\joint_k-\widehat{X}^\joint_{[\theta n]}} >
      n^{\delta/2+\varepsilon}, A_{2n} \Big) + \bP(A^\compl_{2n})\\
    & \le \bP\Big( \sup_{\abs{[\theta n]-k}  \leq
      n^\delta}\Norm{\widehat{X}^\indi_k-\widehat{X}^\indi_{[\theta
      n]}} +\Norm{\widehat{X}^\joint_k-\widehat{X}^\joint_{[\theta
      n]}-\widehat{X}^\indi_k+\widehat{X}^\indi_{[\theta n]}} >
      n^{\delta/2+\varepsilon}, A_{2n} \Big) + \bP(A^\compl_{2n})\\
        &\le \bP\Big( \sup_{\abs{[\theta n]-k}  \leq n^\delta}\Norm{\widehat{X}^\indi_k-\widehat{X}^\indi_{[\theta n]}} >\frac{1}{2}n^{\delta/2+\varepsilon}\Big)\\
        &+\bP\Big(\sup_{\abs{[\theta n]-k}  \leq n^\delta}\Norm{\sum_{i=\min\{k,[\theta n]\}+1}^{\max\{k,[\theta n]\}}\bigl(\widehat{X}^\joint_i-\widehat{X}^\joint_{i-1}\bigr)-\sum_{i=\min\{k,[\theta n]\}+1}^{\max\{k,[\theta n]\}} \bigl(\widehat{X}^\indi_i-\widehat{X}^\indi_{i-1} \bigr)} > \frac{1}{2}n^{\delta/2+\varepsilon}, A_{2n} \Big)\\
        &+ \bP(A^\compl_{2n}).
  \end{align*}

  For the first term we can use $1 \le q < \beta$ and a
  Burkholder inequality as in the proof of Lemma~\ref{lem:ind-to-joint}
  (see Eq.~\eqref{eq:useBurkholder})
  \begin{equation*}
    \bP\Big( \sup_{\abs{[\theta n]-k}  \leq n^\delta}\Norm{\widehat{X}^\indi_k-\widehat{X}^\indi_{[\theta n]}} >\frac{1}{2}n^{\delta/2+\varepsilon}\Big) \le
    \frac{C}{n^{q(\delta/2+\varepsilon)}}
  \end{equation*}
  and note that we have $q(\delta/2+\varepsilon) > \gamma$ when
  $\beta$ is large and $q$ close to $\beta$.

  For the last term we know that by Lemma~\ref{lem:coupl}
  \begin{equation*}
    \bP(A^\compl_{2n}) \le Cn^{-\gamma}
  \end{equation*}
  and we recall that $\gamma$ can be chosen large there (recall the
  discussion below Lemma~\ref{lem:coupl}).

  Lastly we need a suitable upper bound for the remaining term.
  \begin{align*}
    \bP& \Big(\sup_{\abs{[\theta n]-k}  \leq
         n^\delta}\Norm{\sum_{i=\min\{k,[\theta
         n]\}+1}^{\max\{k,[\theta n]\}}
         \bigl(\widehat{X}^\joint_i-\widehat{X}^\joint_{i-1}\bigr)-\sum_{i=\min\{k,[\theta
         n]\}+1}^{\max\{k,[\theta
         n]\}}\bigl(\widehat{X}^\indi_i-\widehat{X}^\indi_{i-1}\bigr)}
         > \frac{1}{2}n^{\delta/2+\varepsilon}, \, A_{2n} \Big)\\
       & \le \bP\Big(\sup_{\abs{[\theta n]-k}  \leq n^\delta}\sum_{i\in
         \mathcal{I}_n^\compl \cap \{\min\{k,[\theta n]\},\dots,\max\{k,[\theta n]\}\}}
         \hspace{-4.55em}
         \bigl(
         \norm{\widehat{X}^\joint_i-\widehat{X}^\joint_{i-1}}+\norm{\widehat{X}^\indi_i-\widehat{X}^\indi_{i-1}}
         \bigr) > \frac{1}{2}n^{\delta/2+\varepsilon}, \, A_{2n} \Big)
  \end{align*}
  \begin{align*}
    &\le\bP\Big( \sum_{i\in \mathcal{I}_n^\compl
      \cap \{[\theta n]-n^\delta,\dots,[\theta n]+n^\delta\}}
      \hspace{-3.5em}
      \bigl(
      \norm{\widehat{X}^\joint_i-\widehat{X}^\joint_{i-1}}+\norm{\widehat{X}^\indi_i-\widehat{X}^\indi_{i-1}}
      \bigr) > \frac{1}{2}n^{\delta/2+\varepsilon}, \, A_{2n} \Big)\\
    &\le\bP\Big( \exists \, i\in \mathcal{I}_n^\compl
      \cap \{[\theta n]-n^\delta,\dots,[\theta n]+n^\delta\}
      \text{ with } \norm{\widehat{X}^\joint_i-\widehat{X}^\joint_{i-1}}>\frac{1}{4}n^{\delta/2+\varepsilon-b_1}\\
    &\hspace{20em}\text{or }\norm{\widehat{X}^\indi_i-\widehat{X}^\indi_{i-1}} > \frac{1}{4}n^{\delta/2+\varepsilon-b_1}, \, A_{2n} \Big)\\
    &\le\sum_{i=[\theta n]-n^\delta}^{[\theta n]+n^\delta} \bP\Big( \norm{\widehat{X}^\joint_i-\widehat{X}^\joint_{i-1}}>\frac{1}{4}n^{\delta/2+\varepsilon-b_1} \Big)
      +\bP\Big(\norm{\widehat{X}^\indi_i-\widehat{X}^\indi_{i-1}} > \frac{1}{4}n^{\delta/2+\varepsilon-b_1} \Big)\\
    &\le \sum_{i=[\theta n]-n^\delta}^{[\theta n]+n^\delta} \bP\Big( T^\joint
      _i - T^\joint_{i-1}>\frac{1}{4R_\kappa}n^{\delta/2+\varepsilon-b_1} \Big) +\bP\Big(T^\indi_i-T^\indi_{i-1} > \frac{1}{4R_\kappa}n^{\delta/2+\varepsilon-b_1} \Big)\\
    &\le 4n^\delta (4R_\kappa)^\beta n^{-\beta(\delta/2+\varepsilon-b_1)}\\
    &\le Cn^{-\gamma},
  \end{align*}
  for any choice of $\gamma$ if $\beta$ is large enough. Combining the
  three estimates completes the proof.
\end{proof}

Now we have all the tools to prove Proposition~\ref{prop:2nd-mom}.

\begin{proof}[Proof of Proposition~\ref{prop:2nd-mom}]
  The \label{proof-prop:2nd-mom} proof uses
  Lemma~\ref{lemma:joint-ind-comparison} and Berry-Esseen estimates.

  First we claim that for any $\varepsilon>0$ (small) and $\gamma>0$ (large), if
  the parameters are tuned right (in particular, $\beta$ must be very large), we have
  \begin{align}
    \label{eq:Matthias6}
    \bP^{\joint}_{(0,0)} \left( \norm{X_N - \widehat{X}_{N/\bE[\tau_2^{\indi}]}}
    > N^{1/4+\varepsilon}\right) \le C N^{-\gamma} .
  \end{align}
  Indeed, note that
  \begin{align}
    \norm{X_N - \widehat{X}_{N/\bE[\tau_2^{\indi}]}}
    \le R_\kappa \max \{ T^\joint_m-T^\joint_{m-1} : m \le N \}
    + \norm{\widehat{X}_{N/\bE[\tau_2^{\indi}]} - \widehat{X}_{V_N}}
  \end{align}
  with $V_N := \max \{ m : T_m \le N \}$. On $\{
  \abs{N/\bE[\tau_2^{\indi}] - V_N} \le N^{1/2+2\varepsilon} \}$,
  which is likely by \eqref{eq:Matthias4}, we can bound the second
  term on the right-hand side by $ \sup_{\abs{k-N/\bE[\tau_2^{\indi}]}
    \leq N^{1/2+2\varepsilon}} \norm{\widehat{X}_k -
    \widehat{X}_{N/\bE[\tau_2^{\indi}]}} $ and then use
  \eqref{eq:Matthias5}. For the first term in the above upper bound
  just note that
  \begin{align*}
    \bP
    & \Big( \max\{ T^\joint_m-T^\joint_{m-1} \colon m\le N\}
      >\frac{1}{2R_\kappa} N^{1/4+\varepsilon} \Big)\\
    & \le \sum_{m=1}^N\bP\Big( T^\joint_m -T^\joint_{m-1} >
      \frac{1}{2R_\kappa}N^{1/4+\varepsilon} \Big)\\
    &\le CN^{1-\beta(1/4+\varepsilon)},
  \end{align*}
  which, again by taking $\beta$ large enough, is smaller than
  $N^{-\gamma}$.

  From \eqref{eq:Matthias6} and Markov's inequality we have
  \begin{align}
    \label{eq:Matthias8a}
    \Pr\left( \{ \omega : P_\omega(\norm{X_N - \widehat{X}^\joint_{N/\bE[\tau_2^{\indi}]}}
    > N^{1/4+\varepsilon}) > N^{-\gamma/2} \}\right) \le C N^{-\gamma/2}
  \end{align}
  and the same ist true for $X'$ replacing $X$ in the above display.

  Thus
  \begin{align}
    \label{eq:Matthias8c}
    \begin{split}
      & \bE[E_\omega[f(\widehat{X}^\joint_m/\sqrt{m})]
        E_\omega[f(\widehat{X}^{\prime,\joint}_m/\sqrt{m})]] \\
      & \qquad
        = \bE[E_\omega[f(X_{m \bE[\tau_2^{\indi}]}/\sqrt{m})]
        E_\omega[f(X^\prime_{m \bE[\tau_2^{\indi}]}/\sqrt{m})]]
        + O(m^{-1/4+\varepsilon}) + O(m^{-\gamma/2}) \\
      & \qquad
        = \bE^\joint[f(X_{m \bE[\tau_2^{\indi}]}/\sqrt{m})
        f(X^\prime_{m \bE[\tau_2^{\indi}]}/\sqrt{m})] + O(m^{-1/4+\varepsilon})  + O(m^{-\gamma/2})\\
      & \qquad
        = \bE^\joint[f(\widehat{X}_m/\sqrt{m})f(\widehat{X}^\prime_m/\sqrt{m})]
        + O(m^{-1/4+\varepsilon})  + O(m^{-\gamma/2}).
    \end{split}
  \end{align}
  Observe that here we need to pass from
  $(\widehat{X}^\joint,\widehat{X}'^\joint)$ to $(X,X')$ in the
  intermediate step because it is true that
  $(E_\omega[f(X_m/\sqrt{m})])^2 = E_\omega[f(X_m/\sqrt{m})]
  E_\omega[f(X'_m/\sqrt{m})]$ by definition while in general
  $(E_\omega[f(\widehat{X}^\joint_m/\sqrt{m})])^2 \neq
  E_\omega[f(\widehat{X}^\joint_m/\sqrt{m})]
  E_\omega[f(\widehat{X}'^\joint_m/\sqrt{m})]$ (note that the
  right-hand side refers to two independent copies of the regeneration
  construction in the same environment whereas the left-hand side
  involves only a single regeneration construction in that
  environment).

  With that we have
  \begin{align*}
    \bE\Bigl[& \Bigl(E_\omega[f(\widehat{X}_m/\sqrt{m})]
     -\wt{\Phi}(f) \Bigr)^2\Bigr] \\
    & \le \Big|\bE^\joint[f(\widehat{X}_m/\sqrt{m})f(\widehat{X}^\prime_m/\sqrt{m})] - \bE^\indi[f(\widehat{X}_m/\sqrt{m})f(\widehat{X}^\prime_m/\sqrt{m})] \Big|\\
    & \quad +\Big|
      \bE^\indi[f(\widehat{X}_m/\sqrt{m})f(\widehat{X}^\prime_m/\sqrt{m})]
      - \wt{\Phi}(f)^2\Big|\\ 
    & \quad + 2|\wt{\Phi}(f)|\cdot\Big|
      \bE^\indi\bigl[f(\widehat{X}_m/\sqrt{m}) \bigr]
      -\bE^\joint\bigl[f(\widehat{X}_m/\sqrt{m}) \bigr] \Big|\\ 
    & \quad + 2|\wt{\Phi}(f)|\cdot\Big|  \wt{\Phi}(f)
      -\bE^\indi\bigl[f(\widehat{X}_m/\sqrt{m}) \bigr] \Big|\\ 
    & \leq 2C_fm^{-a} + Cm^{-1/2} \leq C_fm^{-c}
  \end{align*}
  for a suitable $c>0$, where we used
  Lemma~\ref{lemma:joint-ind-comparison} (observe also that by
  choosing $g \equiv 1$ in Lemma~\ref{lemma:joint-ind-comparison}, we
  obtain a bound for $\bE^\indi\bigl[f(\widehat{X}_m/\sqrt{m}) \bigr]
  -\bE^\joint\bigl[f(\widehat{X}_m/\sqrt{m}) \bigr]$) and Berry-Esseen
  type bounds, since $\widehat{X}_n$ has bounded third moments under
  $\bP^\indi$ for $\beta$ large enough, in the last line.
\end{proof}

\begin{lemma}[Analogue to Lemma~3.12 in
  \cite{BirknerCernyDepperschmidtGantert2013}]
  \label{lemma:convergence for subsequence to full sequence}
  Assume that for some $c>1$, and any bounded Lipschitz function
  $f: \R^d \to \R$
  \begin{align}
    E_\omega\left[f(X_{T^\joint_{k^c}}/k^{c/2})\right]\underset{k\to
    \infty}{\longrightarrow} \wt{\Phi}(f) \quad \text{for a.a. }
    \omega,
  \end{align}
  where $\wt{\Phi}$ is some non-trivial centred $d$-dimensional normal
  law and the $T^\joint_\cdot$ refer to two walks both starting at $0$
  in the same environment. Then we have for any bounded Lipschitz
  function $f$
  \begin{align*}
    E_\omega\left[f(X_{T^\joint_{m}}/m^{1/2})\right] \underset{m \to
    \infty}{\longrightarrow} \wt{\Phi}(f) \quad \text{for a.a. }
    \omega.
  \end{align*}
\end{lemma}
\begin{proof}
  Recall the abbreviation $\widehat{X}_n=X_{T^\joint_n}$ from
  \eqref{def:Xhatn}. Since the distribution of $\{\widehat{X}_i -
  \widehat{X}_{i-1}\}$ does not have exponential tails we have
  $\bE[\exp(\lambda \cdot (\widehat{X}_i - \widehat{X}_{i-1}))]
  =\infty$ for $\lambda \neq 0$, unlike the situation in
  \cite{BirknerCernyDepperschmidtGantert2013}. Therefore we need to
  calculate more carefully than in
  \cite{BirknerCernyDepperschmidtGantert2013} and control the
  probabilities via large deviations for heavy tailed distributions.
  Our first aim is to obtain a summable upper bound on the probability
  \begin{align}
    \label{eq:maxfluctuations}
    \Pr_{0,0}\left( \max_{k^c\leq \ell \leq (k+1)^c}
    \frac{\abs{\widehat{X}^\joint_\ell - \widehat{X}^\joint_{k^c}}}{k^{c/2}}\geq \delta \right),
  \end{align}
  where $\delta>0$. Because of the heavy tails, the probability in
  \eqref{eq:maxfluctuations} does not decay exponentially. We need to
  work around the fact that, along the simultaneous regeneration times
  $T^\joint$, the increments of the random walk are not independent.
  To be able to compare $T^\indi$ with $T^\joint$ we recall
  Remark~\ref{rem:TVdistance-joint-ind-1step} and the discussion
  around it. Define $\hat{\ell} \coloneqq \min\{n\in\N \colon
  T^\indi_n \ge T^\joint_\ell \}$, note that we can define $T^\joint$
  and $T^\indi$ together on the same probability space by using the
  Markov chain \eqref{eq:bigchain}, as well as the following events
  \begin{align*}
    A_{\ell} & \coloneqq \left\{ T^{\joint}_{\ell} \le
               T^\indi_{k^{c-1+\varepsilon}} \right\},\\
    B_k & \coloneqq \left\{ T^\indi_i - T^\indi_{i-1}\le k^\alpha
          \text{ for all } i=1,\dots,k^{c-1+\varepsilon} \right\},\\
    C_\ell & \coloneqq \left\{ \norm{\widehat{X}^\indi_{\ell}} \le
             \frac{1}{2}\delta k^{c/2} \right\},\\
    D_\ell & \coloneqq \left\{ T^\indi_{\hat{\ell}} - T^\joint_\ell <
             \frac{1}{2R_\kappa} \delta k^{c/2} \right\},
  \end{align*}
  where $0<\alpha<1/2$ and $\varepsilon>0$ are small constants to be tuned later.
  We get the following upper bounds for above: starting with
  $A_\ell^\compl$
  \begin{align*}
    \bP_{x,x'} (A_\ell^\compl)
    & \le \bP_{x,x'}(T^\joint_\ell > k^{c-1+\varepsilon})\\
    & \le \bP_{x,x'}\Big(\exists i \le \ell \text{ such that }
      T^\joint_i- T^\joint_{i-1} > \frac{k^{c+\varepsilon}}{\ell} \Big)\\
    & \le \sum_{i=1}^\ell \bP_{x,x'}\Big( T^\joint_i- T^\joint_{i-1} >
      \frac{k^{c-1+\varepsilon}}{\ell} \Big)\\
    & \le \sum_{i=1}^\ell \Big(\frac{k^{c-1+\varepsilon}}{\ell}
      \Big)^{-\beta}\le \ell^{1+\beta}k^{-\beta(c-1+\varepsilon)},
  \end{align*}
  where we understand $\bP_{x,x'}$ as describing the chain from
  \eqref{eq:bigchain} starting from
  $\widehat{X}^\joint_0=x=\widehat{X}^\indi_0,
  \widehat{X}'^\joint_0=x'=\widehat{X}'^\indi_0$. For $B_k^\compl$, by
  Proposition~\ref{prop:JointRegTimesBound}, we obtain
  \begin{align*}
    \bP_{x,x'}(B_k^\compl)
    & = \bP_{x,x'}(\exists i\le k^{c-1+\varepsilon} \text{ such that }
      T^\indi_i-T^\indi_{i-1} >k^\alpha )\\
    & \le k^{c-1+\varepsilon}k^{-\beta \alpha},
  \end{align*}
  for $C_\ell^\compl$ using Azuma's inequality on every coordinate, we
  have
  \begin{align*}
    \bP_{x,x'}(C_\ell^\compl\,\vert\, B_k)
    & = \bP_{x,x'}\Big(\norm{\sum_{i=1}^\ell\widehat{X}^\indi_{i}-\widehat{X}^\indi_{i-1}}
      > \frac{1}{2}\delta k^{c/2} \,\vert\, B_k\Big)\\
    & \le \sum_{j=1}^d \bP_{x,x'}\Big(\norm{\sum_{i=1}^\ell\widehat{X}^\indi_{i}(j)-
      \widehat{X}^\indi_{i-1}(j)} > \frac{1}{2}\delta k^{c/2} \,\vert\, B_k\Big)\\
    & \le 2d \exp\Big( -\frac{\delta^2 k^c}{8\ell R_\kappa k^{2\alpha}} \Big)
  \end{align*}
  and consequently
  \begin{align*}
    \bP_{x,x'}\Big(\bigcup_{\ell \le k^{c-1+\varepsilon}} C_\ell^\compl\Big)
    & \le k^{2(c-1+\varepsilon)}k^{-\beta \alpha} + \sum_{\ell \le k^{c-1+\varepsilon}}
      \bP_{x,x'}(C_\ell^\compl\,\vert\, B_k)\\
    & \le  k^{2(c-1+\varepsilon)}k^{-\beta \alpha} + k^{c-1+\varepsilon}\exp\Big(
      -C(\delta,R_\kappa)k^{1-\varepsilon-2\alpha} \Big).
  \end{align*}
  Lastly, for $D_\ell^\compl$, we have
  \begin{align*}
    \bP_{x,x'}(D_\ell^\compl)
    & \le \bP_{x,x'}(D^\compl_\ell , A_\ell, B_k) + \bP_{x,x'}( A^\compl_\ell )+\bP_{x,x'}( B^\compl_k )\\
    & = \bP_{x,x'}( A^\compl_\ell)+\bP_{x,x'}( B^\compl_k),
  \end{align*}
  where we used the fact that on $A_\ell$ we know that $\hat{\ell}\le
  k^{c-1+\varepsilon}$ and therefore on $A_\ell\cap B_k$ we have
  $T^\indi_{\hat{\ell}}-T^\joint_\ell \le
  T^\indi_{\hat{\ell}}-T^\indi_{\hat{\ell}-1} \le k^\alpha < k^{c/2} $
  since $c>1$ and $\alpha<1/2$. Define $\widetilde{C}_{k}\coloneqq
  \bigcap_{\ell \le k^{c-1+\varepsilon}}C_\ell$, combine the proved
  upper bounds and the fact that for $c>1$ we have $(k+1)^c-k^c \le
  c(k+1)^{c-1}$ to obtain
  \begin{align}
    \notag
    & \bP_{0,0}\left( \max_{k^c\leq \ell \leq (k+1)^c}
      \frac{\abs{\widehat{X}^\joint_\ell - \widehat{X}^\joint_{k^c}}}{k^{c/2}}\geq \delta \right)\\
    \notag
    & = \sum_{x,x'} \bP^\joint_{0,0}\Big(\max_{k^c\leq \ell \leq
      (k+1)^c} \abs{\widehat{X}_\ell - \widehat{X}_{k^c}} \ge
      \delta k^{c/2}\,\Big\vert\,
      \widehat{X}_{k^c}=x,\widehat{X}'_{k^c}=x'\Big)\bP_{0,0}^\joint(\widehat{X}_{k^c}=x,\widehat{X}'_{k^c}=x')\\
    \notag
    & \le \sup_{x,x'} \bP^\joint_{x,x'}(\max_{\ell \in
      \{1,\dots,ck^{c-1} \}} \abs{\widehat{X}_\ell-x}>\delta k^{c/2} )\\
    \label{eq:proof lemma 3.8 eq 1}
    & \le \sum_{\ell=1}^{ck^{c-1}} \sup_{x,x'}
      \bP^\joint_{x,x'}(\abs{\widehat{X}_\ell-x}>\delta k^{c/2})\\
          \notag
    & \le \sum_{\ell=1}^{ck^{c-1}} \sup_{x,x'}
      \bP_{x,x'}(\abs{\widehat{X}^\joint_\ell-x}>\delta k^{c/2},
      A_\ell,\widetilde{C}_{k},D_\ell) + \bP_{x,x'}(A_\ell^\compl) +
      \bP_{x,x'}(\widetilde{C}_{k}^\compl) + \bP_{x,x'}(D_\ell^\compl).
  \end{align}
  Note that the events $A_\ell\cap \widetilde{C}_{k}\cap D_\ell$ and
  $\{\abs{\widehat{X}^\joint_\ell-x}>\delta k^{c/2} \}$ are,
  when joint, a null set and thus,
  \begin{align*}
    \eqref{eq:maxfluctuations}
    & \le \sum_{\ell=1}^{ck^{c-1}}\sup_{x,x'} \bP_{x,x'}(A_\ell^\compl)
      + \bP_{x,x'}(\widetilde{C}_{k}^\compl) +
      \bP_{x,x'}(D_\ell^\compl) \\
    & \le \sum_{\ell=1}^{ck^{c-1}}\sup_{x,x'} \bP_{x,x'}(A_\ell^\compl)
      + \bP_{x,x'}\Big(\bigcup_{\ell \le k^{c-1+\varepsilon}} C_{\ell}^\compl\Big) +
      \bP_{x,x'}(D_\ell^\compl) \\
    & \le \sum_{\ell=1}^{ck^{c-1}}\bigg(
      \ell^{1+\beta}k^{-\beta(c-1+\varepsilon)} +
      k^{2(c-1+\varepsilon)}k^{-\beta \alpha}\\
    & \hspace{2cm} + k^{c-1+\varepsilon}\exp\Big(
      -C(\delta,R_\kappa)k^{1-\varepsilon-2\alpha} \Big)\bigg).
  \end{align*}
  For $\beta$ large enough and $\varepsilon$ small enough such that
  $1-\varepsilon-2\alpha>0$ the upper bound for
  \eqref{eq:maxfluctuations} given above is summable in $k$ and thus
  \begin{align*}
    \limsup_{k\to \infty} \max_{k^c\leq \ell \leq (k+1)^c}
    \frac{\abs{\widehat{X}^\joint_\ell - \widehat{X}^\joint_{k^c}}}{k^{c/2}} =
    0, \qquad \text{for-a.a. } \omega.
  \end{align*}
  From that it follows for $k^c \le m \le (k+1)^c$
  \begin{align*}
    & \abs{E_\omega[f(X_{T^\joint_{m}(0,0)}/m^{1/2})]- \wt{\Phi}(f)}\\
    & \le L_f \left\Vert\frac{\widehat{X}^\joint_m}{\sqrt{m}} -
      \frac{\widehat{X}^\joint_k}{k^{c/2}}\right\Vert +
      \abs{E_\omega[f(X_{T^\joint_{k}(0,0)}/k^{c/2})]-
      \wt{\Phi}(f)}\\
    & \le L_f \left\Vert
      \frac{\widehat{X}^\joint_m}{\sqrt{m}}-\frac{\widehat{X}^\joint_{k^c}}{\sqrt{m}}
      \right\Vert + L_f \left\Vert
      \frac{\widehat{X}^\joint_{k^c}}{\sqrt{m}} -
      \frac{\widehat{X}^\joint_{k^c}}{k^{c/2}} \right\Vert +
      \abs{E_\omega[f(X_{T^\joint_{k}(0,0)}/k^{c/2})]-
      \wt{\Phi}(f)}\\
    & \le L_f\frac{\norm{\widehat{X}^\joint_m -
      \widehat{X}^\joint_{k^c}}}{k^{c/2}}  +
      L_f\frac{\norm{\widehat{X}^\joint_{k^{c}}}}{k^{c/2}} \Big(
      \frac{k^{c/2}}{\sqrt{m}} -1 \Big)+
      \abs{E_\omega[f(X_{T^\joint_{k}(0,0)}/k^{c/2})]-
      \wt{\Phi}(f)}.
  \end{align*}
  The calculations above show that the first term goes to $0$ a.s. and
  we can extract the same result for the second term since equation
  \eqref{eq:proof lemma 3.8 eq 1} yields an upper bound for it.
  Therefore
  \begin{align*}
    \abs{E_\omega[f(X_{T^\joint_{m}(0,0)}/m^{1/2})]- \wt{\Phi}(f)}
    \to 0 \qquad \text{for a.a. } \omega.
  \end{align*}
\end{proof}

Proposition~\ref{prop:2nd-mom} and Lemma~\ref{lemma:convergence for
  subsequence to full sequence} together imply the following quenched
CLT along the regeneration times.
\begin{proposition}
  \label{prop:qCLT.regen.1}
  Let $d \ge 2$.
  For $f: \R^d \to \R$ bounded and Lipschitz we have
  \begin{equation}
    E_\omega\left[f\left(X_{T^\joint_{n}(0,0)}/\sqrt{n}\right)\right]
    \mathop{\longrightarrow}_{n\to\infty} \wt{\Phi}(f)
    \qquad \text{a.s.},
  \end{equation}
  where $\wt{\Phi}$ is the normal law from Proposition~\ref{prop:2nd-mom}.
\end{proposition}
\begin{proof}
  Let $f:\R^d \to \R$ be bounded and Lipschitz, $c' > 1/c \wedge 1$
  with $c$ from Proposition~\ref{prop:2nd-mom}. By
  \eqref{eq:2nd-mom-estimate} and Markov's inequality, abbreviating
  $\widehat{X}_m = \widehat{X}^\joint_m$,
  \begin{align}
    \begin{split}
      & \Pr\Big( \abs{E_\omega [
        f(\widehat{X}_{[n^{c'}]}/\sqrt{[n^{c'}]})] ]
        -\wt{\Phi}(f)}>\varepsilon \Big)\\
      & \leq \frac{\bE \Big[ \big( E_\omega[
        f(\widehat{X}_{[n^{c'}]}/\sqrt{[n^{c'}]})] - \wt{\Phi}(f) \big)^2
        \Big]}{\varepsilon^2}\\
      & \leq C_f n^{-c'c} \varepsilon^{-2},
    \end{split}
  \end{align}
  which is summable and hence by Borel-Cantelli
  \begin{align}
    E_\omega [f(\widehat{X}_{[n^{c'}]}/\sqrt{[n^{c'}]})] \to
    \wt{\Phi}(f) \quad \text{a.s. as }n\to\infty.
  \end{align}
  Now Lemma~\ref{lemma:convergence for subsequence to full sequence} yields
  \begin{align}
    \label{eq:Thm proof convergence along reg. times}
    E_\omega[f(\widehat{X}_m/\sqrt{m})] \underset{m \to
    \infty}{\longrightarrow} \wt{\Phi}(f) \quad \text{for a.a. }
    \omega.
  \end{align}
\end{proof}

\subsection{Proof of Theorem~\ref{thm:LLNuCLTmodel1} for $d \ge 2$}
\label{sect:qCLTd>=2-1}

Here, we complete the proof of Theorem \ref{thm:LLNuCLTmodel1} for
$d \ge 2$.  In view of Proposition~ \ref{prop:qCLT.regen.1}, it
remains to transfer the CLT from the regeneration subsequence to the
full sequence. Thus, we only need to control the behaviour of the
random walk between the regeneration times.

To that end set
$\tau_m \coloneqq T^\joint_m(0,0)- T^\joint_{m-1}(0,0)$ and let
$V_n \coloneqq \max\{m \in \bZ_+ \colon T^\joint_m(0,0) \leq n \}$.

For $\alpha>0$
\begin{align*}
  \begin{split}
    \Pr(\max_{j \le n}\{ j-T_{V_j}\} > c n^\alpha )
    & \leq \Pr(\text{there exists }i\in\{1,\dots,V_n\} \text{ such
      that } \tau_i > cn^\alpha)\\
    & \leq n \Pr(\tau_2 > c n^\alpha) + \Pr(\tau_1 >c n^\alpha)\\
    & \leq C(n+1)n^{-\beta\alpha}
  \end{split}
\end{align*}
which is summable if $1-\beta\alpha < -1$, so
\begin{align}
  \label{eq:proof of thm 5 eq 1}
  \limsup_{n\to \infty} \frac{\max_{j \le n}\{j-T_{V_j}
  \}}{n^{\alpha}} < \infty  \quad \text{ a.s.}
\end{align}
in fact we obtain
\begin{align}
  \label{eq:proof of thm 5 eq 2}
  P_\omega(\max_{j \le n}\{j-T_{V_j} \}> cn^\alpha) \longrightarrow
  0 \quad \text{a.s.}
\end{align}
for an appropriate choice of $\alpha$ and $\beta$. Since we can
choose $\beta$ arbitrarily large it is possible to have the above
probability small for any choice of $\alpha$. Since
$X_{T_{V_n}}=X_{T^\joint_{V_n}(0,0)} = \widehat{X}_{V_n}$ we have
\begin{align}
  \begin{split}
    P_\omega&(\norm{X_n -\widehat{X}_{V_n}} \ge \log(n)cn^\alpha)\\
    &=P_\omega(\norm{X_n -\widehat{X}_{V_n}} \ge \log(n)cn^\alpha,
    \max_{j \le n}\{j-T_{V_j} \}\leq cn^\alpha)\\
    & \hspace{2cm}+P_\omega(\norm{X_n -\widehat{X}_{V_n}} \ge
    \log(n)cn^\alpha,\max_{j \le n}\{j-T_{V_j} \}> cn^\alpha)\\
    & \leq P_\omega(\norm{X_n -\widehat{X}_{V_n}} \ge \log(n)cn^\alpha,
    \max_{j \le n}\{j-T_{V_j} \}\leq cn^\alpha) + P_\omega(\max_{j
      \le n}\{j-T_{V_j} \}> cn^\alpha)
  \end{split}
\end{align}
and
\begin{align}
  \begin{split}
    P_\omega
    & (\norm{X_n -\widehat{X}_{V_n}} \ge \log(n)cn^\alpha, \max_{j \le
      n}\{j-T_{V_j} \}\leq cn^\alpha)\\
    & = P_\omega( \log(n)cn^\alpha \le \norm{X_n -\widehat{X}_{V_n}} \le
    R_\kappa(n-T_{V_n}), \max_{j \le n}\{j-T_{V_j} \}\leq
    cn^\alpha)\\
    & \leq P_\omega(\log(n)cn^\alpha \leq \norm{X_n -\widehat{X}_{V_n}}
    \le R_\kappa cn^\alpha ) \longrightarrow 0 \quad \text{a.s.},
  \end{split}
\end{align}
consequently
\begin{align}
  \label{eq:Thm proof 2}
  P_\omega(\norm{X_n -\widehat{X}_{V_n}} \ge \log(n)cn^\alpha) \longrightarrow 0 \quad \text{a.s.}
\end{align}
By \eqref{eq:60} in Lemma~\ref{lem:ind-to-joint} for any $\varepsilon>0$
\begin{align}
  \label{eq:Thm proof 3}
  P_\omega(\abs{V_n -n/\bE[\tau^\indi_2]}\geq n^{1/2+\varepsilon}) \rightarrow
  0 \quad \text{a.s.}
\end{align}

Using equation~\eqref{eq:Matthias5} from Lemma~\ref{lem:joint_fluctuations}, since we can
choose $\gamma>1$ there, we have by the Borel-Cantelli Lemma for any $\varepsilon>0$
\begin{align}
  \label{eq:Thm proof 4}
  \limsup_{n\to \infty} \sup_{\abs{k-[\theta n]}  \leq n^{1/2+\varepsilon}}
  \frac{\abs{\widehat{X}_k - \widehat{X}_{[\theta n]}}}{n^{1/4+2\varepsilon}} \to 0
  \quad \text{a.s.}
\end{align}

Writing $X_n/\sqrt{n}$ in terms that we can bound by what we showed above
\begin{align}
  \frac{X_n}{\sqrt{n}} = \frac{X_n - \widehat{X}_{V_n}}{\sqrt{n}} +
  \frac{\widehat{X}_{V_n}-\widehat{X}_{[n/\bE[\tau^\indi_2]]}}{\sqrt{n}} +
  \frac{\widehat{X}_{[n/\bE[\tau^\indi_2]]}}{\sqrt{n/\bE[\tau^\indi_2]}}\sqrt{1/\bE[\tau^\indi_2]}
\end{align}
and let $\Phi$ be defined by
$\Phi(f)\coloneqq\wt{\Phi}(f((\bE[\tau^\indi_2]^{-1/2})\cdot))$, i.e.\
$\Phi$ is the image measure of $\wt{\Phi}$ under
$x \to x/\sqrt{\bE[\tau^\indi_2]}$. Then, defining the sets
\begin{align*}
  \begin{split}
    & A_n \coloneqq \{ \abs{X_n - \widehat{X}_{V_n}}\geq n^\alpha\log n \},\\
    & B_n \coloneqq \{ \abs{V_n - n/\bE[\tau^\indi_2]}\geq n^{\varepsilon + 1/2} \},\\
    & C_n \coloneqq \Big\{ \sup_{\abs{k -  n/\bE[\tau^\indi_2]}\leq
      n^{1/2+\varepsilon}} \abs{\widehat{X}_k - \widehat{X}_{[
        n/\bE[\tau^\indi_2]]}}>n^{1/4+2\varepsilon} \Big\},\\
    & D_n \coloneqq A_n^\compl \cap B_n^\compl \cap C_n^\compl,
  \end{split}
\end{align*}
we conclude
\begin{align}
  \begin{split}
    & \abs{E_\omega[f(X_n/\sqrt{n})] - \Phi(f)}\\
    & \le\abs{E_\omega[\indset{D_n}f(X_n/\sqrt{n})]- \Phi(f)} +
    \norm{f}_\infty E_\omega[\indset{D_n^c}],
  \end{split}
\end{align}
where on $D_n$ we get
\begin{align}
  \begin{split}
    & \abs{E_\omega[\indset{D_n}f(X_n/\sqrt{n})] - \Phi(f)}\\
    & \leq CL_f\left(\frac{n^\alpha\log n}{\sqrt{n}} +
      n^{\gamma-1/2}\right) + \left\vert E_\omega\left[
      f\left(\frac{\widehat{X}_{[n/\bE[\tau^\indi_2]]}}{\sqrt{n/\bE[\tau^\indi_2]}}\sqrt{1/\bE[\tau^\indi_2]}\right)
      \right] - \Phi(f) \right\vert\\
    & \rightarrow 0 \quad \text{a.s. as }n \to \infty,
  \end{split}
\end{align}
by \eqref{eq:Thm proof convergence along reg. times}, $\gamma<1/2$
and the fact that $\alpha>0$ can be chosen close to 0 if the
parameters of the model are tuned correctly, i.e. $p$ close to 1,
$s_\inn$ close to $s_\out$ and $s_\out$ much larger than the a
priori bound $s_\mathrm{max}$ from lemma 2.16 of
\cite{BirknerCernyDepperschmidt2016}. Additionally by \eqref{eq:Thm
  proof 2}, \eqref{eq:Thm proof 3} and \eqref{eq:Thm proof 4}
\begin{align}
  \begin{split}
    E_\omega[\indset{D_n^\compl}]\leq P_\omega(A_n) + P_\omega(B_n) +
    P_\omega(C_n)\rightarrow 0 \quad \text{a.s.}.
  \end{split}
\end{align}
This proves convergence for bounded Lipschitz functions which, by
the Portmanteau-theorem, is sufficient to prove the weak convergence
in Theorem \ref{thm:LLNuCLTmodel1}.
\hfill $\qed$

\subsection{Proof of Theorem~\ref{thm:LLNuCLTmodel1} for $d=1$}
\label{sec:dimension1}

The reason we needed to split the proof for $d=1$ is that in this case
the random walks meet often and single excursions away from each other
will typically not be long. In fact, there will be $O(\sqrt{N})$
collisions up to time $N$, so Lemma~\ref{lem:coupl} can not work in
$d=1$.

Therefore we will need to calculate more carefully and consider the
time for an excursion as well as the number of excursions before time
$n$. It turns out that, although a single excursion will not take up
much time, the random walks will split fast enough such that the total
time spent close to each other up until time $n$ will be of order
$o(n)$ in probability. The idea now is to follow the proof in
\cite{BirknerCernyDepperschmidtGantert2013} with a few adjustments,
where the main problem stems from the fact that our bound on the total
variation distance between $\Pr^{\indi}_{x,x'}$ and
$\Pr^{\joint}_{x,x'}$ only has polynomial decay in the distance of the
starting points $x$ and $x'$. Therefore we will introduce so called
\emph{black box intervals} where the random walks are close to each
other and a coupling using Remark~\ref{rem:TVdistance-joint-ind-1step}
will not be possible. While the random walks are not in a black box
interval however, we can make use of
Remark~\ref{rem:TVdistance-joint-ind-1step}.

\medskip

\noindent
Let
$(\widehat{X}^\joint_n,\widehat{X}^{',\joint}_n)_n$
be a pair of random walks in $d=1$ in the same environment observed
along the simultaneous renewal times with transition probabilities
$\hat \Psi^{\joint}((x,x'),(y,y'))$, i.e.\
\begin{align*}
  \hat{\Psi}^{\joint}((x,x'),(y,y'))=\bP^\joint(\widehat{X}_n=y,\widehat{X}'_n=y'\,\vert
  \,\widehat{X}_{n-1}=x,\widehat{X}'_{n-1}=x' )
\end{align*}
and similarly for random walks in independent environments
\begin{align*}
  \hat{\Psi}^{\indi}((x,x'),(y,y'))=\bP^\indi(\widehat{X}_n=y,\widehat{X}'_n=y'\,\vert
  \,\widehat{X}_{n-1}=x,\widehat{X}'_{n-1}=x' ).
\end{align*}
This section will mostly be about
$(\widehat{X}^\joint_n,\widehat{X}^{',\joint}_n)_n$. We therefore abbreviate
$(\widehat{X}_n,\widehat{X}'_n)_n=(\widehat{X}^\joint_n,\widehat{X}^{',\joint}_n)_n$
and will specify when we mean $\widehat{X}^\indi$ and $\widehat{X}^{'\indi}$.
Write
$\hat{\mathcal{F}}_n \coloneqq \sigma(\widehat{X}_i,\widehat{X}'_i, 0\leq i
\leq n)$ for the canonical filtration of $(\widehat{X}_n,\widehat{X}'_n)_n$.

Set
\begin{align*}
  \phi_1(x,x')
  & \coloneqq \sum_{y,y'}(y-x)\hat \Psi^{\joint}((x,x'),(y,y'))\\
  \phi_2(x,x')
  & \coloneqq \sum_{y,y'}(y'-x')\hat \Psi^{\joint}((x,x'),(y,y'))\\
  \phi_{11}(x,x')
  & \coloneqq \sum_{y,y'}(y-x-\phi_1(x,x'))^2\hat \Psi^{\joint}((x,x'),(y,y'))\\
  \phi_{22}(x,x')
  & \coloneqq \sum_{y,y'}(y'-x'-\phi_2(x,x'))^2\hat \Psi^{\joint}((x,x'),(y,y'))\\
  \phi_{12}(x,x')
  & \coloneqq \sum_{y,y'}(y-x-\phi_1(x,x'))(y'-x'-\phi_2(x,x'))\hat
    \Psi^{\joint}((x,x'),(y,y')).
\end{align*}
By Lemma~\ref{prop:JointRegTimesBound} these are bounded,
\begin{align}
  \label{eq:boundAllPhi}
  C_\phi \coloneqq \norm{\phi_1}_\infty \vee \norm{\phi_2}_\infty \vee
  \norm{\phi_{11}}_\infty \vee \norm{\phi_{22}}_\infty \vee
  \norm{\phi_{12}}_\infty < \infty.
\end{align}
Define
\begin{align}
  A^{(1)}_n
  & \coloneqq \sum_{j=0}^{n-1}\phi_1(\widehat{X}_j,\widehat{X}'_j),
    \quad A^{(2)}_n \coloneqq
    \sum_{j=0}^{n-1}\phi_2(\widehat{X}_j,\widehat{X}'_j),\\
  A^{(11)}_n
  & \coloneqq \sum_{j=0}^{n-1}
    \phi_{11}(\widehat{X}_j,\widehat{X}'_j), \quad A^{(22)}_n
    \coloneqq \sum_{j=0}^{n-1}
    \phi_{22}(\widehat{X}_j,\widehat{X}'_j), \quad A^{(12)}_n
    \coloneqq
    \sum_{j=0}^{n-1}\phi_{12}(\widehat{X}_j,\widehat{X}'_j),\\
  M_n
  & \coloneqq \widehat{X}_n - A^{(1)}_n, \quad M'_n \coloneqq \widehat{X}'_n - A^{(2)}_n.
\end{align}
Now $(M_n), (M'_n), (M_n^2-A^{(11)}_n), ({M'_n}^2-A^{(22)}_n)$ and
$(M_nM'_n - A^{(12)}_n)$ are martingales and by
Lemma~\ref{prop:JointRegTimesBound} the distribution of their
increments has polynomial tails.

We write
$\hat{\sigma}^2 \coloneqq \sum_{y,y'} y^2
\hat{\Psi}^\indi((0,0),(y,y'))$ for the variance of a single
increment under $\hat{\Psi}^\indi$.

By Lemma~\ref{prop:JointRegTimesBound} there exist $C_1,a>0$ such
that for $x,x'\in \bZ$ with $\abs{x-x'}\geq n^a$
\begin{align}
  \label{eq:estimatesPhi}
  \abs{\phi_1(x,x')},\abs{\phi_2(x,x')},\abs{\phi_{12}(x,x')} \leq \frac{C_1}{n^2},\\
  \label{eq:estimatesPhi11}
  \abs{\phi_{11}(x,x')-\hat{\sigma}^2},
  \abs{\phi_{22}(x,x')-\hat{\sigma}^2}\leq \frac{C_1}{n^2}.
\end{align}
Here we can choose the $a$ arbitrarily small by
Lemma~\ref{prop:TVdistance-joint-ind-1step} if we tune the parameters right,
e.g. $\beta\geq 2+ 2/a$ with $\beta$ of said Lemma. See for example,
noting that $\bE^\indi_{0,x'}[\widehat{X}_1]=0$,
\begin{align*}
  \abs{\phi_1(x,x')}
  & = \abs{\sum_{(y,y')} (y-x) \hat{\Psi}^\joint((x,x'),(y,y'))}\\
  & =\abs{\sum_{\substack{y,y'\\\abs{y-x}\geq n^a}}(y-x)
  \hat{\Psi}^\joint((x,x'),(y,y')) +
  \sum_{\substack{y,y'\\\abs{y-x}< n^a}}(y-x)
  \hat{\Psi}^\joint((x,x'),(y,y'))}\\
  & \leq
    \bE_{0,x'}^\joint[\abs{\widehat{X}_1}\indset{\abs{\widehat{X}}\geq n^a}] +
    \bE_{0,x'}^\indi[\abs{\widehat{X}_1}\indset{\abs{\widehat{X}}\geq n^a}] \\
  & + \abs{\sum_{\substack{y,y'\\\abs{y-x}< n^a}}(y-x)
  (\hat{\Psi}^\joint((x,x'),(y,y'))-\hat{\Psi}^\indi((x,x'),(y,y')))}\\
  & \leq C(\beta) n^{a(2-\beta)} + 2n^a\abs{x-x'}^{-\beta} \leq C(\beta)(n^{a(2-\beta)}+
    n^{a(1-\beta)}) \leq C(\beta)n^{a(2-\beta)}
\end{align*}
now for $\beta\geq 2 + 2/a$ we have $a(2-\beta)\leq -2$. The other estimates
can be shown analogously.

Let
\begin{align}
  \label{eq:18}
  R_n \coloneqq \#\{ 0\leq j \leq n : \abs{\widehat{X}_j -
  \widehat{X}'_j}\leq n^a \}
\end{align}
be the time that the random walks spend close to each other until time
$n$. Next we want to prove a moment condition for $R_n$ and that the
predictable processes $(A^{(1)}_n)_n$ and $(A^{(2)}_n)_n$ are small on
the diffusive scale.
\begin{lemma}
  \label{lem:bound_for_A1_and_A2}
  1. There exist $0\leq \delta_R< 1/2, c_R < \infty$ such that for all $x_0,x'_0\in\bZ$
  \begin{align}
    \label{eq:bound_R_n_moment}
    \bE^\joint_{x_0,x'_0}[R_n^{3/2}] \leq c_Rn^{1+\delta_R}\quad \text{for all } n.
  \end{align}
  2. There exist $\delta_C>0, c_C< \infty$ such that for all $x_0,x'_0\in \bZ$
  \begin{align}
        \label{eq:bound_for_A1_and_A2}
    \bE^\joint_{x_0,x'_0}\left[ \frac{\abs{A_n^{(1)}}}{\sqrt{n}} \right], \bE^\joint_{x_0,x'_0}\left[
      \frac{\abs{A_n^{(2)}}}{\sqrt{n}} \right] \leq
    \frac{c_C}{n^{\delta_C}} \quad \text{for all }n.
  \end{align}
\end{lemma}
To prove this lemma, we can follow the proof of Lemma~3.14 from
\cite{BirknerCernyDepperschmidt2016}. The necessary ingredients
are a coupling based on Proposition~\ref{prop:TVdistance-joint-ind-1step} as well as the upper
bound from equation~\eqref{eq:sep in d=1} from
Corollary~\ref{cor:separation corollary d=1} below. We provide the detailed
calculations in Section~\ref{sec:proofs}.\\
\medskip

\noindent
To work toward Corollary~\ref{cor:separation corollary d=1} we will introduce some
new notation: Set $\mathcal{R}_{n,0}\coloneqq0$ for $n\in \N$ and for $i\in \N$
\begin{align}
  \label{eq:defn_MathcalDn}
  \mathcal{D}_{n,i}
  & \coloneqq \min\{m>\mathcal{R}_{n,i-1}:\abs{\widehat{X}_m-\widehat{X}'_m}\geq n^{b'} \},\\
  \label{eq:defn_MathcalRn}
  \mathcal{R}_{n,i}
  & \coloneqq \min\{ m>\mathcal{D}_{n,i}: \abs{\widehat{X}_m-\widehat{X}'_m} \leq  n^a \},
\end{align}
with $b' \in (0,1/2)$ and $0<a\ll b'$. We call
$[\mathcal{R}_{n,i-1},\mathcal{D}_{n,i})$ the $i$-th \emph{black box
  interval}. With this definition $R_n$ is the time spent in a black
box interval until time $n$. Note that we can not make use of the
coupling result from Lemma~\ref{prop:TVdistance-joint-ind-1step} in those
intervals.

We differentiate between four possible types of black box intervals,
depending on the relative positions of $\widehat{X}$ and $\widehat{X}'$
at the beginning and end of the interval:
\begin{align}
  W_{n,i}\coloneqq
  \begin{cases}
    1 \quad \text{if }
    \widehat{X}_{\mathcal{R}_{n,i-1}}>\widehat{X}'_{\mathcal{R}_{n,i-1}},
    \widehat{X}_{\mathcal{D}_{n,i}}<\widehat{X}'_{\mathcal{D}_{n,i}},\\
    2 \quad \text{if }
    \widehat{X}_{\mathcal{R}_{n,i-1}}>\widehat{X}'_{\mathcal{R}_{n,i-1}},
    \widehat{X}_{\mathcal{D}_{n,i}}>\widehat{X}'_{\mathcal{D}_{n,i}},\\
    3 \quad \text{if }
    \widehat{X}_{\mathcal{R}_{n,i-1}}<\widehat{X}'_{\mathcal{R}_{n,i-1}},
    \widehat{X}_{\mathcal{D}_{n,i}}>\widehat{X}'_{\mathcal{D}_{n,i}},\\
    4 \quad \text{if }
    \widehat{X}_{\mathcal{R}_{n,i-1}}<\widehat{X}'_{\mathcal{R}_{n,i-1}},
    \widehat{X}_{\mathcal{D}_{n,i}}<\widehat{X}'_{\mathcal{D}_{n,i}}.
  \end{cases}
\end{align}
By construction and the strong Markov property of
$(\widehat{X}_m,\widehat{X}'_m)_m$ we have that: For each $n \in \N$,
$(\mathcal{R}_{n,i}-\mathcal{D}_{n,i})_{i=1,2,\dots}$ is an i.i.d.\
sequence, and
$(W_{n,i},\mathcal{D}_{n,i}-\mathcal{R}_{n,i-1})_{i=2,3,\dots}$ is a
Markov chain. In addition the two objects are independent, the
transition probabilities of the second chain depend only on the first
coordinate and the following separation lemma, similar to
Lemma~\ref{lem:separ} in $d\ge 2$, holds.
\begin{lemma}
  \label{lem:separ in d=1}
  For any $x_0,x'_0 \in \bZ$ and all small enough, positive $\delta$
  there exist $0<b_2<1/8$ and $C,c>0$ such that
  \begin{align}
    \label{eq:25}
    \Pr^\joint_{x_0,x'_0}(H(n^{\delta} )\ge n^{b_2} ) \leq \exp(-Cn^c), \quad n\in \N.
  \end{align}
  Furthermore, there exists $\varepsilon>0$ such that uniformly in $n$
  \begin{align}
    \label{eq:lower bound for W}
    \Pr^\joint(W_{n,2}=w'\, \vert \, W_{n,1}=w) \geq \varepsilon
  \end{align}
  for all pairs $(w,w')\in \{1,2,3,4\}^2$ where a transition is
  ``logically possible''.
\end{lemma}
Note that as a consequence of \eqref{eq:lower bound for W}
$(W_{n,i})_i$ is exponentially mixing. The proof for
equation~\eqref{eq:25} follows basically the same iteration scheme
used in the proof of Lemma~\ref{lem:separ}. For
equation~\eqref{eq:lower bound for W} we follow the steps used in the
proof of Lemma~3.15 in \cite{BirknerCernyDepperschmidtGantert2013} to
prove equation (3.89) therein. The details of the proof can be found
in
Section~\ref{sec:proofs}.

\medskip

\noindent
As a corollary to equation \eqref{eq:25} from Lemma~\ref{lem:separ in
  d=1} we obtain
\begin{corollary}
  \label{cor:separation corollary d=1}
  We can choose $0<b_2<1/8$ and $C,c>0$ such that for any choice of $x_0,x'_0\in\bZ$
  \begin{align}
    \label{eq:sep in d=1}
    \Pr^\joint_{x_0,x'_0}(\mathcal{D}_{n,i}-\mathcal{R}_{n,i-1}\geq n^{b_2}\,
    \vert \, W_{n,i}=w) \leq \exp(-Cn^c), \quad w\in\{1,2,3,4\}, n\in \N.
  \end{align}
\end{corollary}
\begin{proof}
  This is a direct consequence of the fact that
  $n^{b'} - K\log n > n^{b'}-n^a$.
\end{proof}

By construction we have, due to symmetry,
\begin{align}
  \label{eq:12}
  \bP^\joint(W_{n,j}=1)=\bP^\joint(W_{n,j}=3) \quad \text{and}\quad
  \bP^\joint(W_{n,j}=2)=\bP^\joint(W_{n,j}=4)\quad \text{for all }j,n.
\end{align}

Write
$\hat\sigma^2\coloneqq\sum_{y,y'} y^2\hat{\Psi}^\indi((0,0),(y,y'))$
for the variance of a single increment under $\hat{\Psi}$.
\begin{lemma}
  \label{lem:key lemma for d=1}
  There exist $C>0,\wt{b}\in(0,1/4)$ such that for all bounded
  Lipschitz continuous $f:\R^2\to\R$ and all $x_0,x'_0\in\bZ$
  \begin{align}
    \Big\vert \bE^\joint_{x_0,x'_0}\Big[ f\bigg(
    \frac{\widehat{X}_n}{\hat{\sigma}\sqrt{n}},\frac{\widehat{X}'_n}{\hat{\sigma}\sqrt{n}}
    \bigg) \Big] - \bE\big[ f(Z)\big] \Big\vert \leq
    L_f\frac{C}{n^{\wt{b}}}
  \end{align}
  where $Z$ is two-dimensional standard normal and $L_f$ the Lipschitz
  constant of $f$.
\end{lemma}

Since, as has been shown above, $R_n=o(n)$ (with $R_n$ from
\eqref{eq:18}) in probability, we obtain, using the bounds from
\eqref{eq:estimatesPhi} and \eqref{eq:estimatesPhi11},
\begin{align}
  \frac{A^{(11)}_n}{n} \to \hat{\sigma}^2, \quad \frac{A^{(22)}_n}{n}
  \to \hat{\sigma}^2, \quad \frac{A^{(12)}_n}{n} \to 0
\end{align}
in probability as $n\to \infty$. Since $\widehat{X}_n = M_n+A^{(1)}_n$ and
$\widehat{X}'_n = M'_n+A^{(2)}_n$ we make use of the convergence of
$(M_n/\sqrt{n},M'_n/\sqrt{n})$ and the bounds in
\eqref{eq:bound_for_A1_and_A2} to prove Lemma~\ref{lem:key lemma for
  d=1}. To prepare that, for $n\in \N$, let
\begin{align}
  \label{eq:covariance matrix}
  Q_n\coloneqq
  \begin{pmatrix}
    \phi_{11}(\widehat{X}_{n-1},\widehat{X}_{n-1})& \phi_{12}(\widehat{X}_{n-1},\widehat{X}_{n-1})\\
    \phi_{12}(\widehat{X}_{n-1},\widehat{X}_{n-1}) & \phi_{22}(\widehat{X}_{n-1},\widehat{X}_{n-1})
  \end{pmatrix}
\end{align}
be the conditional covariance matrix given $\hat{F}_{n-1}$ of the
random variable $(M_n-M_{n-1},M'_n-M'_{n-1})$ and let
$\lambda_{n,1}\ge \lambda_{n,2}\ge 0$ be its eigenvalues. Equations
\eqref{eq:estimatesPhi}, \eqref{eq:estimatesPhi11} and
\eqref{eq:boundAllPhi} yield bounds on the entries of $Q_n$ and thus,
by stability properties for the eigenvalues of symmetric matrices,
\begin{align}
  \abs{\lambda_{j+1,1}-\hat{\sigma}^2} +
  \abs{\lambda_{j+1,2}-\hat{\sigma}^2} \le C_2\indset{\{
  \abs{\widehat{X}_j-\widehat{X}'_j}\le n^a \}} + \frac{C_2}{n^2}\indset{\{
  \abs{\widehat{X}_j-\widehat{X}'_j}>n^a \}}
\end{align}
for some constant $C_2< \infty$, see \cite{MR0158519}.

In particular,
\begin{align}
  \label{eq:19}
  \sum_{i=1}^2\abs{n\hat{\sigma}^2- \sum_{j=1}^n \lambda_{j,i}} \le C_2R_n+\frac{C_2}{n}
\end{align}
because for $i=1,2$ with
$B_n\coloneqq \{ j\leq n \colon \abs{\widehat{X}_{j-1}-\widehat{X}'_{j-1} }\le n^a \}$
\begin{align*}
  \abs{n\hat{\sigma}^2- \sum_{j=1}^n \lambda_{j,i}}
  & \le \sum_{j=1}^n \abs{\hat{\sigma}^2 - \lambda_{j,i}}\\
  & = \sum_{j\in B_n} \abs{ \hat{\sigma}^2 - \lambda_{j,i} } +
    \sum_{j\notin B_n} \abs{ \hat{\sigma}^2 - \lambda_{j,i} }\\
  & \le R_nC_2 + (n-R_n)\frac{C_2}{n^2}\\
  & \le R_nC_2 + \frac{C_2}{n}.
\end{align*}
\begin{proof}[Proof of Lemma~\ref{lem:key lemma for d=1}]
  Let $f\colon \R^d\to \R$ be a bounded Lipschitz continuous function
  with Lipschitz constant $L_f$ and $Z$ two-dimensional standard
  normal. Using \eqref{eq:19} and \eqref{eq:bound_R_n_moment} and
  Corollary~1.3 in \cite{Rackauskas1995} we conclude that
  \begin{align}
    \label{eq:20}
    \Bigg|\bE^\joint_{x_0,x'_0}\bigg[ f\Big(
    \frac{M_n}{\hat{\sigma}\sqrt{n}},\frac{M'_n}{\hat{\sigma}\sqrt{n}}
    \Big) \bigg] - \bE\big[f(Z)\big] \Bigg| \le
    L_f\frac{C}{n^{b^*}}\qquad \text{for all } n
  \end{align}
  for some $C<\infty$ and $b^*=\frac{1}{3}(\frac{1}{2}-\delta_R)$. For
  the use of Corollary~1.3 in \cite{Rackauskas1995} we read
  $X_k=\Big((M_k-M_{k-1})/\sqrt{\hat{\sigma}^2n},(M'_k-M'_{k-1})/\sqrt{\hat{\sigma}^2n}
  \Big)$ which leads to a covariance matrix with eigenvalues
  $\wt{\lambda}_{k,i} = \lambda_{k,i}/(\hat{\sigma}^2 n)$ for $i=1,2$.
  Moreover note that due to the tail bounds on the regeneration times
  from Lemma~\ref{prop:JointRegTimesBound}, by tuning the parameters
  right, we obtain
  $\sup_k \bE\big[ \norm{(M_k-M_{k-1},M'_k-M'_{k-1})}^3 \big] <
  \infty$. We briefly want to show the calculation for the second
  expectation from the first part of Corollary~1.3 in
  \cite{Rackauskas1995}.
  \begin{align*}
    &\bE^\joint_{x_0,x'_0}\bigg[ \Big(\sum_{i=1}^2 \vert 1- \sum_{k=1}^n \wt{\lambda}_{k,i}^2 \vert\Big)^{3/2} \bigg]\\
    &=(\hat{\sigma}^2 n)^{-3/2}\bE^\joint_{x_0,x'_0}\bigg[ \Big(\sum_{i=1}^2 \vert \hat{\sigma}^2n- \sum_{k=1}^n \lambda_{k,i}^2 \vert\Big)^{3/2} \bigg]\\
    &\le (\hat{\sigma}^2 n)^{-3/2} C\bE^\joint_{x_0,x'_0}\bigg[ \Big(R_n+\frac{1}{n}\Big)^{3/2} \bigg]
  \end{align*}
  from which we can conclude \eqref{eq:20}. Then, combining
  \eqref{eq:20} and \eqref{eq:bound_for_A1_and_A2} yields
  \begin{align}
    \notag
    & \bigg| \bE^\joint_{x_0,x'_0}\Big[
      f\Big(\frac{\widehat{X}_n}{\hat{\sigma}\sqrt{n}},
      \frac{\widehat{X}'_n}{\hat{\sigma}\sqrt{n}}\Big) \Big] - \bE[f(Z)]\bigg|\\
    \notag
    & \le \bigg| \bE^\joint_{x_0,x'_0}\Big[f\Big(\frac{\hat{M}_n}{\hat{\sigma}\sqrt{n}},
      \frac{\hat{M}'_n}{\hat{\sigma}\sqrt{n}}\Big) \Big] - \bE[f(Z)]\bigg|
      + CL_f \bE^\joint_{x_0,x'_0}\bigg[\frac{\abs{A^{(1)}_n}}{\sqrt{n}}
      + \frac{\abs{A^{(2)}_n}}{\sqrt{n}} \bigg]\\
    \label{eq:21}
    & \le L_f\frac{C}{n^{b^*}} + L_f\frac{C}{n^{\delta_C}}
  \end{align}
\end{proof}

Now we have the desired result for $(\widehat{X},\widehat{X}')$, a
pair of random walks observed along joint regeneration times. To prove
Theorem~\ref{thm:LLNuCLTmodel1} for $d\ge 2$ we used
Proposition~\ref{prop:2nd-mom}, Lemma~\ref{prop:TVdistance-joint-ind-1step}
and Lemma~\ref{lemma:convergence for subsequence to full sequence}.
Note that Lemma~\ref{prop:TVdistance-joint-ind-1step} holds for all $d\ge 1$
and in the proof of Lemma~\ref{lemma:convergence for subsequence to
  full sequence} we did not use $d\ge 2$ and in fact this lemma holds
for all $d\ge 1$. To prove Proposition~\ref{prop:2nd-mom} we now use
Lemma~\ref{lem:key lemma for d=1}, noting that the product of two
bounded continuous Lipschitz-functions is again bounded, continuous
and Lipschitz. Therefore we can directly apply Lemma~\ref{lem:key
  lemma for d=1} in the proof of Proposition~\ref{prop:2nd-mom} to
obtain the same result for $d=1$. The proof for
Theorem~\ref{thm:LLNuCLTmodel1} then stays the same.

\section{A more abstract setting}
\label{sec:abstract_proofs}

The proofs for the abstract setting will largely be in the same spirit
as for the toy model. There are some parts that need to be revisited
since they argue specifically with the environment and the behaviour
of the random walk. We will highlight those parts and explain how the
proofs can be adapted.

To apply our results for the random walk on oriented percolation we
need to, similar to \cite{BirknerCernyDepperschmidt2016}, adapt the
regeneration times for two random walks to a coarse grained level.
This will mainly lead to more notation but in essence we transfer the
ideas above to the model introduced in Section~\ref{sec:mgenermod}.

We start by defining the dynamics for the two random walks. Let
$\widehat{U} \coloneqq (\widehat{U}(x,k)\colon x\in\bZ^d,k\in\N_0)$,
$\widehat{U}' \coloneqq (\widehat{U}'(x,k)\colon x\in\bZ^d,k\in\N_0)$
and $\widehat{U}^{''} \coloneqq (\widehat{U}^{''}(x,k)\colon
x\in\bZ^d,k\in\N_0)$ be independent space-time fields of random
variables that are uniformly distributed on $(0,1)$ and independent of
each other.

For our purposes we consider two independent environments $\eta$ and
$\eta'$ that are both according to equation~\eqref{eq:etadynabstr},
where $\eta'$ is built with a family of random variables $U'$, that is
independent of $U$ and identically distributed. Recall from
\eqref{eq:phix}
\begin{align}
  \varphi_X\colon \bZ^{B_{R_X}}_+\times\bZ^{B_{R_X}}_+ \times [0,1]
  \to B_{R_X}
\end{align}
a measurable function that defines the dynamics of the random walks,
where $R_X\in \N$ is an upper bound on the jump size as well as on the
dependence range. Given $\eta$, respectively $\eta'$, we can define
the $\joint$ and $\indi$-dynamics of a pair of random walks. Let
$X_0=X'_0=X^{''}_0=0$ then we define

\begin{align}
  \label{eq:defn random walks abstract}
  \begin{split}
    X_{k+1}(\eta) & \coloneqq X_k(\eta) + \varphi_X \Big( \Theta^{X_k}
                    \eta_{-k}\vert_{B_{R_X}}\, , \, \Theta^{X_k}
                    \eta_{-k-1}\vert_{B_{R_X}}\, , \widehat{U}(X_k,k) \Big), \quad
                    k=0,1,\dots \\
    X'_{k+1}(\eta) & \coloneqq X'_k(\eta) + \varphi_X \Big(
                     \Theta^{X'_k} \eta_{-k}\vert_{B_{R_X}}\, , \, \Theta^{X'_k}
                     \eta_{-k-1}\vert_{B_{R_X}}\, , \widehat{U}'(X'_k,k) \Big), \quad
                     k=0,1,\dots, \\
    X^{''}_{k+1}(\eta') & \coloneqq X^{''}_k(\eta') + \varphi_X \Big(
                          \Theta^{X^{''}_k} \eta'_{-k}\vert_{B_{R_X}}\, , \,
                          \Theta^{X^{''}_k} \eta'_{-k-1}\vert_{B_{R_X}}\, ,
                          \widehat{U}^{''}(X^{''}_k,k) \Big), \quad k=0,1,\dots.
  \end{split}
\end{align}
Note that with this definition we have
$(X^\joint,X^{'\,\joint})=(X,X')$ and
$(X^\indi,X^{'\,\indi})=(X,X^{''})$ for the $\joint$ and $\indi$ pair
of random walks. Where for $(X^\joint, X^{'\,\joint})$ both random
walks again evolve in the same environment $\eta$ and for
$(X^\indi,X^{'\,\indi})$ they evolve in two independent environments.

With the same interpretation as in Section~\ref{sec:preparations}, see
above Proposition~\ref{prop:TVdistance-joint-ind-1step}, we will often
just write $X$ and $X'$ referring to a pair of random walks and from
the context make it clear whether we mean the $\joint$ or $\indi$
pair.

\medskip
\noindent

\begin{remark}
  \label{rem:abstract multiple defn}
  In the following we will define objects needed for the construction
  of the regeneration times for a pair of random walks $(X,X')$. To
  that end we define most objects for $X$ and $X'$ separately. Again
  we want to note that those objects can be different, depending on
  whether we are looking at the pair $(X,X')$ in the $\joint$ or
  $\indi$ dynamics, i.e. at $(X^\joint,X^{',\joint})$ or
  $(X^\indi,X^{',\indi})$. We do this to reduce the amount of notation
  and make it clear from the context which pair we are considering.
\end{remark}

We define
\begin{align}
  \label{eq:coarse grained neighbourhood}
  \begin{split}
    \wt{V}_{\wt{m}}
    & \coloneqq \{ \tilde{x}\in\bZ^d \colon \exists k, (\wt{m}-1)L_{\mathrm{t}} \le
      k \le \wt{m}L_{\mathrm{t}},\norm{X_k-\tilde{x}L_{\mathrm{s}}} \le L_{\mathrm{s}}+R_X \},\\
    \wt{V}'_{\wt{m}}
    & \coloneqq \{ \tilde{x}\in\bZ^d \colon \exists k, (\wt{m}-1)L_{\mathrm{t}} \le
      k \le \wt{m}L_{\mathrm{t}},\norm{X'_k-\tilde{x}L_{\mathrm{s}}} \le L_{\mathrm{s}}+R_{X} \},
  \end{split}
\end{align}
the neighbourhood of the random walks $X$ and $X'$ on the
coarse-grained level at time $\wt{m}$, i.e.\ all points on the
coarse-grained level for which the random walks see part of the
environment. Define the tubes for both random walks
\begin{align}
  \label{eq:coarse grained double tube}
  \begin{split}
    \Tube_{\tilde{n}}
    &\coloneqq \bigcup_{\wt{m}\le \tilde{n}} \Big(\wt{V}_{\wt{m}}\Big) \times\{\wt{m}\},\\
    \Tube'_{\tilde{n}}
    &\coloneqq \bigcup_{\wt{m}\le \tilde{n}} \Big(\wt{V}'_{\wt{m}}\Big) \times\{\wt{m}\}.
  \end{split}
\end{align}
We recall the definition of the determining clusters
$\DC(\tilde{x},\tilde{n})$ from \cite{BirknerCernyDepperschmidt2016},
which play the role of the determining triangles we introduced for the
toy model in \eqref{eq:defn_determining_triangles}. Here the behaviour
of the random walk depends on the coarse-grained oriented percolation
$\widetilde{\xi}$ as well as the random environment $\eta$. Therefore
we need to control both at regeneration times. By
Assumption~\ref{ass:coupling} and Lemma~\ref{lem:abstr-OCcoupl} we
know that on $\tilde{\xi}=1$ the value of $\eta(x,n)$ can be derived
from only observing the driving noise $U$ in a certain finite set
below $(x,n)$; see Remark~3.13 in
\cite{BirknerCernyDepperschmidt2016}. Recall the constant
\begin{equation}
  \label{eq:defn_K_eta}
  K_\eta=R_\eta\big( \lceil \tfrac{\Lt}{\Ls}\rceil +1  \big)
\end{equation}
from the mentioned remark. We can use the following algorithm to
determine the set $\DC(\tilde{x},\tilde{n})$:
\begin{itemize}
\item[1.] Initially, put $\tilde{k}\coloneqq \tilde{n}, \DC(\tilde{x},\tilde{n})\coloneqq \{(\tilde{x},\tilde{n})\}$.
\item[2.] If $\widetilde{\xi}(\tilde{y},\tilde{k})=1$ for all $(\tilde{y},\tilde{k})\in \DC(\tilde{x},\tilde{n})$ : Stop.
\item[3.] Otherwise, for all blocks where this condition fails, add every block one time layer below that could have influenced it, that is replace $\DC(\tilde{x},\tilde{n})$ by
  \begin{equation*}
    \DC(\tilde{x},\tilde{n})\cup \left\{ (\tilde{z},\tilde{k}-1)\colon
      \norm{\tilde{z}-\tilde{y}}\le K_\eta \text{ for some } \tilde{y}
      \text{ with } \widetilde{\xi}(\tilde{y},\tilde{k})=0 \right\},
  \end{equation*}
  put $\tilde{k}=\tilde{k}-1$ and go back to Step 2.
\end{itemize}

The tubes containing the determining clusters are
\begin{align}
  \label{eq:coarse grained determining tube}
  \begin{split}
    \DTube_{\tilde{n}}
    & \coloneqq \bigcup_{(\tilde{x},\wt{\jmath})\in\Tube_{\tilde{n}}} \DC(\tilde{x},\wt{\jmath}),\\
    \DTube'_{\tilde{n}}
    &\coloneqq \bigcup_{(\tilde{x},\wt{\jmath})\in\Tube'_{\tilde{n}}} \DC(\tilde{x},\wt{\jmath}).
  \end{split}
\end{align}

We define a coarse-graining function $\wt{\pi}:\bZ^d\to\bZ^d$ by
\begin{align}
  \label{eq:coarse-graining function}
  \wt{\pi}(x) = \wt{\pi}(x_1,\dots,x_d) = (\tilde{x}_1,\dots,\tilde{x}_d)
  \coloneqq \Big( \Big\lceil \frac{x_1}{L_{\mathrm{s}}}-\frac{1}{2}
  \Big\rceil,\dots,\Big\lceil \frac{x_d}{L_{\mathrm{s}}}-\frac{1}{2} \Big\rceil
  \Big),
\end{align}
and denote by $\wt{\rho}(x)$ the relative position of $x$ inside the
block centred on $\tilde{x}L_{\mathrm{s}}$, i.e.\ we set
\begin{align}
  \label{eq:relative postion function}
  \wt{\rho}(x) \coloneqq x-\tilde{x}L_{\mathrm{s}}.
\end{align}
We then define the coarse-grained random walk
$\wt{X}=(\wt{X}_{\tilde{n}})_{\tilde{n}=0,1,\dots}$ and the relative
positions $\wt{Y}=(\wt{Y}_{\tilde{n}})_{\tilde{n}=0,1,\dots}$ by
\begin{align}
  \label{eq:defn coarse-grained rw}
  \wt{X}_{\tilde{n}}\coloneqq \wt{\pi}(X_{\tilde{n}L_{\mathrm{t}}})
  \quad \text{and} \quad \wt{Y}_{\tilde{n}}\coloneqq \wt{\rho}(X_{\tilde{n}L_{\mathrm{t}}}),
\end{align}
analogously we define $\wt{X}'_{\tilde{n}}$ and $\wt{Y}'_{\tilde{n}}$
using $X'_{\tilde{n}L_{\mathrm{t}}}$ instead of
$X_{\tilde{n}L_{\mathrm{t}}}$. This yields the following relation
between the coarse-grained random walks and their original
counterparts
\begin{align*}
  X_{\tilde{n}L_{\mathrm{t}}} &= \wt{X}_{\tilde{n}}L_{\mathrm{s}} + \wt{Y}_{\tilde{n}},\\
  X'_{\tilde{n}L_{\mathrm{t}}} &= \wt{X}'_{\tilde{n}}L_{\mathrm{s}} + \wt{Y}'_{\tilde{n}}.
\end{align*}

Since we have a new assumption, we need to address the a priori bound
on the speed of $\wt{X}$. This can be done in the same way Lemma~2.16
was proven in \cite{BirknerCernyDepperschmidt2016} and the proof can
be found in Appendix~\ref{sec:abst_detailed_proofs}.
\begin{lemma}[A priori bound on the speed of $\wt{X}$]
  \label{lem:a priori bound for tilde X}
  If $\eta$ satisfies Assumptions~\ref{ass:markovdyn},
  \ref{ass:coupling} and \ref{ass:irred} and the random walk $X$
  satisfies
  Assumption~\ref{ass:abstract_closeness_to_symmetric_transition_kernel}
  then there exist positive constants $\wt{s}_{\mathrm{max}}, c$ and
  $C$ so that for any starting displacement $\wt{y} \in
  [-L_{\mathrm{s}}/2,L_{\mathrm{s}}/2]^d$
  \begin{align}
    \label{eq:a priori bound for tilde X}
    \bP_{0}\Big(\norm{\wt{X}_{\tilde{n}}}>\wt{s}_{\mathrm{max}} \tilde{n} \,\vert\,
    \wt{Y}_0=\wt{y} \Big) \le C \mathrm{e}^{-c \tilde{n}}, \quad \tilde{n}\in\N.
  \end{align}
  The bound $\wt{s}_{\mathrm{max}}$ can be chosen arbitrarily small by
  taking $1-\varepsilon_{\wt{\omega}}$ close to $1$ and
  $\varepsilon_{\mathrm{symm}}$ close to 0.
\end{lemma}
As a direct consequence of Lemma~\ref{lem:a priori bound for tilde X}
we obtain for all $\tilde{x}$, $\tilde{x}'$ and all $\wt{y}$,
$\wt{y}'$
\begin{align*}
  &\bP^\joint_{\tilde{x},\tilde{x}'}(\norm{\wt{X}_{\tilde{n}}-\tilde{x}}\vee\norm{\wt{X}'_{\tilde{n}}-\tilde{x}'}
    >\wt{s}_{\mathrm{max}}\tilde{n}\,\vert\, \wt{Y}_0=\wt{y},\wt{Y}'_0=\wt{y}')\\
  &\le \bP_{0}(\norm{\wt{X}_{\tilde{n}}}>\wt{s}_{\mathrm{max}}\tilde{n}\,\vert\, \wt{Y}_0=\wt{y} )
    + \bP_0(\norm{\wt{X}'_{\tilde{n}}}>\wt{s}_{\mathrm{max}}\tilde{n}\,\vert\, \wt{Y}'_0=\wt{y}' )\\
  &\le Ce^{-c\tilde{n}},
\end{align*}
since $\bP^\joint_{0,0}(\wt{X}_{\tilde{n}} =\cdot)
=\bP_0(\wt{X}_{\tilde{n}} =\cdot)$ and the same holds for $\wt{X}'$ as
well as in the case of $\bP^\indi$ where the random walks are
independent.

Note that $\varepsilon_{\wt{\omega}}$ depends on $\varepsilon_{U}$ and
therefore, since we have $U=^d U'$,
\begin{align*}
  \bP(U\vert_{\block_4(0,0)}\in G_U) = \bP(U'\vert_{\block_4(0,0)}\in
  G_U)\ge 1-\varepsilon_U.
\end{align*}
Thus we can couple $\wt{U}'$ to an i.i.d.\ Bernoulli random field
$\wt{\omega}'$ with $\bP(\wt{\omega}'(\tilde{x},\tilde{n})=1)\ge
1-\varepsilon_{\wt{\omega}}$ (with the same lower bound as for
$\wt{\omega}$)

Since we aim to compare the behaviour of two random walks in a joint
environment with two random walks evolving in two independent
environments, we define the following filtrations:
\begin{align}
  \begin{split}
    \wt{\mathcal{F}}^\joint_{\tilde{n}} \coloneqq
    & \sigma\Big( (\wt{X}_{\wt{\jmath}},\wt{Y}_{\wt{\jmath}}) \colon 0\le
      \wt{\jmath}\le \tilde{n}\Big) \vee
      \sigma\Big((\wt{X}'_{\wt{\jmath}},\wt{Y}'_{\wt{\jmath}}) \colon 0\le
      \wt{\jmath}\le \tilde{n} \Big)\\
    & \vee \sigma\Big(
      \wt{\omega}(\wt{y},\wt{\jmath}),\tilde{\xi}(\wt{y},\wt{\jmath}),
      U\vert_{\block_4(\wt{y},\wt{\jmath})}\colon
      (\wt{y},\wt{\jmath})\in\DTube_{\tilde{n}}\cup \DTube'_{\tilde{n}} \Big).
  \end{split}
\end{align}
and
\begin{align}
  \begin{split}
    \wt{\mathcal{F}}^\indi_{\tilde{n}} \coloneqq
    & \sigma\Big( (\wt{X}_{\wt{\jmath}},\wt{Y}_{\wt{\jmath}}) \colon 0\le
      \wt{\jmath}\le \tilde{n}\Big) \vee
      \sigma\Big((\wt{X}'_{\wt{\jmath}},\wt{Y}'_{\wt{\jmath}}) \colon 0\le
      \wt{\jmath}\le \tilde{n} \Big)\\
    & \vee \sigma\Big(
      \wt{\omega}(\wt{y},\wt{\jmath}),\tilde{\xi}(\wt{y},\wt{\jmath}),
      U\vert_{\block_4(\wt{y},\wt{\jmath})}\colon
      (\wt{y},\wt{\jmath})\in\DTube_{\tilde{n}} \Big)\\
    & \vee \sigma\Big(
      \wt{\omega}'(\wt{y},\wt{\jmath}),\tilde{\xi}'(\wt{y},\wt{\jmath}),
      U'\vert_{\block_4(\wt{y},\wt{\jmath})}\colon
      (\wt{y},\wt{\jmath})\in\DTube'_{\tilde{n}} \Big).
  \end{split}
\end{align}
where for $\wt{\mathcal{F}}^\indi_{\tilde{n}}$ we want to remind the
reader that $\wt{X}'$ is actually the random walk evolving in the
second environment $\eta'$, i.e.\ it is $\wt{X}^{''}$.

To keep with our construction of regeneration times above we need to
check the determining clusters in the neighbourhood of both walkers
and set
\begin{align}
  \wt{D}_{\tilde{n}} \coloneqq \tilde{n} + \max\Big\{
  \mathrm{height}(\DC(\tilde{x},\tilde{n})) \colon \tilde{x}\in
  \wt{V}_{\tilde{n}}\cup\wt{V}'_{\tilde{n}}\Big\}
\end{align}
and define the stopping times
\begin{align}
  \label{eq:abstract defn sigma}
  \tilde{\sigma}_0 \coloneqq 0, \quad \tilde{\sigma}_i \coloneqq \min\Big\{
  \wt{m}>\tilde{\sigma}_{i-1} \colon \max_{\tilde{\sigma}_{i-1}\le \tilde{n}\le
  \wt{m}} \wt{D}_{\tilde{n}}\le \wt{m} \Big\}, \, i\ge 1.
\end{align}

\begin{lemma}[Tails for $\tilde{\sigma}$]
  When $1-\varepsilon_{\wt{\omega}}$ is sufficiently close to $1$
  there exist finite positive constants $c$ and $C$ such that for all
  $\tilde{x}_0,\tilde{x}'_0$
  \begin{align}
    \label{eq:abstract_tail_bounds_sigma}
    \bP_{\tilde{x}_0,\tilde{x}'_0}(\tilde{\sigma}^a_{i+1} - \tilde{\sigma}^a_{i} >
    \tilde{n} \,\vert \, \wt{\mathcal{F}}^a_{\tilde{\sigma}_i}) \le
    C\mathrm{e}^{-c\tilde{n}} \qquad \text{for all }\tilde{n}=1,2,\dots,
    \quad i=0,1,\dots, a\in\{\joint,\indi\} \text{ a.s.},
  \end{align}
  in particular all $\tilde{\sigma}_i$ are a.s. finite. Furthermore
  \begin{align}
    \label{eq:abstract_stochDomSigmaJoint}
    \mathcal{L}\big(
    (\wt{\omega}(\cdot,-\wt{j}-\tilde{\sigma}^a_i))_{\wt{j}=0,1\dots}
    \,\vert \,  \wt{\mathcal{F}}^a_{\tilde{\sigma}_i}\big) \succcurlyeq
    \mathcal{L}\big(
    (\wt{\omega}(\cdot,-\wt{j}))_{\wt{j}=0,1,\dots})\big) \qquad
    \text{for every } i=0,1,\dots,a\in\{\joint,\indi\} \text{ a.s.},
  \end{align}
  where $\succcurlyeq$ denotes stochastic domination.
\end{lemma}
\begin{proof}
  By Lemma~3.12 from \cite{BirknerCernyDepperschmidt2016} we know that
  $\mathrm{height}(\DC(\tilde{x},\tilde{n}))$ has exponential tail
  bounds. Therefore \eqref{eq:abstract_tail_bounds_sigma} can be shown
  analogously to \eqref{eq:expTailsSigmaJoint}. For
  \eqref{eq:abstract_stochDomSigmaJoint} argue as in the proof of
  \eqref{eq:stochDomSigmaJoint} since $\wt{\omega}$ is again an
  i.d.d.\ Bernoulli random field. Also note that for $a=\indi$
  \eqref{eq:abstract_stochDomSigmaJoint} holds for $\wt{\omega}'$ as
  well.
\end{proof}
For two random walks $\wt{X}$ and $\wt{X}'$ we say that $\tilde{n}$ is
a $(b,s)$-cone time point for the decorated path beyond $\wt{m}$ (with
$\wt{m}<\tilde{n}$) if
\begin{align*}
  \DTube_{\tilde{n}} \cap
  & \Big( \bZ^d \times\{-\tilde{n},\dots,-\wt{m}\} \Big)\\
  & \subset \Big\{ (\wt{y},-\wt{j}) \colon \wt{m}\le \wt{j} \le \tilde{n},
    \min\{\norm{\wt{y} - \wt{X}_{\tilde{n}}},\norm{\wt{y}-\wt{X}'_{\tilde{n}}}\}
    \le b+s(\tilde{n}-\wt{j})  \Big\}.
\end{align*}

\begin{lemma}[Cone time points]
  For \label{lem:abstract cone time point} every $\varepsilon>0$ there
  exist suitable $b$ and $s>\wt{s}_{\mathrm{max}}$ such that for all
  finite $\mathcal{\wt{F}}$-stopping times $\wt{T}$ with
  $\wt{T}\in\{\tilde{\sigma}_i\colon i\in\N \}$ a.s.\ and all
  $\wt{k}\in\N$, with
  $\bar{T}\coloneqq \inf\{\tilde{\sigma}_i\colon \tilde{\sigma}_i \ge \wt{k} \}$
  \begin{align}
    \label{eq:abstract cone time points bound}
    \bP_{\tilde{x}_0,\tilde{x}'_0}(\bar{T} \text{ is a }
    (b,s)\text{-cone time point for the decorated paths beyond }
    \wt{T}\,\vert\, \mathcal{\wt{F}}_{\wt{T}})\ge 1-\varepsilon
  \end{align}
  a.s.\ on $\{\wt{T}<\wt{k}\}$. Furthermore
  $0<s-\wt{s}_{\mathrm{max}}\ll 1$ can be chosen small. The inequality
  in \eqref{eq:abstract cone time points bound} holds for
  $\mathcal{\wt{F}}=\mathcal{\wt{F}}^\joint$ as well as
  $\mathcal{\wt{F}}=\mathcal{\wt{F}}^\indi$, where we consider
  $\tilde{\sigma}^\joint$, respectively $\tilde{\sigma}^\indi$ for
  $\tilde{\sigma}$.
\end{lemma}
\begin{proof}
  The proof is analogous to the proof of Lemma~\ref{lem:conepoint}. We
  need to adapt Lemma~\ref{lem:apriori speed conditioned on F} and
  Lemma~\ref{lemma:expTailsTau} for which we only need the exponential
  tail bounds on $\tilde{\sigma}$. The proofs are then a straight
  forward adaption.
\end{proof}
The regeneration times can be constructed by adapting the version from
Section~\ref{sec:preparations} to the coarse-grained setting in a
straightforward manner. We need the same ingredients for a
regeneration, i.e.\ we want the random walks to discover all reasons
for $0$'s of $\tilde{\xi}$ along the path, $\tilde{\xi}$ in the
neighbourhood of both random walks only has $1$'s, we want the
$\wt{\omega}$ to be in a good configuration (we will be more precise
below) and we want the cones constructed along possible regeneration
attempts to include previous cones.

\smallskip
\noindent
The first ingredient, the times at which the random walks have
discovered all reasons for $0$'s of $\tilde{\xi}$ are given by the
sequence $\tilde{\sigma}$ defined in \eqref{eq:abstract defn sigma}.
The good configurations of $\wt{\omega}$ are given by: Let
$\wt{G}_{\tilde{x},\tilde{x}'}(b_\inn,b_\out,s_\inn,s_\out,h) \subset
\{0,1\}^{\dcs(\tilde{x},\tilde{x}',b_\inn,b_\out,s_\inn,s_\out,h)}$ be
the set of possible $\wt{\omega}$-configurations in
$\dcs(\tilde{x},\tilde{x}',b_\inn,b_\out,s_\inn,s_\out,h)$ with the
property

\begin{multline}
  \label{eq:abstract good omega configurations}
  \forall \tilde{\xi}_1(\cdot,0),\tilde{\xi}_2(\cdot,0) \in \{0,1\}^{\bZ^d}
  \text{ with } \tilde{\xi}_1(\cdot,0)\vert_{B_{b_\out}(\tilde{x})\cup
    B_{b_\out}(\tilde{x}')}=\tilde{\xi}_2(\cdot,0)\vert_{B_{b_\out}(\tilde{x})\cup
    B_{b_\out}(\tilde{x}')}\equiv 1 \text{ and}\\
  \wt{\omega} \in \{0,1 \}^{\bZ^d\times\{1,\dots,h\}} \text{ with }
  \wt{\omega}\vert_{\dcs(\tilde{x},\tilde{x}',b_\inn,b_\out,s_\inn,s_\out,h)}
  \in \wt{G}_{\tilde{x},\tilde{x}'}(b_\inn,b_\out,s_\inn,s_\out,h) \colon \\
  \tilde{\xi}_1(\tilde{z},\tilde{n})=\tilde{\xi}_2(\tilde{z},\tilde{n}) \text{ for all }
  (\tilde{z},\tilde{n})\in \cone(\tilde{x};b_\inn,s_\inn,h)\cup
  \cone(\tilde{x}';b_\inn,s_\inn,h)
\end{multline}
where $\tilde{\xi}_1$ and $\tilde{\xi}_2$ are constructed from the same
$\wt{\omega}$.

Note that, by construction, if
$\tilde{\xi}(\cdot,0)\vert_{B_{b_\out}(\tilde{x})\cup
  B_{b_\out}(\tilde{x}') }\equiv 1$ and
\begin{align*}
  \wt{\omega}\vert_{\dcs(\tilde{x},\tilde{x}',b_\inn,b_\out,s_\inn,s_\out,h)}
  \in \wt{G}_{\tilde{x},\tilde{x}'}(b_\inn,b_\out,s_\inn,s_\out,h)
\end{align*}
then
\begin{align*}
  \{ \eta_n(z)\colon (z,n) \in \block(\tilde{z},\tilde{n}),
  (\tilde{z},\tilde{n})\in \cone(\tilde{x};b_\inn,s_\inn,h)\cup
  \cone(\tilde{x}';b_\inn,s_\inn,h) \}
\end{align*}
is a function depending only on $\eta_0(y)$ for $y$ with
$\min\{\norm{y - \tilde{x}},\norm{y-\tilde{x}'} \}\le b_\out
L_{\mathrm{s}}$ and $U\vert_{\block_4(\tilde{z},\tilde{n})}$ for
$(\tilde{z},\tilde{n})\in \cone(\tilde{x};b_\inn,s_\inn,h)\cup
\cone(\tilde{x}';b_\inn,s_\inn,h)$. In particular if we have a
different $\eta'_0$ build from $U'$ with $\eta_0(y)=\eta'_0(y)$ for
all $y$ with $\min\{\norm{y - \tilde{x}},\norm{y-\tilde{x}'} \}\le
b_\out L_{\mathrm{s}}$ and
$U\vert_{\block_4(\tilde{z},\tilde{n})}=U'\vert_{\block_4(\tilde{z},\tilde{n})}$
for $(\tilde{z},\tilde{n})\in \cone(b_\inn,s_\inn,h)\cup
\cone(b_\inn,s_\inn,h)$ then
\begin{align*}
  \eta_n(z)=\eta'_n(z) \text{ for all } (z,n) \in
  \block(\tilde{z},\tilde{n}),(\tilde{z},\tilde{n})\in
  \cone(\tilde{x};b_\inn,s_\inn,h)\cup \cone(\tilde{x}';b_\inn,s_\inn,h).
\end{align*}
This is true since we can construct the value of $\eta$ in any block
where $\tilde{\xi}=0$ locally and only using information about the
environment inside the inner cone, see the determining clusters and
the definition of $\tilde{\sigma}$. For every block inside the cones
where $\tilde{\xi}=1$ we can use that the property of the good
configuration of the cone shell combined with $\tilde{\xi}=1$ at the
base of the cones enforces that $\eta$ in that block is only depended
on $\eta_0$ and $U$ on the blocks inside the cone construction.

Lastly we can use the same sequence $(t_\ell)_\ell$ introduced around
\eqref{eq:space-time-shift.used} from the toy model, only replacing
$s_{\mathrm{max}}$ with $\wt{s}_{\mathrm{max}}$ to find suitable times
at which we can base a double cone such that previous cones will be
included if the random walks do not exceed the a-priori speed.

\smallskip
\noindent
This yields the following construction to find a regeneration time.
\begin{constr}[Regeneration times on the coarse grained level]
  \label{constr:reg-time abstract}
  \begin{enumerate}[1.]
  \item We start by going to the first $\tilde{\sigma}_i$ after
    $t_1$. Check if
    \begin{enumerate}[(i)]
    \item in the spatial $b_\out$-neighbourhood of
      $(\wt{X}_{\tilde{\sigma}_i},-\tilde{\sigma}_i)$ and
      $(\wt{X}'_{\tilde{\sigma}_i},-\tilde{\sigma}_i)$ we have
      $\tilde{\xi}\equiv 1$
    \item the paths together with tubes and determining clusters have
      stayed inside the interior of the corresponding double cone
      shell, i.e.\ inside one of the inner cones, based at the current
      space-time positions
    \item $\wt{\omega}$ restricted to the double cone shell is in the
      good set defined above in \eqref{eq:abstract good omega
        configurations}.
    \end{enumerate}
    All this has positive probability and if it occurs, we have the
    first regeneration times.
  \item If the previous step is unsuccessful we need to check the
    subsequent times in the sequence $(t_\ell)_\ell$. If not
    previously successful, at the $\ell$-th attempt let
    $\tilde{\sigma}_{J(\ell)}$ be the first $\tilde{\sigma}_i$ after
    $t_\ell$. Check
    \begin{enumerate}[(i)]
    \item if $\tilde{\xi}\equiv 1$ in the spatial
      $b_\out$-neighbourhood of
      $(\wt{X}_{\tilde{\sigma}_{J(\ell)}},-\tilde{\sigma}_{J(\ell)})$
      and
      $(\wt{X}'_{\tilde{\sigma}_{J(\ell)}},-\tilde{\sigma}_{J(\ell)})$
    \item if $\tilde{\sigma}_{J(\ell)}$ is a cone-time point for the
      paths beyond $t_{\ell -1}$
    \item if
      $\norm{\wt{X}_{\tilde{\sigma}_{J(\ell)}}}\le \wt{s}_{\mathrm{max}}
      \tilde{\sigma}_{J(\ell)}$ and
      $\norm{\wt{X}'_{\tilde{\sigma}_{J(\ell)}}}\le \wt{s}_{\mathrm{max}}
      \tilde{\sigma}_{J(\ell)}$
    \item $\wt{\omega}$ restricted to the double cone shell is in the
      good set defined above in \eqref{eq:abstract good omega
        configurations}
    \item the path, together with the determining blocks, up to time
      $t_{\ell-1}$ is contained in a box of diameter
      $s_\out t_{\ell-1} + b_\out$ and height $t_{\ell-1}$
    \end{enumerate}
    If all this holds we have found the regeneration time.
  \end{enumerate}
\end{constr}

Then, by the same arguments as in Section~\ref{sec:preparations-1}
(see Propositions~\ref{prop:JointRegTimesBound} and
\ref{prop:IndRegTimesBound} there), we also obtain the following tail
bounds for the regeneration times.
\begin{lemma}
  Denote \label{lem:abstract_JointRegTimesBound}by $\wt{R}_1$ the
  first simultaneous regeneration time, then there are positive
  constants $C$ and $\tilde{\beta}$ so that
  \begin{align}
    \label{eq:537}
    \sup_{\tilde{x}_0,\tilde{x}_0'} \Pr_{\tilde{x}_0,\tilde{x}_0'}^{\joint} (\wt{R}_1 >t) \le
    Ct^{-\tilde{\beta}} \quad \text{and} \quad
    \sup_{\tilde{x}_0,\tilde{x}_0'} \Pr_{\tilde{x}_0,\tilde{x}_0'}^{\indi} (\wt{R}_1 >t) \le
    Ct^{-\tilde{\beta}}.
  \end{align}
  Furthermore, $\tilde{\beta}$ can be chosen arbitrarily large if
  $1-\varepsilon_{\wt{\omega}}$ is suitably close to $1$, which can be
  done by tuning $\varepsilon_U$ appropriately, and by tuning
  $\wt{s}_{\mathrm{max}}$ small enough. The same bounds hold for every
  increment $\wt{R}_i-\wt{R}_{i-1}$ for $i=2,\dots$.
\end{lemma}
\begin{proof}
  The proof is essentially the same as for
  Proposition~\ref{prop:JointRegTimesBound}. Note that the path
  containment property of the construction above holds from some
  finite $\ell_0$ on. Given the construction and all information
  obtained until the $(\ell-1)$-th step, the probability that the
  other conditions hold is uniformly bounded away from $0$ and in fact
  high.

  The constructed regeneration time $\wt{R}_1$ has the following
  properties:
  \begin{enumerate}[(i)]
  \item $\tilde{\xi}(\wt{X}_{\wt{R}_1}+\wt{y},-\wt{R}_1)
    =\tilde{\xi}(\wt{X}'_{\wt{R}_1}+\tilde{z},-\wt{R}_1)$ for all
    $\norm{\wt{y}},\norm{\tilde{z}}\le b_\out$,
  \item the path with determining blocks up to time $\wt{R}_1$ is
    contained in
    \begin{align*}
      \cone\Big((\wt{X}_{\wt{R}_1},-\wt{R}_1);b_\inn,s_\inn,\wt{R}_1\Big)
      \cup
      \cone\Big((\wt{X}'_{\wt{R}_1},-\wt{R}_1);b_\inn,s_\inn,\wt{R}_1\Big);
    \end{align*}
  \item $\wt{\omega}$ in the double cone shell ``centred'' together at
    the base points $(\wt{X}_{\wt{R}_1},-\wt{R}_1)$ and
    $(\wt{X}'_{\wt{R}_1},-\wt{R}_1)$ lies in the good set, that is
    \begin{align*}
      \wt{\omega}\vert_{\dcs\big((\wt{X}_{\wt{R}_1},-\wt{R}_1),(\wt{X}'_{\wt{R}_1},-\wt{R}_1);
      b_\inn,b_\out,s_\inn,s_\out,\wt{R}_1\big)} \in
      G_{\wt{X}_{\wt{R}_1},\wt{X}'_{\wt{R}_1},b_\inn,b_\out,s_\inn,s_\out,\wt{R}_1},
    \end{align*}
  \end{enumerate}
  where $G$ comes from \eqref{eq:GoodOmegaSet} and \eqref{eq:36}.
  Again we can argue as in the proof of
  Proposition~\ref{prop:JointRegTimesBound} that we only require a
  geometric number of $t_\ell$'s to find $\wt{R}_1$ and obtain
  \begin{align*}
    \bP^\joint_{\tilde{x}_0,\tilde{x}'_0}(\wt{R}_1 > \tilde{n}) \le
    \bP^\joint_{\tilde{x}_0,\tilde{x}'_0}(\text{more that }\log \tilde{n}/\log
    c \text{ attempts needed}) \le \delta^{\log \tilde{n}/\log c} =
    \tilde{n}^{-\tilde{\beta}},
  \end{align*}
  where we can make $\tilde{\beta}$ arbitrarily large when
  $\varepsilon_{\wt{\omega}}$ is close to $0$.
\end{proof}
Setting
\begin{align*}
  \hat{\eta}_1 \coloneqq
  & (\eta_{-L_{\mathrm{t}}\wt{R}_1}(x) \colon \norm{x-L_{\mathrm{s}}\wt{X}_{\wt{R}_1}}\le b_\out L_{\mathrm{s}}),\\
  \hat{\eta}'_1 \coloneqq
  & (\eta_{-L_{\mathrm{t}}\wt{R}_1}(x) \colon \norm{x-L_{\mathrm{s}}\wt{X}'_{\wt{R}_1}}\le b_\out L_{\mathrm{s}}),\\
  \widehat{Y}_1 \coloneqq
  & \wt{Y}_{\wt{R}_1}, \text{ the displacement of } X_{L_{\mathrm{t}}\wt{R}_1} \text{ relative to the centre of the}\\
  & L_{\mathrm{s}} \text{-box in which it is contained,}\\
  \widehat{Y}'_1 \coloneqq
  & \wt{Y}'_{\wt{R}_1}, \text{ the displacement of } X'_{L_{\mathrm{t}}\wt{R}_1} \text{ relative to the centre of the}\\
  &L_{\mathrm{s}} \text{-box in which it is contained.}
\end{align*}
By shifting the space-time origin to $(\wt{X}_{\wt{R}_1},-\wt{R}_1)$
we can start anew conditioned on seeing
\begin{enumerate}[(i)]
\item the configuration $\tilde{\xi}\equiv 1$ in the
  $b_\out$-neighbourhood of both $\wt{X}_{\wt{R}_1}$ and
  $\wt{X}'_{\wt{R}_1}$;
\item $\hat{\eta}_1$ in the $b_\out L_{\mathrm{s}}$ box around
  $\wt{X}_{\wt{R}_1}$ and $\hat{\eta}'_1$ for the box around
  $\wt{X}'_{\wt{R}_1}$;
\item the displacements of the walker on the fine level relative to
  the centre of the corresponding coarse-graining box given by
  $\widehat{Y}_1$ and $\widehat{Y}'_1$.
\end{enumerate}
We iterate the above construction and obtain a sequence of random
times $\wt{R}_i$, positions $\wt{X}_{\wt{R}_i}$ and
$\wt{X}'_{\wt{R}_i}$, relative displacements $\widehat{Y}_i,\widehat{Y}'_i$
and local configurations $\hat{\eta}_i,\hat{\eta}'_i$. By construction
\begin{align}
  \label{eq:43}
  \Big(
  \wt{X}_{\wt{R}_i}-\wt{X}_{\wt{R}_{i-1}},\wt{X}'_{\wt{R}_i}-\wt{X}'_{\wt{R}_{i-1}},
  \wt{R}_i-\wt{R}_{i-1},\wt{X}_{\wt{R}_{i-1}},\wt{X}'_{\wt{R}_{i-1}},
  \widehat{Y}_i,\widehat{Y}'_i,\hat{\eta}_i ,\hat{\eta}'_i\Big)_{i\in\N}
\end{align}
is a Markov chain. Furthermore, the increments
$(\wt{X}_{\wt{R}_i}-\wt{X}_{\wt{R}_{i-1}},\wt{X}'_{\wt{R}_i}-\wt{X}'_{\wt{R}_{i-1}},\wt{R}_i-\wt{R}_{i-1})$
only depend on
$(\widehat{Y}_{i-1},\widehat{Y}'_{i-1},\hat{\eta}_{i-1},\hat{\eta}'_{i-1},\wt{X}_{\wt{R}_{i-1}}-\wt{X}'_{\wt{R}_{i-1}})$.
Along the random times $L_{\mathrm{t}}\wt{R}_i$
\begin{align}
  \label{eq:44}
  X_{L_{\mathrm{t}}\wt{R}_i} = \widehat{Y}_i +\sum_{j=1}^i
  L_{\mathrm{s}}(\wt{X}_{\wt{R}_j}-\wt{X}_{\wt{R}_{j-1}})
\end{align}
is the sum of an additive functional of a Markov chain and a function
of the current state of the chain, analogously for $\wt{X}'$.

Recall Remark~\ref{rem:eta-top_eta-bot_independent}. The claim from
this remark holds as follows. Let $\eta_0,\eta_1\in G_\eta$ be two
good configurations and consider the space-time block
$\block(\tilde{x},\tilde{n})$, furthermore let $\eta_{\mathrm{top}}$
and $\eta_{\mathrm{bot}}$ be the configurations at the top, respective
bottom, of $\block(\tilde{x},\tilde{n})$. Then we have
\begin{align}
  \label{eq:eta_rem_1}
  \begin{split}
    \bP&(\eta_{\mathrm{top}}=\eta_1, \eta_{\mathrm{bot}}=\eta_0, \restr{U}{\mathsf{block}_4(\tilde{x},\tilde{n})}\in G_U)\\
       &=\bP(\eta_{\mathrm{bot}}=\eta_0)\int_{G_U}
         \bP(\eta_{\mathrm{top}}=\eta_1 \,\vert\,
         \restr{U}{\mathsf{block}_4(\tilde{x},\tilde{n})}\in du,
         \eta_{\mathrm{bot}}=\eta_{\mathrm{ref}})\bP(\restr{U}{\mathsf{block}_4(\tilde{x},\tilde{n})}\in
         du),
  \end{split}
\end{align}
where $\eta_{\mathrm{ref}}\in G_\eta$ is a reference configuration and
can be any element of $G_\eta$, since by
equation~\eqref{eq:contraction} we obtain the same configuration at
the top for every good configuration at the bottom. Moreover,
\begin{align}
  \label{eq:eta_rem_2}
  \begin{split}
    \bP&(\eta_{\mathrm{bot}} \in G_\eta, \eta_{\mathrm{top}} = \eta_1, \restr{U}{\mathsf{block}_4(\tilde{x},\tilde{n})}\in G_U) \\
       &= \sum_{\eta \in G_\eta} \bP(\eta_{\mathrm{bot}}=\eta, \eta_{\mathrm{top}} = \eta_1, \restr{U}{\mathsf{block}_4(\tilde{x},\tilde{n})}\in G_U) \\
       &=\sum_{\eta \in G_\eta} \bP(\eta_{\mathrm{bot}}=\eta) \int_{G_U} \bP(\eta_{\mathrm{top}}=\eta_1 \,\vert\, \restr{U}{\mathsf{block}_4(\tilde{x},\tilde{n})}\in du, \eta_{\mathrm{bot}}=\eta_{\mathrm{ref}})\bP(\restr{U}{\mathsf{block}_4(\tilde{x},\tilde{n})}\in du)\\
       &=\bP(\eta_{\mathrm{bot}}\in G_\eta) \int_{G_U} \bP(\eta_{\mathrm{top}}=\eta_1 \,\vert\, \restr{U}{\mathsf{block}_4(\tilde{x},\tilde{n})}\in du, \eta_{\mathrm{bot}}=\eta_{\mathrm{ref}})\bP(\restr{U}{\mathsf{block}_4(\tilde{x},\tilde{n})}\in du).
  \end{split}
\end{align}
Combining \eqref{eq:eta_rem_1} and \eqref{eq:eta_rem_2} we obtain
\begin{align}
  \bP&(\eta_{\mathrm{bot}} = \eta_0 \,\vert\, \eta_{\mathrm{top}}=\eta_1, \restr{U}{\mathsf{block}_4(\tilde{x},\tilde{n})}\in G_U, \eta_{\mathrm{bot}} \in G_\eta ) \notag \\
     &= \frac{\eqref{eq:eta_rem_1}}{\eqref{eq:eta_rem_2}} =
       \frac{\bP(\eta_{\mathrm{bot}}=\eta_0)}{\bP(\eta_{\mathrm{bot}}
       \in G_\eta)} = \bP(\eta_{\mathrm{bot}}=\eta_0 \,\vert\,
       \eta_{\mathrm{bot}}\in G_\eta).
       \label{eq:eta_rem_3}
\end{align}
Furthermore, note that since there are only finitely many
configurations for $\widehat{\eta}$ and $\widehat{Y}$ possible at the
time of a regeneration the Lemma~\ref{lem:abstract_JointRegTimesBound}
also implies that there exists a constant $C>0$ such that
\begin{equation}
  \label{eq:61}
  \bP^\indi_{x_0,x'_0} (\wt{R}_i-\wt{R}_{i-1}>t\,\vert\,
  \widehat{\eta}_{i-1}=\eta_{\mathrm{local}},\widehat{Y}_{i-1}=y) \le
  Ct^{-\beta}
\end{equation}
for all good local configurations $\eta_{\mathrm{local}}\in G_\eta$
and all displacements $y\in B_{L_s}$.

\begin{remark}[In the $\indi$-case $(\hat{\eta}_i)_i$ is an independent sequence]
  \label{rem:eta_hat_independent_sequence}
  One immediate consequence of the regeneration construction is that
  the cone shell isolates the values of $\eta$ inside the inner cone
  $\cone(\wt{X}_{\wt{R}_1}, b_\inn,s_\inn, \wt{R}_1)\cup
  \cone(\wt{X}'_{\wt{R}_1}, b_\inn,s_\inn, \wt{R}_1)$ from the
  environment outside the cone construction. Therefore $\eta$ inside
  only depends on $\hat{\eta}_1$, $\hat{\eta}'_1$ and
  $U\vert_{\block_5(\tilde{z},\tilde{n})\cup
    \block_5(\tilde{z},\tilde{n}-1)}$ for $(\tilde{z},\tilde{n}) \in
  \cone(\wt{X}_{\wt{R}_1},b_\inn,s_\inn,\wt{R}_1)\cup
  \cone(\wt{X}'_{\wt{R}_1}, b_\inn,s_\inn, \wt{R}_1)$; see the
  dynamics of $\eta$ in \eqref{eq:etadynabstr}. Furthermore, since
  $\tilde{\xi}(\tilde{y},-\wt{R}_1)=1$ for all
  $\norm{\tilde{y}-\wt{X}_{\wt{R}_1}}\le b_\out$, we know that the
  values of $\eta_{-L_t(\wt{R}_1-1)}$ inside the cone construction
  depend only on $U\vert_{\block_4(\tilde{x},\wt{R}_1)}$ for
  $\norm{\tilde{x}-\wt{X}_{\wt{R}_1}}\le b_\out$, see
  Assumption~\ref{ass:coupling}. Therefore, the values of
  $\eta_{-n}(x)$ inside the inner cone for $n\le L_t(\wt{R}_1-1)$ are
  independent of the values of $\eta$ at the base of the cone
  construction, i.e.\ $\eta_{-L_t\wt{R}_1}\vert_{B_{b_\out
      L_s}(L_s\wt{X}_{\wt{R}_1})}$ and
  $\eta_{-L_t\wt{R}_1}\vert_{B_{b_\out L_s}(L_s\wt{X}'_{\wt{R}_1})}$.
  Thus $(\hat{\eta}_0,\hat{\eta}'_0)$ is also independent of
  $(\hat{\eta}_1,\hat{\eta}'_1)$. Recall the definition of
  deterministic sequence $(t_\ell)_\ell$ of times after which a
  regeneration is attempted; see \eqref{eq:SequenceCondition}. This
  condition enforces that $\hat{\eta}_i$ and $\hat{\eta}'_i$ are in
  the inner cone of the $i+1$-th regeneration, which enables us to
  iterate the above argument for any step in the regeneration scheme
  and conclude $(\hat{\eta}_i,\hat{\eta}'_i)_i$ is a sequence of
  independent random variables.
\end{remark}

\begin{remark}[$(\widehat{X}^\indi_i-\widehat{X}^\indi_{i-1},\widehat{X}^{'\indi}_i-\widehat{X}^{'\indi}_{i-1})_i$ is almost an independent sequence]
  \label{rem:ind_increments_almost_independent}
  The ideas from Remark~\ref{rem:eta_hat_independent_sequence} can
  also be applied to the random walk $(X,X')$. Recall the dynamics of
  the random walk from \eqref{eq:abstr-walk}. For the $k$-th step the
  random walk evaluates $\eta_{-k}$ and $\eta_{-k-1}$ around its
  current position to determine its next step. By construction of the
  regeneration times we know that every part of $\eta$ that the random
  walk needs to evaluate for its next step is inside the inner cone,
  for the last step it needs the values of $\eta$ at the time of the
  regeneration. Because of that the sequence is not independent, but
  $(\hat{X}^\indi_i-\hat{X}^\indi_{i-1},\hat{X}^{'\indi}_i-\hat{X}^{'\indi}_{i-1})$
  is independent of
  $(\hat{X}^\indi_j-\hat{X}^\indi_{j-1},\hat{X}^{'\indi}_j-\hat{X}^{'\indi}_{j-1})$
  if $\abs{i-j}\ge 2$. To see this we fix an $i$. For the increment
  $(\hat{X}^\indi_i-\hat{X}^\indi_{i-1},\hat{X}^{'\indi}_i-\hat{X}^{'\indi}_{i-1})$
  we need to know the values of $\eta_k$ for
  $k=-L_t\wt{R}_{i},\dots,-L_t\wt{R}_{i-1}$, that is the values of
  $\eta$ inside the cone construction for the $i$-th regeneration
  times and at the base of the previous cone
  $\eta_{-L_t\wt{R}_{i-1}}\vert_{B_{b_\out
      L_s}(L_s\wt{X}_{\wt{R}_{i-1}})}$ and
  $\eta'_{-L_t\wt{R}_{i-1}}\vert_{B_{b_\out
      L_s}(L_s\wt{X}'_{\wt{R}_{i-1}})}$. Now take $j$ with
  $\abs{i-j}\ge 2$ and consider the $j$-th regeneration increment. By
  the previous argument and
  Remark~\ref{rem:eta_hat_independent_sequence} we see that the random
  walk looks at an independent part of $\eta$ during this increment
  compared the $i$-th and thus the increments are independent.
\end{remark}

\begin{remark}
  \label{rem:regen_times_almost_independent}
  To decide whether a regeneration time is found, we only have to know
  the environment, that is $\eta$ and $U$ in a local space-time
  window, that is in the cone that is constructed. The only
  information about the environment in the future of the random walks
  are $\hat{\eta}_i$ and $\hat{\eta}'_i$. Since
  $(\hat{\eta}_i,\hat{\eta}'_i)_i$ is an independent sequence and
  $(\hat{\eta}_i,\hat{\eta}'_i)$ only influences the value of
  $\wt{R}_{i+1}-\wt{R}_i$, we conclude that $\wt{R}_{j+1}-\wt{R}_{j}$
  is independent of $\wt{R}_{k+1}-\wt{R}_{k}$ if $\abs{k-j} \ge 2$.
  During this section we will refer to the proofs in the previous
  section. However the increments of the regeneration times in the
  $\indi$-case are not independent in this section but have finite
  range correlations. This property is still strong enough such that a
  CLT, Donskers invariance principle and the law of the iterated
  logarithm hold. By the same arguments we know the same holds for the
  random walks.
\end{remark}

\begin{remark}
  Note that most claims in this section can be made for either the
  coarse-grained random walks $\wt{X},\wt{X}'$ and $\wt{X}^{''}$ along
  the sequence of regeneration times or for the original random walks
  $X,X'$ and $X^{''}$ along $L_t$ times the regeneration times. Since
  we study asymptotic behaviour and the displacements, that is the
  difference between the coarse-grained and original versions of the
  random walks at the regeneration times, are smaller than $L_s$ we
  can easily prove the respective lemmas for both versions.
\end{remark}
Along the simultaneous regeneration times we again obtain a coupling
result comparing the $\joint$ and $\indi$-dynamics.
\begin{lemma}
  There \label{lem:abstract_TVdistance-joint-ind} are constants
  $0<c,\tilde{\beta}<\infty$, such that for all $\tilde{x},\tilde{x}'\in\bZ^d$
  we have
  \begin{align}
    \sum_{\tilde{x},\tilde{x}'\in\bZ^d,\tilde{m}'\in \mathbb{N}}
    &\left| \bP\Big((\wt{X}_{\wt{R}^\joint_1},\wt{X}'_{\wt{R}^\joint_1},\wt{R}^\joint
      _1) = (\tilde{x},\tilde{x}',\tilde{m}') \Big)
      -\bP\Big(
      (\wt{X}_{\wt{R}^\indi_1},\wt{X}^{''}_{\wt{R}^\indi_1},\wt{R}^\indi_1) \in(\tilde{x},\tilde{x}',\tilde{m}')
      \Big) \right| \notag\\
    &\le c\norm{\tilde{x}_0-\tilde{x}'_0}^{-\tilde{\beta}},
  \end{align}
  where $\tilde{x}_0\in\bZ^d$ and $\tilde{x}'_0\in \bZ$ are the
  initial position of the two pairs and the constant $\tilde{\beta}$
  can be chosen large.
\end{lemma}
\begin{proof}
  We can, for the most part, argue as in the proof of
  Proposition~\ref{prop:TVdistance-joint-ind-1step}. We now have to
  check whether $\eta$ changes in a box in which we force the random
  walks to stay. More precisely we define two families of i.i.d.\
  random variables
  \begin{align*}
    U_i \coloneqq \{ U_i(x,n)\colon (x,n)\in\bZ^d\times\bZ \},
  \end{align*}
  where $U_i$ is the random driving noise the environments are build
  on (see \eqref{eq:defn_U}) and $U_1$ is independent of $U_2$.
  Furthermore let another i.i.d.\ family $U_3$ be defined by
  \begin{align*}
    U_3(x,n) \coloneqq
    \begin{cases}
      U_1(x,n), &  \text{if } x_1\le 0,\\
      U_2(x,n), &  \text{if } x_1 >0.
    \end{cases}
  \end{align*}
  Now we consider a pair of random walks $(\wt{X}(U_1),\wt{X}'(U_2))$,
  where we write $\wt{X}(U_1)$ to highlight that $\wt{X}$ moves in the
  environment constructed using $U_1$, and a second pair
  $(\wt{X}(U_3),\wt{X}'(U_3))$ where both move in $U_3$. Therefore the
  distribution of the pair $(\wt{X}(U_1),\wt{X}'(U_2))$ is given by
  $\bP^\indi$ and of $(\wt{X}(U_3),\wt{X}'(U_3))$ by $\bP^\joint$. Let
  \begin{align*}
    \wt{B}
    & \coloneqq \Bigl\{ (z,n) \in \bZ^d\times \bZ: \norm{z-x} \leq \Lt \tfrac{m}{10}+
      L_{\mathrm{s}} (b_\out + \Lt\tfrac{m}{10 R_X} s_\out), n \in
      \bigl\{0,\dots, -L_{\mathrm{t}}\tfrac{m}{10 R_X}\bigr\} \Bigr\},\\
    \wt{B}'
    & \coloneqq \Bigl\{ (z,n) \in \bZ^d\times \bZ: \norm{z-x'} \leq \Lt \tfrac{m}{10} +
      L_{\mathrm{s}}( b_\out + \Lt\tfrac{m}{10 R_X} s_\out), n \in
      \bigl\{0,\dots, -L_{\mathrm{t}}\tfrac{m}{10 R_X}\bigr\} \Bigr\}.
  \end{align*}
  If the first regeneration happens before time $m/(10R_X)$ (here the
  time is for the coarse-grained random walk $\wt{X}$), then these
  boxes will contain the random walks including the cone construction
  for the regeneration. Again we only need to control the probability
  that values of $\eta$ change in these boxes and that the
  regeneration time takes too long. Let $\wt{R}_{1,2}\coloneqq
  \wt{R}_1(\Omega_1,\Omega_2)$ and $\wt{R}_{3,3}\coloneqq
  \wt{R}_1(\Omega_3,\Omega_3)$ be the first regeneration time of the
  pair $(\wt{X}(U_1),\wt{X}'(U_2))$ and $(\wt{X}(U_3),\wt{X}'(U_3))$
  respectively. Note that we can identify $\wt{R}_{1,2}
  \,\widehat{=}\, \wt{R}^\indi_1$ and $\wt{R}_{3,3} \,\widehat{=}\,
  \wt{R}^\joint_1$. By Lemma~\ref{lem:abstract_JointRegTimesBound} we
  have
  \begin{align*}
    \bP(\wt{R}_{1,2}>r) \le Cr^{-\tilde{\beta}}
  \end{align*}
  and
  \begin{align*}
    \bP(\wt{R}_{3,3}>r) \le Cr^{-\tilde{\beta}}.
  \end{align*}
  This leaves to show a suitable upper bound on the probability, that
  $\eta$ does not change inside the boxes $\wt{B}$ and $\wt{B}'$. In
  principle the correlations in $\eta$ have infinite range but on
  $\{\tilde{\xi}(\tilde{x},\tilde{n})=1\}$ we know that
  $\eta\vert_{\block(\tilde{x},\tilde{n})}$ is a function of local
  randomness; see Remark~3.13 in \cite{BirknerCernyDepperschmidt2016}.
  It is then determined by
  $U\vert_{\block_5(\tilde{x},\tilde{n})\cup\block_5(\tilde{x},\tilde{n}-1)}$.
  Furthermore, if we encounter a site where
  $\tilde{\xi}(\tilde{x},\tilde{n})=0$, then the height of
  corresponding determining cluster has exponential tails; see
  Lemma~3.12 in \cite{BirknerCernyDepperschmidt2016}. Since
  \begin{align*}
    \abs{B\cup B'}\le 2 L_{\mathrm{t}}\frac{m}{10R_X}\Bigg( L_{\mathrm{s}}\bigg(
    K_\eta\frac{m}{10}+b_\out+\frac{m}{10R_X}s_\out \bigg) \Bigg)^d
  \end{align*}
  grows polynomially in $m$, the probability that a determining
  cluster for any site in $B\cup B'$ is `too large' is still
  exponentially small. For more detailed steps compare the current
  setting with the proof of
  Proposition~\ref{prop:TVdistance-joint-ind-1step}.
\end{proof}

Now we collect the necessary ingredients, similar to
Section~\ref{sec:preparations} we used to prove the quenched CLT.
Starting with a separation lemma for the coarse-grained random walks.
Recall the definition of $H$ from Lemma~\ref{lemma:exitAnnulus}. We
define $\wt{H}$ to fulfil the the same role for $\wt{X}$, i.e.\
\begin{align*}
  \wt{H}(r) \coloneqq \inf\Big\{ k\in\bZ_+ \colon
  \norm{\wt{X}_{\wt{R}_k}-\wt{X}'_{\wt{R}_k}}_2 \ge r \Big\}.
\end{align*}
\begin{lemma}[Separation lemma]
  Let $d\ge 2$. For \label{lem:abstract separ} all small enough $\delta>0$ and $\varepsilon>0$ there exist $C,c>0$ such that
  \begin{equation}
    \label{eq:abstract sep d>=2}
    \sup_{x_0,x'_0}\bP^\joint_{x_0, x_0'}\Big( \wt{H}(\tilde{n}^\delta) > \tilde{n}^{2\delta+\varepsilon} \Big) \le \exp(-C\tilde{n}^c).
  \end{equation}

  Let $d=1$. For any small enough $\delta>0$ there exists
  $0<b_2<1/8$  and $C,c>0$ such that
  \begin{align}
    \label{eq:abstract sep d=1}
    \sup_{x_0,x'_0}\Pr^{\joint}_{x_0, x_0'}\left( \wt{H}(\tilde{n}^{\delta}) >
    \tilde{n}^{b_2} \right) \leq \exp(-C\tilde{n}^c).
  \end{align}
\end{lemma}
\begin{proof}
  The proof will follow the same basic ideas as the proof of
  Lemma~\ref{lem:separ}. Without loss of generality we assume that
  $\wt{X}_0=\wt{X}'_0=0$. We start by ``forcing'' the random walks to
  a distance of $\varepsilon_1 \log \tilde{n}$. We claim that there
  exist a small $\varepsilon_1>0$ and some $b_4\in(0,1/2), b_5>0$ such
  that
  \begin{align}
    \label{eq:abstract sep lemma step 1}
    \bP^\joint_{x,y} \Big( \wt{H}(\varepsilon_1 \log \tilde{n} ) > \tilde{n}^{b_4} \Big) \le b\tilde{n}^{-b_5},
  \end{align}
  uniformly in $x,y\in\bZ^d$. Using the closeness to a simple random
  walk on $\wt{G}=1$ we construct corridors to separate the random
  walks. If the environment allows the random walks to take only steps
  on $\wt{G}=1$ it can reach a distance of $\varepsilon_1
  \log\tilde{n}$ in as many steps, since by
  Assumption~\ref{ass:abstract_closeness_to_symmetric_transition_kernel}
  there has to exist some $\tilde{x}\in B_{R_{L_{\mathrm{t}}}}(0)$ and
  $\delta_1>0$ such that
  \begin{align*}
    \bP(X_{\tilde{n}L_{\mathrm{t}}}=\tilde{x}+z \,\vert\,
    X_{(\tilde{n}-1)L_{\mathrm{t}}}=z,
    \wt{G}(\tilde{x},\tilde{n})=1,\eta) \ge \delta_1.
  \end{align*}
  Since the reference transition kernel $\kappa_{L_{\mathrm{t}}}$ is
  symmetric we obtain the same lower bound for $-\tilde{x}$ and
  therefore, with probability at least $\delta_1^2$ the random walks
  can separate by a distance of at least $2\norm{\tilde{x}}_\infty>1$
  each step. It remains to show that such corridors exist with
  positive probability. Assume that $\tilde{x}=e_1$ the unit vector of
  the first coordinate. This will result in the slowest possible way
  to separate via the above corridors. Starting from
  $\wt{X}_{\wt{R}_0}=0$ and $\wt{X}'_{\wt{R}_0}=0$ we want to follow
  the paths $\wt{X}_{\wt{R}_i}=\wt{X}_i=-i\cdot \tilde{x}$ and
  $\wt{X}'_{\wt{R}_i}=\wt{X}'_i=i\cdot \tilde{x}$ for all
  $i=1,\dots,\varepsilon_1\log \tilde{n}$. This can be done if the
  paths walk on only $\wt{G}=1$ and therefore regenerate every time at
  the first attempt which we can enforce if $\tilde{\xi}=1$ in the
  $b_\out$-neighbourhood along the paths of the random walks $\wt{X}$;
  note that this neighbourhood is on the coarse-grained level. This
  has a uniform lower bound as was discussed in the proof of
  Lemma~\ref{lem:separ} for the contact process $\xi$ therein.
  Consequently, we can conclude that \eqref{eq:abstract sep lemma step
    1} holds by the same arguments as in the proof of
  Lemma~\ref{lem:separ}.

  Note that it is a bit imprecise to say that a regeneration can occur
  at every step since the construction forces us to make the first
  attempt after $t_1\ge 1$. But the arguments essentially stay the
  same with enforcing the regeneration after $t_1+1$ steps every time.

  \smallskip
  \noindent
  The remaining steps build on a combination of the coupling provided
  by Lemma~\ref{lem:abstract_TVdistance-joint-ind} and the invariance
  principle for $(D_n)_n=(X^\indi_{L_t\wt{R}_n}-
  X^{'\indi}_{L_t\wt{R}_n})_n$. Note that a difference compared to the
  toy model is that for each step we have to carry over the local
  configuration of the environment and the displacement to the next
  step because the underlying random walks depend on it. We want to
  highlight though, that the initial configurations only influence the
  first two steps which means that the invariance principle holds for
  any starting configuration by
  Lemma~\ref{lem:abstract_JointRegTimesBound} or more specifically
  equation~\eqref{eq:61}. To be more precise by
  Remark~\ref{rem:regen_times_almost_independent} $(X^\indi)_i$ and
  $(X^{'\indi})_i$ are almost independent sequences along their
  simultaneous regeneration times by construction, have well behaved
  increments and are independent of each other. We briefly highlight
  how these properties are sufficient to show that the process
  $(X^\indi_{L_t\wt{R}_n}- X^{'\indi}_{L_t\wt{R}_n})_n$ satisfies the
  conditions for Theorem~1 in \cite{DoukhanMassartRio1994}. Using the
  notation therein, we define the mixing coefficients
  \begin{equation*}
    \alpha_n = \sup_{(A,B)\in \sigma(D_i\colon i\le 0)\times \sigma(D_i\colon i\ge n)}\abs{\bP(A\cap B)-\bP(A)\bP(B)}
  \end{equation*}
  which, by Remark~\ref{rem:regen_times_almost_independent}, have the
  property $\alpha(n)=0$ for all $n\ge 2$ and for the tail function of
  the increments we have
  \begin{equation*}
    \bP(\norm{D_i-D_{i-1}}> t) \le ct^{-\tilde{\beta}}
  \end{equation*}
  by Lemma~\ref{lem:abstract_JointRegTimesBound}. Setting
  $\alpha^{-1}(u) = \inf\{ t\colon \alpha_t \le u \}$ we obtain the
  following, by observing $\alpha^{-1}(u)\le 2$ for all $u\ge 0$,
  \begin{align*}
    \int_0^1\alpha^{-1}(u) \big[\bP(\norm{D_1} >u) \big]^2 \,du \le 2c^{2/\tilde{\beta}}\int_0^1 u^{-2/\tilde{\beta}}\,du
  \end{align*}
  and the the integral on the right hand side is finite if we tune
  $\tilde{\beta}>2$.

  Note that as a consequence we also obtain
  Lemma~\ref{lemma:exitAnnulus} for
  $\norm{\wt{X}^\indi-\wt{X}^{'\indi}}$, since the displacements are
  only of constant order. With those ingredients we can replicate the
  steps in the proof of Lemma~\ref{lem:separ}.

  \smallskip
  \noindent
  Lastly note that for $d=1$ we essentially use the same methods and
  can therefore adapt the proof of Lemma~\ref{lem:separ in d=1} as
  well. This concludes the proof for all $d\ge 1$.
\end{proof}

\begin{remark}
  Note that the restart arguments used in Lemma~\ref{lem:separ} cannot
  be applied as directly in the abstract setting since the environment
  $\eta$ around the position of the random walkers at the time of the
  regenerations is not necessarily in the same configuration for every
  regeneration. However, the strong mixing properties remedy that as
  the starting configuration has almost no influence on the long time
  behaviour of the random walkers. To illustrate that let
  $r<r_1<r_2<R$ be positive real numbers and consider
  \begin{align*}
    \bP_{r_2,\widehat{\eta}_0,\widehat{Y}_{0}}(\tau_{r_1}<\tau_R<\tau_r)
    &= \bE_{r_2,\widehat{\eta}_0,\widehat{Y}_{0}}\Big[ \bE_{r_2}\big[
      \ind{\tau_{r_1}<\tau_R<\tau_r}\,\vert\, \mathcal{F}_{\tau_{r_1}}
      \big] \Big]\\
    &=\bE_{r_2,\widehat{\eta}_0,\widehat{Y}_{0}}\Big[
      \ind{\tau_{r_1}<\tau_R}\bE_{r_2,\widehat{\eta}_0,\widehat{Y}_{0}}\big[
      \ind{\tau'_R<\tau'_r}\,\vert\, \mathcal{F}_{\tau_{r_1}} \big]
      \Big]\\
    &=\bE_{r_2,\widehat{\eta}_0,\widehat{Y}_{0}}\Big[
      \ind{\tau_{r_1}<\tau_R}\bE_{r_1,\widehat{\eta}_{\tau_{r_1}},\widehat{Y}_{\tau_{r_1}}}\big[
      \ind{\tau_R<\tau_r} \big] \Big],
  \end{align*}
  where we note that $\widehat{\eta}_{\tau_{r_1}}$ is independent of
  the starting configuration $\widehat{\eta}_0$ and distributed
  according to the invariant measure. Thus we can restart from the
  radius $r_1$ with a new configuration
  $(\widehat{\eta}'_0,\widehat{Y}'_0)$ drawn from the invariant
  measure.
\end{remark}
Next we use this separation lemma to prove that the number of steps
with a successful coupling is high. In this setting we say that the
coupling is successful for the $i$-th step if
\begin{align*}
  \wt{X}'_{\wt{R}^\joint_i}-\wt{X}'_{\wt{R}^\joint_{i-1}}
  &= \wt{X}^{''}_{\wt{R}^\indi_i}-\wt{X}^{''}_{\wt{R}^\indi_{i-1}}, \qquad \text{and}\\
  \wt{R}^\joint_i - \wt{R}^\joint_{i-1}
  &= \wt{R}^\indi_i - \wt{R}^\indi_{i-1}.
\end{align*}
Note that by construction this directly implies
$\wt{X}_{L_t\wt{R}^\joint_i}-\wt{X}_{L_t\wt{R}^\joint_{i-1}} =
\wt{X}_{L_t\wt{R}^\indi_i}-\wt{X}_{L_t\wt{R}^\indi_{i-1}}$. Recalling
the notation at the start of this section, \eqref{eq:defn random walks
  abstract} and the text below, a successful coupling means an
increment of the $\joint$ and $\indi$ pair is equal. Note that the
following results, which are the key ingredients to prove the quenched
CLT, are all for the coarse-grained pair. Since the displacements are
only of constant order however, this has almost no impact on the proof
of the quenched CLT for the original pair.

\begin{lemma}
  Let $d\geq 2$. For any \label{lem:abstract coupl} $b_1>0$ and $c>0$
  there exists $\tilde{\beta}$ such that
  \begin{align}
    \label{eq:505}
    \sup_{x_0,x'_0}\Pr^{\joint}_{x_0, x_0'} (\text{at most $N^{b_1}$ uncoupled
    steps before time $N$}) \ge 1-N^{-c}.
  \end{align}
  Let $d=1$. For any $\varepsilon>0$ and $\gamma>0$ there exists
  $\tilde{\beta}$ such that
  \begin{equation}
    \label{eq:506}
    \sup_{x_0,x'_0}\Pr^{\joint}_{x_0, x_0'} (\text{at most $N^{1/2+\varepsilon}$ uncoupled
      steps before time $N$}) \ge 1-N^{-\gamma}.
  \end{equation}
\end{lemma}

\begin{proof}
  The proof for Lemma~\ref{lem:coupl} requires
  Proposition~\ref{prop:TVdistance-joint-ind-1step},
  Lemma~\ref{lem:separ} which we have already proven for the abstract
  setting and the invariance principle for the coarse-grained pair in
  the $\indi$-case. Furthermore, we can again choose
  $\varepsilon_{\wt{\omega}}$ close to 0 and $\tilde{\beta}$ large by
  tuning the parameters of the model. Therefore we have all
  ingredients to adapt the proof of Lemma~\ref{lem:coupl} to the
  current setting.
\end{proof}

We have the comparison lemma used to prove the key proposition.

\begin{lemma}
  Let \label{lemma:abstract joint-ind-comparison} $d\geq 2$. Then,
  there exist $a,C>0$ such that for every pair of bounded Lipschitz
  functions $f,g:\R^d\to \R$
  \begin{align}
    \begin{split}
      \abs{& \bE^{\joint}_{0,0}[f(\wt{X}_{\wt{R}_n}/\sqrt{n})g(\wt{X}'_{\wt{R}_n}/\sqrt{n})]
             - \bE^{\indi}_{0,0}[f(\wt{X}_{\wt{R}_n}/\sqrt{n})g(\wt{X}'_{\wt{R}_n}/\sqrt{n})]}\\
           & \leq C(1 + \norm{f}_\infty L_f)(1+\norm{g}_\infty L_g) n^{-a},
    \end{split}
  \end{align}
  where $L_f \coloneqq\sup_{x\neq y}\abs{f(y) - f(x)}/\norm{x-y}$ and
  $L_g$ are the Lipschitz constants of $f$ and $g$.
\end{lemma}
\begin{proof}
  The main tools for the proof of
  Lemma~\ref{lemma:joint-ind-comparison} are Lemma~\ref{lem:coupl},
  Propositions~\ref{prop:JointRegTimesBound} and
  \ref{prop:IndRegTimesBound} and finite range for the transition
  kernels from Assumption~\ref{ass:finite-range}. Both lemmas have
  been adapted to the abstract setting and finite range is an one of
  the assumptions made for the random walks.
\end{proof}

We use the same idea, as in Lemma~\ref{lem:ind-to-joint}, of
transferring properties for the $\indi$-regeneration times to the
$\joint$ ones. To that end define $\wt{\tau}^\indi_m \coloneqq
\wt{R}^\indi_m-\wt{R}^\indi_{m-1}$, let $\wt{V}^\indi_n \coloneqq
\max\{ m\in\bZ_+\colon \wt{R}^\indi_m \le n \}$ and define
$\wt{V}^\joint_n$ analogously.
\begin{lemma}
  We have \label{lem:abstract ind-to-joint}
  \begin{equation}
    \limsup_{n\to \infty} \frac{\abs{\wt{V}^\joint_n -
        n/\bE[\wt{\tau}^\indi_2]}}{\sqrt{n\log\log n}} < \infty \qquad
    \text{a.s.}
  \end{equation}
  Furthermore for any $\varepsilon>0$ and $C,\gamma>0$ there exist
  $\tilde{\beta}$ so that
  \begin{equation}
    \sup_{x,x'}\bP^\joint_{x,x'}\Big( \exists m\le n\colon
    \abs{\wt{R}_m - m\bE[\wt{\tau}^\indi_2]} > n^{1/2+\varepsilon}
    \Big) \le Cn^{-\gamma}.
  \end{equation}
\end{lemma}
\begin{proof}
  The proof follows by the same arguments as the proof of
  Lemma~\ref{lem:ind-to-joint}, noticing that we can tune $\wt{\beta}$
  similar to $\beta$ such that $(\wt{R}^\indi)_n$ satisfies the moment
  condition and by Remark~\ref{rem:regen_times_almost_independent} the
  mixing condition of Theorem~3 from \cite{OodairaYoshihara1971}.
  Therefore $(\wt{R}^\indi)_n$ obeys the law of the iterated logarithm
  and we can argue for $\wt{V}^\indi$ again that
  \begin{equation*}
    \limsup_{n\to \infty} \frac{\abs{\wt{V}^\indi_n -
        n/\bE[\wt{\tau}^\indi_2]}}{\sqrt{n\log\log n}} < \infty \qquad
    \text{a.s.}
  \end{equation*}
  The other arguments are analogous to the proof of
  Lemma~\ref{lem:ind-to-joint} with the additional step to split
  $\wt{R}_n$ into two sums of independent increments, i.e. the even
  and odd increments $\sum_{i=1}^{n/2}\wt{R}_{2i}- \wt{R}_{2i-1}$ and
  $\sum_{i=0}^{n/2-1}\wt{R}_{2i+1}-\wt{R}_{2i}$.
\end{proof}

The following lemma will give some control over the fluctuations of
the random walks.
\begin{lemma}
  \label{lem:abstract_joint_fluctuations}
  For any $\delta, \varepsilon>0$ and $\gamma>0$, if parameters are tuned right
  (in part., $\beta$ must be very large)
  \begin{align}
    \label{eq:abstract_Matthias5}
    \sup_{0 \le \theta \le 1} \Pr^{\joint}\left( \sup_{\abs{k-[\theta n]}  \leq n^\delta}
    \norm{\wt{X}_{\wt{R}_k} - \wt{X}_{\wt{R}_{[\theta n]}}} >  n^{\delta/2+\varepsilon}\right)
    \le C n^{-\gamma}.
  \end{align}
\end{lemma}

\begin{proof}
  The proof follows along the same steps as the proof of
  Lemma~\ref{lem:joint_fluctuations} making use of
  Lemma~\ref{lem:abstract ind-to-joint} to transfer properties of the
  $\indi$-dynamics to the $\joint$-dynamics, with a small adaptation
  due to $(\wt{X}_{\wt{R}^\indi_i}-\wt{X}_{\wt{R}^\indi_{i-1}})_i$ and
  $(\wt{R}^\indi_i-\wt{R}^\indi_{i-1})_i$ not being independent
  sequences. Using the observations from
  Remark~\ref{rem:regen_times_almost_independent}
  \begin{align*}
    &\bP\Big( \sup_{\abs{[\theta n]-k } \leq n^\delta}
      \norm{\wt{X}^\indi_{\wt{R}_k} - \wt{X}^\indi_{\wt{R}_{[\theta
      n]}}} > \frac{1}{2} n^{\delta/2 +\varepsilon} \Big)\\
    &\le \bP\Big( \sup_{\abs{[\theta n]-k } \leq n^\delta} \norm{
      \sum_{i=\min\{ k,[\theta n] \}+1}^{\max\{ k,[\theta n] \}}
      \wt{X}^\indi_{\wt{R}_{i}} -\wt{X}^\indi_{\wt{R}_{i-1}} } >
      \frac{1}{2}n^{\delta/2+\varepsilon}  \Big)\\
    &\le \bP\Big(  \sup_{\abs{[\theta n]-k } \leq n^\delta}
      \norm{\sum_{\substack{i=\min\{ k,[\theta n] \}+1,\\i \text{
    even}}}^{\max\{ k,[\theta n] \}}
    \wt{X}^\indi_{\wt{R}_{i}}-\wt{X}^\indi_{\wt{R}_{i-1}} }
    +\norm{\sum_{\substack{i=\min\{ k,[\theta n] \}+1,\\i \text{
    odd}}}^{\max\{ k,[\theta n] \}}
    \wt{X}^\indi_{\wt{R}_{i}}-\wt{X}^\indi_{\wt{R}_{i-1}} } >
    \frac{1}{2}n^{\delta/2+\varepsilon}\Big)\\
    &\le \bP\Big( \sup_{\abs{[\theta n]-k } \leq n^\delta}
      \norm{\sum_{\substack{i=\min\{ k,[\theta n] \}+1,\\i \text{
    even}}}^{\max\{ k,[\theta n] \}}
    \wt{X}^\indi_{\wt{R}_{i}}-\wt{X}^\indi_{\wt{R}_{i-1}} } >
    \frac{1}{4}n^{\delta/2+\varepsilon} \Big)\\
    &\hspace{2cm}+ \bP\Big( \sup_{\abs{[\theta n]-k } \leq
      n^\delta} \norm{\sum_{\substack{i=\min\{ k,[\theta
      n] \}+1,\\i \text{ odd}}}^{\max\{ k,[\theta n] \}}
    \wt{X}^\indi_{\wt{R}_{i}}-\wt{X}^\indi_{\wt{R}_{i-1}} } >
    \frac{1}{4}n^{\delta/2+\varepsilon} \Big)\\
  \end{align*}
  we split the sum into two sums of independent random variables and
  can use the same arguments from the proof of
  Lemma~\ref{lem:joint_fluctuations} on both probabilities of the last
  line.
\end{proof}

The key proposition translates to.

\begin{proposition}
  \label{prop:abstract 2nd-mom}
  There exists $c>0$ and a non-trivial centred $d$-dimensional normal
  law $\wt{\Phi}$ such that for $f: \R^d \to \R$ bounded and Lipschitz
  we have
  \begin{align}
    \label{eq:abstract_2nd-mom-estimate}
    \bE\left[ \left( E_U[f(\wt{X}^\joint_{\wt{R}_m}/\sqrt{m})]-\wt{\Phi}(f)
    \right)^2 \right] \leq C_f m^{-c}.
  \end{align}
\end{proposition}
\begin{proof}
  Adapt the proof of Proposition~\ref{prop:2nd-mom}.
\end{proof}

We can expand the convergence from a subsequence to the full sequence.
\begin{lemma}
  \label{lemma:abstract_convergence_for_subsequence_to_full_sequence}
  Assume that for some $c>1$, and any bounded Lipschitz function $f:
  \R^d \to \R$
  \begin{align}
    \label{eq:42}
    E_U[f(\wt{X}_{\wt{R}^\joint_{k^c}(0,0)}/k^{c/2})]\underset{k\to
    \infty}{\longrightarrow} \wt{\Phi}(f) \quad \text{for a.a. }
    U,
  \end{align}
  where $\wt{\Phi}$ is some non-trivial centered $d$-dimensional
  normal law. Then we have for any bounded Lipschitz function $f$
  \begin{align*}
    E_U[f(\wt{X}_{\wt{R}^\joint_{m}(0,0)}/m^{1/2})] \underset{m \to
    \infty}{\longrightarrow} \wt{\Phi}(f) \quad \text{for a.a. } U.
  \end{align*}
\end{lemma}
\begin{proof}
  The proof of Lemma~\ref{lemma:convergence for subsequence to full
    sequence} requires finite range for the random walk steps, $\beta$
  large enough and Proposition~\ref{prop:JointRegTimesBound}. We have
  all of this in the abstract setting.
\end{proof}
This gives us the necessary ingredients to prove the quenched CLT for
the random walk $X$, by comparing it to the coarse-grained random walk
$\widetilde{X}$ along its regeneration times. The proof is similar to
the proof of Theorem~\ref{thm:LLNuCLTmodel1} but has another
complication due to the coarse-graining.
\begin{proof}[Proof of Theorem~\ref{thm:abstract_quenched_clt}]
  During the course of this proof we abbreviate $\widehat{X}_m =
  \widetilde{X}^\joint_{\widetilde{R}_m}$ for the coarse-grained
  process at its $m$-th regeneration time. Let $f\colon \R^d \to \R$
  be a bounded Lipschitz function and $c'>1/c\wedge 1$ with $c$ from
  Proposition~\ref{prop:abstract 2nd-mom}. Using
  Proposition~\ref{prop:abstract 2nd-mom} and Markov's inequality we
  find
  \begin{align*}
    \bP&\Bigg( \Abs{E_U\Big[ f(\widehat{X}_{[n^{c'}]}/\sqrt{[n^{c'}]}) \Big] - \widetilde{\Phi}(f)} >\varepsilon \Bigg)\\
       &\le \frac{\bE\bigg[ \Big( E_U\Big[ f(\widehat{X}_{[n^{c'}]}/\sqrt{[n^{c'}]}) \Big] - \widetilde{\Phi}(f) \Big)^2 \bigg]}{\varepsilon^2}\\
       &\le C_f n^{-cc'}\varepsilon^{-2},
  \end{align*}
  which leads to, using Borel--Cantelli and
  Lemma~\ref{lemma:abstract_convergence_for_subsequence_to_full_sequence},
  \begin{equation}
    \label{eq:48}
    E_U\Big[ f(\widehat{X}_m/\sqrt{m}) \Big] \xrightarrow[m\to\infty]{} \widetilde{\Phi}(f)\qquad \text{for a.a. } U.
  \end{equation}
  We want to show a sufficiently good behaviour of the random walk
  between regeneration times. The difference to the proof of
  Theorem~\ref{thm:LLNuCLTmodel1} is the space-time scaling we now
  need to consider due to the comparison between $X$, the walk itself,
  and its coarse-graining $\widetilde{X}$. To that end recall the
  relation $X_{\widetilde{n}L_t} =
  \widetilde{X}_{\widetilde{n}}L_s+\widetilde{Y}_{\widetilde{n}}$. Let
  $\widetilde{\tau}_m \coloneqq \widetilde{R}^\joint_m(0,0) -
  \widetilde{R}^\joint_{m-1}(0,0) $ the $m$-th increment of the
  regeneration times on the coarse-grained level. Furthermore, we
  abbreviate $\widetilde{R}_m = \widetilde{R}^\joint_m(0,0)$ and let
  $\widetilde{W}_n \coloneqq \max\{ m\in\Z_+ \colon
  L_t\widetilde{R}_m\le n \}$. Note that the scaling by $L_t$ is due
  to the coarse-graining. By Lemma~\ref{lem:abstract ind-to-joint} we
  have
  \begin{equation}
    \label{eq:46}
    \limsup\limits_{n\to \infty} \frac{\abs{L_t\widetilde{W}_n - n/\bE[\widetilde{\tau}_2]}}{\sqrt{n\log\log n}} < \infty \quad \text{a.s.}
  \end{equation}
  By Lemma~\ref{lem:abstract_JointRegTimesBound}, for any $\alpha>0$,
  \begin{align*}
    \bP\Big( \max_{j \le n} \{j-L_t\widetilde{R}_{\widetilde{W}_j} \}
    > cn^\alpha \Big)
    &\le \bP \Big( \text{there exists}i\in\{1,\dots,\widetilde{W}_n \} \text{ such
      that } L_t\widetilde{\tau}_i > cn^\alpha \Big)\\
    &\le n \bP( \widetilde{\tau}_2 > cn^\alpha/L_t)  + \bP(\widetilde{\tau}_1 > cn^\alpha/L_t)\\
    &\le C(n+1) n^{-\tilde{\beta} \alpha},
  \end{align*}
  which is summable if $1-\widetilde{\beta}\alpha< -1 $. Therefore we obtain
  \begin{equation}
    \label{eq:45}
    P_U\Big(  \max_{j\le n}\{ j-L_t\widetilde{R}_{\widetilde{W}_j} \}> cn^\alpha\Big) \xrightarrow[n\to \infty]{} 0 \quad \text{a.s.}
  \end{equation}
  for an appropriate choice of $\alpha$ and $\tilde{\beta}$. Note that
  $\tilde{\beta}$ can be chosen arbitrarily large by
  Lemma~\ref{lem:abstract_JointRegTimesBound} and can therefore be
  adjusted to any choice of $\alpha>0$. This yields the following
  control on the behaviour of the random walk
  \begin{align}
    \begin{split}
      P_U\Big( &\norm{X_n - L_s\widehat{X}_{\widetilde{W}_n}}\ge cn^\alpha\log n\Big)\\
               &\le P_U\Big( \norm{X_n - L_s\widehat{X}_{\widetilde{W}_n}}\ge cn^\alpha\log n, \max_{j\le n}\{ j-L_t\widetilde{R}_{\widetilde{W}_j} \}\le cn^\alpha \Big)\\
               &\hspace{3cm}+ P_U\Big( \norm{X_n - L_s\widehat{X}_{\widetilde{W}_n}}\ge cn^\alpha\log n, \max_{j\le n}\{ j-L_t\widetilde{R}_{\widetilde{W}_j} \}> cn^\alpha \Big).
    \end{split}
  \end{align}
  Due to the relation between $X$ and its coarse-grained version
  $\widetilde{X}$ we have $\norm{X_n -
    L_s\widehat{X}_{\widetilde{W}_n}}= \norm{X_n -
    X_{L_t\widetilde{R}_{\widetilde{W}_n}} -
    \widetilde{Y}_{\widetilde{R}_{\widetilde{W}_n}} }$
  \begin{align*}
    P_U\Big(& \norm{X_n - L_s\widehat{X}_{\widetilde{W}_n}}\ge cn^\alpha\log n, \max_{j\le n}\{ j-L_t\widetilde{R}_{\widetilde{W}_j} \}\le cn^\alpha \Big)\\
            &\le P_U\Big( cn^\alpha \log n \le \norm{X_n -
              L_s\widehat{X}_{\widetilde{W}_n}} \le
              R_X(n-L_t\wt{R}_{\widetilde{W}_n})+L_s, \max_{j\le n}\{
              j-L_t\widetilde{R}_{\widetilde{W}_j} \}\le cn^\alpha
              \Big)\\
            &\le P_U\Big(  cn^\alpha \log n \le \norm{X_n - L_s\widehat{X}_{\widetilde{W}_n}} \le cR_Xn^\alpha \Big) \xrightarrow[n\to \infty]{} 0 \quad \text{a.s.}
  \end{align*}
  Combining this with \eqref{eq:45} we obtain
  \begin{equation}
    \label{eq:49}
    P_U\Big( \norm{X_n - L_s\widehat{X}_{\widetilde{W}_n}}\ge cn^\alpha\log n\Big) \xrightarrow[n\to \infty]{} 0\quad \text{a.s.}
  \end{equation}
  By equation~\eqref{eq:46}, for any $\varepsilon>0$,
  \begin{equation}
    \label{eq:50}
    P_U\big( \abs{\widetilde{W}_n - n/(L_t\bE[\widetilde{\tau}_2])} \ge n^{1/2 + \varepsilon} \big) \xrightarrow{} 0 \quad \text{a.s.}
  \end{equation}

  By Lemma~\ref{lem:abstract_joint_fluctuations} we have for
  $\delta,\varepsilon>0$
  \begin{equation}
    \bP\Big( \sup_{\abs{k-[\theta n]}  \leq n^\delta}
    \abs{\widehat{X}_k - \widehat{X}_{[\theta n]}} >
    n^{\delta/2+\varepsilon} \Big) \le Cn^{-\gamma},
  \end{equation}
  for some $\gamma>0$ that we can tune arbitrarily large if we have
  $\tilde{\beta}$ large. Therefore, using Borel--Cantelli yields
  \begin{equation}
    \label{eq:51}
    \limsup_{n\to \infty} \sup_{\abs{k-[\theta n]}  \leq n^\delta}
    \frac{\abs{\widehat{X}_k-\widehat{X}_{[\theta
          n]}}}{n^{\delta/2+\varepsilon}} \xrightarrow{} 0 \quad
    \text{a.s.}
  \end{equation}
  We write $X_n/\sqrt{n}$ as
  \begin{equation}
    \label{eq:47}
    \frac{X_n}{\sqrt{n}} = \frac{X_n -
      L_s\widehat{X}_{\widetilde{W}_n}}{\sqrt{n}} +
    \frac{L_s\widehat{X}_{\widetilde{W}_n} -L_s
      \widehat{X}_{[n/(L_t\bE[\widetilde{\tau}_2])]}}{\sqrt{n}} +
    \frac{L_s
      \widehat{X}_{[n/(L_t\bE[\widetilde{\tau}_2])]}}{\sqrt{n/(L_t\bE[\widetilde{\tau}_2])}}\sqrt{1/(L_t\bE[\widetilde{\tau}_2])}
  \end{equation}
  and define the sets
  \begin{align*}
    \widetilde{A}_n &\coloneqq \{ \abs{X_n - L_s\widehat{X}_{\widetilde{W}_n}} \ge n^\alpha \log n \},\\
    \widetilde{B}_n &\coloneqq \{ \abs{\widetilde{W}_n - n/(L_t\bE[\widetilde{\tau}_2])} \ge n^{\varepsilon+1/2} \},\\
    \widetilde{C}_n &\coloneqq \{ \sup_{\abs{k - n/(L_t\bE[\widetilde{\tau}_2])} \le n^{1/2+\varepsilon} } \abs{\widehat{X}_k - \widehat{X}_{[n/(L_t\bE[\widetilde{\tau}_2])]}} > n^{1/2 - \varepsilon} \},\\
    \widetilde{D}_n &\coloneqq \widetilde{A}_n^\compl \cap \widetilde{B}_n^\compl \cap \widetilde{C}_n^\compl.
  \end{align*}
  Denote by $\Phi$ the image measure of $\widetilde{\Phi}$ under $x\to
  L_s x/\sqrt{L_t\bE[\widetilde{\tau}_2]}$, i.e.\ $\Phi(f) =
  \widetilde{\Phi}(f(L_s(L_t\bE[\widetilde{\tau}_2])^{-1/2}\cdot))$.
  It follows
  \begin{align}
    &\abs{E_U[f(X_n/\sqrt{n})] - \Phi(f)}\\
    &\le \abs{E_U[\indset{\widetilde{D}_n}f(X_n/\sqrt{n})] - \Phi(f)} + \norm{f}_\infty E_U[\indset{\widetilde{D}_n^\compl}],
  \end{align}
  where on the set $\widetilde{D}_n$ we use \eqref{eq:47} and get
  \begin{align}
    \begin{split}
      &\abs{E_U[\indset{\widetilde{D}_n}f(X_n/\sqrt{n})] -\Phi(f)}\\
      &\le \abs{E_U[\indset{\widetilde{D}_n}(f(X_n/\sqrt{n}) - f(L_s\widehat{X}_{\widetilde{W}_n}/\sqrt{n}) )]}\\
      &\hspace{2cm}+ \abs{E_U[\indset{\widetilde{D}_n}(f(L_s\widehat{X}_{\widetilde{W}_n}/\sqrt{n}) - f(L_s\widehat{X}_{[n/(L_t\bE[\widetilde{\tau}_2])]}/\sqrt{n}))]}\\
      &\hspace{4cm}+ \abs{E_U[\indset{\widetilde{D}_n}( f(L_s\widehat{X}_{[n/(L_t\bE[\widetilde{\tau}_2])]}/\sqrt{n}))] - \Phi(f)}\\
      &\le CL_f \Big(\frac{n^\alpha\log n}{\sqrt{n}} +
        \frac{n^{1/2-\varepsilon}}{\sqrt{n}}\Big) + \Abs{E_U\bigg[
        f\bigg(
        \frac{L_s\widehat{X}_{[n/(L_t\bE[\widetilde{\tau}_2])]}}{\sqrt{n/(L_t\bE[\widetilde{\tau}_2])}}
        \sqrt{1/(L_t\bE[\widetilde{\tau}_2])}\bigg) \bigg] -\Phi(f) }
      \\
      &\xrightarrow[n\to \infty]{} 0 \qquad \text{a.s.}
    \end{split}
  \end{align}
  This hold by using \eqref{eq:48} and the fact that $\alpha$ can be
  chosen close to $0$. Additionally we have
  \begin{equation}
    E_U[\indset{\widetilde{D}_n^\compl}] \le P_U(\widetilde{A}_n) +
    P_U(\widetilde{B}_n) + P_U(\widetilde{C}_n) \xrightarrow[n\to
    \infty]{} 0 \quad \text{a.s.}
  \end{equation}
  by \eqref{eq:49}, \eqref{eq:50} and \eqref{eq:51}. Therefore we have
  proven the theorem for bounded Lipschitz function, which, by the
  Portmanteau-theorem, implies weak convergence.
\end{proof}

\subsection{Dimension $d=1$}
For dimension $d=1$ we need to adapt some more results from
Section~\ref{sec:dimension1}. We transfer the notation from
Section~\ref{sec:dimension1} to the coarse-grained random walk by
setting
\begin{align*}
  \bar{X}_i \coloneqq \wt{X}_{\wt{R}_i},\qquad
  \bar{X}'_i \coloneqq \wt{X}'_{\wt{R}_i}.
\end{align*}
We write all objects introduced in Section~\ref{sec:dimension1} with a
bar meaning the same object but exchanging $\widehat{X}$ with
$\bar{X}$ and $\widehat{X}'$ with $\bar{X}'$, furthermore for
notational convenience we write $\bar{R}_i = L_t\widetilde{R}_i$ and
define
\begin{align*}
  \bar{\phi}_1(x,x') \coloneqq \sum_{y,y'} (y-x)
  \bP^\joint\big(X_{\bar{R}_1}=y, X'_{\bar{R}_1}=y'\,\vert\, X_0=x,
  X'_0=x'\big).
\end{align*}
Note that $\bar{\phi}_2,\bar{\phi}_{11},\bar{\phi}_{22}$ and
$\bar{\phi}_{12}$ are derived analogously from
$\phi_2,\phi_{11},\phi_{22}$ and $\phi_{12}$ respectively. We set
\begin{align*}
  \bar{A}^{(1)}_n &\coloneqq \sum_{j=0}^{n-1} \bar{\phi}_1(X_{\bar{R}_j},X'_{\bar{R}_j}), \quad \bar{A}^{(2)}_n \coloneqq \sum_{j=0}^{n-1} \bar{\phi}_2(X_{\bar{R}_j},X'_{\bar{R}_j})\\
  \bar{A}^{(11)}_n &\coloneqq \sum_{j=0}^{n-1} \bar{\phi}_{11}(X_{\bar{R}_j},X'_{\bar{R}_j}), \quad \bar{A}^{(22)}_n \coloneqq \sum_{j=0}^{n-1} \bar{\phi}_{22}(X_{\bar{R}_j},X'_{\bar{R}_j})\\
  \bar{A}^{(12)}_n &\coloneqq \sum_{j=0}^{n-1} \bar{\phi}_{12}(X_{\bar{R}_j},X'_{\bar{R}_j})
\end{align*}
and define
\begin{align*}
  \bar{M}_n \coloneqq X_{\bar{R}_n} - \bar{A}^{(1)}_n, \quad \bar{M}'_n \coloneqq X'_{\bar{R}_n} - \bar{A}^{(2)}_n.
\end{align*}
Furthermore, let, similar to \eqref{eq:18},
\begin{align*}
  B_n\coloneqq \#\{0\le j \le n\colon \abs{X_{\bar{R}_j}-X'_{\bar{R}_j}} \le 2L_sn^a  \}
\end{align*}
be the number of times a black box is visited from the point of view
of the original random walks $X$ and $X'$. Setting 
\begin{equation*}
  \bar{B}_n \coloneqq \#\{0\le j \le n\colon \abs{\bar{X}_j-\bar{X}'_j} \le n^a  \}
\end{equation*}
as the number of times a black box is visited by the coarse-grained
walks, we see that $\bar{B}_n\le B_n$, which is the reason for the
choice of the different bound in the definition of $B_n$.
\begin{lemma}
  \label{lem:abstract bound_for_A1_and_A2}
  \begin{enumerate}[1.]
  \item There exist
    $0\leq \delta_{\bar{B}}\leq 1/2, c_{\bar{B}} < \infty$ such that
    \begin{align}
      \label{eq:abstract bound_R_n_moment}
      \bE^\joint[B_n^{3/2}] \leq c_{\bar{B}}n^{1+\delta_{\bar{B}}}\quad \text{for all } n.
    \end{align}
  \item There exist $\delta_{\bar{C}}>0, c_{\bar{C}}< \infty$ such
    that
    \begin{align}
      \label{eq:abstract bound_for_A1_and_A2}
      \bE^\joint_{0,0}\left[ \frac{\abs{\bar{A}_n^{(1)}}}{\sqrt{n}} \right], \bE^\joint_{0,0}\left[
      \frac{\abs{\bar{A}_n^{(2)}}}{\sqrt{n}} \right] \leq
      \frac{c_{\bar{C}}}{n^{\delta_{\bar{C}}}} \quad \text{for all }n.
    \end{align}
  \end{enumerate}
\end{lemma}

\begin{proof}
  The proof for \eqref{eq:bound_R_n_moment} in
  Lemma~\ref{lem:bound_for_A1_and_A2} uses very general methods in
  combination with the separation lemma for $d=1$ and
  Proposition~\ref{prop:TVdistance-joint-ind-1step}, note that the
  difference now only lies in the fact that we use different
  regeneration times $\bar{R}_i$ for which we have already shown all
  necessary bounds. Therefore this can be adapted in a straightforward
  manner for \eqref{eq:abstract bound_R_n_moment}.

  \smallskip
  \noindent
  For the proof of \eqref{eq:bound_for_A1_and_A2} we use a bound on
  mixingales and the separation lemma for $d=1$. Both hold for the
  abstract setting and therefore \eqref{eq:abstract
    bound_for_A1_and_A2} holds as well.
\end{proof}
We set $\bar{\sigma}^2\coloneqq \sum_{y,y'} y^2 \bP^\indi(\bar{X}_1
=y,\bar{X}'_1=y' \,\vert\, \bar{X}_0=0,\bar{X}'_0=0)$
\begin{lemma}
  There \label{lem:abstract_key_lemma_for_d=1} exist
  $C>0,\wt{b}\in(0,1/4)$ such that for all bounded Lipschitz
  continuous $f:\R^2\to\R$
  \begin{align}
    \Big\vert \bE^\joint_{0,0}\Big[ f\bigg(
    \frac{\bar{X}_n}{\bar{\sigma}\sqrt{n}},\frac{\bar{X}'_n}{\bar{\sigma}\sqrt{n}}
    \bigg) \Big] - \bE\big[ f(Z)\big] \Big\vert \leq
    L_f\frac{C}{n^{\wt{b}}}
  \end{align}
  where $Z$ is two-dimensional standard normal and $L_f$ the Lipschitz
  constant of $f$.
\end{lemma}
\begin{proof}
  The idea is to write $\bar{X}_i = L_sX_{L_t\wt{R}_i} + \wt{Y}_i$ and
  $\bar{X}'=L_sX'_{L_t\wt{R}_i} + \wt{Y}'_i$ and use the martingale
  decomposition for $X$ and $X'$. To illustrate we give the starting
  argument in more detail. Since $f$ is Lipschitz, we obtain
  \begin{multline*}
    \Big\vert \bE^\joint_{0,0}\Big[ f\bigg(
    \frac{\bar{X}_n}{\bar{\sigma}\sqrt{n}},\frac{\bar{X}'_n}{\bar{\sigma}\sqrt{n}}
    \bigg) \Big] - \bE\big[ f(Z)\big] \Big\vert\\
    \le \Big\vert \bE^\joint_{0,0}\Big[ f\bigg(
    \frac{\bar{X}_n}{\bar{\sigma}\sqrt{n}},\frac{\bar{X}'_n}{\bar{\sigma}\sqrt{n}}
    \bigg) - f\bigg( \frac{L_sX_{L_t\wt{R}_n}}{\bar{\sigma}\sqrt{n}},
    \frac{L_sX'_{L_t\wt{R}_n}}{\bar{\sigma}\sqrt{n}} \bigg)\Big] \Big\vert\\
    + \Big\vert \bE^\joint_{0,0}\Big[ f\bigg( \frac{L_sX_{L_t\wt{R}_n}}{\bar{\sigma}\sqrt{n}}, \frac{L_sX'_{L_t\wt{R}_n}}{\bar{\sigma}\sqrt{n}} \bigg)\Big] - \bE\big[ f(Z)\big]\Big\vert \\
    \le \frac{C L_f}{\sqrt{n}}
    + \Big\vert \bE^\joint_{0,0}\Big[ f\bigg( \frac{L_sX_{L_t\wt{R}_n}}{\bar{\sigma}\sqrt{n}}, \frac{L_sX'_{L_t\wt{R}_n}}{\bar{\sigma}\sqrt{n}} \bigg)\Big] - \bE\big[ f(Z)\big]\Big\vert
  \end{multline*}
  Now we can then use the same arguments for $(\bar{M}_n,\bar{M}'_n)$
  as around \eqref{eq:covariance matrix}. The rest of the proof then
  follows by the same arguments as in the proof of Lemma~\ref{lem:key
    lemma for d=1}, noting that $\abs{\wt{Y}_i},\abs{\wt{Y}'_i}\le
  L_s$ and therefore the additional term from the above equation
  vanishes in the limit.
\end{proof}

Therefore the quenched CLT holds for the coarse-grained random walk
$\bar{X}$ in $d\ge 1$.

\appendix
\section{Detailed proofs of auxiliary results from Section~\ref{sec:preparations}}
\label{sec:proofs}
\subsection{Proof of Lemma~\ref{th:lemma2.13Analog}}
\label{sect:proof.th:lemma2.13Analog}
The analogous result for the single cone shell is Lemma~2.13 in
\cite{BirknerCernyDepperschmidt2016}. We will use similar arguments
but, as already pointed out before, an additional complication arises
from the overlapping parts of the cones. Throughout the proof for
$r>0$ and $x \in \mathbb{Z}^d$ we denote by $B_r(x)$ the closed
$\ell^2$ ball of radius $r$ around $x$. We will write $\cs(x_\base)$
and $\cs(x_\base')$ as an abbreviation for
$\cs(x_\base;b_\inn,b_\out, s_\inn, s_\out,\infty)$ and
$\cs(x_\base';b_\inn,b_\out,s_\inn,s_\out,\infty)$ respectively.

We split the proof in two cases according to $d=1$ or $d \ge 2$. We
will reuse some arguments from the case $d=1$ for higher dimensions
and thus begin with the case $d=1$.

\medskip
\noindent
\emph{Case $d=1$}. Without loss of generality we may assume
$x_\base < x_\base'$. We will focus on the differences to the version
of this lemma with only one cone. For that we distinguish according to
the distance of the bases of the cones.

\smallskip

First let $\norm{x_\base' -x_\base}_2 \leq 2 b_\out$. Since, in this
case, the bases of the two outer cones already overlap, it is
impossible for any path $\gamma $ to cross
$\dcs(x_\base,x_\base';\infty)$ from between $x_\base$ and $x_\base'$
without hitting one of the bases. (It easy to see how the picture in
Figure~\ref{fig:coneshells_in_d=1} changes in this case.) In this case
we can use the same arguments as in
\cite{BirknerCernyDepperschmidt2016}, since we can combine the two
cones to a single larger cone with the cone shell being
\begin{multline*}
        \dcs(x_\base,x'_\base) \coloneqq\left\{ (x,n) \in\bZ\times\bZ_+:
        x_\base - b_\out - n
        s_\out \leq x \leq x_\base - b_\inn - n s_\inn \right\}\\
        \cup \left\{ (x,n) \in\bZ\times\bZ_+: x_\base' + b_\inn + n s_\inn
        \leq x \leq x_\base' + b_\out + n s_\out \right\}.
\end{multline*}

\smallskip
\noindent
Now let $\norm{x_\base' -x_\base}_2 > 2 b_\out$. In particular, the
two cones do not overlap at time $t=0$. The two cone shells are each
made up of two wedges
\begin{align*}
        c^1_\ell
        & \coloneqq \left\{ (x,n) \in\bZ\times\bZ_+ : x_\base - b_\out -
        n s_\out \leq x \leq x_\base - b_\inn - n s_\inn \right\},\\
        c^1_r
        & \coloneqq \left\{ (x,n) \in\bZ\times\bZ_+ : x_\base + b_\inn +
        n s_\inn \leq x \leq x_\base + b_\out + n s_\out \right\},\\
        c^2_\ell
        & \coloneqq \left\{ (x,n) \in\bZ\times\bZ_+ : x_\base' - b_\out
        - n s_\out \leq x \leq x_\base' - b_\inn - n s_\inn \right\},\\
        c^2_r
        & \coloneqq \left\{ (x,n) \in\bZ\times\bZ_+ :  x_\base' + b_\inn
        + n s_\inn \leq x \leq x_\base' + b_\out + n s_\out \right\}.
\end{align*}
We build the cone shell for the double cone using the above wedges. We
have to isolate the inner cones using the double cone shell and need
to make sure, that the cone shell doesn't evolve into any inner cone.
Obviously this isn't a problem for $c^1_\ell$ and $c^2_r$ since these
two wedges both evolve away from the other cone. It remains to find a
suitable hight at which we cut and merge $c^1_r$ and $c^2_\ell$ to
avoid their propagation into the inner cones. Given the parameters of
the cones let
\begin{align}
  \label{eq:defn_meeting_time_inner_outer_cone}
  t^* \coloneqq \frac{(x_\base'-x_\base)-b_\out- b_\inn}{s_\out+s_\inn}
\end{align}
be the time at which the inner cones meet the respective outer cones
of the other double cone. Then the cone shell is the following set
\begin{align*} 
  \dcs(x_\base,x'_\base) \coloneqq c^1_\ell \cup c^2_r \cup \bigl(c^1_r \cap \bZ\times[0,\lceil t^*\rceil]
  \bigr) \cup \bigl( c^2_\ell \cap \bZ \times [0,\lceil t^*\rceil] \bigr).
\end{align*}
Note that we slightly deviate here from the definition given above in
equation~\eqref{eq:defn_double_cone_shell}. The cone shell given here
is a bit smaller and the arguments therefore hold as well for the a
cone shell defined by \eqref{eq:defn_double_cone_shell}. A sketch of
this cone shell can be seen in Figure~\ref{fig:coneshells_in_d=1}.

The cone shell is thus made up of four wedges and in Claim~2.15 from
\cite{BirknerCernyDepperschmidt2016} it was shown that for any
$\varepsilon'>0$, by tuning the parameters correctly, the contact
process $\eta^{c^1_\ell}$ restricted to the wedge $c^1_\ell$ survives
with probability at least $1-\varepsilon'$. The same holds for the
other wedges and therefore $\eta^{\dcs}$ survives in every wedge with
probability at least $1-4\varepsilon'$. For the outer wedges
$c^1_\ell$ and $c^2_r$ every path crossing the cone shell has to hit
the contact process if it survives. For the inner wedges we have to
argue a bit more carefully since we cut them at a certain height. It
is theoretically possible that the contact process survives up until
that point in time in both wedges but evolves in such a way that the
clusters of points visited (infected) by the contact process do not
intersect. Then there exists a path which crosses the cone shell but
doesn't hit the contact process $\eta^\dcs$.

\noindent
Theorem~2 from \cite{CoxMaricSchinazi2010} tells us that, if we
condition on the event
\begin{align*}
  B \coloneqq \bigl\{ \eta^{c^1_r}_t \neq 0, \forall t \geq 0 \bigr\} \cap
  \bigl\{\eta^{c^2_\ell}_t \neq 0, \forall t \geq 0 \bigr\},
\end{align*}
then the rightmost particle of the contact process is close to the
right border of the wedge. It follows, for every $\wt{\varepsilon} >0$
there exists $T >0$ such that for every $n \geq T$
\begin{align}
  \label{eq:41}
  \mathbb{P}\Bigl( \max\{x : \eta^{c^1_r}_n(x) =1  \} \leq
  \frac{1}{2}(b_\inn + b_\out + t(s_\inn+s_\out)) \,\Big\vert\,B \Bigr) < 1
  - \wt{\varepsilon}.
\end{align}
An analogous bound holds for the left most particle in the wedge
$c^2_\ell$. So $\eta^{c^1_r}$ and $\eta^{c^2_\ell}$, if they both
survive, will eventually meet. Note that we in fact only need the
contact process to survive until we cut the inner wedges which happens
at the time $t^*$ from \eqref{eq:defn_meeting_time_inner_outer_cone}.
The probability of this joint survival event is greater than
$1-\varepsilon'$.

Recall that $\norm{x_\base-x'_\base}> 2b_\out$, therefore by
\eqref{eq:defn_meeting_time_inner_outer_cone} we can increase $t^*$ by
increasing the base $b_\out$ of the outer cones which results in
\eqref{eq:41}. Furthermore, conditioned on the event $B$, the contact
processes $\eta^{c^1_r}$ and $\eta^{c^2_\ell}$ will survive and meet
each other with probability at least $1-2\wt\varepsilon$. Thus any
path crossing the cone shell has to hit the contact process
$\eta^\dcs$ at some point with probability at least $1-\varepsilon$
for some $\varepsilon>0$. So using the same arguments as in
\cite{BirknerCernyDepperschmidt2016} we obtain the claim for $d=1$.

\medskip
\noindent
\emph{Case $d>1$}. Recall the definition of $M^\dcs_n$ in
\eqref{eq:defn_middle_tube_dcs}. Note that for every path
\begin{align*}
        \gamma=((x_m,m),(x_{m+1},m+1),\dots,(x_{m'},m'))
\end{align*}
crossing $\dcs(x_\base,x_\base')$ there exists $i \in [m,m']$ with
$x_i \in M^\dcs_i$, i.e.\ every path will hit at least one point in
$\mathcal{M}$. Without loss of generality let $n$ be the time at which
the path hits $\mathcal{M}$. By Lemma~2.9 from
\cite{BirknerCernyDepperschmidt2016} it follows that with high
probability after $n$ steps the contact process started from a single
site will coincide with the contact process started from the upper
invariant measure in a ball of radius $n s_\coupl$ around the starting
site.

Let $\rho>0$ be a small constant such that for every $x\in M^\dcs_n$
we have
\begin{align*}
        B_{n \rho s_\coupl }(x)\times\{ \lfloor (1-\rho)n \rfloor \}\subset
        \dcs(x_\base,x'_\base).
\end{align*}
To get into the setting of \cite{BirknerCernyDepperschmidt2016} we
need to find unit vectors $\hat{v}^i \in \R^d, 1 \leq i \leq \hat{N}$,
with $\hat{N}$ sufficiently large so that we can ``cover'' the middle
tubes $M^\dcs_n$ in the following sense: For every $x\in M^\dcs_n$
there exists an $i \leq \hat{N}$ such that for $\rho$ small there is
$\delta=\delta(\rho) >0$ with the property
\begin{align}
  \label{eq:vectorProperty}
  \begin{split}
    & \text{the length of the intersection of the half-line } \{
      t\hat{v}^i : t \geq 0 \} \text{ with the (real) ball }\\
    & \{ y \in \R^d : \norm{x-y}_2 \leq n \rho s_{\coupl} \} \text{ is
      at least } \delta n.
  \end{split}
\end{align}
Following the idea in \cite{BirknerCernyDepperschmidt2016} we can
choose such vectors for the single cone shells $\cs(x_\base)$ and
$\cs(x_\base')$ respectively. The union $\{v^{i}:i=1,\dots, N \}\cup
\{v'^{,i}:i=1,\dots, N' \}$ of those two sets of vectors has the above
property for the double cone shell $\dcs(x_\base,x_\base')$. Now we
approximate the half-line $\{ tv^i: t\geq 0 \}$ with a self-avoiding
nearest neighbour path in $\bZ^d$ $\alpha^i = (\alpha^i(j))_{j \in
  \mathbb{N}}$ given by
\begin{enumerate}[(a)]
\item $\alpha^i_0 = x_\base$,
\item $\alpha^i$ makes steps only in direction of $v^i$,
\item $\alpha^i$ stays close to $tv^i$, that is $\{
  \alpha^i_j:j\in \N \} \subset \{ tv^i + z: t\geq 0,z\in \R^d,
  \norm{z} \leq d \}$.
\end{enumerate}
Analogously we define $\alpha'^{,i}$ in direction of $v'^{,i}$ for
$1\leq i \leq N'$. Next we define contact processes $\eta^{(i)}$
restricted to the set $\mathcal{W}^i \coloneqq (\alpha^i \times \bZ_+)
\cap \dcs(x_\base,x'_\base)$ and analogous $\eta'^{,(i)}$ restricted
to $\mathcal{W}'^{,i} \coloneqq (\alpha'^{,i} \times \bZ_+) \cap
\dcs(x_\base,x'_\base)$ and started from $\indset{\mathcal{W}^i
  \cap(\bZ^d\times \{0\})}$ and $\indset{\mathcal{W}'^{,i}\cap (\bZ^d
  \times\{0\})}$ respectively. By the definition of
$\dcs(x_\base,x'_\base)$ we avoid the problem of having any of the
$\eta^{(i)}$ or $\eta'^{,(i)}$ evolving into the inner double cone.

Let $N^*=N+N'$ be the total number of vectors above and
$\{\hat{\eta}^{(i)}:i\le N^* \}$ the collection of contact processes
defined. Define
$S\coloneqq\{ \hat{\eta}^{(i)} \text{ survives for every }i\leq N^*
\}$. Given $\varepsilon>0$, choose $\varepsilon^*$ so that
$(1-\varepsilon^*)^{N^*}\ge 1-\varepsilon/2$. Given that we get
\begin{align}
        \bP(S) \geq 1-\varepsilon/2.
\end{align}
The rest follows exactly as in the proof of Lemma~2.13 from
\cite{BirknerCernyDepperschmidt2016}.

\subsection{Proof of Lemma~\ref{lemma:expTailsSigmaJoint}}
\label{proof:lemma:expTailsSigmaJoint}
In order to verify \eqref{eq:stochDomSigmaJoint} note that at stopping
times $\sigma^{\simi}_i$ all the reasons for zeros along the path of
the random walks have been explored and are contained in
$\mathcal{F}^\simi_{\sigma^{\simi}_i}$. On the other hand the
knowledge of ones of $\eta$ enforces the existence of certain open
paths for the $\omega$'s. Thus, \eqref{eq:stochDomSigmaJoint} follows
from the FKG inequality for the $\omega$'s.

For the proof of \eqref{eq:expTailsSigmaJoint} we consider first the
case $i=0$. We write
$\widehat{R}_{\kappa} \coloneqq (2R_{\kappa} +1)^d$ and
$\widehat{R}_\loc \coloneqq (2R_\loc +1)^d$ for the number of
elements of a ball with radius $R_\kappa$ or $R_\loc$ respectively.
The event $\{ \sigma^{\simi}_1 > n \}$ enforces that there are
space-time points $(y_j,-j)$ with $\eta_{-j}(y_j)=0$ for
$j=0,1,\dots,n$ in the $R_\loc$-vicinity of the paths of the two
random walks $X$ and $X'$. Here, it is enough to have such a point
in only one of the two $R_\loc$-vicinities. By Lemma~2.11 from
\cite{BirknerCernyDepperschmidt2016} the probability that
$\eta_{-j}(y_j)=0$ for a fixed choice of $y_0,\dots,y_n$ is bounded
from above by $\varepsilon(p)^{n+1}$.

Now we prove an upper bound for the number of relevant vectors
$(y_0,y_1,\dots, y_n) \in (\bZ^d)^{n+1}$. Each of the two random
walks has $\widehat{R}_\kappa^n$ possible $n$-step paths. Thus, the
two walks together have at most $\widehat{R}_\kappa^{2n}$ possible
$n$-step paths. Assume that at exactly $k$ times,
$0 \leq m_1 < m_2 < \dots < m_k \leq n$, for sites
$(y_{m_i}, -m_i) \in B_{R_\loc}(X_{m_i})\times \{ -m_i\} \cup
B_{R_\loc}(X_{m_i}')\times\{ -m_i \}$ we have
$\eta_{-m_i} (y_{m_i})=0$ and hence the corresponding
``determining'' triangle $D(y_{m_i},-m_i)$ is not empty.

Set $m_{k+1}=n$. Then the height of $D(y_{m_i},-m_i)$ is bounded
from below by $m_{i+1}-m_i$, because the triangles have to overlap
until time $n$ to enforce $\sigma^{\simi}_1>n$. For a fixed $n$-step
path of $(X,X')$ and fixed $m_1 < m_2 < \dots < m_k $, there are at
most $(2\widehat{R}_\loc)^k$ many choices for the $y_{m_i}$,
$i=1,\dots,k$. Here we have $2\widehat{R}_\loc$ choices for every
$y_{m_i}$ since those points are in
$B_{R_\loc}(X_{m_i}) \cup B_{R_\loc}(X_{m_i}')$. And inside
$D(y_{m_i},-m_i)$ we have at most
$\widehat{R}_\kappa^{m_{i+1}-m_i-1}$ choices to pick
$y_{m_i+1},y_{m_i+2},\dots, y_{m_{i+1}-1}$ (start with $y_{m_i}$,
then follow the longest open path which is not connected to
$\bZ^d \times \{ -\infty \}$, these sites are necessarily zeros of
$\eta$). Thus, there are at most
\begin{align*}
        \widehat{R}_\kappa^{2n} \sum_{k=1}^n \sum_{m_1 < m_2 <\dots <m_k
                \leq m_{k+1}=n}(2\widehat{R}_\loc)^k
        & \prod_{i=1}^k \widehat{R}_\kappa^{m_{i+1}-m_i-1}\\
        & = \widehat{R}_\kappa^{2n}\sum_{k=1}^n \binom{n}{k}
        (2\widehat{R}_\loc)^k\widehat{R}_\kappa^{n-k} \leq
        \widehat{R}_\kappa^{2n}(2\widehat{R}_\loc+\widehat{R}_ \kappa)^n
\end{align*}
possible choices of $(y_0,y_1,\dots,y_n)$ and hence we have
\begin{align*}
        \mathbb{P}(\sigma^{\simi}_1 > n) \leq
        (\widehat{R}_\kappa^{2}(2\widehat{R}_\loc+\widehat{R}_\kappa)\varepsilon(p))^n.
\end{align*}
The right hand side decays exponentially when $p$ is close to $1$.
The general case $i>0$ in \eqref{eq:expTailsSigmaJoint} follows by
employing \eqref{eq:stochDomSigmaJoint} and the argument for $i=0$.

\subsection{Proof of Lemma~\ref{lemma:expTailsTau}}
\label{sec:proof:lemma:expTailsTau}
We start with the proof of \eqref{eq:expTailsTau1}. We have, abbreviating $\bP_{x_0,x'_0}=\bP$
\begin{align*}
        \Pr& (\wt{\tau}^{\simi}_1(t) \geq n \,\vert \,
        \mathcal{F}^{\simi}_{\sigma^{\simi}_i})\\
        & = \Pr(\sigma^{\simi}_{i+1} \geq t \vee (n+\sigma^{\simi}_i) \,
        \vert \, \mathcal{F}^{\simi}_{\sigma^{\simi}_i}) + \sum_{j
                >i} \sum_{\ell = \sigma^{\simi}_i+1}^{t-1}
        \Pr(\sigma^{\simi}_j=\ell,\sigma^{\simi}_{j+1}\geq t \vee
        (\ell+n)\, \vert \, \mathcal{F}^{\simi}_{\sigma^{\simi}_i} )\\
        & \leq Ce^{-cn} + \sum_{\ell=\sigma^{\simi}_i +1}^{t-1}
        Ce^{-c((t-\ell)\vee n)}\Pr(\exists j>i : \sigma^{\simi}_j = \ell \,
        \vert \, \mathcal{F}^{\simi}_{\sigma^{\simi}_i})\\
        & \leq Ce^{-cn} + \ind{\sigma^{\simi}_i \leq t-n-2}
        \sum_{\ell=\sigma^{\simi}_i +1 }^{t-n-1} C e^{-c(t-\ell)} +
        \ind{n+1 \leq t} \sum_{\ell=t-n}^{t-1} Ce^{-cn}\\
        & \leq C(1 + \frac{e^{-c}}{1-e^{-c}} +n)e^{-cn}
\end{align*}
where we used Lemma~\ref{lemma:expTailsSigmaJoint} and
\begin{align*}
        \Pr\bigl(\sigma^{\simi}_j = \ell, \sigma^{\simi}_{j+1} \geq t \vee (\ell+n)
        \, \vert \, \mathcal{F}^{\simi}_{\sigma^{\simi}_i}\bigr)
        = \mathbb{E}\bigl[\ind{\sigma^{\simi}_j=\ell}\Pr(\sigma^{\simi}_{j+1}
        -\sigma^{\simi}_j \geq (t-\ell)\vee n \, \vert \,
        \mathcal{F}^{\simi}_{\sigma^{\simi}_j})\vert
        \mathcal{F}^{\simi}_{\sigma^{\simi}_i} \bigr]
\end{align*}
in the first inequality.

\medskip To prove \eqref{eq:expTailsTaum} we assume
$\sigma^{\simi}_i \leq t-n-m-1$ since otherwise the conditional
probability appearing in that display equals $0$. We have
\begin{align}
        \label{eq:29}
        \begin{split}
                \Pr& (R_t \geq i+m,\wt{\tau}^{\simi}_m(t) \geq n\, \vert \,
                \mathcal{F}^{\simi}_{\sigma^{\simi}_i}) \\
                & = \Pr(R_t = i+m,\wt{\tau}^{\simi}_m(t) \geq n\, \vert \,
                \mathcal{F}^{\simi}_{\sigma^{\simi}_i}) +\Pr(R_t >
                i+m,\wt{\tau}^{\simi}_m(t) \geq n\, \vert \,
                \mathcal{F}^{\simi}_{\sigma^{\simi}_i})\\
                & = \Pr(\sigma^{\simi}_{i+1} - \sigma^{\simi}_i \geq n, R_t=i+m \,
                \vert \, \mathcal{F}^{\simi}_{\sigma^{\simi}_i}) + \Pr(R_t >
                i+m,\wt{\tau}^{\simi}_m(t) \geq n\, \vert \,
                \mathcal{F}^{\simi}_{\sigma^{\simi}_i})\\
                & \leq Ce^{-cn} + \sum_{j>i} \sum_{k=\sigma^{\simi}_i+1}^{t-m-n}
                \sum_{\ell=k+n}^{t-m+1} \Pr(\sigma^{\simi}_j =
                k,\sigma^{\simi}_{j+1}=l,\sigma^{\simi}_{j+m-1}<t,
                \sigma^{\simi}_{j+m}\geq t \, \vert \,
                \mathcal{F}^{\simi}_{\sigma^{\simi}_i})\\
                & \leq Ce^{-cn} + \sum_{j>i} \sum_{k=\sigma^{\simi}_i
                        +1}^{t-m-n}\sum_{\ell=k+n}^{t-m+1}
                \Pr(\sigma^{\simi}_j=k,\sigma^{\simi}_{j+1}=l\, \vert \,
                \mathcal{F}^{\simi}_{\sigma^{\simi}_i}) \times
                (m-1)Ce^{-c(t-\ell)/(m-1)}\\
                & \leq Ce^{-cn} + C(m-1) \sum_{k=\sigma^{\simi}_i+1}^{t-m-n}
                \sum_{\ell=k+n}^{t-m+1}
                e^{-c(t-\ell)/(m-1)}\sum_{j>i}\Pr(\sigma^{\simi}_j=k \, \vert \,
                \mathcal{F}^{\simi}_{\sigma^{\simi}_i})\times Ce^{-c(\ell-k)}\\
                & \leq Ce^{-cn} + C^2(m-1)
                \sum_{k=\sigma^{\simi}_i+1}^{t-m-n}e^{ck-ct/(m-1)}
                \sum_{\ell=k+n}^{t-m+1} \exp\Bigl(-c\frac{m-2}{m-1}l\Bigr)
        \end{split}
\end{align}
where we used in the second inequality that
\begin{align*}
        \{ \sigma^{\simi}_{j+1}=\ell, \sigma^{\simi}_{j+m-1} < t,
        \sigma^{\simi}_{j+m} >t \} \subset \bigcup_{r=j+2}^{j+m} \{
        \sigma^{\simi}_r - \sigma^{\simi}_{r-1} \geq \frac{t-\ell}{m-1} \}
\end{align*}
together with Lemma~\ref{lemma:expTailsSigmaJoint}. For $m=2$ the
inequality proven in \eqref{eq:29} yields
\begin{align*}
        \Pr(R_t \geq i+2, \wt{\tau}^{\simi}_m(t) \geq n \, \vert \,
        \mathcal{F}^{\simi}_{\sigma^{\simi}_i}) \leq Ce^{-cn} + C^2
        \sum_{k=\sigma^{\simi}_i +1}^{t-n}(t-k-n)e^{-c(t-k)} \leq C'
        e^{-cn}\sum_{\ell=0}^{\infty} \ell e^{-c\ell}
\end{align*}
whereas for $m>2$ we obtain
\begin{align*}
        \Pr(R_t \geq i+m,
        & \wt{\tau}^{\simi}_m(t) \geq n \, \vert \,
        \mathcal{F}^{\simi}_{\sigma^{\simi}_i}) \\
        & \leq Ce^{-cn} + C^2(m-1) \sum_{k=\sigma^{\simi}_i+1}^{t-m-n}e^{ck-ct/(m-1)}
        \frac{\exp(-c\frac{m-2}{m-1}(k+n))}{1-\exp(-c\frac{m-2}{m-1})}\\
        & = Ce^{-cn} + C^2(m-1)\frac{\exp(-c\frac{m-2}{m-1}n
                -c\frac{t}{m-1})}{1-\exp(-c\frac{m-2}{m-1})}
        \sum_{k=\sigma^{\simi}_i +1}^{t-m-n}e^{ck/(m-1)}\\
        & \leq Ce^{-cn} + \frac{C^2(m-1)}{1-\exp(-c\frac{m-2}{m-1})}
        \exp(-c\frac{m-2}{m-1}n -c\frac{t}{m-1})
        \frac{e^{c\frac{t-n}{m-1}}}{e^{c/(m-1)}-1}\\
        & = Ce^{-cn} + \frac{C^2(m-1)}{1-\exp(-c\frac{m-2}{m-1})}
        \frac{e^{-cn}}{e^{c/(m-1)}-1} \leq C'm^2e^{-cn}.
\end{align*}
where we note that $1-\exp(-c\frac{m-2}{m-1}) \ge 1 - \exp(c/2)$ and
$\exp(c/(m-1))-1 = \frac{c}{m+1}\sum_{k=0}^\infty
\frac{(c/(m-1))^k}{(k+1)!} > \frac{c}{m-1}$.

\subsection{Proof of Lemma~\ref{lem:apriori speed conditioned on F}}
\label{sec:proof:lem:apriori speed conditioned on F}
By Lemma~A.1 of \cite{BirknerCernyDepperschmidt2016}, we may assume
that $T=\sigma^{\simi}_\ell$ for some $\ell \in \N$. For
\eqref{eq:aprioriAlongStoppingTimesX} we combine
Corollary~\ref{corollary:dryPoints} with the proof of Lemma~2.16 of
\cite{BirknerCernyDepperschmidt2016}. Since the proof is analogous
to the proof of Lemma~2.20 in \cite{BirknerCernyDepperschmidt2016},
we briefly show the part of the proof that differs. Let $\Gamma_n$
be the set of all $n$-step paths $\gamma$ on $\bZ^d$ starting from
$\gamma_0 = X_T$ with the restriction
$\norm{\gamma_i - \gamma_{i-1}} \leq R_\kappa$, $i=1,\dots,n$, where
$R_\kappa$ is the range of the kernels $\kappa_n$ from
Assumption~\ref{ass:finite-range}. For $\gamma \in \Gamma_{k-T}$ and
$T \leq i_1 < i_2 < \dots < i_m \leq k$ we define
\begin{align*}
        D^{\gamma}_{i_1,\dots,i_m}
        & \coloneqq \{ \eta_{-\ell}(\gamma_\ell) =0 \text{ for all }
        \ell \in \{i_1,\dots,i_m \} \},\\
        W^{\gamma}_{i_1,\dots,i_m}
        & \coloneqq \{ \eta_{-\ell}(\gamma_\ell) =1 \text{ for all } \ell \in \{
        T,\dots,k \}\setminus \{ i_1,\dots,i_m \} \}.
\end{align*}
Let $H_{n} \coloneqq \# \{ T \leq i \leq n : \eta_{-i}(X_i)=0 \}$ be
the number of sites with $\eta=0$ the walker visits from time $T$ up to
time $n$ and set
$K\coloneqq \max_{x \in \bZ^d} \{ \kappa_{\rf}(x) \} +
\varepsilon_\rf$. We abbreviate again $\bP_{x_0,x'_0}=\bP$, it follows
\begin{align}
        \label{eq:33}
        \Pr(H_{k} = m \, \vert \, \mathcal{F}^{\simi}_T)
        = \sum_{T\leq i_1 < \dots < i_m \leq k} \sum_{\gamma \in
                \Gamma_{k-T}} \Pr((X_T,\dots,X_k)=\gamma,
        W^{\gamma}_{i_1,\dots,i_m},D^{\gamma}_{i_1,\dots,i_m}\, \vert \,
        \mathcal{F}^{\simi}_T).
\end{align}
Note that  if $i_m<k$ we get
\begin{align*}
        \Pr&((X_T,\dots,X_k)=\gamma,
        W^{\gamma}_{i_1,\dots,i_m},D^{\gamma}_{i_1,\dots,i_m}\, \vert
        \, \mathcal{F}^{\simi}_T)\\
        &= \mathbb{E}\Big[\ind{(X_T,\dots,X_{i_m-1})=(\gamma_T,
                \dots,\gamma_{i_m-1})}
        \indset{W^{\gamma\vert_{i_m-1}}_{i_1,\dots,i_{m-1}}}\\
        & \qquad \mathbb{E}[\ind{(X_{i_m},\dots,X_k) =
                (\gamma_{i_m},\dots,\gamma_k)}\ind{\eta_{-\ell}(\gamma_\ell)=1,
                i_m+1\leq l \leq k}  \indset{D^{\gamma}_{i_1,\dots,i_m}}\,
        \vert \, \mathcal{F}^{\simi}_{i_m-1}]\, \vert \,
        \mathcal{F}^{\simi}_T \Big]
\end{align*}
where $\gamma\vert_t = (\gamma_1,\dots,\gamma_t)$ is the path
$\gamma$ restricted to the first $t$ components. By conditioning on
$\mathcal{F}_{k-1},\dots,\mathcal{F}_{i_m}$ successively we get
\begin{multline*}
        \Pr((X_T,\dots,X_k)=\gamma, W^{\gamma}_{i_1,\dots,i_m},
        D^{\gamma}_{i_1,\dots,i_m}\, \vert \, \mathcal{F}^{\simi}_T)\\
        \leq \mathbb{E}[ \ind{(X_T,\dots,X_{i_m-1})=(\gamma_T,
                \dots,\gamma_{i_m-1})}
        \indset{W^{\gamma\vert_{i_m-1}}_{i_1,\dots,i_{m-1}}} K^{k-i_m-1}
        \mathbb{E}[ \indset{D^{\gamma}_{i_1,\dots,i_m}}\, \vert \,
        \mathcal{F}^{\simi}_{i_m-1}]\, \vert \, \mathcal{F}^{\simi}_T ]
\end{multline*}
because the path will only hit ones of $\eta$ for the steps after
$i_m$. If we repeat the same argument for every $i_j$ we obtain
\begin{align*}
        \Pr(H_{k} = m \, \vert \, \mathcal{F}^{\simi}_T)
        & \leq \sum_{T\leq i_1 < \dots < i_m \leq k} \sum_{\gamma \in
                \Gamma_{k-T}} K^{k-T-m}\Pr(D^{\gamma}_{i_1,\dots,i_m} \, \vert
        \, \mathcal{F}^{\simi}_T) \\
        & \leq  \sum_{T\leq i_1 < \dots < i_m \leq k}
        R_\kappa^{d(k-T)}K^{k-T-m}\varepsilon(p)^m =
        \binom{k-T}{m}R_\kappa^{d(k-T)}K^{k-T-m}\varepsilon(p)^m,
\end{align*}
where we used Corollary~\ref{corollary:dryPoints} in the second
inequality. From this point the rest of the proof consists of the
same calculations as in the proof of Lemma~2.16 of
\cite{BirknerCernyDepperschmidt2016}, where the walk is compared
with the reference walk for steps from sites on which $\eta$ equals
$1$ and one uses the a priori bound,
Assumption~\ref{ass:finite-range}, for step sizes from (the few)
other sites. This proves \eqref{eq:aprioriAlongStoppingTimesX}.
Since $(X'_k-X'_T)$ has the same maginal distribution as $(X_k-X_T)$ we also obtain
\eqref{eq:aprioriAlongStoppingTimesX'} by the same arguments.

For
\eqref{eq:aprioriAlongStoppingTimesX2} define
$T' \coloneqq \inf(\{ \sigma^{\simi}_i : i \in \N \} \cap
[j,\infty))$ the first $\sigma^{\simi}_i$ after $j$. Inequality
\eqref{eq:expTailsTau1} shows that
$\Pr(T' - j > \varepsilon(k-j) \, \vert \, \mathcal{F}^{\simi}_T)$ is
exponentially small in $k-j$ on $\{T < j \}$. Using
\eqref{eq:aprioriAlongStoppingTimesX} we obtain
\begin{align*}
        \Pr& (\norm{X_k - X_j} > (1+\varepsilon)s_{\mathrm{max}}(k-j),
        T'-j\leq\frac{\varepsilon s_{\mathrm{max}}}{R_\kappa}(k-j)
        \,\vert\, \mathcal{F}^{\simi}_T)\\
        & \leq \Pr(\norm{X_k - X_{T'}}+\norm{X_{T'}-X_j} >
        (1+\varepsilon)s_{\mathrm{max}}(k-j), T'-j\leq\frac{\varepsilon
                s_{\mathrm{max}}}{R_\kappa}(k-j) \,\vert\, \mathcal{F}^{\simi}_T)\\
        & \leq \Pr(\norm{X_k - X_{T'}} + \varepsilon s_{\mathrm{max}}(k-j)
        > (1+\varepsilon)s_{\mathrm{max}}(k-j),T'-j\leq\frac{\varepsilon
                s_{\mathrm{max}}}{R_\kappa}(k-j) \,\vert\, \mathcal{F}^{\simi}_T)\\
        & \leq \Pr(\norm{X_k - X_{T'}} >
        s_{\mathrm{max}}(k-j),T'-j\leq\frac{\varepsilon
                s_{\mathrm{max}}}{R_\kappa}(k-j) \,\vert\, \mathcal{F}^{\simi}_T)\\
        & \leq \mathbb{E}[\ind{T' -j \leq \frac{\varepsilon
                        s_{\mathrm{max}}}{R_\kappa}(k-j)}\Pr( \norm{X_k - X_{T'}} >
        s_{\mathrm{max}}(k-j) \,\vert\,
        \mathcal{F}^{\simi}_{T'})\,\vert\, \mathcal{F}^{\simi}_T] \\
        & \leq \mathbb{E}[\ind{T' -j \leq \frac{\varepsilon
                        s_{\mathrm{max}}}{R_\kappa}(k-j)} Ce^{-c(k-T')} \, \vert \,
        \mathcal{F}^{\simi}_T]\\
        & \leq \mathbb{E}[\ind{T' -j \leq \frac{\varepsilon
                        s_{\mathrm{max}}}{R_\kappa}(k-j)} Ce^{-c(k-j) +
                c\frac{\varepsilon s_{\mathrm{max}}}{R_\kappa}(k-j)} \, \vert \,
        \mathcal{F}^{\simi}_T]\\
        & \leq Ce^{-c(k-j) + c\frac{\varepsilon
                        s_{\mathrm{max}}}{R_\kappa}(k-j)}.
\end{align*}
This yields
\begin{align*}
        \Pr (\norm{X_k - X_j}
        & > (1+\varepsilon)s_{\mathrm{max}}(k-j) \, \vert \, \mathcal{F}^{\simi}_T)\\
        & = \Pr(\norm{X_k - X_j} > (1+\varepsilon)s_{\mathrm{max}}(k-j),
        T'-j>\frac{\varepsilon  s_{\mathrm{max}}}{R_\kappa}(k-j)
        \,\vert\, \mathcal{F}^{\simi}_T)\\
        & \qquad  +  \Pr(\norm{X_k - X_j} >
        (1+\varepsilon)s_{\mathrm{max}}(k-j), T'-j\leq\frac{\varepsilon
                s_{\mathrm{max}}}{R_\kappa}(k-j) \,\vert\, \mathcal{F}^{\simi}_T)\\
        & \leq Ce^{-c\frac{\varepsilon s_{\mathrm{max}}}{R_\kappa}(k-j)} +
        Ce^{-c(k-j) + c\frac{\varepsilon s_{\mathrm{max}}}{R_\kappa}(k-j)}\\
        & \leq Ce^{-c'(k-j)},
\end{align*}
where we used \eqref{eq:expTailsTau1} of
Lemma~\ref{lemma:expTailsTau} for the first inequality. Finally,
\eqref{eq:aprioriAlongStoppingTimesX'2} follows by the same
arguments.

\subsection{Proof of Lemma~\ref{prop:TVdistance-joint-ind-1step}}
\label{sec:proof:lem:TVdistance-joint-ind}
From equation~\eqref{eq:TVdistance-joint-ind-1step in prop} it is clear that the
exact starting positions of the random walks are negligible and
only the distance is important. Due to notational convenience we
will restrict ourselves in the proof to a distance in $\N$. The
proof can be adapted to all possible distances in $\R_+$.
Furthermore we choose starting positions on the axis of the first
coordinate. In the proof we want to split space in two parts
separating the staring points by a hyperplane with equal distance
to both. This can be done for all starting positions but requires
tedious notation that would not further the understanding of the
proof.

\medskip

\noindent
Therefore, to simplify the notation throughout the proof, we fix a
positive even integer $m$ and we prove the assertion for starting
positions $x=(-m/2,0,\dots,0)$ and $x'=(m/2,0,\dots,0)$ in $\bZ^d$. We
can do this since for the proof only the distance between the starting
positions is going to be relevant, not the exact positions.

Every environment is defined by a configuration
$\omega \in \{0,1\}^{\bZ^d\times \bZ}$ and dynamics of the random
walks are then given by the family of transition
kernels $\kappa = \{\kappa_n(x,\cdot):n\in \bZ, x \in \bZ^d\}$. Let
\begin{align}
        \Omega_i \coloneqq \{ (\omega_i(z,n) : (z,n) \in
        \bZ^d\times \bZ \}, \quad i=1,2
\end{align}
be two independent families of random variables, where the random
variable $\omega_i(z,n)$ are i.i.d.\ Bernoulli distributed with
parameter $p > p_c$. We introduce a composite environment $\Omega_3
\coloneqq \{ \omega_3(z,n) : (z,n) \in \bZ^d\times \bZ \}$, where
$\omega_3$ is constructed using $\omega_1$ and $\omega_2$ and defined
by
\begin{align*}
  \omega_3(z,n) \coloneqq
  \begin{cases}
    \omega_1(z,n) &: z_1 \leq 0,\\
    \omega_2(z,n) &: z_1 > 0,
  \end{cases}
\end{align*}
where $z_1$ is the first coordinate of $z$. Define
\begin{align*}
  B & \coloneqq \Bigl\{ (z,n) \in \bZ^d\times \bZ: \norm{z-x} \leq
      \frac{m}{10} + b_\out + \frac{m}{10 R_\kappa} s_\out, n \in
      \bigl\{0,\dots, -\frac{m}{10 R_\kappa}\bigr\} \Bigr\},\\
  B' & \coloneqq \Bigl\{ (z,n) \in \bZ^d\times \bZ: \norm{z-x'} \leq
       \frac{m}{10} + b_\out + \frac{m}{10 R_\kappa} s_\out, n \in
       \bigl\{0,\dots, -\frac{m}{10 R_\kappa}\bigr\} \Bigr\}.
\end{align*}
These boxes will contain the random walks and cones constructed for
the first regeneration if it happens before time $\frac{m}{10
  R_\kappa}$. This also means that the $R_\loc$-vicinity of the random
walks is inside these boxes and thus up until time $m/(10R_\kappa)$
the random walks in $(\Omega_1,\Omega_2)$ and $(\Omega_3,\Omega_3)$
will behave the same as long as the values of $\eta$ inside $B$ and
$B'$ are the same in both environments (this means, by
Assumption~\ref{ass:local}, that the transition kernels
$\kappa_n(x,\cdot)$ are the same).

Define $T_{1,2}=T_1(\Omega_1,\Omega_2)$ and
$T_{3,3}=T_1(\Omega_3,\Omega_3)$ as the first simultaneous
regeneration times using $\Omega_1, \Omega_2$ and $\Omega_3$
respectively. Note that $T_{3,3}=T^\joint_1$ and $T_{1,2}=T^\indi_1$
and we have already shown in
Propositions~\ref{prop:JointRegTimesBound} and
\ref{prop:IndRegTimesBound} that
\begin{align}
  \Pr_{x,x'}(T_{3,3} > r) \leq C r^{-\beta}
\end{align}
and
\begin{align}
  \Pr_{x,x'}(T_{1,2} > r) \leq C r^{-\beta}.
\end{align}
Now if $\eta_n(z,\Omega_1) = \eta_n(z,\Omega_3)$ for all $(z,n)\in B$
and $\eta_n(z,\Omega_2) = \eta_n(z,\Omega_3)$ for all $(z,n)\in B'$ we
can couple the two random walks at their first regeneration time in
$(\Omega_1,\Omega_2)$ with the two in $(\Omega_3,\Omega_3)$ until time
$\frac{m}{10R_\kappa}$, since the values of $\omega$ and $\eta$ in the
parts of the environment that the random walks explore are equal in
both cases and therefore their distribution is equal. To that end we
define the sets
\begin{align*}
  D_1 & \coloneqq \{ \text{for all } (z,n) \in B: \eta_n(z,\Omega_1) =
        \eta_n(z,\Omega_3) \} \cap \{ \text{for all } (z,n) \in B':
        \eta_n(z,\Omega_2) = \eta_n(z,\Omega_3)\},\\
  D_2 & \coloneqq \{ T_{1,2} \leq \frac{m}{10 R_\kappa},
        T_{3,3} \leq \frac{m}{10 R_\kappa} \}.
\end{align*}
On $D_1 \cap D_2$ we have $T_{1,2} = T_{3,3}$ and
since $T_{1,2} \leq \frac{m}{10 R_\kappa}$ the random walks
are still in the box and thus have the same distribution for their
position at the first regeneration in both cases. That means
\begin{align*}
  \bP_{x,x'}(X_{T_{1,2}} =y,X'_{T_{1,2}} =y')
  = \bP_{x,x'}(X_{T_{3,3}} =y,X'_{T_{3,3}} =y')
\end{align*}
on $D_1 \cap D_2$. With Propositions~\ref{prop:JointRegTimesBound} and
\ref{prop:IndRegTimesBound} we get an upper bound for the probability
of $D_2^c$. On $D_1^c$ there needs to exist a space-time site $(z,n)$
in $B$ with $\eta_n(z,\Omega_1)\neq \eta_n(z,\Omega_3)$ or $(z',n')$
in $B'$ with $\eta_{n'}(z',\Omega_2) \neq \eta_{n'}(z',\Omega_3)$.

\medskip

Assume there exists such a site $(z,n) \in B$ with
$\eta_{n}(z,\Omega_1)\neq\eta_{n}(z,\Omega_3)$. There are two cases in
which that can occur. First, $\eta_{n}(z,\Omega_1)=1$ and
$\eta_{n}(z,\Omega_3)=0$. That is $(z,n)$ is connected to $-\infty$ in
$\Omega_1$ but not in $\Omega_3$. This means that by changing the
values of $\omega$ in the positive half plane
$\bZ_+\times\bZ^{d-1}\times\bZ$ we cut of all infinitely long open
paths starting from $(z,n)$. Thus the only open paths connecting
$(z,n)$ to $-\infty$ are via the half plane
$\bZ_+\times\bZ^{d-1}\times\bZ$ and in $\Omega_3$ the contact process
started from $(z,n)$ lives for at least
$\wt{m}=\frac{m}{2}-\frac{m}{10}-b_\out-\frac{m}{10R_\kappa}s_\out$
steps, since that is the distance of $B$ and the positive half plane,
but dies out eventually. By Lemma~1.3 from \cite{Steiber2017}
\begin{align*}
  \bP(\eta^{\{(z,n)\}}_{\wt{m}} \neq 0 \text{ and } \eta^{\{(z,n)\}}
  \text{ eventually dies out}) \le Ce^{-\tilde{c}\wt{m}}\le Ce^{-cm}.
\end{align*}
The case where $\eta_{n}(z,\Omega_1)=0$ and $\eta_{n}(z,\Omega_3)=1$
follows the same arguments as above. We now know that the only open
paths connecting $(z,n)$ to $-\infty$ in $\Omega_3$ have to be cut
off in $\Omega_1$, which again has probability less than $Ce^{-cm}$.
Analogous arguments can be made for sites in $B'$. Since
$\abs{B\cup B'} = 2 (\frac{m}{10}+b_\out +
\frac{m}{10R_\kappa}s_\out)^d\cdot \frac{m}{10R_\kappa}$ we obtain
$\bP(D_1^c)\le Ce^{-cm}$. Since
$\bP_{x,x'}(X_{T_{1,2}} =y,X'_{T_{1,2}} =y') =
\bP^\indi_{x,x'}(\widehat{X}_1 =y,\widehat{X}'_1=y')$ and
$\bP_{x,x'}(X_{T_{3,3}} =y,X'_{T_{3,3}} =y') =
\bP^\joint_{x,x'}(\widehat{X}_1 =y,\widehat{X}'_1=y')$ we conclude
\begin{multline}
  \label{eq:7}
  \frac{1}{2}\sum_{(y,y')\in \bZ^d \times \bZ^d}
  \abs{\Pr^{\indi}_{x,x'}(\widehat{X}_1 =y, \widehat{X}'_1 = y') -
    \Pr^{\joint}_{x,x'}(\widehat{X}_1 =y, \widehat{X}'_1 = y')}\\
  \leq \Pr(D_1^c) + \Pr(D_2^c) \leq C e^{-cm} + C m^{-\beta} \leq C
  m^{-\beta}
\end{multline}

\subsection{Proof of Lemma~\ref{lem:separ in d=1}}
To prove \eqref{eq:25} the steps are analogous to
Lemma~\ref{lem:separ}. The only difference is that we use the
harmonic function for $d=1$ from Lemma~\ref{lemma:exitAnnulus}
\begin{align}
        \frac{r - r_1}{r_2-r_1}
\end{align}
for the proof. The steps will be the same so we will only highlight
the changes.

Assume that we start with $x_0=x'_0$. This will yield an upper bound
on the other starting pairs since we have no starting distance. By the
same arguments as in the proof of Lemma~\ref{lem:separ} we obtain for
some $\delta_0>0$, by constructing suitable corridors, the lower bound
\begin{equation*}
  \bP_{x_0,x'_0}^\joint( H(\varepsilon_1 \log n) \le \varepsilon_1 \log n ) \ge \delta_0^{\varepsilon_1 \log n} = n^{\varepsilon_1 \log \delta_0}
\end{equation*}
and for some large constant $K\gg \varepsilon_1$ and $\abs{x-x'}\ge \varepsilon_1 \log n$
\begin{align*}
  &\bP^\joint_{x,x'} \bigg(H(K\log n) < h\Big(\frac{1}{2}\varepsilon_1\log n\Big) \wedge (K\log n)^3 \bigg)\\
  &\ge \bP^\indi_{x,x'} \bigg(H(K\log n) < h\Big(\frac{1}{2}\varepsilon_1\log n\Big) \wedge (K\log n)^3 \bigg) - (K\log n)^3\Big(\frac{1}{2}\varepsilon_1 \log n\Big)^{-\beta}\\
  &\ge\bP^\indi_{x,x'} \bigg(H(K\log n) < h\Big(\frac{1}{2}\varepsilon_1 \log n\Big)\bigg) - C(K\log n)^{-1} - (K\log n)^3\Big(\frac{1}{2}\varepsilon_1 \log n\Big)^{-\beta}\\
  &\ge (1-\varepsilon) \frac{1}{(2K/\varepsilon_1) -1} - C(K\log n)^{-1} - (K\log n)^3\Big(\frac{1}{2}\varepsilon_1\log n\Big)^{-\beta}\\
  &\ge \frac{1}{4} \frac{\varepsilon_1}{K}
\end{align*}
for $n$ and $\beta$ large enough. Combining these with probability
greater than $n^{\varepsilon_1\log \delta_0} \frac{\varepsilon_1}{4K}$
we need at most $\varepsilon_1\log n + (K\log n)^3$ many steps to
reach a distance of at least $K\log n$.

If the random walks are already at distance $K\log n$ we want to start
the iteration of Step 2 and the following from the proof of
Lemma~\ref{lem:separ}. Therefore we need a lower bound in $d=1$, when
starting from a distance of $\norm{x-x'}\ge K\log n$, for
\begin{align}
  \label{eq:proof sep lemma d=1 - eq1}
  \Pr_{x,x'}^\indi(H(\log^2 n) < h(1/2 K\log n))
\end{align}
and starting from a distance of $\norm{x-x'}\ge \log^k n$ we want to
find a lower bound for
\begin{align}
  \label{eq:proof sep lemma d=1 - eq2}
  \Pr_{x,x'}^\indi(H(\log^{k+1} n) < h(1/2 \log^{k} n)).
\end{align}
Here we need to make smaller steps since our previous approximations
are not sufficient. For \eqref{eq:proof sep lemma d=1 - eq1} we get
the lower bound (note that here $\norm{x-x'}\ge K\log n$)
\begin{align}
  \Pr_{x,x'}^\indi(H(\log^2 n) < h(1/2 K\log n)) \geq
  (1-\varepsilon)\frac{K}{2\log n - K}
\end{align}
and for \eqref{eq:proof sep lemma d=1 - eq2} (note that here
$\norm{x-x'}\ge \log^k n$)
\begin{align}
  \Pr_{x,x'}^\indi(H(\log^{k+1} n) < h(1/2 \log^{k} n)) \geq
  (1-\varepsilon)\frac{1}{2\log n - 1}.
\end{align}

That means for one step,
starting from $\norm{x-x'} \ge \log^k n$,
\begin{align*}
  \Pr^\joint_{x,x'}
  & (H(\log^{(k+1)} n) < h(1/2 \log^k n))\\
  & \geq \Pr^\joint_{x,x'}(H(\log^{(k+1)} n) < h(1/2 \log^k n) \wedge (\log^{k+1}n)^3)\\
  & \geq \Pr^\indi_{x,x'}(H(\log^{(k+1)} n) < h(1/2 \log^k n) \wedge
    (\log^{k+1}n)^3) - C(\log^{k+1}n)^3(1/2 \log^k n)^{-\beta}\\
  & \geq \Pr^\indi_{x,x'}(H(\log^{(k+1)} n) < h(1/2 \log^k n)) - (\log
    n)^{-(k+1)} - C(\log^{k+1}n)^3(1/2 \log^k n)^{-\beta}\\
  & \geq (1-\varepsilon)\frac{1}{2}(\log n)^{-1} - (\log n)^{-(k+1)} -
    C(\log^{k+1}n)^3(1/2 \log^k n)^{-\beta}\\
  & \geq \frac{1}{4}(\log n)^{-1}
\end{align*}
for large $n$ and $\beta>0$ large enough, where we again use the
invariance principle, Lemma~\ref{lemma:exitAnnulus} and
Proposition~\ref{prop:TVdistance-joint-ind-1step}.

Now we need $b' \log n /\log\log n$ steps in the iteration to reach a
distance of $n^{b'}$, with the $k$-th step taking at most time
$(\log^{k+1}n)^3$. Consequently, we need at most $\sum_{k=2}^{b'\log
  n/\log\log n} (\log n)^{3k}$ time for one such attempt from distance
$K\log n$ to $n^{b'}$ and therefore from $x=x'$ we need at most
\begin{align*}
  \varepsilon_1\log n + (K\log n)^3 + \sum_{k=2}^{b'\log n/\log\log n} (\log n)^{3k} \le n^{4b'}
\end{align*}
for $n$ large enough. We obtain a lower bound on the probability to
make the whole distance in a single attempt
\begin{align*}
  &n^{\varepsilon_1\log \delta_0}\frac{\varepsilon_1}{4K}\prod_{k=2}^{b'\log n/\log\log n}\frac{1}{4 \log n}\\
  &\ge n^{\varepsilon_1\log \delta_0}\frac{\varepsilon_1}{4K} \exp\Big( -b'\log n \frac{\log \log n^4}{\log \log n} \Big)\\
  &\ge  \frac{\varepsilon_1}{4K} n^{-2b' + \varepsilon_1\log \delta_0}
\end{align*}
It is easy to see that for any $b'>0$ small enough there exist a
$0<b_2<1/8$ and $\varepsilon_1$ small enough such that $\alpha
\coloneqq b_2-4b'>2b'-\varepsilon_1\log \delta_0$ and with that
\begin{align*}
  & \Pr^\joint_{x,x'}(H(n^{b'})\geq n^{b_2})\\
  & \leq (1-n^{-2b'+\varepsilon_1\log \delta_0})^{n^\alpha}\\
  & \leq \exp(-n^{-2b'+\varepsilon_1\log \delta_0+\alpha})
\end{align*}
where the last term will tend to 0 as $n\to \infty$.

For \eqref{eq:lower bound for W} we need to show that there exists a
uniform lower bound in $n$ away from zero for the random walks to
change their positions after they come close to each other before they
reach a distance of $n^{b'}$. So let's say at time
$m\coloneqq\mathcal{R}_{n,i}$ we have $\widehat{X}_m>\widehat{X}'_m$.
More precisely, write $D_j \coloneqq \widehat{X}_j-\widehat{X}'_j$ and
pick a small $\varepsilon>0$ to be tuned later. Using the same methods
as in the proof of \eqref{eq:25} we can show a suitable uniform lower
bound on the probability that $D$ reaches $(-\infty,0]$ before
$n^{b'}$. By the same arguments as above and noting that
\begin{align*}
  \bP^\indi \bigl(h(\log^{k-1}n) < H(\log^{k+1}n) \,\vert\, D_0=\log^k n\bigr)
  \ge (1-\wt{\varepsilon})\frac{\log^{k+1}n - \log^k
  n}{\log^{k+1}n - \log^{k-1} n},
\end{align*}
we obtain for $k\in \N$
\begin{align}
  \label{eq:24}
  \begin{split}
    \bP^\joint \bigl(h(\log^{k-1}n) < H(\log^{k+1}n) \,\vert\, D_0=\log^k n\bigr)
    &\ge c\frac{\log^{k+1}n - \log^k n}{\log^{k+1}n - \log^{k-1} n}\\
    &=c\Big(1-\frac{\log n - 1}{\log^2n - 1}\Big).
  \end{split}
\end{align}
And similarly
\begin{align*}
  \bP^\joint (h(n^a) < H(n^{b'}) \,\vert\, D_0=n^a\log n)
  & \ge c \frac{n^{b'-a} -n^a\log n}{n^{b'-a} -n^a}\\
  & = c\Big(1-\frac{\log n -1}{n^{b'-a}-1}\Big).
\end{align*}
Therefore we have, with positive probability, $\frac{n^{b'-a}}{\log
  n}$ many tries for the process $D$ to reach $(-\infty,0]$ before it
hits $n^{b'}$. By the Markov property and using \eqref{eq:24} we need
$a\log n/\log\log n$ iterations for $D$ to reach $\varepsilon\log n$.
From $\varepsilon \log n$ we build corridors, as in Step~1 in the
proof of Lemma~\ref{lem:separ}, to achieve $D<0$. The probability for
such a corridor to exist is $\exp(-c\varepsilon\log n)$. Thus, for one
attempt to reach $(-\infty,0]$, the probability to be successful is
greater than $n^{ca\log n/\log \log n}n^{-c\varepsilon}$ and we have
$\frac{n^{b'-a}}{\log n}$ such attempts with positive probability.
Combining those two facts, we obtain a uniform lower bound on the
probability for $D$ to reach $(-\infty,0]$ away from zero. Which
concludes the proof.

\subsection{Proof of Lemma~\ref{lem:bound_for_A1_and_A2}}
Let $Y$ have the distribution
\begin{align}
  \label{eq:23}
  \bP(Y\geq \ell) = \hat{\Psi}^{\indi}\Big( \inf\{ m \geq 0 :
  \widehat{X}_m < \widehat{X}'_m \}\geq \ell \,\vert\,
  (\widehat{X}_0,\widehat{X}'_0) = (1,0) \Big), \quad \ell \in \N
\end{align}
and let $V$ be an independent Bernoulli$(1-1/n)$-distributed random
variable. A coupling based on
Proposition~\ref{prop:TVdistance-joint-ind-1step} shows that
$\mathcal{R}_{n,1}-\mathcal{D}_{n,1}$ is stochastically larger than
$((1-V)+VY)\wedge n$, since the distance between
$\widehat{X}_{D_{n-1}}$ and $\widehat{X}'_{D_{n-1}}$ is $n^{b'}$ and
$n^{b'}- n^a\gg 1$ and we can choose $a$ and $\beta$ in such a way
that the probability for the coupling between $\hat{\Psi}^\indi$ and
$\hat{\Psi}^\joint$ fails during the first $n$ steps is less than
$1/n$. By well known estimates of one-dimensional random walks (e.g.
see Theorem~12.17 in \cite{Kallenberg20213rdEdition} and Theorem~1.a.
in \cite{Feller1971} on page 415), there exist $c>0$ and $c_Y>0$ such
that uniformly in $n\ge 2$,
\begin{align}
  \label{eq:14}
  \bE[\mathrm{e}^{-\lambda((1-V)+VY)}]
  & \leq \exp(-c_Y\sqrt{\lambda}), \quad \lambda\geq 0 \quad \text{ and }\\
  \label{eq:tails_of_R_minus_D}
  \bP^\joint_{x_0,x'_0}(\mathcal{R}_{n,1}-\mathcal{D}_{n,1}\geq \ell)
  & \geq \frac{c}{\sqrt{\ell}}, \quad \ell=1,\dots,n.
\end{align}
Inequality \eqref{eq:14} is trivial for $\lambda\ge 1$ and
$\lambda=0$, for $\lambda\in(0,1)$ we have, using Theorem~12.17 from
\cite{Kallenberg20213rdEdition} with $u=0$ and
$s=\mathrm{e}^{-\lambda}$,
\begin{align*}
  \bE[\mathrm{e}^{-\lambda((1-V)+VY)}]
  & = \frac{1}{n}\mathrm{e}^{-\lambda} + \frac{n-1}{n}\bE[\mathrm{e}^{-\lambda Y}] \\
  & \leq \frac{1}{n}\mathrm{e}^{-\lambda} + \frac{n-1}{n}\Big(1 -
    \exp\Big\{ -\frac{1}{2}\sum_{m=1}^\infty
    \frac{\mathrm{e}^{-\lambda m}}{m} \Big\}  \Big)\\
  & \leq \frac{1}{n}\mathrm{e}^{-\lambda} + \frac{n-1}{n}\Big(1 -
    \exp\Big\{ \frac{1}{2}\log(1-\mathrm{e}^{-\lambda}) \Big\} \Big)\\
  & =\frac{1}{n}\mathrm{e}^{-\lambda} + \frac{n-1}{n}\Big(1 -
    (1-\mathrm{e}^{-\lambda})^{\frac{1}{2}}  \Big).
\end{align*}
Note that
\begin{align*}
        1 - \exp(-\sqrt{\lambda}) \leq (1-\exp(-\lambda))^{\frac{1}{2}}
\end{align*}
for all $\lambda \in (0,1)$. Furthermore we have
\begin{align}
  \label{eq:15}
  \frac{1}{n}\mathrm{e}^{-\lambda} \leq \exp(-\sqrt{\lambda})
\end{align}
if
\begin{align*}
  -\lambda - \log(n) \leq -\sqrt{\lambda},
\end{align*}
which holds for all $n\ge 2$. And thus \eqref{eq:14} holds uniformly
in $n\ge 2$ with $c_Y=1$. Let $I_n \coloneqq \max\{
i:\mathcal{R}_{n,i} \leq n \}$ be the number of ``black boxes'' that
we see up to time $n$. By equation \eqref{eq:tails_of_R_minus_D} we
have $I_n = \mathcal{O}(\sqrt{n})$ in probability and in fact
\begin{align}
  \label{eq:bound_second_moment_In}
  \bE^\joint_{x_0,x'_0}[I_n^2] \leq Cn.
\end{align}
This can be proven by using the lower bounds of equation
\eqref{eq:tails_of_R_minus_D}. Note that since $I_n$ is a
$\N_0$-valued random variable
\begin{align*}
  \bE^\joint_{x_0,x'_0}[I_n^2]
  & = \sum_{\ell =1}^\infty \bP^\joint_{x_0,x'_0}(I_n^2\ge \ell)
    =\sum_{\ell =1}^\infty \bP^\joint_{x_0,x'_0}(I_n\ge \sqrt{\ell})
    =\sum_{\ell =1}^\infty \bP^\joint_{x_0,x'_0}(\mathcal{R}_{n,\sqrt{\ell}}\le n)\\
  & \le \sum_{\ell =1}^\infty
    \bP^\joint_{x_0,x'_0}\Big(\sum_{i=1}^{\sqrt{\ell}}\mathcal{R}_{n,i}-\mathcal{D}_{n,i}\le n\Big)\\
  & \le \sum_{\ell =1}^\infty \bP^\joint_{x_0,x'_0}\Big(\text{for all }i\le
    \sqrt{\ell} \colon \mathcal{R}_{n,i}-\mathcal{D}_{n,i}\le n\Big)\\
  & \le \sum_{\ell =1}^\infty \bP^\joint_{x_0,x'_0}\Big(
    \mathcal{R}_{n,1}-\mathcal{D}_{n,1}\le n\Big)^{\sqrt{\ell}}\\
  & \le \sum_{\ell =1}^\infty
    \bigg(1-\frac{c}{\sqrt{n}}\bigg)^{\sqrt{\ell}}\\
  & \le n \sum_{k=1}^\infty    \exp(-c\sqrt{k-1}) \le Cn.
\end{align*}
The last line follows by grouping the first $n$ summands, the second
$n$ summands and so forth and bounding them together. For example for
$k n \le \ell \le (k+1)n$ we have
\begin{align*}
        \bigg(1-\frac{c}{\sqrt{n}}\bigg)^{\sqrt{\ell}} \le
        \bigg(1-\frac{c}{\sqrt{n}}\bigg)^{\sqrt{k n}} \le
        \exp(-c\sqrt{k}).
\end{align*}
More quantitatively, there exists $c>0$ such that for
$1\leq k \leq n$
\begin{align}
  \label{eq:22}
  \bP^\joint_{x_0,x'_0}(I_n \geq k) \leq \exp(-ck^2/n)
\end{align}
and so in particular
\begin{align}
  \label{eq:17}
  \bE^\joint_{x_0,x'_0}[I_n\indset{I_n\geq n^{3/4}}] = \sum_{k=\lceil n^{3/4}
  \rceil}^n\bP^\joint_{x_0,x'_0}(I_n \geq k) \leq n \mathrm{e}^{-c\sqrt{n}}.
\end{align}

The inequality in \eqref{eq:22} can be obtained by the following
arguments. Let $Y_1,Y_2,\dots$ be i.i.d.\ copies of
$\big((1-V)+VY\big)$ defined above in \eqref{eq:23}, then
\begin{align*}
  \bP\Big( \big((1-V)+VY\big)\wedge n  \ge \ell \Big) \le
  \bP^\joint_{x_0,x'_0}(\mathcal{R}_{n,1}-\mathcal{D}_{n,1}\ge \ell)
\end{align*}
and thus
\begin{align*}
  \bP^\joint\bigg( \sum_{i=1}^\ell (Y_i \wedge n) \le n\bigg) \ge \bP^\joint\bigg(
  \sum_{i=1}^\ell \mathcal{R}_{n,i}-\mathcal{D}_{n,i} \le n \bigg)
  \ge \bP^\joint\Big( \mathcal{R}_{n,\ell} \le n \Big) = \bP^\joint(I_n \ge \ell).
\end{align*}
Combining the two we obtain by \eqref{eq:14} for $\lambda >0$
\begin{align*}
  \bP^\joint_{x_0,x'_0}(I_n \ge \ell)
  & \le \bP\bigg( \sum_{i=1}^\ell (Y_i \wedge n) \le n \bigg) =
    \bP(Y_1+\cdots+Y_\ell \le n ) \\
  & \le \mathrm{e}^{\lambda n} \bE\Big[ \exp\big(-\lambda Y_1\big)
    \Big]^\ell \le \mathrm{e}^{\lambda n -c_Y\sqrt{\lambda} \ell}.
\end{align*}
Choosing $\lambda=(c_Yk/n)^2$ we see that \eqref{eq:22} holds.

Note that
\begin{align}
  \label{eq:10}
  R_n \leq \sum_{j=1}^{I_n+1}(\mathcal{D}_{n,j}-\mathcal{R}_{n,j-1}),
\end{align}
and using \eqref{eq:sep in d=1} we get $R_n=o(n)$ in probability.
Now using \eqref{eq:bound_second_moment_In} together with
\eqref{eq:10} and \eqref{eq:sep in d=1} implies
\eqref{eq:bound_R_n_moment}:
\begin{align}
  \notag
  \bE^\joint_{x_0,x'_0}[R_n^2]
  &= \bE^\joint_{x_0,x'_0}[R_n^2\indset{\{\exists j\leq
    n:\mathcal{D}_{n,j+1}-\mathcal{R}_{n,j} \geq n^{b_2}\}}] +
    \bE^\joint_{x_0,x'_0}[R_n^2\indset{\{\forall j\leq
    n:\mathcal{D}_{n,j+1}-\mathcal{R}_{n,j} < n^{b_2}\}}]\\
  \notag
  & \leq n^2\bP^\joint_{x_0,x'_0}(\exists j\leq
    n:\mathcal{D}_{n,j+1}-\mathcal{R}_{n,j} \geq n^{b_2})
    +n^{2b_2}\bE^\joint_{x_0,x'_0}[I_{n+1}^2]\\
  \label{eq:28}
  & \leq Cn^{1+2b_2}
\end{align}
and $\bE^\joint_{x_0,x'_0}[R_n^{3/2}] \leq
(\bE^\joint_{x_0,x'_0}[R_n^2])^{3/4}$.

For \eqref{eq:bound_for_A1_and_A2} we define
\begin{align*}
  D_{n,m} \coloneqq
  A^{(1)}_{\mathcal{D}_{n,m}}-A^{(1)}_{\mathcal{R}_{n,m-1}},\quad
  D_{n,m}' \coloneqq A^{(2)}_{\mathcal{D}_{n,m}}-A^{(2)}_{\mathcal{R}_{n,m-1}}.
\end{align*}
By symmetry we have
\begin{align}
  \label{eq:11}
  \bE^\joint[D_{n,j}]=0, \bE^\joint[D_{n,j}\,\vert\, W_{n,j}=1]
  & = -\bE^\joint[D_{n,j}\,\vert\,W_{n,j}=3],\\
  \bE^\joint[D_{n,j}\,\vert\, W_{n,j}=2]
  & = -\bE^\joint[D_{n,j}\,\vert\,W_{n,j}=4],
\end{align}
by Lemma~\ref{lem:separ in d=1} and \eqref{eq:boundAllPhi} we get that
for some $C<\infty$ and $b>0$ uniformly in $j,n\in\N$, for
$w\in\{1,2,3,4\}$
\begin{align}
  \notag
  \bE^\joint_{x_0,x'_0}
  &[\abs{D_{n,j}}\,\vert\, W_{n,j}=w]\\
  \notag
  & = \bE^\joint_{x_0,x'_0}\Bigg[\abs{\sum_{i=\mathcal{R}_{n,j-1}}^{\mathcal{D}_{n,j}}
    \phi_1(\widehat{X}_i,\widehat{X}'_i)}\,\Big\vert\, W_{n,j}=w \Bigg]\\
  \notag
  & \leq C_\phi \bE^\joint_{x_0,x'_0}\Big[ (\mathcal{D}_{n,j}-\mathcal{R}_{n,j-1}) \,\vert\, W_{n,j}=w\Big]\\
  \notag
  & \leq C_\phi n^{b_2} + C_\phi \sum_{\ell =2}^\infty n^{b_2 \ell}
    \bP^\joint_{x_0,x'_0}(\mathcal{D}_{n,j}-\mathcal{R}_{n,j-1} \geq n^{(\ell-1)b_2}\,\vert\, W_{n,j}=w)\\
  \label{eq:13}
  & \leq Cn^{b_2}
\end{align}
and analogously for $D'_{n,j}$. Set $\mathcal{G}_j\coloneqq
\hat{\mathcal{F}}_{\mathcal{D}_{n,j}}$ (the $\sigma$-field of the
$\mathcal{D}_{n,j}$-past) for $j\in \N$ and for $j\leq 0$ let
$\mathcal{G}_j$ be the trivial $\sigma$-algebra. Note that $D_{n,j}$
and $D'_{n,j}$ are $\mathcal{G}_j$-adapted for $j\in\N$. For $k< m$ we
have
\begin{align*}
  \bE^\joint_{x_0,x'_0}[D_{n,m}\,\vert\,
  \mathcal{G}_k]=\bE^\joint_{x_0,x'_0}\Big[\bE[D_{n,m}\,\vert\, 
  W_{n,m}]\,\vert\, \mathcal{G}_k\Big]
\end{align*}
by construction and $(W_{n,j})_j$ is (uniformly in $n$) exponentially
mixing, thus, observing \eqref{eq:11}, \eqref{eq:12}, and
\eqref{eq:13}
\begin{align}
  \bE^\joint_{x_0,x'_0}\Big[ (\bE[D_{n,m}\,\vert\, \mathcal{G}_{m-j}])^2 \Big] \leq
  Cn^{2b_2}\mathrm{e}^{-cj}, \quad m,j\in\N,n\in\N
\end{align}
for some $C,c\in(0,\infty)$ and analogous bounds for $D'_{n,m}$.
Indeed, abbreviating $\bE=\bE^\joint_{x_0,x'_0}$,
\begin{align*}
  \bE\Big[ (\bE[D_{n,m}\,\vert\, \mathcal{G}_{m-j}])^2 \Big]
  & = \bE\Big[ (\bE[\bE[D_{n,m}\,\vert\, W_{n,m}]\,\vert\, \mathcal{G}_{m-j}])^2 \Big]\\
  & =\bE\Bigg[ \Big(\bE\Big[\sum_{i=1}^4
    \indset{\{W_{n,m}=i\}}\bE[D_{n,m}\,\vert\, W_{n,m}=i]\,\vert\,
    \mathcal{G}_{m-j}\Big]\Big)^2 \Bigg]\\
  & = \sum_{\ell,i=1}^4 \bE[D_{n,m}\,\vert\, W_{n,m}=i]\bE[D_{n,m}\,\vert\, W_{n,m}=\ell]\\
  & \hspace{2cm}\times\bE\big[ \bP(W_{n,m}=i\,\vert\,
    \mathcal{G}_{m-j})\bP(W_{n,m}=\ell\,\vert\, \mathcal{G}_{m-j})
    \big]\\
  & \leq \sum_{\ell,i=1}^4 \vert\bE[D_{n,m}\,\vert\,
    W_{n,m}=i]\bE[D_{n,m}\,\vert\, W_{n,m}=\ell]\vert\\
  & \hspace{2cm}\times (\bP(W_{n,m}=i)+ \mathrm{e}^{-cj})(\bP(W_{n,m}=\ell)+ \mathrm{e}^{-cj})\\
  & = \sum_{\ell,i=1}^4\vert\bE[D_{n,m}\,\vert\,
    W_{n,m}=i]\bE[D_{n,m}\,\vert\, W_{n,m}=\ell]\vert
    \mathrm{e}^{-2cj}\\
  & \leq Cn^{2b_2}\mathrm{e}^{-2cj}
\end{align*}
Let $S_{n,m}\coloneqq \sum_{j=1}^{m}D_{n,j}$ and
$S'_{n,m}\coloneqq \sum_{j=1}^{m}D'_{n,j}$ then for each $n\in\N$,
$(S_{n,m})_m$ is a mixingale; see \cite{HallHeyde:1980}, p.~19.

Using McLeish's analogue of Doobs $\mathcal{L}^2$-inequality for
mixingales, we get
\begin{align}
  \bE^\joint_{x_0,x'_0}\Big[ \max_{m=1,\dots,n^{3/4}} S_{n,m}^2 \Big] \leq
  K\sum_{i=1}^{n^{3/4}} n^{2b_2} \leq
  Kn^{\frac{3}{4}+\frac{3}{2}b_2},
\end{align}
and thus
\begin{align}
  \label{eq:16}
  \bE^\joint_{x_0,x'_0}\Bigg[ \frac{\abs{S_{n,I_n}}}{\sqrt{n}}\ind{I_n\le n^{3/4} } \Bigg]
  \le \frac{1}{\sqrt{n}}
  \Bigg(\bE^\joint_{x_0,x'_0}\Big[\max_{m=1,\dots,n^{3/4}} S_{n,m}^2 \Big] \Bigg)^{1/2}
  \le K^{1/2}n^{-\frac{1}{8}+\frac{3}{4}b_2}.
\end{align}
Note that $b_2$ comes from Corollary~\ref{cor:separation corollary
  d=1} and can be chosen smaller than $1/8$ which results in the right
hand side converging to zero for $n\to\infty$. By
\eqref{eq:estimatesPhi} we have 
\begin{align*}
  A^{(1)}_n
  & = \sum_{j=0}^{n-1} \phi_1(\widehat{X}_j,\widehat{X}'_j)\\
  & =\sum_{j=1}^{I_n} D_{n,j} + \sum_{j=1}^{I_{n+1}}
    \sum_{i=\mathcal{D}_{n,j}\wedge n}^{\mathcal{R}_{n,j}\wedge n}
    \phi_1(\widehat{X}_i,\widehat{X}'_i)\\
  & \leq \sum_{j=1}^{I_n} D_{n,j} +
    \sum_{j=1}^{I_{n+1}}\big((\mathcal{R}_{n,j}\wedge n)-
    (\mathcal{D}_{n,j}\wedge n)\big)\frac{C_1}{n^2}\\
  & \leq \sum_{j=1}^{I_n} D_{n,j} + \frac{C_1}{n},
\end{align*}
and so
\begin{align*}
  \bE^\joint_{x_0,x'_0}\Bigg[ \frac{\vert A^{(1)}_n \vert}{\sqrt{n}} \Bigg]
  & \leq \frac{c}{\sqrt{n}}+ \frac{1}{\sqrt{n}}\bE^\joint_{x_0,x'_0}\Big[ \vert
    S_{n,I_n}\vert \ind{I_n \le n^{3/4}}\Big]
    +\frac{1}{\sqrt{n}}\bE^\joint_{x_0,x'_0}\Big[\vert S_{n,I_n}\vert\ind{I_n >
    n^{3/4}}\Big]\\
  & \leq \frac{c}{\sqrt{n}}+ \frac{1}{\sqrt{n}}\bE^\joint_{x_0,x'_0}\Big[ \vert
    S_{n,I_n}\vert \ind{I_n \le n^{3/4}}\Big]
    +\frac{1}{\sqrt{n}}Cn^{b_2}\bE^\joint_{x_0,x'_0}\Big[I_n\ind{I_n > n^{3/4}}\Big] \\
  & \hspace{2cm}+ \frac{1}{\sqrt{n}}Cn\bP^\joint_{x_0,x'_0}(\exists j\le n\colon
    \mathcal{D}_{n,j}-\mathcal{R}_{n,j-1}\geq n^{b_2}).
\end{align*}
Using \eqref{eq:16}, \eqref{eq:17} and \eqref{eq:sep in d=1}
respectively on the last three terms on the right hand side yields
\eqref{eq:bound_for_A1_and_A2} for $A^{(1)}_n$ and analogous
calculations for $D'$ instead of $D$ yield
\eqref{eq:bound_for_A1_and_A2} for $A^{(2)}_n$.

\subsection{Coupling in $d=1$ and its consequences}
In $d=1$ we have to calculate more carefully since we aren't able to
prove a coupling lemma of similar effect. The problem arises from the
fact that the random walks spend roughly $\sqrt{N}$ time close to each
other which, on first glance, should make the above estimates unusable
for $d=1$. As it turns out, we can prove the above estimates. The
price we have to pay for that is tuning $\beta$ even larger, as will
become obvious from the proofs below. In this section $\beta$ will
always refer to the constant from
Proposition~\ref{prop:TVdistance-joint-ind-1step}, which can be made
arbitrarily large by tuning the parameters of the model correctly.

We start with the $d=1$-version of the coupling lemma
\begin{lemma}
  \label{lem:coupl_d=1}
  For any $\varepsilon>0$ and $\gamma>0$ there exists $\beta$ such
  that for all $x,x'\in\bZ$ we have
  \begin{equation}
    \bP_{x,x'}\Big( \text{at most }N^{1/2+\varepsilon} \text{
      uncoupled steps before }N \Big) \ge 1-N^{-\gamma(\varepsilon)}. 
  \end{equation}
\end{lemma}

\begin{proof}
  Recall the definitions of $R_N$, \eqref{eq:18}, $\mathcal{D}_{N,i},
  \mathcal{R}_{N,i}$, \eqref{eq:defn_MathcalDn} and
  \eqref{eq:defn_MathcalRn}, and $I_N = \max\{ i\colon
  \mathcal{R}_{N,i} \le N\}$. We use $R_N$ as the number of steps at
  which the random walks are too close to each other, and thus a
  coupling is not possible. Therefore we obtain
  \begin{align*}
    \bP_{x,x'}&\Big( \text{more than }N^{1/2+\varepsilon} \text{ uncoupled steps before }N \Big)\\
              &\le \bP_{x,x'}\Big( \text{more than }N^{1/2+\varepsilon} \text{ uncoupled steps before }N,R_N \le N^{1/2+\varepsilon} \Big)+ \bP_{x,x'} \Big( R_N > N^{1/2+\varepsilon} \Big)\\
              &\le \bP_{x,x'} \Big( \exists \text{ uncoupled step from a starting distance } N^a  \Big)+\bP_{x,x'} \Big( R_N > N^{1/2+\varepsilon} \Big)\\
              &\le N^{1-\beta a} +\bP_{x,x'} \Big( R_N > N^{1/2+\varepsilon} \Big).
  \end{align*}
  Where $1-\beta a <-\gamma$ for large enough $\beta$. Now we make use
  of properties we already obtained for the recalled random variables.
  Writing $\bP$ for $\bP_{x,x'}$ and noting that $R_N \le
  \sum_{j=1}^{I_N+1} \mathcal{D}_{N,j}-\mathcal{R}_{N,j-1}$
  \begin{align*}
    \bP\Big( R_N>N^{1/2+\varepsilon} \Big)
    &\le \bP\Big( I_N> N^{1/2+\varepsilon/2} \Big) + \bP\Big( R_N>N^{1/2+\varepsilon}, I_N\le N^{1/2+\varepsilon/2} \Big)\\
    &\le \exp\Big(-cN^{1+\varepsilon}/N \Big) + \bP\Big( \sum_{j=1}^{I_N+1} \mathcal{D}_{N,j}-\mathcal{R}_{N,j-1}>N^{1/2+\varepsilon}, I_N \le N^{1/2+\varepsilon/2} \Big)
  \end{align*}
  where we used \eqref{eq:22}. For the second term we get
  \begin{align*}
    \bP&\Big( \sum_{j=1}^{I_N+1} \mathcal{D}_{N,j}-\mathcal{R}_{N,j-1}>N^{1/2+\varepsilon}, I_N \le N^{1/2+\varepsilon/2} \Big)\\
       &\le \bP\Big( \exists j\le N^{1/2+\varepsilon/2}\colon \mathcal{D}_{N,j}-\mathcal{R}_{N,j-1}>N^{\varepsilon/2} \Big)\\
       &\le N^{1/2+\varepsilon/2}\exp\Big( -CN^c \Big)
  \end{align*}
  where we used the Separation lemma for $d=1$, see
  Lemma~\ref{lem:separ in d=1}. Combine the last two displays and we
  obtain the claim. Furthermore, since
  $(\mathcal{D}_{N,j}-\mathcal{R}_{N,j-1})_j$ is and i.i.d.\ sequence
  and by \eqref{eq:tails_of_R_minus_D} we can ``shift'' the starting
  time. That means, for
  \begin{equation*}
    A_N(\varepsilon,m)\coloneqq \text{at most }N^{1/2+\varepsilon}\text{ uncoupled steps from time }m \text{ to time } N+m
  \end{equation*}
  we have
  \begin{equation}
    \label{eq:B1}
    \bP(A_N(\varepsilon,m)) \ge  1- N^{-\gamma(\varepsilon)}.
  \end{equation}
\end{proof}
Next we tackle the fluctuations of the sequence $(T^\joint_m)_m$.
\begin{lemma}
  For any $\varepsilon^*>0$ and $\gamma>0$ there exists $\beta$ large
  enough such that for all $x,x'\in\bZ$
  \begin{equation}
    \bP^\joint_{x,x'}\Big( \exists m\le N\colon \abs{T_m -
      m\bE[\tau^\indi_2]}>N^{1/2+\varepsilon^*} \Big) \le
    CN^{-\gamma}. 
  \end{equation}
\end{lemma}

\begin{proof}
  The steps are in spirit the same as for the $d\ge 2$ case in the
  proof of Lemma~\ref{lem:ind-to-joint}. The main difference is that
  we consider the event $A_N(\varepsilon) = \text{ at most }
  N^{1/2+\varepsilon}\text{ many uncoupled steps before } N$, with
  $\varepsilon<\varepsilon^*$ instead of $A_N$ since we have to allow
  for more uncoupled steps in $d=1$. Now following the same ideas as
  in the $d\ge 2$ case we obtain
  \begin{align*}
    \bP^\joint_{x,x'}
    &\Big( \exists m\le N\colon \abs{T_m - m\bE[\tau^\indi_2]}>N^{1/2+\varepsilon^*} \Big)\\
    &\le \bP_{x,x'} \Big( \exists m\le N\colon \abs{T^\joint_m - m\bE[\tau^\indi_2]}>N^{1/2+\varepsilon^*}, A_N(\varepsilon) \Big) + \bP_{x,x'}\Big( A^\compl_N(\varepsilon) \Big)\\
    &\le \bP_{x,x'} \Big( \exists m\le N\colon \abs{T^\indi_m - m\bE[\tau^\indi_2]}+\abs{T^\indi_m-T^\joint_m}>N^{1/2+\varepsilon^*}, A_N(\varepsilon) \Big) + N^{-\gamma}\\
    &\le \bP_{x,x'} \Big( \exists m\le N\colon \abs{T^\indi_m - m\bE[\tau^\indi_2]}>\frac{1}{2}N^{1/2+\varepsilon^*} \Big)\\
    &\hspace{2cm} + \bP_{x,x'} \Big( \exists m\le N\colon \abs{T^\indi_m - T^\joint_m}>\frac{1}{2}N^{1/2+\varepsilon^*},A_N(\varepsilon) \Big) + N^{-\gamma}\\
    &\le CN^{-\gamma} + \bP_{x,x'} \Big( \exists m\le N\colon \sum_{\substack{i\in \text{uncoupled steps}\\\text{before }m}}\abs{T^\indi_i - T^\indi_{i-1}}+\abs{T^\joint_i-T^\joint_{i-1}}>\frac{1}{2}N^{1/2+\varepsilon^*},A_N(\varepsilon) \Big)\\
    &= CN^{-\gamma} + \bP_{x,x'} \Big( \sum_{\substack{i\in \text{uncoupled steps}\\\text{before }N}}\abs{T^\indi_i - T^\indi_{i-1}}+\abs{T^\joint_i-T^\joint_{i-1}}>\frac{1}{2}N^{1/2+\varepsilon^*},A_N(\varepsilon) \Big)\\
    &\le CN^{-\gamma} + \bP_{x,x'} \Big( \exists i\in \text{uncoupled steps before }N\colon \abs{T^\indi_i - T^\indi_{i-1}}>\frac{1}{4}N^{1/2+\varepsilon^*-1/2-\varepsilon},A_N(\varepsilon) \Big)\\
    &\hspace{2cm}+\bP_{x,x'} \Big( \exists i\in \text{uncoupled steps before }N\colon \abs{T^\joint_i-T^\joint_{i-1}}>\frac{1}{4}N^{1/2+\varepsilon^*-1/2-\varepsilon},A_N(\varepsilon) \Big)\\
    &\le N^{-\gamma} + \sum_{i=1}^N \bP_{x,x'}\Big( \abs{T^\indi_i - T^\indi_{i-1}}>\frac{1}{4}N^{\varepsilon^*-\varepsilon} \Big) +\bP_{x,x'}\Big( \abs{T^\joint_i - T^\joint_{i-1}}>\frac{1}{4}N^{\varepsilon^*-\varepsilon} \Big)\\
    &\le N^{-\gamma} 2N^{2-\beta(\varepsilon^*-\varepsilon)}
  \end{align*}
  where the last line can be bound by $CN^{-\gamma}$ for large enough $\beta$.
\end{proof}

To make use of results for the random walk along the regeneration
times we need some tool to compare this random walk with the original.
To that end we define the last regeneration times before $N$ in both
dynamics, i.e.\ $V^\joint_N\coloneqq \max\{m\colon T^\joint_m \le N\}$
and analogous for $V^\indi_N$. The following lemma allows a comparison
of $V^\joint_N$ to a deterministic time
\begin{lemma}
  For any $\varepsilon>0$ and $\gamma>0$ there is $\beta$ large enough
  such that for all $x,x'\in\bZ$ we have 
  \begin{equation}
    \bP_{x,x'}\Big( \abs{V^\joint_N -N/\bE[\tau^\indi_2]}>N^{1/2+\varepsilon} \Big) \le N^{-\gamma}.
  \end{equation}
  Moreover, abbreviating $P_{x,x',\omega}=P_\omega$
  \begin{equation}
    P_\omega \Big( \abs{V^\joint_N -N/\bE[\tau^\indi_2]}>N^{1/2+\varepsilon} \Big) \to 0, \qquad \text{a.s.}
  \end{equation}
\end{lemma}

\begin{proof}
  We start by comparing $V^\joint_N$ to $V^\indi_N$. Let
  $A_N(\varepsilon/4)=\text{at most }N^{1/2+\varepsilon/4}\text{
    uncoupled steps before }N$ (again $\bP=\bP_{x,x'}$) 
  \begin{align*}
    \bP&\Big( \abs{V^\joint_N-V^\indi_N}> N^{1/2+\varepsilon/2} \Big)\\
       &\le \bP(A^\compl_N(\varepsilon/4)) +\bP\Big( \abs{V^\joint_N-V^\indi_N}> N^{1/2+\varepsilon/2}, A_N(\varepsilon/4) \Big)\\
       &\le CN^{-\gamma} + \bP\Big( \sum_{\substack{i\in
         \text{uncoupled steps}\\\text{before }N}} \abs{T^\joint_i -
    T^\joint_{i-1} - T^\indi_i + T^\indi_{i-1}}>
    N^{1/2+\varepsilon/2}, A_N(\varepsilon/4) \Big)\\ 
       &\le CN^{-\gamma} + \bP\Big( \exists i\in \text{uncoupled steps before }N\colon\abs{T^\joint_i - T^\joint_{i-1} - T^\indi_i + T^\indi_{i-1}}> N^{\varepsilon/4},A_N(\varepsilon/4) \Big)\\
       &\le CN^{-\gamma} + \sum_{i=1}^N \bP\Big( \abs{T^\joint_i-T^\joint_{i-1}}>\frac{1}{2}N^{\varepsilon/4} \Big) + \bP\Big( \abs{T^\indi_i-T^\indi_{i-1}}>\frac{1}{2}N^{\varepsilon/4} \Big)\\
       &\le CN^{-\gamma} + CN^{1-\beta\varepsilon/4},
  \end{align*}
  where the last line can be bound by $CN^{-\gamma}$ for $\beta$ large enough.
  \begin{align*}
    \bP&\Big( \abs{V^\joint_N -N/\bE[\tau^\indi_2]}>N^{1/2+\varepsilon} \Big)\\
       &\le \bP\Big( \abs{V^\joint_N -V^\indi_N}>N^{1/2+\varepsilon} \Big)+ \bP\Big( \abs{V^\indi_N -N/\bE[\tau^\indi_2]}>N^{1/2+\varepsilon} \Big)\\
       &\le CN^{-\gamma},
  \end{align*}
  by the CLT for renewal processes and the above estimate.
  Furthermore,
  \begin{align*}
    \bP\Big( \{P_\omega\big( \abs{V^\joint_N -
    N/\bE[\tau^\indi_2]}>N^{1/2+\varepsilon} \big) > N^{-a} \}\Big)
    \le \frac{1}{N^{-a}} \bP\Big( \abs{V^\joint_N -
    N/\bE[\tau^\indi_2]}>N^{1/2+\varepsilon}\Big) \le CN^{-\gamma+a}, 
  \end{align*}
  which can be made summable since we can tune $\gamma$ arbitrarily large. Now the claim follows by the Borel--Cantelli lemma.
\end{proof}

Another important step is to control the fluctuations in $(\widehat{X}^\joint_n)_n$.
\begin{lemma}
  For any $\varepsilon>0$ and $\gamma>0$ there exists $\beta$ such that for all $x,x'\in\bZ$ and all $0<\theta<1$
  \begin{equation}
    \bP^\joint_{x,x'}\Big( \sup_{\abs{k-[\theta N]}\le N^\delta}
    \abs{\widehat{X}_k-\widehat{X}_{[\theta N]}}>N^{\delta/2 +
      \varepsilon}   \Big) \le CN^{-\gamma}. 
  \end{equation}
\end{lemma}

\begin{proof}
  Let $A_N(\varepsilon/2,m) = \text{at most }N^{1/2+\varepsilon/2}
  \text{ uncoupled steps from time }m \text{ to } m+N$. Note that this
  event was already introduced in the text above \eqref{eq:B1}.
  \begin{align*}
    \bP
    &\Big( \sup_{\abs{k-[\theta N]}\le N^\delta} \abs{\widehat{X}^\joint_k-\widehat{X}^\joint_{[\theta N]}}>N^{\delta/2 + \varepsilon} \Big)\\
    &\le \bP\Big( \sup_{\abs{k-[\theta N]}\le N^\delta} \abs{\widehat{X}^\indi_k-\widehat{X}^\indi_{[\theta N]}}>N^{\delta/2 + \varepsilon} \Big)\\
    &\hspace{1cm}+ \bP\Big( \sup_{\abs{k-[\theta N]}\le N^\delta} \abs{\sum_{i=\min\{k,[\theta N]\}+1}^{\max\{k,[\theta N]\}}\widehat{X}^\joint_i-\widehat{X}^\joint_{i-1}-\widehat{X}^\indi_i+\widehat{X}^\indi_{i-1}}>N^{\delta/2 + \varepsilon}, A_{2N^{\delta}(\varepsilon/2,[\theta N]-N^\delta)} \Big)\\
    &\hspace{2cm}+ CN^{-\gamma}\\
    &\le CN^{-\gamma} + \bP\Big( \sup_{\abs{k-[\theta N]}\le N^\delta} \sum_{\substack{i \in \text{uncoupled steps}\\\text{from }\min \text{ to } \max}}\abs{\widehat{X}^\joint_i-\widehat{X}^\joint_{i-1}}+\abs{\widehat{X}^\indi_i-\widehat{X}^\indi_{i-1}}>N^{\delta/2 + \varepsilon}, A_{2N^{\delta}(\varepsilon/2,[\theta N]-N^\delta)} \Big)\\
    &\le CN^{-\gamma} + \bP\Big( \sum_{i \in \text{uncoupled}\cap\{ [\theta N]-N^\delta,\dots,[\theta N]+N^\delta\}}\abs{\widehat{X}^\joint_i-\widehat{X}^\joint_{i-1}}+\abs{\widehat{X}^\indi_i-\widehat{X}^\indi_{i-1}}>N^{\varepsilon/2}, A_{2N^{\delta}(\varepsilon/2,[\theta N]-N^\delta)} \Big)\\
    &\le CN^{-\gamma} + CN^{\delta-\beta\varepsilon/2},
  \end{align*}
  which, again, can be bound by $CN^{-\gamma}$ for $\beta$ large enough.
\end{proof}

Furthermore we obtain the following lemma.
\begin{lemma}
  For any $\varepsilon>0$ and $\gamma>0$ there exists $\beta$ such that for all $x,x'\in\bZ$
  \begin{equation}
    \bP^\joint_{x,x'}\Big( \abs{X_N-\widehat{X}_{N/\bE[\tau^\indi_2]}}>N^{1/4+\varepsilon} \Big) \le CN^{-\gamma}.
  \end{equation}
\end{lemma}
\begin{proof}
  As in the case $d\ge 2$, note that
  \begin{equation*}
    \abs{X_N-\widehat{X}^\joint_{N/\bE[\tau^\indi_2]}} \le R_\kappa
    \max\{T^\joint_m -T^\joint_{m-1}\colon m\le N \} +
    \abs{\widehat{X}^\joint_{N/\bE[\tau^\indi_2]}-\widehat{X}^\joint_{V_N}}. 
  \end{equation*}
  On $\{ \abs{V^\joint_N -N/\bE[\tau^\indi_2] }\le N^{1/2+\varepsilon}
  \}$, which is likely by the above estimates, we can bound the second
  term by $\sup_{\abs{k-N/\bE[\tau^\indi_2]}\le N^{1/2+\varepsilon}}
  \abs{\widehat{X}_k-\widehat{X}_{N/\bE[\tau^\indi_2]}}$ and use the
  last lemma. The upper bound for the first term is an easy
  consequence of the tail bounds of $T^\joint_k-T^\joint_{k-1}$.
\end{proof}
Now we have all ingredients to prove
\begin{align*}
  \bE[E_\omega[f(\widehat{X}^\joint_m/\sqrt{m})]E_\omega[f(\widehat{X}^{',\joint}_m/\sqrt{m})]]=
  \bE^\joint[f(\widehat{X}^\joint_m/\sqrt{m})f(\widehat{X}^{',\joint}_m/\sqrt{m})
  ] + O(m^{-1/4+\varepsilon}). 
\end{align*}
From which we now can prove Proposition~\ref{prop:2nd-mom} for $d=1$.

\section{Proofs for Section~\ref{sec:abstract_proofs}}
\label{sec:abst_detailed_proofs}

\subsection{Proof of Lemma~\ref{lem:a priori bound for tilde X}}

\begin{proof}[Proof of Lemma~\ref{lem:a priori bound for tilde X}]
  The proof follows the ideas from the proof of Lemma~2.16 from
  \cite{BirknerCernyDepperschmidt2016}. The main difference is that we
  consider the paths for $\wt{X}$ and $X$ together to be able to apply
  Assumption~\ref{ass:abstract_closeness_to_symmetric_transition_kernel}.
  Let $\wt{\Gamma}_{\tilde{n}}$ be the set of $\tilde{n}$-step paths
  $\wt{\gamma}$ satisfying $\norm{\wt{\gamma}_i-\wt{\gamma}_{i-1}}\le
  R_XL_{\mathrm{t}}/L_{\mathrm{s}}$. Note that $\wt{\gamma}$ is a path
  of $\wt{X}$ on the coarse-grained grid. Furthermore, let
  \begin{align*}
    D^{\wt{\gamma}}_{i_1,\dots,i_k}
    & \coloneqq \Big\{ \wt{G}(\wt{\gamma}_\ell,-\ell)=0 \text{ for all
      } \ell\in\{i_1,\dots,i_k \}  \Big\},\\
    W^{\wt{\gamma}}_{i_1,\dots,i_k}
    & \coloneqq \Big\{ \wt{G}(\wt{\gamma}_\ell,-\ell)=1 \text{ for all
      } \ell\in[\tilde{n}]\setminus\{i_1,\dots,i_k \} \Big\}
  \end{align*}
  be the sets of sites where $\wt{X}$ finds $\wt{G}=0$, respectively
  $\wt{G}=1$, where $\wt{G}$ is the random field introduced in
  Lemma~\ref{lem:abstr-OCcoupl} in equation~\eqref{e:tildeG}. Let
  $H_{\tilde{n}} \coloneqq \#\{ 0\le i \le \tilde{n} \colon
  \wt{G}(\wt{X}_i,-i)=0 \}$. Let $\Gamma_{\tilde{n}}(\wt{\gamma})$ be
  the set of possible positions for
  $(X_0,X_{L_{\mathrm{t}}},\dots,X_{\tilde{n}L_{\mathrm{t}}})$ that
  can have resulted in the path $\wt{\gamma}$ for $\wt{X}$. Set
  $K=\varepsilon_{\mathrm{symm}} + \max \kappa_{L_{\mathrm{t}}}$ and
  abbreviate $\bP=\bP_0$
  \begin{align*}
    &\bP(H_{\tilde{n}}=k)\\
    &= \sum_{0\le i_1<\dots<i_k\le \tilde{n}} \sum_{\wt{\gamma} \in
      \wt{\Gamma}_{\tilde{n}}} \bP\Big((\wt{X}_0,\wt{X}_1,\dots,\wt{X}_{\tilde{n}})
      = \wt{\gamma},D^{\wt{\gamma}}_{i_1,\dots,i_k},W^{\wt{\gamma}}_{i_1,\dots,i_k} \Big)\\
    & = \sum_{0\le i_1<\dots<i_k\le \tilde{n}} \sum_{\wt{\gamma} \in
      \wt{\Gamma}_{\tilde{n}}} \sum_{\gamma\in\Gamma_{\tilde{n}}({\wt{\gamma}})}
      \bP\Big((\wt{X}_0,\wt{X}_1,\dots,\wt{X}_{\tilde{n}})
      =\wt{\gamma},(X_0,X_{L_{\mathrm{t}}},\dots,X_{\tilde{n}L_{\mathrm{t}}})=\gamma
      ,D^{\wt{\gamma}}_{i_1,\dots,i_k},W^{\wt{\gamma}}_{i_1,\dots,i_k} \Big)\\
    & \le \sum_{0\le i_1<\dots<i_k\le \tilde{n}} \sum_{\wt{\gamma} \in
      \wt{\Gamma}_{\tilde{n}}}
      \sum_{\gamma\in\Gamma_{\tilde{n}}({\wt{\gamma}})}\bP\Big(
      (X_0,X_{L_{\mathrm{t}}},\dots,X_{\tilde{n}L_{\mathrm{t}}})=\gamma
      ,D^{\wt{\gamma}}_{i_1,\dots,i_k},W^{\wt{\gamma}}_{i_1,\dots,i_k} \Big)\\
    & \le  \sum_{0\le i_1<\dots<i_k\le \tilde{n}} \sum_{\wt{\gamma} \in
      \wt{\Gamma}_{\tilde{n}}}
      \sum_{\gamma\in\Gamma_{\tilde{n}}({\wt{\gamma}})} K^{\tilde{n}-k}
      \bP\Big(D^{\wt{\gamma}}_{i_1,\dots,i_k} \Big).
  \end{align*}
  Now we use the fact that by Lemma~\ref{lem:abstr-OCcoupl} we can
  couple $\wt{G}$ with a contact process build on a Bernoulli-field
  with success probability at least $1-\varepsilon_{\wt{\omega}}$ and
  therefore we can use Lemma~2.11 from
  \cite{BirknerCernyDepperschmidt2016} to obtain
  \begin{align*}
    \bP\Big(D^{\wt{\gamma}}_{i_1,\dots,i_k} \Big) \le
    \varepsilon(1-\varepsilon_{\wt{\omega}})^k,
  \end{align*}
  where $\varepsilon(\cdot)$ is the function from Lemma~2.11 and we
  can see in its proof that $\lim_{p\to 1}\varepsilon(p)=0$. Note
  that, if $\varepsilon_U$ is small, we can choose
  $\varepsilon_{\wt{\omega}}$ from Lemma~\ref{lem:abstr-OCcoupl} small
  and therefore $\varepsilon(1-\varepsilon_{\wt{\omega}})$ can be made
  arbitrarily small. Combining the last two equations we conclude
  \begin{align*}
    \bP(H_{\tilde{n}}=k)
    & \le \sum_{0\le i_1<\dots<i_k\le \tilde{n}} \sum_{\wt{\gamma} \in
      \wt{\Gamma}_{\tilde{n}}}
      \sum_{\gamma\in\Gamma_{\tilde{n}}({\wt{\gamma}})}
      K^{\tilde{n}-k}\varepsilon(1-\varepsilon_{\wt{\omega}})^k\\
    & \le \sum_{0\le i_1<\dots<i_k\le \tilde{n}}
      \abs{\wt{\Gamma}_{\tilde{n}}} L_{\mathrm{s}}^{d\tilde{n}}
      K^{\tilde{n}-k}\varepsilon(1-\varepsilon_{\wt{\omega}})^k\\
    & \le \binom{\tilde{n}}{k} (R_XL_{\mathrm{t}})^{d\tilde{n}}K^{\tilde{n}-k}\varepsilon(1-\varepsilon_{\wt{\omega}})^k,
  \end{align*}
  where we used $\abs{\Gamma_{\tilde{n}}(\wt{\gamma})}\le
  L_{\mathrm{s}}^{d\tilde{n}}$ and $\abs{\wt{\Gamma}_{\tilde{n}}}\le
  (R_XL_{\mathrm{t}}/L_{\mathrm{s}})^{d\tilde{n}}$. Consequently we
  have for $\delta >0$
  \begin{align*}
    \bP(H_{\tilde{n}}\ge \delta \tilde{n})
    & \le \sum_{k=\lfloor \delta \tilde{n} \rfloor}^{\tilde{n}}
      \binom{\tilde{n}}{k} (R_XL_{\mathrm{t}})^{d\tilde{n}}K^{\tilde{n}-k}\varepsilon(1-\varepsilon_{\wt{\omega}})^k\\
    & \le 2^{\tilde{n}}(R_XL_{\mathrm{t}})^{d\tilde{n}} K^{\tilde{n}} \sum_{k=\lfloor
      \delta \tilde{n} \rfloor}^{\tilde{n}}
      K^{-k}\varepsilon(1-\varepsilon_{\wt{\omega}})^k\\
    & \le 2^{\tilde{n}}(R_XL_{\mathrm{t}})^{d\tilde{n}} K^{\tilde{n}}
      \frac{(\varepsilon(1-\varepsilon_{\wt{\omega}})/K)^{\delta \tilde{n}}}{1-\varepsilon(1-\varepsilon_{\wt{\omega}})/K}\\
    &\le c_1\mathrm{e}^{-c_2 \tilde{n}},
  \end{align*}
  since we have $\varepsilon(1-\varepsilon_{\wt{\omega}})$ small
  independent of $\delta$ and $\varepsilon_{\mathrm{symm}}$. The
  displacement at the start is at most of size $L_{\mathrm{s}}^d$ and
  therefore of constant order. Thus this doesn't effect the speed. Now
  we are in the setting of the proof of Lemma~2.16 again and the rest
  follows by the same arguments.
\end{proof}


\begin{thebibliography}{ACdCdH22}

\bibitem[ACdCdH22]{AvenaChinodaCostadenHollander2022}
Luca Avena, Yuki Chino, Conrado da~Costa, and Frank den Hollander, \emph{Random
  walk in cooling random environment: recurrence versus transience and mixed
  fluctuations}, Ann. Inst. Henri Poincar\'{e} Probab. Stat. \textbf{58}
  (2022), no.~2, 967--1009. \MR{4421615}

\bibitem[AdSV13]{AvenaDosSantosVoellering:2013}
Luca Avena, Renato~Soares dos Santos, and Florian V{\"o}llering,
  \emph{Transient random walk in symmetric exclusion: limit theorems and an
  {E}instein relation}, ALEA Lat. Am. J. Probab. Math. Stat. \textbf{10}
  (2013), no.~2, 693--709. \MR{3108811}

\bibitem[AJV18]{AvenaJaraVoellering:2014}
L.~Avena, M.~Jara, and F.~V\"{o}llering, \emph{Explicit {LDP} for a slowed {RW}
  driven by a symmetric exclusion process}, Probab. Theory Related Fields
  \textbf{171} (2018), no.~3-4, 865--915. \MR{3827224}

\bibitem[B{\v{C}}D16]{BirknerCernyDepperschmidt2016}
Matthias Birkner, Ji\v{r}\'{\i} {\v{C}}ern\'y, and Andrej Depperschmidt,
  \emph{Random walks in dynamic random environments and ancestry under local
  population regulation}, Electron. J. Probab. \textbf{21} (2016), 1--43.

\bibitem[B{\v{C}}DG13]{BirknerCernyDepperschmidtGantert2013}
Matthias Birkner, Ji{\v{r}}{\'{\i}} {\v{C}}ern{\'y}, Andrej Depperschmidt, and
  Nina Gantert, \emph{Directed random walk on the backbone of an oriented
  percolation cluster}, Electron. J. Probab. \textbf{18} (2013), no. 80, 35.
  \MR{3101646}

\bibitem[BD07]{BirknerDepperschmidt2007}
Matthias Birkner and Andrej Depperschmidt, \emph{Survival and complete
  convergence for a spatial branching system with local regulation}, Ann. Appl.
  Probab. \textbf{17} (2007), no.~5-6, 1777--1807. \MR{2358641 (2008j:60226)}

\bibitem[BH17]{BethuelsenHeydenreich2015}
Stein~Andreas Bethuelsen and Markus Heydenreich, \emph{Law of large numbers for
  random walks on attractive spin-flip dynamics}, Stochastic Process. Appl.
  \textbf{127} (2017), no.~7, 2346--2372. \MR{3652417}

\bibitem[BHdS{\etalchar{+}}19]{BlondelHilarioetal2019}
Oriane Blondel, Marcelo~R. Hil\'{a}rio, Renato~S. dos Santos, Vladas
  Sidoravicius, and Augusto Teixeira, \emph{Random walk on random walks: higher
  dimensions}, Electron. J. Probab. \textbf{24} (2019), Paper No. 80, 33.
  \MR{4003133}

\bibitem[BHdS{\etalchar{+}}20]{BlondelHilarioetal2020}
\bysame, \emph{Random walk on random walks: low densities}, Ann. Appl. Probab.
  \textbf{30} (2020), no.~4, 1614--1641. \MR{4132636}

\bibitem[Bur73]{Burkholder1973}
D.~L. Burkholder, \emph{{Distribution Function Inequalities for Martingales}},
  The Annals of Probability \textbf{1} (1973), no.~1, 19 -- 42.

\bibitem[CMS10]{CoxMaricSchinazi2010}
J.~Theodore Cox, Nevena Maric, and Rinaldo~B. Schinazi, \emph{Contact process
  in a wedge}, J. Stat. Phys. \textbf{139} (2010), no. 3, 506--517.

\bibitem[Dep08]{Depperschmidt08}
Andrej Depperschmidt, \emph{Survival, complete convergence and decay of
  correlations for a spatial branching system with local regulation}, Doctoral
  thesis, Technische Universität Berlin, Fakultät II - Mathematik und
  Naturwissenschaften, Berlin, 2008.

\bibitem[dHKS14]{HKS14}
Frank den Hollander, Harry Kesten, and Vladas Sidoravicius, \emph{Random walk
  in a high density dynamic random environment}, Indag. Math. (N.S.)
  \textbf{25} (2014), no.~4, 785--799. \MR{3217035}

\bibitem[DMR94]{DoukhanMassartRio1994}
Paul Doukhan, Pascal Massart, and Emmanuel Rio, \emph{The functional central
  limit theorem for strongly mixing processes}, Ann. Inst. H. Poincar\'{e}
  Probab. Statist. \textbf{30} (1994), no.~1, 63--82. \MR{1262892}

\bibitem[Fel71]{Feller1971}
William Feller, \emph{An introduction to probability theory and its
  applications. {V}ol. {II}.}, Second edition, John Wiley \& Sons Inc., New
  York, 1971. \MR{0270403 (42 \#5292)}

\bibitem[FF63]{MR0158519}
D.~K. Faddeev and V.~N. Faddeeva, \emph{Computational methods of linear
  algebra}, Translated by Robert C. Williams, W. H. Freeman and Co., San
  Francisco, 1963. \MR{0158519 (28 \#1742)}

\bibitem[Gut88]{Gut1988}
Allan Gut, \emph{Stopped random walks}, Applied Probability. A Series of the
  Applied Probability Trust, vol.~5, Springer-Verlag, New York, 1988, Limit
  theorems and applications. \MR{916870}

\bibitem[HH80]{HallHeyde:1980}
P.~Hall and C.~C. Heyde, \emph{Martingale limit theory and its application},
  Academic Press Inc. [Harcourt Brace Jovanovich Publishers], New York, 1980,
  Probability and Mathematical Statistics.

\bibitem[JM20]{MiltonMenezes2020}
Milton Jara and Ot\'{a}vio Menezes, \emph{Symmetric exclusion as a random
  environment: invariance principle}, Ann. Probab. \textbf{48} (2020), no.~6,
  3124--3149. \MR{4164462}

\bibitem[Kal21]{Kallenberg20213rdEdition}
Olav Kallenberg, \emph{Foundations of modern probability}, third ed.,
  Probability Theory and Stochastic Modelling, vol.~99, Springer, Cham, 2021.
  \MR{4226142}

\bibitem[MV15]{MV15}
Thomas Mountford and Maria~E. Vares, \emph{Random walks generated by
  equilibrium contact processes}, Electron. J. Probab. \textbf{20} (2015), no.
  3, 17. \MR{3311216}

\bibitem[Nag82]{Nagaev82}
Sergey~V. Nagaev, \emph{On the asymptotic behavior of one-sided large deviation
  probabilities}, Theory Probab. Appl. \textbf{26} (1982), 362--366 (English).

\bibitem[OY71]{OodairaYoshihara1971}
Hiroshi Oodaira and Ken-ichi Yoshihara, \emph{The law of the iterated logarithm
  for stationary processes satisfying mixing conditions}, Kodai Math. Sem. Rep.
  \textbf{23} (1971), 311--334. \MR{307311}

\bibitem[Per02]{Perkins1999}
Edwin Perkins, \emph{Dawson-{W}atanabe superprocesses and measure-valued
  diffusions}, Lectures on probability theory and statistics ({S}aint-{F}lour,
  1999), Lecture Notes in Math., vol. 1781, Springer, Berlin, 2002,
  pp.~125--324. \MR{1915445}

\bibitem[Rac95]{Rackauskas1995}
Alfredas Rackauskas, \emph{On the rate of convergence in the martingale {CLT}},
  Statistics \& Probability Letters \textbf{23} (1995), no.~3, 221--226.

\bibitem[RV13]{RedigVoellering:2013}
Frank Redig and Florian V{\"o}llering, \emph{Random walks in dynamic random
  environments: a transference principle}, Ann. Probab. \textbf{41} (2013),
  no.~5, 3157--3180. \MR{3127878}

\bibitem[Ste17]{Steiber2017}
Sebastian Steiber, \emph{Ancestral lineages in the contact process: scaling and
  hitting properties}, Ph.D. thesis, University of Mainz, 2017.

\end{thebibliography}

\newcommand{\etalchar}[1]{$^{#1}$}
\providecommand{\bysame}{\leavevmode\hbox to3em{\hrulefill}\thinspace}
\providecommand{\MR}{\relax\ifhmode\unskip\space\fi MR }
\providecommand{\MRhref}[2]{%
  \href{http://www.ams.org/mathscinet-getitem?mr=#1}{#2}
}
\providecommand{\href}[2]{#2}

\end{document}